\documentclass[reqno]{amsart}
\usepackage{amsmath,amssymb}
\usepackage{amsmath,amssymb,cite}
\usepackage{graphics,epsfig,color}
\usepackage[mathscr]{euscript}
\usepackage{stmaryrd}
\usepackage{xcolor}
\usepackage{enumitem}
\usepackage{a4wide}
\usepackage{tikz}
\usepackage{pgfplots}
\theoremstyle{plain}
\newtheorem{theorem}{Theorem}[section]
\newtheorem{proposition}[theorem]{Proposition}
\newtheorem{lemma}[theorem]{Lemma}
\newtheorem{corollary}[theorem]{Corollary}

\theoremstyle{definition}
\newtheorem{definition}[theorem]{Definition}

\newtheorem{assertion}[theorem]{Assertion}
\theoremstyle{remark}
\newtheorem{remark}[theorem]{Remark}

\makeatletter
\@addtoreset{equation}{section}
\makeatother

\newcommand{\mc}[1]{{\mathcal #1}}
\newcommand{\mb}[1]{{\mathbf #1}}
\newcommand{\mf}[1]{{\mathfrak #1}}

\newcommand{\bb}[1]{{\mathbb #1}}
\newcommand{\ms}[1]{{\mathscr #1}}
\newcommand{\mss}[1]{{\textsf #1}}
\newcommand{\mtt}[1]{{\mathtt #1}}
\newcommand{\mz}[1]{{\mathscr #1}}

\newcommand{\lb}{\llbracket}
\newcommand{\rb}{\rrbracket}

\newcommand{\<}{\langle}
\renewcommand{\>}{\rangle}

\newcommand{\cb}[1]{{\color{blue} #1}}

\title[$\Gamma-$expansion of non reversible one-dimensional
diffusions]{Full $\Gamma-$expansion for the level-two large deviation
rate functionals of non-reversible one-dimensional diffusions with
periodic boundary conditions}

\author{Claudio Landim, and Christian Maura}
\address{Claudio Landim
  \hfill\break\indent IMPA \hfill\break\indent Estrada Dona Castorina
  110, \hfill\break\indent
J. Botanico, 22460 Rio de Janeiro, Brazil\hfill\break\indent
{\normalfont and} \hfill\break\indent
Univ. Rouen Normandie, \hfill\break\indent
CNRS, LMRS UMR 6085, \hfill\break\indent
F-76000 Rouen, France.} 
\email{landim@impa.br}

\address{Christian Maura
  \hfill\break\indent IMPA \hfill\break\indent Estrada Dona Castorina
  110, \hfill\break\indent
J. Botanico, 22460 Rio de Janeiro, Brazil} 
\email{christian.maura@impa.br}

\begin{document}

\begin{abstract}
Consider the diffusion process 
\begin{equation*}
dX_{\epsilon}(t) = \mss b(X_{\epsilon}(t)) \, dt
+ \sqrt{2\, \epsilon\, \mss a(X_\epsilon(t))} \, dW_{t},
\end{equation*}
on the one-dimensional torus $\bb T = [0,1)$.  Here $\epsilon$ is the
temperature, $W_{t}$ a Brownian motion on $\bb T$
and $\mss a$, $\mss b$ functions of class $C^{2}(\bb T)$ satisfying
further conditions. Denote by $\mss P(\bb T)$ the set of probability
measures on $\bb T$  equipped with the weak topology, and by
$\ms I_{\epsilon}\colon \mss P(\bb T)\to [0,+\infty)$ the level two large
deviation rate functional of the diffusion $X_{\epsilon}(\cdot)$. We
derive a full $\Gamma-$expansion of $\ms I_{\epsilon}$, as
$\epsilon \to 0$, expressing it as
\begin{equation*}
\ms I_{\epsilon} = \frac{1}{\epsilon} \;\ms J^{(-1)} \; +\; \ms J^{(0)} \;+\;
\sum_{p=1}^{\widehat{\mf q}}\frac{1}{\theta^{(p)}_{\epsilon}}\;\ms J^{(p)}\,,
\end{equation*}
where $\ms J^{(-1)}$, $\ms J^{(0)}$,
$\ms J^{(p)} \colon \mss P(\bb T)\to [0,+\infty]$ represent rate
functionals, independent of $\epsilon$, and $\theta^{(p)}_{\epsilon}$ are
the time-scales at which the Markov process $X_{\epsilon}(\cdot)$
exhibits a metastable behaviour.
\end{abstract}

\noindent


\maketitle

\section{Introduction}
\label{secI}

The metastable behavior of Markov processes has attracted considerable
attention in recent years \cite{OV05, BH15, l-review}. In \cite{bl10,
bl12}, the authors proposed characterizing such processes as those
that can be effectively approximated by simpler Markov chains,
typically defined on finite state spaces \cite{b25}. This approach,
based on analyzing the trace of the process within the wells, has
since been successfully applied to a wide range of models in
statistical mechanics and diffusion processes \cite{bl12b, bdg17,
ls18, ks21, ks22}. This line of research culminated in
\cite{lms25}, which established that a Markov process is metastable in
the above sense if, and only if, the solutions to the associated
resolvent equations are asymptotically constant within the wells.
This property was proved for non-reversible diffusion processes
\cite{ls22, lls24, lls25, lm}, for zero range dynamics \cite{or18,
lms23}, and for  symmetric inclusion dynamics \cite{kl26}.

Let $x_\epsilon(\cdot)$ be the diffusion process 
\begin{equation*}
dx_\epsilon (t) \;=\; -\, \nabla U ( x_\epsilon (t) ) \, dt
\;+\; \sqrt{2\,\epsilon} \, dW(t) \;,
\end{equation*}
where $\epsilon$ stands for the temperature,
$U\colon \bb R^d \to \bb R$ a Morse potential, and $W(\cdot)$ a
$d$-dimensional Brownian motion. Denote by $I_\epsilon(\cdot)$ the
associated (Donsker and Varadhan) level two large deviations rate
functional \cite{dv75}.

Assuming the genericity condition that each well has a distinct depth,
Di Gesù and Mariani \cite{gm17} proved that the rate functional
$I_\epsilon (\cdot)$ converges, in the sense of $\Gamma$-convergence
-- the natural notion of convergence for rate functionals
\cite{mar}. More precisely, they showed that there exist rate
functionals $J_{-1}$, $J_0$, $J_p\colon \mss P(\bb R^d) \to \bb R_+$,
$1\le p < \mf q_0$ such that
\begin{equation}
\label{i01}
\Gamma-\lim_{\epsilon\to 0} \epsilon \,  I_\epsilon \,=\, J_{-1}\,, \quad
\Gamma-\lim_{\epsilon\to 0} I_\epsilon \,=\, J_0 \,, \quad
\Gamma-\lim_{\epsilon\to 0} \theta_\epsilon^{(p)}
I_\epsilon \,=\, J_{p}\, .
\end{equation}
Here, $\mss P(\bb R^d)$ denotes the space of probability measures on
$\bb R^d$ endowed with the weak topology; $\mf q_0$ is the number of
local minima of the potential $U(\cdot)$;
$\theta_\epsilon^{(p)} = e^{h_p/\epsilon}$, where $h_p$ represents the
depth of the $p$th shallowest well; and
$\Gamma-\lim_{\epsilon\to 0} a_\epsilon  I_\epsilon = J$ means that
the sequence of rate functionals $a_\epsilon I_\epsilon$ $\Gamma$-
converges to $J$.

Note that $a_\epsilon I_\epsilon$ represents the level two large
deviations rate functional associated with the diffusion process
$x_\epsilon(\cdot)$ accelerated by the factor $a_\epsilon$, that is,
the process $x_\epsilon(\cdot)$ observed on the time-scale
$a_\epsilon$. The preceding result therefore characterizes the
asymptotic behavior of the rate functional across $\mf q+1$ distinct
time-scales.

The time-scale $1$ corresponds to the regime in which the diffusion
escapes from unstable stationary points (such as saddle points or
local maxima of the potential $U(\cdot)$, for instance), while
$\theta_\epsilon^{(p)}$ represents the time-scale at which the process
$x_\epsilon (\cdot)$ starting from a well of depth $h_p$ transitions
to a deeper well.

The assumption that each well has a distinct depth was removed in
\cite{llm25}, following the works \cite{l23, bgl24}, which established
the $\Gamma$-convergence of the level-two large deviations rate
functional for finite-state Markov chains, and \cite{lms23b}, which
investigated the analogous problem for reversible random walks in a
potential field. These results were further extended in \cite{kl25} to
the large deviations rate functional of the pair empirical measure and
current, as originally derived by Bertini, Faggionato, and Gabrielli
in \cite{bfg14}.

While the proof in \cite{gm17} is purely analytical, relying on a
spectral analysis of the generator, the approaches developed in
\cite{l23, bgl24, lms23b, llm25} are probabilistic in nature and
exploit the metastable behavior of the underlying processes.

According to \cite[Corollary~5.3]{kl25}, the generator of a reversible
Markov chain on a finite state space can be reconstructed from its
large deviations rate functional. Consequently, the $\Gamma$-expansion
presented in~\eqref{i01} not only identifies the time scales at which
metastable behavior arises, but also determines the dynamics of the
reduced Markov chain, provided it is a reversible finite-state
process.

In this article, we investigate the $\Gamma$-expansion of the level
two large deviations rate-functional for non-reversible
one-dimensional diffusions with periodic boundary conditions.  The
main result, Theorem~\ref{mt3}, provides the full \(\Gamma\)-expansion
for this functional and states that:
\begin{itemize}
\item The shortest time scale, $\epsilon$, yields a functional
$\ms J^{(-1)}$ that penalizes measures not concentrated on critical
points of the drift.
\item The next scale, $1$,  yields $\ms J^{(0)}$ which
assigns positive cost to measures charging unstable critical points.
\item The metastable scales \(\theta_\epsilon^{(p)}\) produce
functionals $\ms J^{(p)}$ that are lifted versions of the level-two
rate functionals of finite-state reduced Markov chains
$\widehat{\mathbb{X}}_p$.
\end{itemize}

The proof combines analytic and probabilistic techniques. We use a
variational representation of the rate functional to derive lower
bounds, and we construct explicit recovery sequences using solutions
of resolvent equations and localized Gibbs measures to obtain upper
bounds. A key ingredient is the resolvent characterization of
metastability~\cite{lms25}, which allows us to relate the behavior of
the rate functional at scale $\theta_\epsilon^{(p)}$ to the
generator of the reduced chain $\widehat{\mathbb{X}}_p$.

The main novelty of our work is that we do not assume the process to
be reversible. To our knowledge, with the exception of \cite{l23, kl25},
where non-reversible finite-state Markov chains are considered, this is
the first proof of $\Gamma$-convergence for a non-reversible dynamics.

This non-reversibility introduces significant technical challenges:
while for reversible dynamics the large deviations rate functional
admits an explicit expression, in the non-reversible setting it is
defined only through a variational formulation. The absence of an
explicit formula explains why previous studies have generally
restricted attention to reversible dynamics.

Furthermore, the method to prove the $\Gamma-\limsup$ differs completely
in the reversible and non-reversible cases. In the reversible case,
see Section \ref{sec6} and \ref{sec2}, the recovery sequence of a
measure $\mu$ concentrated on a set of local minima $m$, is
constructed by fixing the values at $m$ and extending it harmonically
to the entire space. The procedure is entirely different in the
non-reversible case, and this is the main message of this article.  In
the non-reversible case, see Section \ref{sec8}, one has to proceed as
follows. Fix a measure $\mu$ concentrated on a set $\mc M$ of local
minima $m$. Find the tilting of the reduced Markov chain jump rates to
turn the measure $\mu$ typical for the reduced dynamics. Extend the
tilting harmonically to the entire space. Consider the diffusion
tilted by this extension and its associated stationary state. Prove
that this $\epsilon$-sequence of stationary states for the tilted
dynamics is a recovery sequence of the original measure. This is
carried out in Section \ref{sec8} in the last metastable
time-scale, the only one that gives rises to a non-reversible reduced
dynamics

The main challenge in the area of metastable diffusions is to
establish the metastable behavior of a diffusion process in dimension
$d\ge 2$ when the stationary distribution is not explicitly
known. Existing approaches to proving metastability crucially depend
on explicit computations involving the stationary state. In contrast,
the large deviations rate functional is available under weak
assumptions and does not require such explicit knowledge. This
suggests the possibility of developing analytical methods to study the
asymptotic behavior of this functional, thereby recovering the
metastable dynamics of the system.

In this direction, it may be particularly fruitful to investigate the
asymptotic behavior of the joint current–empirical measure large
deviations rate functional, which has the notable advantage of being
explicit even in the non-reversible case, unlike the level-$2$ large
deviations rate functional, which is defined only through a
variational characterization.

\section{Notation and Results}
\label{secNR}

In this section, we introduce the model, and state the main results.

\subsection*{\bf The diffusion process.}

Let $\bb T = [0,1)$ be the one-dimensional torus of length $1$.
Consider the one-dimensional diffusion processes with periodic
boundary conditions,
\begin{equation}
\label{03a}
dX_\epsilon (t) \;=\; \mss b (X_\epsilon (t))\, dt
\;+\; \sqrt{2\,\epsilon \, \mss a(X_\epsilon(t))} \,
dW_t\;,
\end{equation}
where $\epsilon>0$ is the temperature, and $W_t$ a Brownian motion on
the torus $\bb T$. Throughout this article, we assume that the
following conditions are fulfilled. Denote by $\cb{C(\bb T)}$ the
space of continuous functions $F\colon \bb T\to \bb R$, and by
$\cb{C^p(\bb T)}$, $p\ge 1$, the elements of $C(\bb T)$ with $p$
continuous derivatives.

\smallskip\noindent {\bf Condition H:} The drift
$\mss b\colon \bb T \to \bb R$ is of class $C^2(\bb T)$. It vanishes
at a finite number of points represented by $x_j$, $1\le j\le 2N$, and
$\mss b'(x_j) \neq 0$. The diffusion coefficient
$\mss a\colon \bb T \to \bb R$ is of class $C^2(\bb T)$, and is
strictly positive: $\mss a(x) \geq c_0 > 0$.  \smallskip

The generator of the diffusion process \eqref{03a}, denoted by
$\ms L_\epsilon$, is given by
\begin{equation*}
{\color{blue}\ms L_\epsilon f} \;=\; \mss b f'
\;+\; \epsilon \, \mss a\, f''\;, \quad f\in C^2(\bb T)\,.
\end{equation*}

Denote by $B$ the integral of the quotient $\mss b\,/\, \mss a$ over
the torus: 
\begin{equation}
\label{04}
{\color{blue}B}\;:=\; \int_{\bb T} \frac{\mss b(\theta)}{\mss
a(\theta)}\, d\theta \;.
\end{equation}
If $B=0$, there exists a potential $U\colon\bb T\to \bb R$ such that
$\mss b(\theta) \,/\, \mss a(\theta) = -\, U'(\theta)$.  In this case
the stationary state of the diffusion process  is given by
\begin{equation*}
\pi _\epsilon(d\theta) = \frac{1}{Z_\epsilon}
\frac{1}{\mss a(\theta)}\exp\{- U(\theta)/\epsilon\} \, d\theta\,,
\end{equation*}
for a suitable normalization factor $Z_\epsilon$, and the process
is reversible with respect to this measure.

In contrast, if the integral does not vanish the process is not
reversible, and the stationary state is not a Gibbs measure.
Faggionato and Gabrielli derived  the quasi-potential for this
dynamics in \cite{fg12}. The metastable behavior among the deepest
traps has been investigated in \cite{ls19}

Without loss of generality, we assume, from now on, that $B\ge 0$.
Our analysis includes therefore the case $B\neq 0$, where the process
$X_\epsilon(t)$ is non-reversible.

\subsection*{Level $2$ large deviations}

Denote by $\color{blue} \mss P(\bb T)$ the set of probability measures
on $\bb T$ equipped with the weak topology, and by $L_\epsilon (t)$ the
empirical measure of the process $X_\epsilon(t)$ defined as:
\begin{equation*}
{ \color{blue} L_\epsilon (t) }
\;:=\; \frac{1}{t} \int_0^t \delta_{\{X_\epsilon(s)\}}\; ds \;, \quad
t>0\,, 
\end{equation*}
where $\color{blue} \delta_{\{x\}}$, $x\in \bb T$, represents the
Dirac measure supported at $\{x\}$. Thus, $L_\epsilon (t)$ is a random
element of $\mss P(\bb T)$ and $L_\epsilon (t) (\mc A)$,
$\mc A \subset \bb T$, stands for the proportion of time the process
$X_\epsilon(t)$ stays at $\mc A$ in the time interval $[0,t]$.

As the Markov chain $X_\epsilon(t)$ is irreducible, by the ergodic
theorem, for any starting point $x\in \bb T$, as $t\to\infty$, the
empirical measure $L_\epsilon (t)$ converges in probability to the
stationary state, denoted by  $\color{blue} \pi_\epsilon$.

Donsker and Varadhan \cite{dv75} proved the associated large
deviations principle. More precisely, they showed that for any subset
$A$ of $\mss P(\bb T)$,
\begin{equation*}
\begin{aligned}
-\, \inf_{\mu\in A^o} \ms I_\epsilon (\mu) \;&\le\; \liminf_{t\to \infty}
\min_{x\in \bb T}\, \frac{1}{t}\,
\ln \bb P^\epsilon_x \big[ \, L_\epsilon (t) \in A\,\big] \\
&\qquad \;\le\; \limsup_{t \to \infty} \max_{x\in \bb T}\,
\frac{1}{t}\, \ln \bb P^\epsilon_x \big[ \, L_\epsilon (t) \in A\,\big]
\;\le\; -\, \inf_{\mu\in \overline{A}} \ms I_\epsilon (\mu)\;.
\end{aligned}
\end{equation*}
In this formula, $\color{blue} \bb P^\epsilon_{\! x}$, $x\in \bb T$,
stands for the distribution of the process $X_\epsilon(t)$ starting
from $x$, $A^o$, $\overline{A}$ represent the interior, closure of the
set $A$, respectively, and
$\ms I_\epsilon \colon \mss P(\bb T) \to [0,+\infty)$ is the level two
large deviations rate functional given by
\begin{equation*}
{\color{blue} \ms I_\epsilon (\mu)} \;:=\; \sup_u \,-\,
\int_{\bb T} \frac{(\ms L_\epsilon u)(x)}{u(x)} \; \mu(dx)\;,
\end{equation*}
where the supremum is performed over all strictly positive functions
$u\colon \bb T \to (0,\infty)$ in the domain of the generator.

\subsection*{$\Gamma-$expansion}

We investigate in this article the $\Gamma$-convergence of the action
functional $\ms I_\epsilon$.  We first recall the definition of a
$\Gamma$-expansion in a general setting, and refer to \cite{mar} for
further properties of $\Gamma$-convergence and its relation to large
deviations.

Fix a Polish space $\mc X$ and a sequence $(U_n : n\in\bb N)$ of
functionals on $\mc X$, $U_n\colon \mc X \to [0,+\infty]$.  The
sequence $U_n$ is said to \emph{$\Gamma$-converge} to the functional
$U\colon \mc X\to [0,+\infty]$ if the two following conditions are
met:

\begin{itemize}
\item [(i)]\emph{$\Gamma-\liminf$.} The functional $U$ is a
$\Gamma-\liminf$ for the sequence $U_n$: For each $x\in\mc X$ and each
sequence $x_n\to x$,
\begin{equation*}
\liminf_{n\to \infty} U_n(x_n) \;\ge\; U(x) \;.
\end{equation*}

\item [(ii)]\emph{$\Gamma-\limsup$.} The functional $U$ is a
$\Gamma-\limsup$ for the sequence $U_n$: For each $x\in\mc X$ there
exists a sequence $x_n\to x$ (called the recovery sequence) such that
\begin{equation*}
\label{30}
\limsup_{n\to \infty} U_n(x_n) \;\le\; U(x)\;.
\end{equation*}
\end{itemize}

Recall from the introduction that $\theta_\epsilon \, \ms I_\epsilon$
corresponds to the large deviations rate functional of the diffusion
process $X_\epsilon (\cdot)$ seen at the time-scale $\theta_\epsilon$,
that is, accelerated by this quantity. For this reason, if the
functional $U_n$ introduced above is multiplied by a constant $a_n$,
we interpret $a_n$ as a time-scale.

For two positive monotone sequences $(\alpha_n:n\geq 1)$ and
$(\beta_n:n\geq 1)$, denote by $\alpha_n\prec \beta_n$ if
$\alpha_n/\beta_n \to 0$ as $n\to +\infty$. In particular, we write
$1\prec\alpha_n$, if $\alpha_n\to +\infty$, and $\beta_n\prec 1$, if
$\beta_n\to 0$, as $n\to +\infty$.

In the context of the previous definition, suppose that there exists a
positive monotone sequence $(\alpha^{(0)}_n:n\geq 1)$ such that
$\alpha_n^{(0) \, }U_n$ $\Gamma-$converges to a functional $U^{(0)}$.
If the $0-$level set of the functional $U^{(0)}$ is a singleton, the
exploration over time-scales larger than $\alpha_n^{(0)}$
ends. Otherwise, we consider the $\Gamma-$convergence of the sequence
$\alpha_n^{(1)}U_n$, in a time-scale $(\alpha_n^{(1)}:n\geq 1)$
satisfying $\alpha_n^{(0)}\prec \alpha_n^{(1)}$, to distinguish the
elements in the $0$-level set of $U^{(0)}$.

Suppose, therefore, that the $0-$level set of the functional $U^{(0)}$ is
not a singleton, and that there exists a sequence $\alpha_n^{(1)}$
such that $\alpha_n^{(0)}\prec \alpha_n^{(1)}$, and
$\alpha_n^{(1)}U_n$ $\Gamma-$converges to some functional denoted by
$U^{(1)}$. It is clear that $U^{(1)}(x)$ is finite only if $x$ belongs
to the $0-$level set of $U^{(0)}$. We say that $\alpha_n^{(1)}$ is the
correct speed whenever $U^{(1)} (x)$ happens to be finite if $x$
belongs to the $0-$level set of $U^{(0)}$. If the $0-$level set of
$U^{(1)}$ is a singleton, we stop exploring longer
time-scales. Otherwise, we iterate the procedure. This recursion stops
at the first step $\mf q$ at which we obtain a functional
$U^{(\mf q)}:\mc X\to \bb R$ whose $0-$level set is a singleton.

We turn to the description of the iterative procedure over shorter
time-scales. If for every positive sequence $(\alpha_n:n\geq 1)$
satisfying $\alpha_n \prec \alpha_n^{(0)}$, the $\Gamma$-limit of
$\alpha_n U_n$ is $0$, this procedure stops. Otherwise, we seek for a
sequence $\alpha_n^{(-1)}$ and a functional $U^{(-1)}:\mc X\to \bb R$
such that $\alpha_n^{(-1)}\prec \alpha_n^{(0)}$, $\alpha_n^{(-1)}U_n$
$\Gamma-$converges to $U^{(-1)}$, and $U^{(-1)}(x) = 0$ if, and only if,
$U(x) < +\infty$.

If $U^{(-1)}$ is such that $\alpha_nU^{(-1)}\to 0$, for every sequence
$(\alpha_n:n\geq 1)$ such that $\alpha_n \prec \alpha_n^{(-1)}$, we
stop exploring shorter time-scales. Otherwise, we iterate the
procedure. This recursion continues until the first step $\mf r$ at
which we find a time-scale $\alpha_n^{(-\mf r)}$ and a functional
$U^{(-\mf r)}:\mc X\to \bb R$, not equal to $0$, such that
$\alpha_n^{-\mf r}U_n$ $\Gamma-$converges to $U^{-(\mf r)}$, and
$\alpha_nU_n$ $\Gamma$-converges to $0$, for every
$\alpha_n \prec \alpha_n^{-(\mf r)}$.

By the end of these two procedures we obtain a family of speeds and
functionals that captures the complete asymptotics of $U_n$. These
considerations are summarized in the following

\begin{definition}\label{gamma_expansion_def}
Consider a non empty Polish space $\mc X$ and a sequence of
functionals $U_n: \mc X \to [0,+\infty]$. A full $\Gamma-$expansion of
$U_n$ is given by the speeds $(\alpha^{(p)}_n, n\geq 1)$,
$-\mf r \leq p \leq \mf q$, and the functionals
$U^{(p)}: \mc X \to [0,+\infty]$, $-\mf r \leq p \leq \mf q$, if:

\begin{enumerate}

\item The speeds $\alpha^{(-\mf r)}_n, \dots, \alpha^{(\mf q)}_n$ are
positive sequences such that $\alpha^{(p)}_n\prec \alpha^{(p+1)}_n $,
for every $-\mf r \leq p < \mf q$;

\item $\alpha_{\epsilon}^{(p)} U_n$ $\Gamma-$converges to $U^{(p)}$,
$-\mf r \leq p \leq \mf q$;

\item $U^{(p+1)}(x)$ is finite if, and only if, $x$ belongs to the
$0-$level set of $U^{(p)}$, for every $-\mf r \leq p < \mf q$;

\item $\beta_n U_n$ $\Gamma-$converges to $0$, for every sequence
$\beta_n \prec \alpha_n^{(-\mf r)}$; and

\item the $0-$level set of $U^{(\mf q)}$ is a singleton.
\end{enumerate}

In this case, we express the expansion as 
\begin{equation*}
U_n = \sum_{p\in -\mf r}^{\mf q}\frac{1}{\alpha_{n}^{(p)}} \, U^{(p)}.
\end{equation*}
\end{definition}

In this article, we derive the full $\Gamma-$expansion of the level
two large deviation rate functional $\ms I_\epsilon$. We establish the
connection of the $\Gamma$-expansion to the metastable behavior of the
diffusion $X_\epsilon(\cdot)$.

\subsection*{Pre-metastable time-scales}

Denote by $\mc C\subset \bb T$ the set of the critical points of the ODE
\begin{equation}
\label{crit_R}
\dot x(t) \,=\,  \mss b(x(t)) \,, \quad 
{\color{blue} \mc C} \,:=\,  \{\theta \in \bb T: \mss b(\theta) = 0\}\;.
\end{equation}
Let $\mc M$, $\mc W$ be the sets of stable and unstable critical
points, respectively,
\begin{equation*}
\label{mnmx_R}
{\color{blue} \mc M} := \{m\in \mc C: -\, \mss b'(m)>0\},
\hspace{5mm}
{\color{blue} \mc W} := \{\sigma\in \mc C: -\, \mss b'(\sigma)<0 \}.
\end{equation*}
By condition H, every critical point is either stable or unstable, so
that $\mc C = \mc M\cup \mc W$.  

For $\mu\in \mss P(\bb T)$, denote by {\color{blue}$\mu_{\rm ac}$},
{\color{blue}$\mu_{\rm s}$} the absolutely continuous and the singular
part, respectively, of the measure $\mu$ with respect to the Lebesgue
measure denoted by {\color{blue}$\lambda$}. Let {\color{blue}$m_\mu$}
be the Radon-Nikodym derivative of $\mu_{\rm ac}$ with respect to
$\lambda$.

Let $\ms J^{(-1)}: \mss P(\bb T)\to [0,+\infty)$ be the functional
given by
\begin{equation}
\label{1st_rate_f} 
{\color{blue} \ms J^{(-1)}(\mu)}\::=\:
\frac{1}{4} \,  \int_{\bb T} \frac{\mss b^2}{\mss a} \; d\mu
\;-\; \frac{1}{4} \,  \frac{1}{\< 1/(\mss am_\mu)\>} \, B^2 \;,
\end{equation}
where $B$ is the constant introduced in \eqref{04}. In this formula
and below, for a nonnegative function $h\colon \bb T\to \bb R_+$,
$\color{blue} \< h\> = \int_{\bb T} h(x) \, dx$. Clearly, if $1/m_\mu$
is not Lebesgue integrable, the second term vanishes. The first term
in the $\Gamma-$expansion of $\ms J_\epsilon$ is provided by the
following result.

\begin{proposition}
\label{l01}
Assume that condition {\rm (H)} is in force. Then,
$\epsilon\, \ms I_\epsilon$ $\Gamma$-converges to $\ms J^{(-1)}$.  The
$0$-level set of the functional $\ms J^{(-1)}$ consists of all convex
combinations of Dirac measures supported on the elements of $\mc C$.
Moreover, $\alpha_\epsilon\, \ms I_\epsilon$ $\Gamma$-converges to $0$
for any positive sequence $\alpha_\epsilon$ such that
$\alpha_\epsilon \, \epsilon^{-1}\to 0$.
\end{proposition}

We turn to the second term in the expansion. Recall that we denote by
$\delta_{\{z\}}$, $z\in \bb T$ or $\bb R$, the Dirac measure supported
on $\{z\}$. Let $\cb{\mss P(\bb M)}$ be the set of probability measures
on a metric space $\bb M$ equipped with the Borel $\sigma$-algebra.
Let $\ms J^{(0)} \colon\mss P(\bb T)\to [0,+\infty]$ be given by

\begin{equation}
\label{2nd_rate_f}
{\color{blue}
\ms J^{(0)}(\mu)}\::=\: \begin{cases}
\sum_{z \in \mc C}\ w(z) \zeta(z),
&\text{if there exists $w\in \mss P(\mc C)$
such that}\ \mu = \sum_{z \in \mc C}\ w(z)\, \delta_{\{z\}} \\
+\infty &\text{otherwise}\ ,
\end{cases}
\end{equation}
where
\begin{equation}
\label{2f_1}
{\color{blue}
\zeta(z)}:= \max\{\:\mss b'(z) \,,\, 0\:\}. 
\end{equation}
Clearly $\zeta(z) >0$ if $z$ is an unstable solution of the ODE
\eqref{crit_R}, and it vanishes if it is a stable solution.

The second result of the paper reads as follows. It corresponds to the
time-scale at which the diffusion process $X_\epsilon(\cdot)$ escapes
from the local maxima and provides the cost for doing so.

\begin{proposition}
\label{s04}  
Assume that condition (H) is in force. Then, $\ms I_\epsilon$
$\Gamma-$converges to the functional $\ms J^{(0)}$ given by
\eqref{2nd_rate_f}. The $0$-level set of the functional $\ms J^{(0)}$
consists of all convex combinations of Dirac measures supported on the
elements of $\mc M$.
\end{proposition}

\subsection*{Metastable time-scales}

We turn to the time-scales at which the diffusion process escapes from
the stable solutions of the ODE \eqref{crit_R}. This is easier to
explain if we lift the diffusion to the line.  Regard $\mss b(\cdot)$
and $\mss a(\cdot)$ as periodic functions on $\bb R$, and denote by
$S: \bb R \to \bb R$ the potential given by
\begin{equation*}
\label{potential}
{\color{blue}S(x)} = -\int_0^x \frac{\mss b(z)}{\mss a(z)}\;dz.
\end{equation*}
The function $S(\cdot)$ is periodic if, and only if, the constant $B$
introduced in \eqref{04} vanishes. Consider the one-dimensional
diffusion process on $\bb R$ given by
\begin{equation}
\label{03}
dY_\epsilon (t) \;=\; \mss b(Y_\epsilon (t))\, dt \;+\;
\sqrt{2\, \epsilon\, \mss a(Y_\epsilon(t))} \,
dW_t\;,
\end{equation}
where $W_t$ is the Brownian motion on $\bb R$. 

\smallskip\noindent{\bf First order metastable structure.}  Denote by
$\mz C$, $\mz M$, $\mz W$ the lifting to $\bb R$ of the sets $\mc C$,
$\mc M$, $\mc W$, respectively.  Label the sets $\mz M$ and $\mz W$ in
the following way:
\begin{equation*}
\mz M:= \{m_k: k\in \bb Z\}\;, \hspace{5mm} \mz W:=\{\sigma_k: k\in \bb Z \},  
\end{equation*}
so that $m_{k+N}= m_k +1$ and $\sigma_{k+N} = \sigma_k+1$, for every
$k\in \bb Z$. Without loss of generality, assume that
\begin{equation*}
\label{1st_ordering}
m_{-1} < 0 < \sigma_0 < m_0 < \dots < \sigma_{N-1} < m_{N-1} < 1 < \sigma_N.
\end{equation*}
Let $\mz M_1(k)$, $k\in \bb Z$, be the sets defined by 
\begin{equation}
\label{1st_partition}
{\color{blue}\mz M_1(k)}:= \{m_k\}\, \;\;
\text{and set} \;\; {\color{blue}\mz S_1}:= \{\mz M_1(k): k\in \bb Z\}.
\end{equation}

For every $k\in \bb Z$, let 
\begin{equation*}
\label{f01}
{\color{blue}  h^{+}_{k}}  := S(\sigma_{k+1}) - S(m_{k}),
\hspace{10mm}
{\color{blue}  h^{-}_{k}} := S(\sigma_{k}) - S(m_{k}),
\end{equation*}
which can be regarded as the escape barriers on the right and left of
the set $\mz M_{1}(k)$, respectively. For each
$k\in \bb Z$ define the height (of the escape barrier) of
$\mz M_{1}(k)$ as the minimum of its right and left escape
barriers:
\begin{equation*}
\label{height}
{\color{blue}   h_k}  := \min\{h^{-}_{k}, h^{+}_{k}\}.
\end{equation*}
Clearly, since $\mss b$ and $\mss a$ are periodic, 
\begin{equation}
\label{per_h}
h^{\pm}_{k+N} = h^{\pm}_{k}\;, \hspace{5mm} h_{k+N} = h_{k}\,,
\quad
\forall\; k \in \bb Z\,.
\end{equation}
Let $\mf h_1$, and
$\theta_{\epsilon}^{(1)}$, be the smallest height, and the associated
time-scale, respectively:
\begin{equation}
\label{f02}
{\color{blue}\mf h_1} =
\min\{h_k: k\in \bb Z\}, \hspace{5mm} {\color{blue}\theta_{\epsilon}^{(1)}} =
e^{\mf h_1/\epsilon}.
\end{equation}
By \eqref{per_h}, $\mf h_1>0$. Thus,
$\theta^{(1)}_\epsilon \uparrow +\infty$ as $\epsilon \to 0$.

For $k\in\bb Z$, let
\begin{equation}
\label{22}
{\color{blue}\pi_{1}(k)} := \cb{\pi_1(\{m_k\}) } \,:=\, 
\frac{1}{\mss a(m_k)}\sqrt{\frac{2\pi}{S''(m_k)}} \;,
\;\;
{\color{blue} \sigma_1(k,k+1)}
\, :=\, \sqrt{\frac{2\pi}{-\,S''(\sigma_{k+1})}}\,.
\end{equation}
Here, the indices $k$, $k+1$ of $\sigma_1(k,k+1)$ indicate that we
consider the $S$-local maximum between the local minima $m_k$ and
$m_{k+1}$.  Let $\color{blue} \sigma_1(k+1,k) = \sigma_1(k,k+1)$.
Define the jump rates
\begin{equation}
\label{73}
{\color{blue} R_1 (\mz M_{1}(k),\mz M_{1}(k\pm 1))} \; := \;
\frac{1}{\pi_{1}(k)\, \sigma_1(k,k\pm 1) }
\; \mb 1\{ h^{\pm }_{k} = \mf h_1\} \;.
\end{equation}
Therefore, the jump rate from $\mz M_{1}(k)$ to $\mz M_{1}(k\pm 1)$
vanishes if $h^{\pm }_{k} > \mf h_1$. [By definition
$\mf h_1 \le h^{\pm }_{j}$ for all $j\in\bb Z$.] Recall the definition
of the set $\mz S_1$ given in \eqref{1st_partition}. Denote by
$\color{blue}  \bb X_1 (\cdot)$ the $\mz S_{1}$-valued Markov
chain induced by the generator $\bb L_1$ defined as
\begin{equation*}
\label{72b}
{\color{blue} (\bb L_1 f)(\mz M_{1}(k))} \,:=\, 
\sum_{a = \pm 1} R_1(\mz M_{1}(k), \mz M_{1}(k+a))\,
\big[\, f(\mz M_{1}(k+a)) - f(\mz M_{1}(k))\, \big]\;, \quad k\in \bb
Z\;. 
\end{equation*}

As $\mss a(\cdot)$, $\mss b(\cdot)$ are periodic functions, the jump
rates defined in \eqref{73} are periodic as well:
$R_1 (\mz M_{1}(k + N),\mz M_{1}(k+N\pm 1)) = R_1 (\mz M_{1}(k),\mz
M_{1}(k\pm 1))$, $k\in \bb Z$. Therefore, if
$\Pi\colon \bb R \to \bb T$ represents the projection defined by
\begin{equation}
\label{proj_1}
{\color{blue}\Pi(x)} = x-\lfloor x \rfloor\,, \quad\text{where
$\lfloor r \rfloor$ is the integer part of a real number $r$}\,,
\end{equation}
$\cb{ \widehat{\bb X}_1 (\cdot) }: = \Pi( \bb X_1 (\cdot))$ is a
continuous-time Markov chain taking values in
$\cb{\mc S_1 }:= \{ \mc M_1(k) ; 0\le k<N\}$, where
$\cb{\mc M_1(k)} = \{\mf m_k\}$, $\cb{\mf m_k} := \Pi(m_k)$.  Denote
by $\cb{ \widehat{\bb L}_1}$ the generator of the projected Markov
chain $\widehat{\bb X}_1(\cdot)$, and let
$\bb I^{(1)} \colon \mss P (\mc S_1) \to [0,+\infty)$ be the associated
large deviations rate functional:
\begin{equation*}
\cb{\bb I^{(1)}(\mu)} \,:=\,
\sup_{H}-\, \sum_{\mc M\in \mc S_{1}}
\frac{(\widehat{\bb L}_1 H)(\mc M)} {H(\mc M)}\mu(\mc M)\,,
\end{equation*}
where the supremum is performed over all positive functions
$H\colon \mc S_1 \to \bb R$. Let
$\ms J^{(1)} \colon\mss P(\bb T)\to [0,+\infty]$ be the functional
given by
\begin{equation}
\label{23}
{\color{blue}
\ms J^{(1)}(\mu)}\::=\:
\begin{cases}
\bb I^{(1)}(\omega), &\text{if } \mu(\cdot)
= \sum_{\mc M\in \mc S_{1}} \omega(\mc M)\,
\delta_{\mc M}(\cdot) \text{ for } \omega \in \mss P(\mc S_{1}) \\
+\infty &\text{otherwise.}
\end{cases}
\end{equation}

\begin{proposition}
\label{s03}
Assume that condition H is in force.
Then, $\theta_\epsilon^{(1)}\, \ms I_\epsilon$ $\Gamma$-converges to
$\ms J^{(1)}$.
\end{proposition}

If the Markov chain $\widehat{\bb X}_1$ has only one closed
irreducible class, then the metastable description of
$X_\epsilon(\cdot)$ ended. Otherwise, $X_\epsilon(\cdot)$ exhibits a
metastable behavior in a longer timescale.

\smallskip\noindent{\bf Higher order metastable structure.}  In this
subsection, we describe the metastable behavior of the process
$Y_\epsilon(\cdot)$ across all time scales. The exposition is based on
a recursive construction proposed in \cite{lm}. 

The recursive construction produces an integer $\mf q\ge 1$, and  triplets
\begin{equation*}
\label{meta_trip}
{\color{blue}
\Gamma_q= (\mtt P_{q},\; \mf h_{q},\;
\bb L_{q})},   \quad 1\le q\le\mf q\,,
\end{equation*}
where $\mtt P_{q} := \mz S_q \cup \{ \mz T_q\}$ is a
partition of the set $\mz M$,
${\color{blue} \mz S_q} := \{ \mz M_{q}(j) : j\in \bb Z\}$ is a family
of disjoint, non-empty subsets of $\mz M$, $\mf h_q>0$ is an energy
barrier, and $\bb L_q$ is the generator of a
$\mz S_q$-valued Markov chain, denoted by $\bb X_q(\cdot)$.
The first triplet $\Gamma_{1}$ has been constructed in the
previous subsection.

It is shown in \cite{lm} that each triplet $\Gamma_q$,
$1\le q\le \mf q$, satisfies postulates $\mc P_1(q)$--$\mc P_{10}(q)$
listed in Appendix \ref{secA1} of this article. We present below the
properties of this constructions needed to state the main theorem.
The reader can check that they are fulfilled by the triplet
$\Gamma_{1}$. Fix $2\le q\le \mf q$.

\begin{enumerate}

\item[(a)] The sets $\mz M_q(j)$ are well ordered in the sense that $x<y$
for all $x\in \mz M_q(j)$, $y\in \mz M_q(j+1)$, $j\in \bb Z$. [$\mc P_2(q)$].

\item[(b)] The potential $S(\cdot)$ restricted to the sets $\mz M_q(j)$ is
constant. For each $j\in \bb Z$, $S(y)=S(x)$ for all $y$,
$x\in \mz M_q(j)$. The common value is denoted by
$\cb{ S(\mz M_q(j)) } $.  [$\mc P_3(q)]$.

\item[(c)] The sets $\mz M_q(j)$ are periodic in the sense that there
exists and integer $\cb{\mf u_q}$ such that
$\mz M_q(j+\mf u_q) = \mz M_q(j) +1$. By the previous construction
$\mf u_1=N$. [$\mc P_3(q)$].

\item[(d)] By construction, each set $\mz M_{q}(j)$, $j\in \bb Z$, is
obtained as the union of a $\bb X_{q-1}$-closed irreducible
class. That is, for each $j\in\bb Z$, there exists a finite subset
$A_{q,j}$ of $\bb Z$ such that $\{ \mc M_{q-1}(i) : i\in A_{q,j} \}$
is a $\bb X_{q-1}$-closed irreducible class, and
$\mz M_{q}(j) = \cup_{i\in A_{q,j}} \mz M_{q-1}(i)$.
\end{enumerate}

We may now define the escape barrier $\mf h_{q}$.  Denote by
$m^+_{q,i}$, $m^-_{q,i}$, $i\in \bb Z$, the rightmost, leftmost
element of $\mz M_{q}(i)$, respectively:
\begin{equation}
\label{17}
{\color{blue} m^-_{q,i}}
\,:=\,  \min\{m\in \mz M_{q}(i)\}\;, \hspace{5mm}
{\color{blue} m^+_{q,i}}
\,:=\,  \max\{m\in \mz M_{q}(i)\}.
\end{equation}
Let $\mz W^{(q)}_{k,k+1}$ be the set of the absolute maximum of $S$ on
the intervals $[m^+_{q, k}, m^-_{q, k+1}]$:
\begin{equation}
\label{max_br}
{\color{blue}\mz W^{(q)}_{k,k+1} } \; := \;
\operatorname*{arg\,max}_{x\;\in\;
[m^+_{q, k} \,,\,  m^-_{q, k+1}]} S(x) \,,\quad
\cb{\mz W_{q} } \,:=\, \bigcup_{k\in \bb Z} \mz W^{(q)}_{k,k+1}
\;.
\end{equation}
Clearly, the potential $S(\cdot)$ is constant on
$\mz W^{(q)}_{k,k+1}$. Denote by $\cb{S(\mz W^{(q)}_{k,k+1})}$ the
common value.  Let $\sigma_{k, k+1}^{p,+}$, $\sigma_{k, k+1}^{p,-}$ be
the rightmost and leftmost maxima of $\mz W^{(p)}_{k,k+1}$,
respectively:
\begin{equation}
\label{rl_max:R}
{\color{blue}\sigma^{p,+}_{k, k+1}} \,:= \, \max\{\sigma\in
\mz W^{(p)}_{k, k+1}\}\,, \quad
{\color{blue}
\sigma^{p,-}_{k, k+1}} \,:=\,
\min\{\sigma\in \mz W^{(p)}_{k-1, k}\}\,, \quad k\in \bb Z\,.
\end{equation}

Since the sets $\mz M_{q}(k)$, $k\in\bb Z$, are ordered and the
elements of each set have the same depth, we may define the heights of
the left and right escape barriers of the set $\mz M_{q}(k)$,
$k\in \bb Z$, as
\begin{equation}
\label{18}
\begin{gathered}
{\color{blue} h^{q,+}_{k} }  \,:=\,
S(\mz W^{(q)}_{k,k+1} ) \,-\,  S(\mz M_{q}(k))\;,
\quad
{\color{blue} h^{q,-}_{k} }  \,:=\,
S(\mz W^{(q)}_{k-1,k})  \,-\,  S(\mz M_{q}(k)) \;.
\end{gathered}
\end{equation}
Define the height $h^{q}_{k}$ of the escape barrier of
$\mz M_{q}(k)$ as the minimum of its left and right heights. Let
$\mf h_{q}$ be the smallest height and
$\theta_\epsilon^{(q)}$ the associated time-scale:
\begin{equation}
\label{19}
{\color{blue} h^{q}_{k}} := h^{q,-}_{k}\wedge h^{q,+}_{k} \;, \quad 
{\color{blue}\mf h_{q}} := \min
\big\{ \, h^{q}_{k} : k\in \bb Z\, \big\} \,, \quad
\cb{\theta_\epsilon^{(q)}}  \,:=\, e^{\mf h_{q}/\epsilon}\,.
\end{equation}
By property (c) of the sets $\mz M_q(j)$,
$h^{q}_{k+\mf u_{q}} \,=\, h^{q}_{k}$ for all $k\,\in\, \bb Z$.
Hence, the heights $h^{q}_{k}$, $k\in\bb Z$, take at most $\mf u_{q}$
different values, and $\mf h^{q} \in (0,\infty)$. Moreover, by
\cite[Proposition 2.3]{lm},
\begin{equation}
\label{27}
\mf h_{q} \,>\, \mf h_{q-1}\,.
\end{equation}

Let the weight $\pi_{q}(k)$ of the set of local minima
$\mz M_{q}(k)$ and the weight $\sigma_{q}(k,k+1)$ of the set of
maxima $\mz W^{(q)}_{k,k+1}$, $k\in \bb Z$, be defined as
\begin{equation}
\label{p_weights}
{\color{blue}\pi_{q}(k)} \,:=\, \pi_1(\mz M_{q}(k) )\, =\, 
\sum_{m\in \mz M_{q}(k)}
\frac{1}{\mss a(m)}\sqrt{\frac{2\pi}{S''(m)}}\,, \quad
{\color{blue}\sigma_{q}(k,k+1)} \,:= \,
\sum_{ \sigma\in \mz W^{(q)}_{k,k+1}}
\sqrt{\frac{2\pi}{ -\, S''(\sigma)}}\;\cdot
\end{equation}
Define the jump rates as
\begin{equation}
\label{50b}
\begin{gathered}
{\color{blue}R_{q}(\mz M_{q}(k),\mz M_{q}(k+1))} \; :=
\; \frac{1}{\pi_{q}(k) \, \sigma_{q}(k,k+1)}\;
\mb 1\{\, h^{q,+}_{k} \,=\,  \mf h_{q} \,\} \;,
\\
{\color{blue}R_{q}(\mz M_{q}(k),\mz M_{q}(k-1))} \; :=
\; \frac{1}{\pi_{q}(k)\, \sigma_{q}(k-1,k)}
\; \mb 1\{\,  h^{q,-}_{k} \,=\,  \mf h_{q} \,\} \;.
\end{gathered}
\end{equation}
Let ${\color{blue}\bb X_{q}(\cdot)}$ be the $\mz S_{q}-$valued
Markov chain induced by the generator $\bb L_{q}$ given by
\begin{equation}
\label{p_generator}
{\color{blue}(\bb L_{q}f)\, (\mz M_{q}(k))}
\;=\; \sum_{a = \pm 1} R_{q}(\mz M_{q}(k), \mz
M_{q}(k\pm a)) \,
[\, f(\mz M_{q}(k+a)) - f(\mz M_{q}(k)) \,]\;,
\end{equation}
for any $f\colon \mz S_{q} \to \bb R$ and
$\mz M_{q}(k)\in \mz S_{q}$.

Observe that the Markov chain $\bb X_q$ only jumps to
``nearest-neighbor'' sets: $R_{q}(\mz M_{q}(j),\mz M_{q}(k)) =0$ if
$|k-j|\neq 1$. This remark together with property (d) explain property
(a) of the sets $\mz M_q(j)$.

\begin{remark}
\label{rm2}
This remarks also holds for $q=1$.  Since
$\pi_{q}(k)\, \sigma_{q}(k,k\pm 1)$, $k\in \bb Z$, is strictly
positive and finite, by definition of $\mf h_{q}$, there exists
$k'\in \bb Z$ and $a\in \{-1, 1\}$ such that
$R_{q}(\mz M_{q}(k'),\mz M_{q}(k'+a))>0$. Therefore, the state
$\mz M_{q}(k')$ is either a transient state or it belongs to a
$\bb X_{q}$-closed irreducible class with at least two elements.
\end{remark}

\begin{remark}
\label{rm3}
By conditions (c), the variable $\mf u_q$, $1\le q\le \mf q$
represents the number of states of the Markov chain $\bb X_{q}$ up to
equivalence. By conditions (d), for $2\le q\le \mf q$, $\mf u_q$ is
also the number of $\bb X_{q-1}$-closed irreducible classes, up to
equivalence. By Remark \ref{rm2}, the sequence $\mf u_p$ strictly
decreases.  The recursive construction ends when the Markov chain
$\bb X_p(\cdot)$ has a unique closed irreducible class, the whole set
$\ms S_p$, or none (if all states are transient).
\end{remark}

As $\mss a(\cdot)$, $\mss b(\cdot)$ are periodic functions, by
property (c) of the sets $\mz M_q(j)$, the jump rates
$R_q(\cdot ,\, \cdot)$ are periodic as well:
$R_q (\mz M_{q}(k + \mf u_q),\mz M_{q}(k+ \mf u_q \pm 1)) = R_q (\mz
M_{q}(k),\mz M_{q}(k\pm 1))$, $k\in \bb Z$. Therefore,
$\cb{ \widehat{\bb X}_q (\cdot) }: = \Pi( \bb X_q(\cdot))$ is a
continuous-time Markov chain taking values in
$\cb{\mc S_q }:= \{ \mc M_q(k) ; 0\le k<\mf u_q\}$, where
$\cb{\mc M_q(k)} = \Pi(\mz M_q(k))$.

Let $\cb{\widehat{\mf q}}$ be the first integer $p\ge 1$ such that
$\widehat{\bb X}_{p}$ has a unique closed irreducible class. Clearly,
$\widehat{\mf q} \le \mf q$, and it may happen that
$\widehat{\mf q} = \mf q-1$. See Figure \ref{fig-1f} in Appendix
\ref{secA1}. Indeed, suppose that the Markov chain $\bb X_{p}$ has
infinitely many closed irreducible classes, which are all
equivalent. In this case, we may construct a new layer by the
recursive procedure. In this new layer, the states will be the closed
irreducible classes of level $p$. As they are all equivalent and the
jump rates periodic, either all states are transient or there is only
one closed irreducible class. In particular, the recursive procedure
ends at this stage and $\mf q=p+1$. On the other hand, the projected
Markov chain $\widehat{\bb X}_{p}$ has only one closed irreducible
class so that $\widehat{\mf q}=p$. We describe in Lemma \ref{s11} all
possible behaviors of the Markov chain $\bb X_{\widehat{\mf q}}$ and
the relation between $\widehat{\mf q}$ and $\mf q$ in each case.

Denote by $\cb{ \widehat{\bb L}_q}$, $1\le q\le \widehat{\mf q}$, the
generator of the projected Markov chain $\widehat{\bb X}_q(\cdot)$,
and let $\bb I^{(q)} \colon \mss P (\mc S_q) \to [0,+\infty)$ be the
associated large deviations rate functional:
\begin{equation}
\label{25}
\cb{ \bb I^{(q)}(\mu)} \,:=\,
\sup_{H}-\, \sum_{\mc M\in \mc S_{q}}
\frac{(\widehat{\bb L}_q H)(\mc M)} {H(\mc M)}\mu(\mc M)\,,
\end{equation}
where the supremum is performed over all positive functions
$H\colon \mc S_1 \to \bb R$.

The definition of the rate functional $\ms J^{(q)}$ requires some
further notation. Extend the definitions of he weights $\pi_1(\cdot)$,
introduced in \eqref{22}, to the torus by setting
$\pi_1(\{\mf m\}) = \pi_p(\Pi^{-1} (\{\mf m\}))$. This is well defined
because $\mss a(\cdot)$ and $\mss b(\cdot)$ are periodic.  Let
$\mu^{p}_{\mc M_{p}(k)} \in \mss P(\bb T)$, $1\le p\le \widehat{\mf q}$,
$1\le k\le \mf u_p$, be the measures defined by
\begin{equation}
\label{21}
\cb{\mu^{p}_{\mc M_{p}(k)}(\cdot)}
\,:=\,  \frac{1}{\pi_{p}(k)}
\sum_{\mf m  \in  \mc M_p(k)} 
\pi_{1}(\{\mf m\} ) \, \delta_{\mf m} (\cdot)\,, 
\end{equation}
where $\pi_p(k)$ has been introduced in \eqref{p_weights}, and turns
$\mu^{p}_{\mc M_{p}(k)}$ into a probability measure.  Informally,
$\mu_{\mc M_p(k)}^p(\{m\})$, $m\in \mc M_p(k)$, represents the
proportion of time the diffusion $X_\epsilon(\cdot)$ spends close to
the point $m$ when it visits the set $\mc M_p(k)$.

For $2\le p\le \widehat{\mf q}$, let
$\ms J^{(p)} \colon\mss P(\bb T)\to [0,+\infty]$ be the functional
given by
\begin{equation}
\label{23b}
{\color{blue}
\ms J^{(p)}(\mu)}\::=\: \begin{cases}
\bb I^{(p)}(\omega), &\text{if } \mu(\cdot)
= \sum_{\mc M\in \mc S_{p}}\omega(\mc M)\,
\mu^{p}_{\mc M}(\cdot) \text{ for } \omega \in \mss P(\mc S_{p}) \\
+\infty &\text{otherwise.} \end{cases}
\end{equation}
This functional can be regarded as the lifted version of $\bb I^{(p)}$
to $\mss P(\bb T)$.  We are now in position to state the main result of
this article.

\begin{theorem}
\label{mt3}
Assume that condition H is in force.  Then, for $2\le p\le
\widehat{\mf q}$,
$\theta_\epsilon^{(p)}\, \ms I_\epsilon$ $\Gamma$-converges to
$\ms J^{(p)}$. Moreover, all conditions of Definition
\ref{gamma_expansion_def} are fulfilled and we can write
\begin{equation*}
\ms I_\epsilon \;=\; \frac{1}{\epsilon}\;\ms J^{(-1)} \;+\; \ms J^{(0)}  \;+\;
\sum_{p=1}^{\widehat{\mf q}} \frac{1}{\theta_{\epsilon}^{(p)}} \, \ms J^{(p)}  \;.
\end{equation*}
\end{theorem}

\subsection*{Organization of the paper}
The rest of the article is organized as follows. In Section \ref{sec1},
we derive an explicit formula for the rate functional
$\ms I_\epsilon$ and prove Proposition \ref{l01}, which constitutes the first
step of the expansion. Section \ref{sec0004} is devoted to the proof of Proposition
\ref{s04}. In Section \ref{sec5}, we introduce a representation
formula for the rate-function of the reduced Markov chain
$\widehat{\bb X}_{p}(\cdot)$, $1\leq p\leq \widehat{\mf q}$, originally derived
in \cite{l23}, and state the resolvent condition to
metastability. These results are the main technical and theoretical
tools required to derive the higher order expansion. In Section
\ref{sec5.2}, we verify the third condition in Definition
\ref{gamma_expansion_def} and establish the $\Gamma-\liminf$ at every
metastable scale. The $\Gamma-\limsup$ at metastable regimes is
discussed along the remaining three sections. In Section \ref{sec6},
we introduce a method to prove the $\Gamma-\limsup$ at every
metastable scale with the exception of the last one. In Section
\ref{sec2} we extend this method to the last time-scale, under the
additional assumption that the Markov chain $\widehat{\bb X}_{\widehat{\mf
q}}(\cdot)$ is either reducible or reversible. Since this argument
breaks down in the complementary case, Section \ref{sec8} is devoted
to the remaining situation in which $\widehat{\bb X}_{\widehat{\mf q}}(\cdot)$ is
irreducible and non-reversible. There we establish the
$\Gamma-\limsup$ at the last metastable time-scale and complete the
proofs of Proposition \ref{s03} and Theorem \ref{mt3}.

\section{Proof of Proposition \ref{l01}}
\label{sec1}

In this section, we establish several properties of the action
functional $\ms I_\epsilon$ and prove Proposition~\ref{l01}. The
section is organized as follows.  In Lemmata \ref{l02} and \ref{l03},
we provide an explicit formula for $\ms I_\epsilon(\mu)$ when the
measure $\mu$ is absolutely continuous with respect to the Lebesgue
measure and has a smooth, strictly positive density. Using this
explicit formula, in Corollary \ref{l05} we construct a recovery
sequence and prove the $\Gamma-\limsup$ for absolutely continuous
measures. This recovery sequence is obtained by smoothing the density
and shifting it away from the origin.

To extend the $\Gamma-\limsup$ to singular measures, we first construct
in Lemma \ref{l04} a recovery sequence for a Dirac measure supported
at a single point. Then, in Corollary \ref{l06}, using the convexity
of the rate functional $\ms I_\epsilon(\cdot)$ and the linearity of
$\ms J^{(-1)}(\cdot)$ when restricted to singular measures, we obtain
a recovery sequence and establish the $\Gamma-\limsup$ for singular
measures.  By decomposing a general measure into its absolutely
continuous and singular parts, we can extend the $\Gamma-\limsup$ proof
to all measures, again relying on the convexity of
$\ms I_\epsilon(\cdot)$. This final step is carried out in the second
part of the proof of Proposition \ref{l01}, towards the end of this
section.

The $\Gamma-\liminf$ part of the proof is obtained in two steps. First,
exploiting the variational representation of the rate functional
$\ms I_\epsilon(\cdot)$ and choosing test functions of the form
$u_\epsilon = e^{F/\epsilon}$, for a smooth function $F(\cdot)$, we
derive a lower bound. Optimizing over $F$ completes the argument.
Since $\ms J^{(-1)}(\mu)$ is finite for all measures
$\mu \in \mss P(\bb T)$, it follows that $\ms J^{(-1)}(\mu)$ appears as
the first term in the $\Gamma$-expansion of the rate functional
$\ms I_\epsilon(\cdot)$. This is the content of Lemma
\ref{shortest:scale}, and it completes the proof of Proposition
\ref{l01}.

\subsection{Properties of the rate functional $\ms I_\epsilon$}

For a function
$H\in C^2(\bb T)$, denote by $\ms L_{H,\epsilon}$ the generator given
by
\begin{equation*}
{\color{blue}(\ms L_{H,\epsilon} F)(x)}
\;=\; \{\, \mss b \,+\, 2\,\epsilon\, \mss a\, H' \,\} F' \;+\;
\epsilon\, \mss a\, F''\;. 
\end{equation*}
Thus, the diffusion induced by the generator $\ms L_{H,\epsilon}$ has
an extra drift $2\epsilon\, \mss a\, H' $ with respect to the one induced by
$\ms L_{\epsilon}$.

\begin{lemma}
\label{l02}
Let $m\in C^2(\bb T)$ be the density of a probability measure
($m\ge 0$, $\int_{\bb T} m(x)\, dx =1$). Suppose that there exists
$c>0$ such that $m(x) \ge c$ for all $x\in\bb T$.  Let
$H\colon \bb T \to \bb R$ be given by
\begin{equation}
\label{02}
H(x) \;=\; \frac{1}{2} \, \log \big( \mss a(x)\, m (x) \big)
\;+\; \frac{1}{2\epsilon} 
\frac{\<\mss b/\mss a\>}{\<1/(\mss a\,m)\> }  \,
\int_0^x \frac{1}{\mss a(y)\, m (y)} \, dy  \,-\,
\frac{1}{2\epsilon} 
\int_0^x \frac{\mss b(y)}{\mss a(y)} \, dy \;,
\end{equation}
where, recall, $\<h\> = \int_{\bb T}h(y)\,dy$. Then, $H$ belongs to
$C^2(\bb T)$ and $m$ is the unique stationary state for the diffusion
induced by the generator $\ms L_{H,\epsilon}$.
\end{lemma}

\begin{proof}
An elementary computation yields that $H$ belongs to $C^2(\bb T)$ and
that
\begin{equation*}
\epsilon \, \big(\mss a \,m \big)''
\,-\, \big(\mss b\, m \big)'
\,-\, 2\epsilon \, \big(\mss a\, H'\, m \big)' \;=\; 0\;.
\end{equation*}
Hence, by integration by part, 
\begin{equation*}
\int_{\bb T}  (\ms L_{H,\epsilon} F)(x) \, m(x) \, dx \;=\; 0\,,
\end{equation*}
for any function $F\in C^2(\bb T)$, as claimed.
\end{proof}

\begin{lemma}
\label{l03}
Let $\mu\in\mss P(\bb T)$ be such that $\mu(dx) = m(x)\, dx$, where
$m\in C^2(\bb T)$ and there exists $c>0$ such that $m(x) \ge c$ for
all $x\in\bb T$.  Then,
\begin{equation*}
\ms I_\epsilon (\mu) \;=\; -\,
\int_{\bb T} \frac{\ms L_\epsilon e^H}{e^{H}} \; d\mu
\;=\; \epsilon \int_{\bb T} \mss a\,[H']^2  \; d\mu \;=\;
\frac{1 }{4\, \epsilon}\, \int_{\bb T} \mss a\,\Big\{ \, \epsilon\,
\Big[ \frac{m'}{m}
+ \frac{\mss a'}{\mss a}\Big] \,+\,  \frac{B \,/\, (\mss a\,m)} { \int_{\bb
T} 1\,/ (\mss a\, m) \,dy } \,-\, \frac{\mss b}{\mss a}  \, \Big\}^2 \, d \mu\;,
\end{equation*}
where $H$ is given by \eqref{02}.
\end{lemma}

\begin{proof}
Fix $G\in C^2(\bb T)$. An elementary computation yields that
\begin{equation}
\label{26}
\frac{\ms L_\epsilon e^G}{e^{G}} \;=\;
\mss b \, G' \, +\, \epsilon\,  \mss a \, (G')^2
\,+\; \epsilon\, \mss a\, G''  \,.
\end{equation}
Adding and subtracting
$2\,\epsilon \, \mss a\, G' \, H'$, yields that
\begin{equation}
\label{05}
\frac{\ms L_\epsilon e^G}{e^{G}} \;=\;
\ms L_{H,\epsilon} G \;+\; \epsilon\, \mss a\, [G']^2 \;-\; 2\,\epsilon \, \mss a \, G'
\, H'\;.
\end{equation}
Completing the square, and since, by Lemma \ref{l02}, $\mu$ is the
stationary state of the diffusion induced by the generator $\ms L_{H,\epsilon}$,
\begin{equation*}
-\, \int_{\bb T} \frac{\ms L_\epsilon e^G}{e^{G}} \; d\mu
\;\le\; -\, \int_{\bb T} \ms L_{H,\epsilon}  G \; d\mu \;+\;
\epsilon \int_{\bb T} \mss a\, [H']^2  \; d\mu \;=\;
\epsilon \int_{\bb T} \mss a\, [H']^2  \; d\mu  \;.
\end{equation*}
Taking $G=H$ in equation \eqref{05} yields that
\begin{equation*}
-\, \int_{\bb T} \frac{\ms L_\epsilon e^H}{e^{H}} \; d\mu \;=\;
\epsilon \int_{\bb T} \mss a\, [H']^2  \; d\mu  \;,
\end{equation*}
which completes the proof of the lemma in view of the penultimate
displayed equation inequality.
\end{proof}

\begin{corollary}
\label{l05}
Let $\mu(dx) = m(x)\, dx \in\mss P(\bb T)$ be a measure absolutely
continuous with respect to the Lebesgue measure. Then, there exists a
sequence $\mu_\epsilon \in \mss P(\bb T)$, such that $\mu_\epsilon$
converges weakly to $\mu$ as $\epsilon \to 0$, and
\begin{equation*}
\limsup_{\epsilon\to 0} \epsilon\, \ms I_\epsilon (\mu_\epsilon) \;\le
\; \frac{1 }{4}\,  \Big\{ \, \int_{\bb T}  \frac{\mss b^2}{\mss a}\, d \mu
\;-\; \frac{1} { \<1/(\mss a\,m) \>  }\, B^2\,  \Big\} \;.
\end{equation*}
\end{corollary}

\begin{proof}
If $\mu$ satisfies the assumptions of the previous lemma, consider the
constant sequence $\mu_\epsilon = \mu$. The assertion of the corollary
is a straightforward consequence of the formula for
$\ms I_\epsilon (\mu) $ presented in Lemma \ref{l03}.

Fix a measure $\mu(dx) = m(x)\, dx \in\mss P(\bb T)$ such that
$m(x) \ge c$ for all $x\in\bb T$ and some $c>0$. In this case,
smoothing the density $m$ we obtain a sequence $m_n$ satisfying the
hypotheses of Lemma \ref{l03} and such that $\mu_n(dx) = m_n(x) \, dx$
converges weakly to $\mu(dx)$. Clearly,
$\ms J (\mu_n) \to \ms J(\mu)$. A diagonal argument permits to
construct a sequence $\mu_\epsilon \in \mss P(\bb T)$ such that
$\mu_\epsilon \to \mu$ weakly as $\epsilon\to 0$, and
$\limsup_{\epsilon\to 0} \epsilon\, \ms I_\epsilon (\mu_\epsilon)
\;\le \ms J(\mu)$.

Finally, fix a measure $\mu(dx) = m(x)\, dx \in\mss P(\bb T)$. Let
$m_\epsilon (x) = (m(x) + \epsilon)/(1+\epsilon)$. Then,
$\mu_\epsilon (dx) = m_\epsilon (x)\, dx$ converges weakly to
$\mu(dx) $ as $\epsilon \to 0$ and
$\lim_{\epsilon\to 0} \ms J (\mu_\epsilon) = \ms J(\mu)$.
The same diagonal argument permits to conclude the proof.
\end{proof}

\subsection{Recovery sequence for singular measures}

Fix $x \in \bb T$, define the interval
${\color{blue}I_x}:=(x-1/4,\,x+1/4)$, and consider a function
$G\in C^2(\bb T)$ such that
\begin{align*}
 G(y) \,=\, 
 \frac{\Vert y-x\Vert^2}{2}\,,\,\, y\in I_x\,,
 \quad \quad  G(y) \,\geq\, \frac{1}{32}\,,\,\, y\in \bb T\setminus I_x\,.
\end{align*}
In particular, $G(y)=0$ if, and only if, $y=x$,
and $G'(x)=0$.  Let
\begin{align}
\label{06}
{\color{blue}\mu_\epsilon (dy)}
\, := \; \frac{e^{- G(y)/\epsilon}}{\<e^{-G/\epsilon}\>}\; dy  \, . 
\end{align}

\begin{lemma}
\label{l04}
Fix $x \in \bb T$, and consider the sequence of measures
$\mu_\epsilon $ given by \eqref{06}. Then, as $\epsilon\to 0$,
$\mu_\epsilon$ converges weakly to $\delta_x$, and
\begin{equation*}
\lim_{\epsilon\to 0} \epsilon\,
\ms I_\epsilon (\mu_\epsilon)
\;=\; \frac{1}{4} \, \frac{\mss b(x)^2}{\mss a(x)}\, .
\end{equation*}
\end{lemma}

\begin{proof}
For every $\phi\in C(\bb T)$ and $\delta \in (0,1)$,
\begin{equation*}
\left|\langle \mu_\epsilon ,\phi\rangle -\langle \delta_x,\phi\rangle
\right|
\;\le\; \sup_{\substack{y\in \bb T\\
\Vert y-x\Vert<\delta}}\left|\phi(y)-\phi(x)\right| \;+\;
2\,\|\phi\|_\infty\, \frac{1}{\<e^{-G/\epsilon}\>}  \, \int_{\Vert y- x\Vert>\delta}
e^{- G(y)/\epsilon} \, dy\;,
\end{equation*}
which, by the continuity of $\phi$ at $x\in \bb T$ and the quadratic
behavior of $G$ at $x$, vanishes taking first $\epsilon\to 0$, and
then $\delta \to 0$. This proves that $\mu_\epsilon $ converges to
$\delta_x$ as $\epsilon\to 0$.

On the other hand, by Lemma \ref{l03},
\begin{equation*}
\epsilon\, \ms I_\epsilon (\mu_\epsilon) \;=\;
\frac{1 }{4}\, \int_{\bb T} \mss a\,\Big\{ \, \epsilon\,
\Big[\frac{m'_\epsilon}{m_\epsilon}
\,+\, \frac{\mss a'}{\mss a}\Big]
\,+\, \frac{B \,/\, (\mss a\,m_\epsilon)} {\<1/(\mss a\,m_\epsilon)\>}
\,-\, \frac{\mss b}{\mss a}  \, \Big\}^2 \, d \mu_\epsilon\;,
\end{equation*}
where $m_\epsilon \,=\, e^{- G/\epsilon}/\<e^{-G/\epsilon}\>$. Expand
the square. It is not difficult to show that all terms vanish as
$\epsilon\to 0$, except
$(1/4) \, \int_{\bb T} \mss b^2\,/\,\mss a \, d \mu_\epsilon$ which
converges to $\mss b(x)^2\,/4\,\mss a(x)$ since $\mu_\epsilon$
weakly converges to $\delta_x$.
\end{proof}

\begin{corollary}
\label{l06}
Fix $\mu \in \mss P(\bb T)$. Then, there exists a
sequence $\mu_\epsilon \in \mss P(\bb T)$, such that $\mu_\epsilon$
converges weakly to $\mu$ as $\epsilon \to 0$, and
\begin{equation*}
\limsup_{\epsilon\to 0} \epsilon\, \ms I_\epsilon (\mu_\epsilon) \;\le
\; \frac{1 }{4}\,  \int_{\bb T} \frac{\mss b^2}{\mss a}  \, d \mu \;,
\end{equation*}
\end{corollary}

\begin{proof}
Fix $\mu \in \mss P(\bb T)$, and let $\ms J_0\colon \mss P(\bb T) \to \bb
R_+$ be given by
$\ms J_0(\mu) = (1/4)\, \int_{\bb T} \mss b^2\,/\,\mss a \, d \mu$.

By Lemma \ref{l04}, the assertion of the corollary holds for Dirac
measures.  Consider a sequence $\mu_n$ of finite convex combinations
of Dirac measures which converges weakly to $\mu$. Since $\mss b$ and
$\mss a$ are continuous, $\ms J_0(\mu_n) \to \ms J_0(\mu)$. On the
other hand, by the convexity of $\ms I_\epsilon$, the linearity of
$\ms J_0$ and Lemma \ref{l04},
\begin{equation*}
\limsup_{\epsilon\to 0} \epsilon\, \ms I_\epsilon (\mu_n) \;\le
\; \ms J_0(\mu_n) 
\end{equation*}
for all $n\ge 1$. A diagonal argument provides a sequence which
satisfies the assertions of the corollary.
\end{proof}

Observe that the previous corollary provides a recovery sequence and
the $\Gamma-\limsup$ for singular measures.

\subsection*{$\Gamma$-convergence of $\epsilon\, \ms I_\epsilon$}
\label{sec3.1}

In this subsection, we prove the $\Gamma$-convergence of the action
functional $\epsilon\, \ms I_\epsilon$, as stated in Proposition
\ref{l01}.

\begin{proof}[Proof of Proposition \ref{l01}]
We prove the $\Gamma-\liminf$ and $\Gamma-\limsup$ inequalities
separately.
	
\smallskip\noindent \emph{$\Gamma-\liminf$}.  Fix
$\mu \in \mss P(\bb T)$, and let $\mu_\epsilon \in \mss P(\bb T)$ be a
sequence such that $\mu_\epsilon \to\mu$.  Let
$u_\epsilon:= e^{F/\epsilon}$ for some $F\in C^2( \bb T)$. By
\eqref{26}, by the definition of $\ms I_\epsilon$ and since $F$ belongs to
$C^2( \bb T)$,
\begin{equation*}
\epsilon\, \ms I_\epsilon (\mu_\epsilon) \;\ge\; \epsilon\,
\int_{\bb T} -\frac{\ms L_\epsilon u_\epsilon} {u_\epsilon} \;
d\mu_\epsilon \;=\; -\,
\int_{\bb T}  \big(\, \mss b + \mss a\, F'\,\big) \, F' \;
d\mu_\epsilon \;+\; O(\epsilon)\;.
\end{equation*}
Therefore, as $\mu_\epsilon\to \mu$, and $\mss a, \mss b$ and $F'$ are
continuous functions,
\begin{equation*}
\liminf_{\epsilon\to 0} \epsilon\, \ms I_\epsilon(\mu_\epsilon)
\;\ge\; -\, \int_{\bb T} \big(\, \mss b + \mss a\, F'\,\big) \, F' \; d\mu\;.
\end{equation*}

Fix a function $h$ in $L^2(\mu) \cap L^1(dx)$, such that $h\geq 0$,
$\<h/\mss a\>>0$, and let
$H = (1/2) \, \big(\frac{\mss b}{\mss a} - \frac{B\,h/\mss a}{\<
h\,/\,\mss a\>}\,\big)$. Note that $\< H \> =0$. Consider a sequence
$F_n$ of $C^2(\bb T)$ functions such that $F'_n$ converges to $-H$ in
$L^2(\mu)$ to conclude that
\begin{equation}
\label{prop:scale0:eq1}
\liminf_{\epsilon\to 0} \epsilon\, \ms I_\epsilon(\mu_\epsilon)
\;\ge\;  \frac{1}{4} \, \Big\{ \, \int_{\bb T} \frac{\mss b^2}{\mss a} \; d\mu
\;-\; \frac{\int_{\bb T} h^2 / \mss a \; d\mu}{\< h / \mss a\>^2}
\, B^2 \,  \Big\} \;\cdot
\end{equation}
To complete the proof, it remains to optimize in $h$. This is done in
the next paragraph.

Recall that we denote by $\mu_{\rm ac}$, $\mu_{\rm s}$ the absolutely
continuous and the singular part, respectively, of the measure $\mu$
with respect to the Lebesgue measure $\lambda$, and by $m_\mu$ the
Radon-Nikodym derivative of $\mu_{\rm ac}$ with respect to $\lambda$.

Since $\mu_{\rm s}$ is the singular part, there exists a measurable
set $A$ such that $\mu_{\rm s} (A) = 0$ and $\lambda (A) = 1$.  Let
$h_\epsilon = \chi_{A} \, (1+\epsilon)/(m_\mu + \epsilon)$. It is easy
to see that the function $h_\epsilon$ satisfies the hypothesis needed
to derive \eqref{prop:scale0:eq1}. By definition of $A$,
\begin{equation*}
\frac{\int_{\bb T} h_\epsilon^2 / \mss a \; d\mu}{
\< h_\epsilon / \mss a\>^2}
\;=\; \frac{\<\,m_{\mu}/(\,\mss a\,[m_{\mu}+\epsilon]^2\,)\,\>}
{\<\,1/(\,\mss a\,[m_{\mu}+\epsilon]\,)\,\>^2} \;\cdot
\end{equation*}
If $m_\mu^{-1}$ is Lebesgue integrable, by the monotone convergence
theorem, this expression converges to
$1\,/\,\< 1/\, \mss a \,m_\mu\>$. If it is not, observe that it is
bounded by
\begin{equation*}
\frac{\<1/(\,\mss a\,[m_{\mu}+\epsilon]\,)\,\>}
{\<\,1/(\,\mss a\,[m_{\mu}+\epsilon]\,)\,\>^2}
\;=\;  \frac{1}
{\<1/(\,\mss a\,[m_{\mu}+\epsilon]\,)\,\>} \;\to\;
0\;=\; \frac{1}
{ \< 1/\mss a\, m_\mu\>} \;\cdot
\end{equation*}
Thus, in both cases, the expression on the right-hand side of the
penultimate formula converges to $1/\< 1/\mss a m_\mu\>$ as
$\epsilon\to 0$. This shows that
\begin{equation*}
\liminf_{n\to \infty} \epsilon\, \ms I_\epsilon(\mu_\epsilon)
\;\ge\;  \frac{1}{4} \, \Big\{ \, \int_{\bb T} \frac{\mss b^2}{\mss a} \; d\mu
\;-\; \frac{1}{\< 1/\mss a m_\mu\>} \, B^2 \,  \Big\} \;,
\end{equation*}
and completes the proof of the $\Gamma-\liminf$.

\smallskip\noindent \emph{$\Gamma-\limsup$}.  Fix
$\mu\in \mss P(\bb T)$, and recall that we denote by $\mu_{\rm ac}$,
$\mu_{\rm s}$ the absolutely continuous and the singular part,
respectively, of the measure $\mu$ with respect to the Lebesgue
measure. Let $r = \mu_{\rm ac} (\bb T)$ to rewrite $\mu$ as
$\mu = r \mu^*_{\rm ac} + (1-r) \mu^*_{\rm s}$, where
$\mu^*_{\rm ac} = (1/r) \mu_{\rm ac} \in \mss P(\bb T)$,
$\mu^*_{\rm s} = [1/(1-r)] \mu_{\rm s} \in \mss P(\bb T)$.
We assume here that $0<r<1$, but the argument is simpler if $r(1-r)
=0$. 

Denote by $\mu_\epsilon$ the measure given by Corollary \ref{l05}
associated to $\mu^*_{\rm ac}$, and by $\nu_\epsilon$ the one given by
Corollary \ref{l06} associated to $\mu^*_{\rm s}$. By these results,
as $\epsilon \to 0$,
$\varrho_\epsilon = r \mu_\epsilon + (1-r) \nu_\epsilon$ converges
weakly to $r \mu^*_{\rm ac} + (1-r) \mu^*_{\rm s} = \mu$. On the other
hand, by convexity of $\ms I_\epsilon$,
$\ms I_\epsilon (\varrho_\epsilon) \le r\, \ms I_\epsilon (
\mu_\epsilon) + (1-r)\, \ms I_\epsilon (\nu_\epsilon)$.  Hence, by
Corollaries \ref{l05} and \ref{l06},
\begin{equation*}
\begin{aligned}
& \limsup_{n\to \infty} \epsilon\, \ms I_\epsilon(\varrho_\epsilon)
\;\le\;
\frac{r}{4} \, \Big\{ \, \int_{\bb T} \frac{\mss b^2}{\mss a} \; d\mu^*_{\rm ac}
\;-\; \frac{1}{\< 1/\mss a m_{\mu^*_{\rm ac}}\>} \, B^2 \,  \Big\}
\;+\; \frac{1-r}{4} \, \int_{\bb T} \frac{\mss b^2}{\mss a} \; d\mu^*_{\rm s} \\
&\qquad =\;
\frac{1}{4} \, \Big\{ \, \int_{\bb T} \frac{\mss b^2}{\mss a} \; d\mu
\;-\; \frac{1}{\< 1/r \mss a m_{\mu^*_{\rm ac}}\>} \, B^2 \,  \Big\}
\; =\;
\frac{1}{4} \, \Big\{ \, \int_{\bb T} \frac{\mss b^2}{\mss a} \; d\mu
\;-\; \frac{1}{\< 1/\mss a m_{\mu}\>} \, B^2 \,  \Big\}
\;.
\end{aligned}
\end{equation*}
This establishes the $\Gamma-\limsup$. Lemma \ref{shortest:scale}
below completes the proof of the proposition.
\end{proof}

We conclude this section showing that $\epsilon$ is the shortest
time-scale at which the rate functional $\ms J_\epsilon$ exhibits a
nontrivial limit in the sense of $\Gamma-$convergence.

\begin{lemma}\label{shortest:scale}
Let $(\beta_\epsilon: \epsilon>0)$ be any positive sequence such that
$\beta_\epsilon \prec \epsilon$. Then, $\beta_\epsilon \ms I_\epsilon$
$\Gamma-$converges to $0$.
\end{lemma}
\begin{proof}
Fix a sequence $\beta_\epsilon$ such that
$\beta_\epsilon\prec \epsilon$. We prove the $\Gamma-\liminf$ and
$\Gamma-\limsup$ inequalities separately.

\smallskip\noindent \emph{$\Gamma-\liminf$.} Fix $\mu\in \mss P(\bb T)$
and let $\mu_\epsilon\in \mss P(\bb T)$ be a sequence such that
$\mu_\epsilon$ converges weakly to $\mu$. By \eqref{1st_rate_f},
$\ms J^{(-1)}(\mu)<+\infty$. Then, by Proposition \ref{l01} and the
definition of $\Gamma-$convergence,
\begin{equation*}
\liminf_{\epsilon\to 0}\beta_\epsilon \,\ms I(\mu_\epsilon)
\,\geq\, \liminf_{\epsilon\to 0}
\frac{\beta_\epsilon}{\epsilon}\,\ms J^{(-1)}(\mu) \,=\, 0\,.
\end{equation*}

\smallskip\noindent \emph{$\Gamma-\limsup$.} Fix
$\mu\in \mss P(\bb T)$. By Proposition \ref{l01}, choose a sequence
$\mu_\epsilon\in \mss P(\bb T)$ such that $\mu_\epsilon$ converges
weakly to $\mu$ and
\begin{equation*}
\limsup_{\epsilon \to 0}\epsilon\ms I_{\epsilon}(\mu_\epsilon)
\,\leq \, \ms J^{(-1)}(\mu). 
\end{equation*}
Then, since $\beta_\epsilon \prec \epsilon$,
\begin{equation*}
\limsup_{\epsilon \to 0}\beta_\epsilon\ms I_\epsilon(\mu_\epsilon)
\leq \limsup_{\epsilon\to 0}\frac{\beta_\epsilon}{\epsilon}\ms J^{(-1)}(\mu) = 0,  
\end{equation*}
given that $\ms J^{(-1)}(\mu)< +\infty$.
\end{proof}

\section{Proof of Proposition \ref{s04}}
\label{sec0004}

In this section we establish Proposition \ref{s04}. The section is divided in two main parts.  In the first one, we address the
$\Gamma$–$\liminf$ of the functional $\ms I_\epsilon$. This is
Proposition \ref{p4.1}, whose proof is developed in three steps. In
Lemma \ref{s01}, we show that the zero-level set of the functional
$\ms J^{(-1)}$ consists of convex combinations of Dirac measures
supported on $\mc C$, the set of critical points. This observation
allows us to restrict our analysis to this convex hull. A key
ingredient in the proof of the $\Gamma-\liminf$ is the introduction of
local versions of the potential $S$ defined via cut-off
functions. However, this approach raises the technical challenge of
obtaining suitable estimates on the intervals containing the supports
of the derivatives of these cut-offs. Lemma \ref{s02} provides the
necessary bounds to overcome this difficulty. We conclude this part by
assembling all the intermediate results to complete the proof of the
$\Gamma$–$\liminf$ of $\ms I_\epsilon$.

In the second part of the section, we turn to the $\Gamma$–$\limsup$
of $\ms I_\epsilon$, as stated in Proposition \ref{p4_limsup}. The
main step is Lemma \ref{l4_limsup}, which constructs a recovery
sequence for the $\Gamma$–$\limsup$ corresponding to Dirac measures
supported on $\mc C$. By the convexity of $\ms I_\epsilon$, this
suffices to prove Proposition \ref{p4_limsup}. As in the previous
subsection, we verify that the constructed sequence is indeed a
recovery sequence by applying the formula from Lemma \ref{l03} to
estimate $\ms I_\epsilon$.

\subsection*{Second order $\Gamma-\liminf$}
\label{sec4.1}

The main result of this subsection reads as follows.

\begin{proposition}
\label{p4.1}   
Assume that condition {\rm (H)} is in force. Then, the functional
$\ms J^{(0)}$ introduced in \eqref{2nd_rate_f} is a $\Gamma-\liminf$ of
the sequence of rate functionals $\ms I_\epsilon$.
\end{proposition}

The proof requires some preliminary estimates. We first show that the
$0$-level set of the functional $\ms J^{(-1)}(\cdot)$ is exactly the set of all convex
combinations of the Dirac measures $\delta_x$, $x\in \mc C$.

\begin{lemma}
\label{s01}
A probability measure $\mu \in \mss P(\bb T)$ belongs to the $0$-level
set of the functional $\ms J^{(-1)}(\cdot)$ if, and only if,
$\mu(\cdot) = \sum_{z\in \mc C}\omega(z)\, \delta_z(\cdot)$ for some
$\omega\in \mss P(\mc C)$.
\end{lemma}

\begin{proof}
Assume that
$\mu(\cdot) = \sum_{z\in \mc C}\omega(z)\, \delta_z(\cdot)$ for some
$\omega\in \mss P(\mc C)$, and recall the notation introduced in
\eqref{1st_rate_f}. Clearly $m_\mu=0$, and the second  term on the right-hand
side of \eqref{1st_rate_f} vanishes. The first one also vanishes because $\mss
b(x)=0$ for all $x\in \mc C$.

Conversely, suppose that $\ms J^{(-1)}(\mu) =0$ for some
$\mu \in \mss P(\bb T)$. We have to show that $\mu(\mc C^c)=0$.

By Schwarz inequality,
\begin{equation}
\label{07}
\int_{\bb T} \frac{\mss b^2}{\mss a} \; d\mu_{\rm ac}
\;-\;  \frac{1}{\< 1/(\mss am_\mu)\>} \, B^2 \,\ge\, 0 \;.
\end{equation}
Hence, if $\ms J^{(-1)}(\mu) =0$, the previous difference vanishes and
\begin{equation*} 
\int_{\bb T} \frac{\mss b^2}{\mss a} \; d\mu_{\rm s} \,=\, 0\,.
\end{equation*}
Since $\mss b(x)^2/\mss a(x) >0$ for $x\not\in \mc C$,
$\mu_{\rm s} (\mc C^c)=0$.

We turn to the absolutely continuous part $\mu_{\rm ac}$.  If
$\< 1/m_\mu\> =+\infty$, since the difference \eqref{07} vanishes, the
first integral vanishes. Therefore, as $\mss b(x)^2/\mss a(x) >0$ for
$x\not\in \mc C$, $\mu_{\rm ac} (\mc C^c)=0$, which completes the
proof of the lemma in the case $\< 1/m_\mu\> =+\infty$.

It remains to consider the case $\< 1/m_\mu\> <+\infty$. As  the
difference \eqref{07} vanishes (and Schwarz inequality becomes an
identity), there exists $c\in \bb R$ such that
$\mss b \sqrt{m_\mu} / \sqrt{\mss a} = c\, 1/\sqrt{\mss a\, m_\mu}$
Lebesgue a.s. The constant $c$ can not be $0$ because we assumed that
$\< 1/m_\mu\> <+\infty$. If $c\neq 0$, then, as $m_\mu>0$ Lebesgue
a.s., $\mss b = c/m_\mu$ a.s., which is a contradiction because $\mss
b$ changes sign while $m_\mu$ is always positive. This completes the
proof of the lemma.
\end{proof}

Let $\cb{B_r(x)}\subset \bb T$ be the ball of radius $r>0$ centered at
$x\in \bb T$, and, for a finite subset $A$ of $\bb T$, set
$\cb{B_r(A)}: = \cup_{x\in A} B_r(x)$. Denote by $\cb{r_0}$ the
minimal distance between points in $\mc C$. The proof of the
$\Gamma-\liminf$ relies on the following estimate.

\begin{lemma}
\label{s02}
Fix $\mu \in \mss P(\bb T)$, and a sequence
$(\mu_{\epsilon})_{\epsilon>0}$, $\mu_\epsilon \in \mss P(\bb T)$,
converging weakly to $\mu$. Suppose that
\begin{equation*}
\liminf_{\epsilon \to 0}\ms I_{\epsilon}(\mu_{\epsilon}) \,<\,
\infty \,.
\end{equation*}
Then, for all $r<r_0/3$, $x\in \mc C$,
\begin{equation*}
\label{bound_int_1st}
\liminf_{\epsilon \to 0} \frac{1}{\epsilon}\, \mu_{\epsilon} (\bb T
\setminus B_r(\mc C)\,\big) \,<\, \infty\,.
\end{equation*} 
\end{lemma}

\begin{proof}
By the definition of the rate functional $\ms I_{\epsilon}$ and by
hypothesis, for any functions $F$ in $C^2(\bb T)$
\begin{align*}
\liminf_{\epsilon \to 0}
- \int_{\bb T} \frac{\ms L_{\epsilon} e^{F(x)/\epsilon}}
{e^{F (x)/\epsilon}}
\;d\mu_{\epsilon}(x)
\,\le\,
\liminf_{\epsilon \to 0}\ms I_{\epsilon}(\mu_{\epsilon}) \,<\, \infty\,.
\end{align*}
Thus, by \eqref{26}, there exists a finite constant $C_0$ such
that
\begin{align*}
\liminf_{\epsilon \to 0} \Big\{
-\, \frac{1}{\epsilon} \, \int_{\bb T}
\big[\, \mss b(x) \, +\, \mss a(x) \, F'(x)\, \big]
F'(x)\;d\mu_{\epsilon}(x) \;-\; \int_{\bb T}\mss
a(x)\, F''(x)\;d\mu_{\epsilon}(x) \,\Big\}\,\le\, C_0\,.
\end{align*}
As $F$ belongs to $C^2(\bb T)$, the second integral is bounded so that
\begin{align}
\label{11}
\liminf_{\epsilon \to 0} \Big\{
-\, \frac{1}{\epsilon} \, \int_{\bb T} \mss a(x) \,
\big[\, \mss c(x) \, +\,  F'(x)\, \big]
F'(x)\;d\mu_{\epsilon}(x) \,\Big\}\,\le\, C_0\,,
\end{align}
for a new constant $C_0$ whose value may change from line to line and
which may depend on the function $F$. Here
$\mss c(x) = \mss b(x) /\mss a(x)$.

Fix $x\in \bb T$, and suppose that $\mss c(x)\ge 0$. In this case,
$[\, \mss c(x) \, +\, F'(x)\, \big] F'(x)\le 0$ if, and only if,
$F'(x) \in [-\mss c(x) , 0]$. If $\mss c(x)\le 0$, then
$[\, \mss c(x) \, +\, F'(x)\, \big] F'(x)\le 0$ if, and only if, $F'(x)
\in [0, - \mss c(x)]$. Therefore, to fulfill the inequality, $F$ has
to decrease in intervals where $\mss b(\cdot)$ is positive and to
increase in the intervals where $\mss b(\cdot)$ is negative. In
consequence, the test function $F(\cdot)$ may not have support in an
interval formed by two consecutive critical points of $S(\cdot)$, and
$F'(\cdot)$  has to vanish at these critical points.

In view of this observation, the construction of the test functions
$F$ is clear. Fix three consecutive critical points
$x_{\mtt l} < x_0 < x_{\mtt r}$, and $\delta>0$. Assume, without loss
of generality, that $\mss c (x)\le 0$ for $x\in [x_{\mtt l}, x_0]$,
and, in consequence, $\mss c (x)\ge 0$ for $x\in [x_0, x_{\mtt
r}]$. As $\mss c(\cdot)$ is continuous, there exists $\eta>0$ such
that $\mss c(x) \le -\eta$ for all
$x\in [x_{\mtt l}+\delta, x_0-\delta]$, and $\mss c(x) \ge \eta$ for
all $x\in [x_0+\delta , x_{\mtt r}-\delta]$.

Fix a point $z_{\mtt l}\in [x_{\mtt l}+2\delta , x_0 - 5\delta]$ and a point
$z_{\mtt r}\in [x_0+5\delta , x_{\mtt r} - 2\delta]$. Let
$F\colon \bb T \to \bb R$ be given by
\begin{equation*}
F(x) \,=\, (\eta/2)\, (x- z_{\mtt l})\, \quad \text{for}\;\; x\in
[z_{\mtt l}  , z_{\mtt l} + 3\delta] \,.
\end{equation*}
In the interval $[z_{\mtt l}+ 3\delta , z_{\mtt r} - 3\delta]$, $F$ is
set to be constant, and
\begin{equation*}
F(x) \,=\, (\eta/2)\, ( z_{\mtt r}  - x)\,
\quad \text{for}\;\; x\in
[z_{\mtt r}-3\delta , z_{\mtt r}]\,.
\end{equation*}
Complete the definition of $F(\cdot)$ by setting $F(y)=0$ for
$y\not\in [z_{\mtt l}  , z_{\mtt r} ]$. In this way,
$F(\cdot)$ is a continuous function.

By definition of $F$, on the interval
$[z_{\mtt l} , z_{\mtt l} +3\delta]$, since $-\, \mss c(\cdot)$ is
bounded below by $\eta$ on this interval, $\mss c(x) \, +\,  F'(x) \le
-\, \eta/2$, so that
\begin{equation}
\label{12}
-\, \mss a(x) \,
\big[\, \mss c(x) \, +\,  F'(x)\, \big] \, F'(x)
\,\ge\, (\eta/2)^2 \, \inf_{x\in \bb T} \mss a(x)\,.
\end{equation}
A similar bound holds in the interval
$[z_{\mtt r}- 3\delta , z_{\mtt r} ]$ because $\mss c(\cdot)$ is
bounded below by $\eta$. On the complement of these intervals $F'$
vanishes. 

If $F(\cdot)$ were in the domain of the generator, we could directly
substitute it into equation \eqref{11} to complete the proof. However,
since $F'(\cdot)$ is discontinuous we must smooth $F(\cdot)$ in a way
that preserves the previously established bounds in order to complete
the proof of the lemma.

Let $\cb{\Phi}\colon [-1,1]\to \bb R_+$ be a smooth approximation of
the identity: the support of $\Phi(\cdot)$ is contained in the set
$(-1,1)$, and $\int \Phi(x)\, dx=1$. Let
$\cb{\Phi_a(x)}: = (1/a) \, \Phi(x/a)$, $0<a<1$.

Let $F_\delta = F * \Phi_\delta$. The function $F_\delta$ is smooth,
and thus belongs to the domain of the generator.  We examine the
left-hand side of \eqref{12} with $F_\delta $ in place of $F$. By
definition of $F(\cdot)$,
$F'_\delta (x) = \int F'(x-y)\, \Phi_\delta (y)\, dy$ vanishes in the
complement of
$[z_{\mtt l}- \delta , z_{\mtt l} +4 \delta] \cup [z_{\mtt r}- 4\delta
, z_{\mtt r} +\delta ]$. On the intervals
$[z_{\mtt l}- \delta , z_{\mtt l} + \delta] $,
$[z_{\mtt l} + 2\delta , z_{\mtt l} +4 \delta] $, as
$-\, \mss c(x) \ge \eta$, and $0\le F'_\delta (x) \le \eta/2$, the
left-hand side of \eqref{12} with $F_\delta $ in place of $F$ is
positive. On the interval
$[z_{\mtt l} + \delta , z_{\mtt l} +2 \delta] $,
$F_\delta (x) = F'(x)$ so that \eqref{12} holds for $F_\delta$. A
similar analysis can be carried out on the interval
$[z_{\mtt r}- 4\delta , z_{\mtt r} +\delta ]$. Therefore,
\begin{align*}
& -\, \, \int_{\bb T} \mss a(x) \,
\big[\, \mss c(x) \, +\,  F_\delta'(x)\, \big]
F_\delta'(x)\;d\mu_{\epsilon}(x)
\\
& \quad \,\ge\,
(1/2) \, \eta^2 \, \inf_{x\in \bb T} \mss a(x)
\, \big\{ \, \mu_\epsilon ([z_{\mtt l} + \delta , z_{\mtt l} +2
\delta])
\,+\,  \mu_\epsilon ([z_{\mtt r} -2 \delta , z_{\mtt r} - \delta])  \,\big\}
\,.
\end{align*}

As the previous intervals have length $\delta$, choosing $z_{\mtt l}$,
$z_{\mtt r}$ appropriately yields that
\begin{align*}
& \sum_{j=0}^{N-1}
\, \big\{ \, \mu_\epsilon ([ \sigma_j + 5\delta , m_j - 5\delta ])
\,+\,  \mu_\epsilon ([  m_j + 5\delta , \sigma_{j+1} - 5\delta])
\,\big\}
\\
& \quad  \le\, \frac{-\, 2}{\delta \, \eta^2}
\, \frac{1}{\inf_{x\in \bb T} \mss a(x)}
\, \int_{\bb T} \mss a(x) \,
\big[\, \mss c(x) \, +\,  F_\delta'(x)\, \big]
F_\delta'(x)\;d\mu_{\epsilon}(x)
\,.
\end{align*}
This completes the proof of the lemma in view of \eqref{11}.
\end{proof}

Given two real valued functions $f,g$, where $g$ is positive, we say
that $\cb{f(x)=O(g(x))}$, if there exist some positive constant $C>0$
such that $|f(x)|\leq Cg(x)$, for every $x$ in the domain of $f$. If
$f$ is also positive, we write $\cb{f(x) = \Theta(g(x))}$ if $f$ and
$g$ have the same domain, and both $f(x)=O(g(x))$ and $g(x)=O(f(x))$
hold. We are now in position to prove Proposition \ref{p4.1}

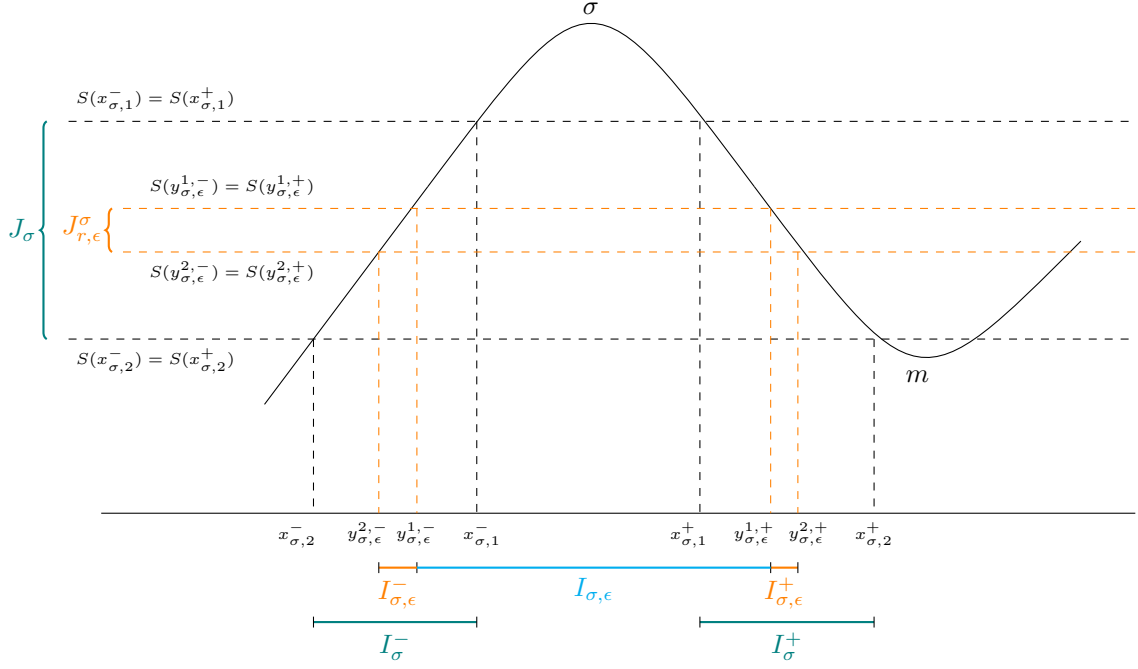
\begin{figure}
\centering
\begin{tikzpicture}[scale=1.44, xshift=-5cm]
\draw[rounded corners] (0.5,3) -- (2,5);
\draw[rounded corners] (2,5) .. controls (3.5,7) .. (5,5);  
\draw[rounded corners] (5,5) .. controls (6.5, 3) .. (8, 4.5);

\fill(3.5,6.5)node[above]{$\sigma$};
\fill(6.5,3.4)node[below]{$m$};

\draw[solid](-1,2)--(8.5,2);





\fill(-0.5,3.2)node[above, font=\tiny]{$S(x_{\sigma, 2}^-) = S(x_{\sigma,2}^+)$};

\fill(-0.5,5.6)node[above, font=\tiny]{$S(x_{\sigma, 1}^-)=S(x_{\sigma,1}^+)$};

\fill(0.2, 4.4)node[below, font=\tiny]{$S(y_{\sigma,\epsilon}^{2,-})=S(y_{\sigma,\epsilon}^{2,+})$};

\fill(0.2,4.8)node[above,font=\tiny]{$S(y_{\sigma,\epsilon}^{1,-})=S(y_{\sigma,\epsilon}^{1,+})$};

\draw[decorate, decoration={brace}, teal, thick]  (-1.5,3.6) -- node[left, teal] {$J_\sigma$}  (-1.5,5.6);

\draw[decorate, decoration={brace}, orange, thick]  (-0.9,4.4) -- node[left, orange] {$J_{r,\epsilon}^{\sigma}$}  (-0.9,4.8);

\draw[dashed](-1.3,3.6)--(8.5,3.6);
\draw[dashed](-1.3,5.6)--(8.5,5.6);

\draw[dashed](0.95,3.6)--(0.95,2);
\draw[dashed](2.45,5.6)--(2.45,2);

\fill(0.8,2)node[below, font=\tiny]{$x_{\sigma,2}^-$};
\fill(2.5,2)node[below, font=\tiny]{$x_{\sigma,1}^-$};

\draw[dashed](6.1,3.6)--(6.1,2);
\draw[dashed](4.5,5.6)--(4.5,2);

\draw[dashed, orange](5.4,4.4)--(5.4,2);
\draw[dashed, orange](5.15,4.8)--(5.15,2);

\draw[dashed, orange](-0.8,4.4)--(8.5,4.4);
\draw[dashed, orange](-0.8,4.8)--(8.5,4.8);

\draw[dashed, orange](1.55,4.4)--(1.55,2);
\draw[dashed, orange](1.9,4.8)--(1.9,2);

\fill(4.4,2)node[below, font=\tiny]{$x_{\sigma,1}^+$};
\fill(6.1,2)node[below, font=\tiny]{$x_{\sigma, 2}^+$};

\fill(5.5,2)node[below, font=\tiny]{$y_{\sigma,\epsilon}^{2,+}$};
\fill(5,2)node[below, font=\tiny]{$y_{\sigma, \epsilon}^{1,+}$};

\draw[solid, teal, thick]  (4.5,1) -- node[below, teal] {$I_\sigma^+$}  (6.1,1);

\draw[solid](6.1,0.95)--(6.1,1.05);
\draw[solid](4.5,0.95)--(4.5,1.05);

\draw[solid, orange, thick]  (5.15,1.5) -- node[below, orange] {$I_{\sigma,\epsilon}^+$}  (5.4,1.5);

\draw[solid](5.4,1.45)--(5.4,1.55);
\draw[solid](5.15,1.45)--(5.15,1.55);

\draw[solid](1.9,1.45)--(1.9,1.55);

\draw[solid, cyan, thick]  (1.9,1.5) -- node[below, cyan] {$I_{\sigma,\epsilon}$}  (5.15,1.5);



\fill(1.45,2)node[below, font=\tiny]{$y_{\sigma,\epsilon}^{2,-}$};

\fill(1.9,2)node[below, font=\tiny]{$y_{\sigma,\epsilon}^{1,-}$};

\draw[solid, teal, thick]  (0.95,1) -- node[below, teal] {$I_\sigma^-$}  (2.45,1);

\draw[solid, orange, thick]  (1.55,1.5) -- node[below, orange] {$I_{\sigma,\epsilon}^-$}  (1.9,1.5);

\draw[solid](1.55,1.45)--(1.55,1.55);
\draw[solid](0.95,0.95)--(0.95,1.05);
\draw[solid](2.45,0.95)--(2.45,1.05);

\end{tikzpicture}
\caption{Here we illustrate the intervals constructed in the proof of
Proposition \ref{p4.1}. The intervals in orange satisfy
$\mu_\epsilon(I_{\sigma, \epsilon}^- \cup I_{\sigma, \epsilon}^+) =
O(\epsilon^2)$. The union of these sets contain the support of the
derivative of the cut-off function $\Phi_{\sigma,\epsilon}$. More
precisely, $\Phi_{\sigma, \epsilon}$ is strictly increasing on
$I_{\sigma, \epsilon}^-$, strictly decreasing on
$I_{\sigma, \epsilon}^+$, it is equal to $1$ on $I_{\sigma,\epsilon}$,
and it vanishes on
$(I_{\sigma, \epsilon}^- \cup I_{\sigma,\epsilon} \cup I_{\sigma,
\epsilon}^+)^c$.}
\label{fig001}
\end{figure}

\begin{proof}[Proof of Proposition \ref{p4.1}]
Fix $\mu \in \mss P(\bb T)$ and
$(\mu_\epsilon)_{\epsilon>0}\subset \mss P(\bb T)$ such that
$\mu_\epsilon$ converges weakly to $\mu$. If
$\ms J^{(0)}(\mu)=+\infty$, then by \eqref{2nd_rate_f} and Lemma
\ref{s01}, $\ms J^{(-1)}(\mu)>0$. Thus, by Proposition \ref{l01} and
the definition of $\Gamma-$convergence,
\begin{equation*}
\liminf_{\epsilon\to 0}\ms I_\epsilon(\mu_\epsilon)
\;\geq\; \lim_{\delta  \to 0} \delta^{-1}
\liminf_{\epsilon\to 0}\epsilon\, \ms I_\epsilon(\mu_\epsilon)
\;\geq\; \lim_{\delta \to 0} \delta^{-1}\ms J^{(-1)} (\mu) \;=\; +
\infty \,=\, \ms J^{(0)}(\mu) \,.
\end{equation*}
This proves the $\Gamma-\liminf$ for measures $\mu$ such that $\ms
J^{(0)}(\mu)=+\infty$.  

We turn to the case where $\ms J^{(0)}(\mu)<+\infty$. Suppose 
that
\begin{equation*}
\label{liminf1:mu}
\mu(\cdot) = \sum_{z\in \mc C}\omega(z)\delta_z(\cdot),
\end{equation*}
for some $\omega\in \mss P(\mc C)$. Without loss of generality, we may
also assume that
\begin{equation}
\label{prop:liminf1:asump}
\liminf_{\epsilon\to 0}\ms I_\epsilon(\mu_\epsilon) < +\infty,
\end{equation}
otherwise the result trivially holds. By the definition of the
rate functional $\ms I_\epsilon$, for any function $F\in C^2(\bb T)$
\begin{equation*}
\liminf_{\epsilon\to 0}\ms I_\epsilon(\mu_\epsilon)
\,\geq\, \liminf_{\epsilon \to 0}-\int_{\bb T}
\frac{\ms L_\epsilon e^{F(x)/\epsilon}}{e^{F(x)/\epsilon}}d\mu_\epsilon(x)\,.
\end{equation*}
By \eqref{26}, and the definition of the functional $\ms J^{(0)}$, it
is enough to find a sequence of test functions
$F_\epsilon \in C^{2}(\bb T)$, such that
\begin{equation}
\label{prop:liminf1:eq0}
\liminf_{\epsilon\to 0}-\int_{\bb T}
\Big\{\, \frac{1} {\epsilon} \, \big[\,
\mss b(x)+\mss a(x)F'_\epsilon(x) \,\big]
\, F'_\epsilon(x) \, + \, \mss a(x)\, F''_\epsilon(x)\,
\Big\}\, d\mu_\epsilon 
\,\geq\, \sum_{z\in \mc W}\omega(x)\, \mss b'(z)\,.
\end{equation}

To obtain the derivative $\mss b'$ in the limit, for each
$\sigma\in \mc W$ we define a local version of the potential $S$ with
the help of a cut-off function $\Phi_{\epsilon,\sigma}$ whose support
is an interval containing $\sigma$, and no other critical point. The
test function is constructed as a linear combination of these local
potentials.  Because of the factor $\epsilon^{-1}$ appearing in the
first term of the integral in \eqref{prop:liminf1:eq0}, and the
presence of the second derivative of the test function in the second
term of the same integral, the interval containing the support of the
derivative of the cut-off function needs to be carefully defined. Mind
that, by the definition of $\ms J^{(0)}$, we can neglect the local
minima.

We now proceed to the construction of the test function. We start with
the construction of the intervals needed in the definition of the
cut-off function, see Figure \ref{fig001} above. Fix
$\sigma\in \mc W$. Choose four points
$x_{\sigma, 1}^{-}, x_{\sigma, 2}^{-}, x_{\sigma, 1}^{+}, x_{\sigma,
2}^{+}\in (0,1)$ such that (a)
$x_{\sigma, 2}^{-} < x_{\sigma, 1}^{-} < \sigma < x_{\sigma, 1}^{+} <
x_{\sigma, 2}^{+}$, (b) $S(x_{\sigma,2}^{-}) = S(x_{\sigma,2}^{+})$,
$S(x_{\sigma,1}^{-})=S(x_{\sigma,1}^{+})$, and (c) the only critical
point in $[x_{\sigma,2}^{-}, x_{\sigma,2}^{+}]$ is $\sigma$:
\begin{equation}
\label{prop:liminf1:test1}
[x_{\sigma,2}^{-}, x_{\sigma,2}^{+}]\cap \mc C = \{\sigma\}\,.
\end{equation}

Let $I_\sigma^{-}, I_\sigma^{+}\subset (0,1)$ be the intervals defined
by $I_\sigma^- := [x_{\sigma,2}^{-}, x_{\sigma,1}^{-}]$ and
$I_\sigma^+ := [x_{\sigma,1}^{+}, x_{\sigma,2}^{+}]$. Since $\sigma$
is a point of local maxima, $S$ is strictly increasing on $I_\sigma^-$
while strictly decreasing on $I_\sigma^+$. Then,
$S(I_\sigma^-) = S(I_\sigma^+)$. Denote this interval by $J_\sigma$,
\begin{equation}
\label{prop:liminf1:imageint}
J_\sigma := S(I_\sigma^-) = S(I_\sigma^+). 
\end{equation}
By \eqref{prop:liminf1:asump}, \eqref{prop:liminf1:test1} and Lemma
\ref{s02},
\begin{equation}
\label{liminf1:supp1}
\mu_\epsilon(I_\sigma^-\cup I_\epsilon^+) = O(\epsilon).
\end{equation}

Partition the interval $J_\sigma$ in $\lceil1/\epsilon\rceil$ sub
intervals of length $\lambda (J_\sigma) / \lceil1/\epsilon\rceil$,
where, recall, $\lambda(\cdot)$ stands for the Lebesgue measure.
Denote these intervals by $ J_{r,\epsilon}^\sigma$ so that
\begin{equation*}
\label{prop4.1:eq2}
J_\sigma = \bigsqcup_{r =
1}^{\lceil1/\epsilon\rceil} J_{r,\epsilon}^\sigma\,. 
\end{equation*} 
By \eqref{prop:liminf1:imageint} and \eqref{liminf1:supp1}, choose
some interval $J_{r,\epsilon}^\sigma$ in the partition such that
\begin{equation}
\label{liminf1:suppcut1}
\mu_\epsilon \big( \, [\, I^-_{\sigma} \cup I^+_{\sigma} \,] 
\, \cap \, \Pi(S^{-1}(J_{r,\epsilon}^\sigma)) \, \big)
\,=\, O(\epsilon^2),
\quad \lambda(J_{r,\epsilon}^\sigma) = \lambda(J_\sigma) /
\lceil1/\epsilon\rceil\,, 
\end{equation}
where $\Pi$ is the projection defined in \eqref{proj_1}. Let
$I_{\sigma,\epsilon}^- \subset I_\sigma^-$ and
$I_{\sigma, \epsilon}^+ \subset I_\sigma^+$ be the intervals such that
\begin{equation*}
I_{\sigma,\epsilon}^- \cup I_{\sigma,\epsilon}^+ \,=\,
(I^-_{\sigma} \cup I^+_{\sigma} )
\, \cap \,  \Pi(S^{-1}(J_{r,\epsilon}^\sigma))\,.
\end{equation*}
By \eqref{liminf1:suppcut1} and since the function $S(\cdot)$ is
Lipschitz,
\begin{equation}
\label{liminf1:suppcut2}
\mu_\epsilon(I_{\sigma, \epsilon}^{\pm}) = O(\epsilon^2)\,,
\quad \lambda(I_{\sigma, \epsilon}^{\pm}) = \Theta(\epsilon)\,,
\end{equation}
where the notation $O(\cdot)$, $\Theta(\cdot)$ has been introduced
just before the beginning of the proposition's proof.
Let $y_{\sigma, \epsilon}^{1,\pm}$,
$y_{\sigma, \epsilon}^{2,\pm} \in (0,1)$ be the endpoints of the
intervals $I_{\sigma, \epsilon}^{\pm}$,
$I_{\sigma, \epsilon}^{-} = [y_{\sigma, \epsilon}^{2,-}, y_{\sigma,
\epsilon}^{1,-}]$,
$I_{\sigma, \epsilon}^{+} = [y_{\sigma, \epsilon}^{1,+}, y_{\sigma,
\epsilon}^{2,+}]$, so that
\begin{equation}
\label{liminf1:interv}
x_{\sigma, 2}^- < y_{\sigma, \epsilon}^{2,-}
< y_{\sigma, \epsilon}^{1,-} < x_{\sigma, 1}^{-}
< \sigma < x_{\sigma, 1}^{+} < y_{\sigma, \epsilon}^{1,+}
< y_{\sigma, \epsilon}^{2,+} < x_{\sigma, 2}^{+}\,.
\end{equation}

Denote by $I_{\sigma,\epsilon}$ the interval given by
$\cb{I_{\sigma,\epsilon}} := [y_{\sigma, \epsilon}^{1,-}, y_{\sigma,
\epsilon}^{1,+}]$, and let $\Phi_{\sigma, \epsilon}$ be some
nonnegative smooth function satisfying the following properties

\begin{enumerate}
\item $\Phi_{\sigma, \epsilon}(x) = 1$, for every
$x\in I_{\sigma, \epsilon}$;

\item $\Phi_{\sigma, \epsilon}$ is strictly increasing on
$I_{\sigma,\epsilon}^-$ while strictly decreasing on
$I_{\sigma,\epsilon}^+$;

\item
$\sup_{x\in I_{\sigma, \epsilon}^{-} \cup I_{\sigma,
\epsilon}^{+}}|\Phi_{\sigma, \epsilon}'(x)| = \Theta(\epsilon^{-1})$;

\item
$\sup_{x\in I_{\sigma, \epsilon}^{-} \cup I_{\sigma,
\epsilon}^{+}}|\Phi_{\sigma, \epsilon}''(x)| = \Theta(\epsilon^{-2})$;
and

\item $\Phi_{\sigma,\epsilon}(x) = 0$, for every
$x\in (I_{\sigma,\epsilon}\cup I_{\sigma,\epsilon}^-\cup I_{\sigma,
\epsilon}^+)^c$.
\end{enumerate}

Fix $c_{\sigma,\epsilon}\in J_{r,\epsilon}^\sigma$ and let
$S_{\sigma, \epsilon}: \bb R\to \bb R$ be the translated potential
$\cb{S_{\sigma, \epsilon}(x)} := S(x) - c_{\sigma,\epsilon}$.  By
\eqref{liminf1:suppcut1},
\begin{equation}
\label{liminf1:bound:pot}
\sup_{x\in I_{\sigma, \epsilon}^{-}\cup I_{\sigma,\epsilon}^{+}}
|S_{\sigma,\epsilon}(x)| = O(\epsilon)\,.
\end{equation}
Clearly, the product
$\Phi_{\sigma, \epsilon}\cdot S_{\sigma,\epsilon}$ defines an element
in $C^2(\bb T)$. We perform the same construction for each local
maxima $\sigma\in \mc W$ and let $F_\epsilon: \bb T \to \bb R$ be the
function
\begin{equation}
\label{prop4.3:eq5}
F_\epsilon(x) := \sum_{\sigma \in \mc W}\Phi_{\sigma, \epsilon}(x)
\, S_{\sigma, \epsilon}(x)\,.
\end{equation}

By linearity, $F_\epsilon$ is contained in the domain of
$\ms L_{\epsilon}$. It remains to show that the sequence $F_\epsilon$
satisfies \eqref{prop:liminf1:eq0}. By \eqref{prop4.3:eq5}, the left
hand side of \eqref{prop:liminf1:eq0} is equal to
\begin{align}
\label{prop4.3:eq6}
&- \sum_{\sigma \in \mc W}\;\int_{\bb T}
\left\{\frac{\mss b(x)\,+\, \mss a(x)
\left(\Phi_{\sigma, \epsilon} \cdot
S_{\sigma,\epsilon}\right)'(x)}{\epsilon}\right\}
\left(\Phi_{\sigma, \epsilon}\cdot S_{\sigma, \epsilon}\right)'(x)\;
d\mu_{\epsilon}(x) \notag
\\
&- \sum_{\sigma \in \mc W}\int_{\bb T} \mss a(x)
\left(\Phi_{\sigma, \epsilon}\cdot S_{\sigma, \epsilon}\right)''(x)
\;d\mu_{\epsilon}(x)\,.
\end{align} 
We claim that the first sum of \eqref{prop4.3:eq6} vanishes as
$\epsilon \to 0$, while the $\liminf$ of the second sum is bounded below
by $ \sum_{z\in \mc W}\omega(x)\mss b'(z)$. We start with the first
sum. Fix $\sigma \in \mc W$. By item (1) in the definition of
$\Phi_{\sigma, \epsilon}$ and the fact that
$S_{\sigma,\epsilon}' = - \mss b/\mss a$, the expression inside the
brackets vanishes on $I_{\sigma,\epsilon}$. By item (5) in the
definition of $\Phi_{\sigma, \epsilon}$, the remaining factor vanishes
on
$\big(I_{\sigma,\epsilon}\cup I_{\sigma,\epsilon}^-\cup I_{\sigma,
\epsilon}^+\big)^{c}$. Then, it is enough to show that
\begin{equation*}
\lim_{\epsilon\to 0} 
\int_{I_{\sigma,\epsilon}^-\cup I_{\sigma,\epsilon}^+}
\left\{\frac{\mss b(x)\,+\, \mss a(x)
\left(\Phi_{\sigma, \epsilon} \cdot
S_{\sigma, \epsilon}\right)'(x)}{\epsilon}\right\}
\left(\Phi_{\sigma, \epsilon}\cdot S_{\sigma, \epsilon}\right)'(x)\;
d\mu_{\epsilon}(x)  \,=\, 0 \,.
\end{equation*}
Calculate the derivatives. By the triangle inequality, the previous
integral is bounded above by
\begin{equation*}
\label{prop4.3:eq7}
\begin{aligned}
&\frac{1}{\epsilon}\;\, \int_{I_{\sigma,\epsilon}^-\cup I_{\sigma,\epsilon}^+}
\frac {1} {\mss a(x) }\,   \mss b(x)^2 \, \big[\, 1 \,-\,  \Phi_{\sigma,
\epsilon}(x)\,\big] \, 
\Phi_{\sigma, \epsilon}(x) \, \;d\mu_{\epsilon}(x)
\\
&+ \frac{2}{\epsilon}\;
\int_{I_{\sigma,\epsilon}^-\cup I_{\sigma,\epsilon}^+}
\Big\{\, \big|\, \mss b(x) \, \Phi_{\sigma, \epsilon}'(x)
\, \Phi_{\sigma, \epsilon}(x) \, S_{\sigma, \epsilon}(x)  \,  \big|
\,+\,
\big|\, \mss b(x)\, \Phi'(x)\,  S_{\sigma, \epsilon}(x) \,  \big|\,
\Big\}
\;d\mu_{\epsilon}(x)
\\
& + \frac{1}{\epsilon}\; \int_{I_{\sigma,\epsilon}^-\cup
I_{\sigma,\epsilon}^+}
\mss a(x) \,  \big[\, \Phi_{\sigma, \epsilon}'(x)\,
S_{\sigma, \epsilon}(x)\, \big]^2 \;d\mu_{\epsilon}(x)\,.
\end{aligned} 
\end{equation*}

Since $\mss b, \mss a$ are bounded, $\mss a$ is strictly positive, and
$0\leq \Phi \leq 1$, the previous expression is bounded above by
\begin{equation*}
\frac{C_0}{\epsilon}\; \Big\{
\mu_{\epsilon}(I_{\sigma,\epsilon}^-\cup I_{\sigma,\epsilon}^+)
\,+\,  \int_{I_{\sigma,\epsilon}^-\cup I_{\sigma,\epsilon}^+}
\Big[ \, \big| \, \Phi_{\sigma, \epsilon}'(x)\,
S_{\sigma, r(\epsilon)}(x) \, \big| \,+\,
\big[ \, \Phi_{\sigma, \epsilon}'(x)\,
S_{\sigma, r(\epsilon)}(x) \, \big]^2 \,\Big] \;d\mu_{\epsilon}(x)
\,\Big\}  
\end{equation*}
for some finite constant $C_0$. This expression is less than or equal
to
\begin{align*}
\label{prop4.3:eq8}
\frac{C_0}{\epsilon}\; \{ \, 1\, + a^2_\epsilon\}\,
\mu_{\epsilon}(I_{\sigma,\epsilon}^-\cup I_{\sigma,\epsilon}^+)
\;\;\text{where}\;\;
a_\epsilon \,=\,
\sup_{x\in I_{\sigma,\epsilon}^-\cup
I_{\sigma,\epsilon}^+}|\Phi_{\sigma,\epsilon}'(x)|\;\sup_{x\in
I_{\sigma,\epsilon}^-\cup I_{\sigma,\epsilon}^+}
|S_{\sigma,r(\epsilon)}(x)| \,.
\end{align*}
By \eqref{liminf1:suppcut2}, \eqref{liminf1:bound:pot} and item (3) in
the definition of $\Phi_{\sigma, \epsilon}$, this expression vanishes
as $\epsilon$ goes to $0$. This proves that the first term in
\eqref{prop4.3:eq6} converges to $0$ as $\epsilon \to 0$.

We turn to the second sum in \eqref{prop4.3:eq6}. Fix
$\sigma\in \mc W$. By the definitions of $\Phi_{\sigma,\epsilon}$ and
$S_{\sigma,\epsilon}$, the term in the sum corresponding to $\sigma$
is equal to
\begin{equation*}
\label{prop4.3:eq9}
\begin{aligned}
&\int_{I_{\sigma,\epsilon}}
\Big\{\, \mss b'(x)
\,-\, \frac{\mss b(x)\, \mss a'(x)}{\mss a(x)} \,\Big\}
\;d\mu_{\epsilon}(x)
\,-\,  \int_{I_{\sigma,\epsilon}^-\cup I_{\sigma,\epsilon}^+}
\mss a(x) \, \Phi_{\sigma, \epsilon}''(x)\, S_{\sigma, \epsilon}(x) \;d\mu_{\epsilon}(x)
\\
& + \int_{I_{\sigma,\epsilon}^-\cup I_{\sigma,\epsilon}^+}
\Big\{\, 2\, \Phi_{\sigma, \epsilon}'(x)\, \mss b(x)
\,+\,  \Phi_{\sigma, \epsilon}(x)  \,
\Big[\, \mss b'(x) - \frac{\mss a'(x)\,
\mss b(x)} {\mss a(x)} \, \Big]\, \Big\} \;d\mu_{\epsilon}(x)\,.
\end{aligned}
\end{equation*}
By items (3) and (4) in the definition of $\Phi_{\sigma,\epsilon}$,
\eqref{liminf1:suppcut2}, and \eqref{liminf1:bound:pot}, the integrals
over $I_{\sigma,\epsilon}^-\cup I_{\sigma,\epsilon}^+$ vanish as
$\epsilon\to 0$. We turn to the integral over
$I_{\sigma,\epsilon}$. Recall \eqref{liminf1:interv}, let
$I_\sigma:= [x_{\sigma,2}^-,x_{\sigma,2}^+]$ and rewrite the integral
as
\begin{equation*}
\int_{I_\sigma} \Big\{\,
\mss b'(x) - \frac{\mss b(x)\, \mss a'(x)}{\mss a(x)}\, \Big\}\, d\mu_\epsilon(x)
- \int_{I_\sigma\setminus I_{\sigma,\epsilon}} \Big\{\,
\mss b'(x) - \frac{\mss b(x)\, \mss a'(x)}{\mss a(x)} \, \Big\}\,
d\mu_\epsilon(x)\,.
\end{equation*}
By \eqref{liminf1:mu} and the assumption that $\mu_\epsilon$ converges
weakly to $\mu$,
$\mu_\epsilon(I_\sigma\setminus I_{\sigma,\epsilon}) \to 0$ as
$\epsilon \to 0$. Hence, since the integrand is bounded, the second
integral in the previous displayed equation vanishes as
$\epsilon \to 0$. To finish the proof, note that by \eqref{liminf1:mu}
and the fact that $\mss b(\sigma) = 0$,  the first
integral converges to $\omega(\sigma)\, \mss b'(\sigma)$, as
$\epsilon\to 0$.
\end{proof}

\subsection*{Second order $\Gamma-\limsup$}
\label{sec4.2}

The main result of this subsection reads as follows.

\begin{proposition}
\label{p4_limsup}
Assume that condition (H) is in force. Then, the functional
$\ms J^{0}$ is a $\Gamma-\limsup$ for the sequence $\ms I_\epsilon$.
\end{proposition}

The proof of this proposition requires some preparation. Fix a
probability measure $\mu\in \mss P(\bb T)$. If
$\ms J^{(0)}(\mu) = +\infty$, there is nothing to prove. Hence, by
definition of $\ms J^{(0)}(\cdot)$, we may assume that
$\mu = \sum_{z \in \mc C} w(z)\, \delta_z$ for some
$w\in \mss P(\mc C)$.

It is actually enough to prove that for every critical point
$\mf c\in \mc C$, there exists a sequence
$\mu_{\mf c, \epsilon}\in \mss P(\bb T)$ which converges weakly to
$\delta_{\mf c}$ and such that
$\limsup_{\epsilon \to 0}\ms I_\epsilon(\mu_{\mf c, \epsilon})\leq \ms
J^{(0)}(\delta_{\mf c})$. Indeed, in this case, setting
$\nu_{\epsilon} := \sum_{\mf c \in \mc C}\omega(\mf c)\, \mu_{\mf c,
\epsilon}$ yields a sequence which converges weakly to
$\mu = \sum_{z \in \mc C} w(z)\, \delta_z$ and such that, by
convexity,
\begin{equation*}
\limsup_{\epsilon \to 0}\ms I_\epsilon(\nu_\epsilon)
\,\leq\, \sum_{\mf c \in \mc C}\omega(\mf c) \,
\limsup_{\epsilon \to 0}\ms I_\epsilon(\mu_{\mf c, \epsilon})
\,\leq\,  \sum_{\mf c\in \mc C} \omega(\mf c) \, \ms J^{(0)}(\delta_{\mf c})
\,=\, \ms J^{(0)}(\mu)\,,
\end{equation*}
where the last inequality follows by \eqref{2nd_rate_f}.

Suppose, from now on, that $\mu = \delta_{\mf c}$, for some
$\mf c\in \mc C$, and assume, without loss of generality, that
$\mf c = 0$. To complete the proof of the $\Gamma-\limsup$, we need to
prove the following result.

\begin{lemma}
\label{l4_limsup}
There exists a sequence of probability measures
$\mu_{\epsilon}\in \mss P(\bb T)$ which converges weakly to
$\delta_{0}$, and such that
\begin{equation*}
\limsup_{\epsilon \to 0}\ms I_\epsilon(\mu_\epsilon) \;\leq\;
b'(0)\vee 0\,. 
\end{equation*}
\end{lemma}
    
Fix $0<\delta_0<\delta_1<1/32$ and let $\Phi \in C^{2} (\bb R)$ be a
non-negative function, bounded by $1$, whose support is contained in
$( -1/2 +\delta_0 , 1/2-\delta_0)$ and such that $\Phi(x) = 1$ for
$x\in (-1/2+\delta_1,1/2-\delta_1)$. We also require $(\Phi')^2/\Phi$
to be bounded. This is not a serious condition. Let $\Psi$ be a function
which fulfills all requirements for $\Phi$, but the last one. Define
$\Phi$ as $\Psi^2$. The first requirements hold for $\Phi$, and
$(\Phi')^2/\Phi = (\Psi')^2$, which is a bounded function.

Let $G_\epsilon:\bb R \to \bb R$, $\epsilon>0$, be the function given by
\begin{equation*}
G_\epsilon(x) = e^{-C\frac{x^2}{2\epsilon}},
\end{equation*}
where $C:= |\mss b'(0)|/\mss a(0)$. Consider the probability measure
$\mu_\epsilon\in \mss P(\bb T)$ defined by
\begin{equation*}
\mu_\epsilon(dx) := m_\epsilon(x) \, dx \,:=\,
\frac{1}{Z_\epsilon} G_\epsilon(x) \Phi(x) \, dx\,, 
\end{equation*}
where $Z_\epsilon$ is a normalization constant that turns
$\mu_\epsilon$ a probability measure.  Clearly, the support of the
measure $\mu_\epsilon$ is contained in the interval $(-1/2,1/2)$, and
$\mu_\epsilon$ can be considered a probability measure on the
torus. Moreover, by a change of variables, as $\Phi(\cdot)$ is
continuous at the origin and equal to $1$,
\begin{equation}
\label{15}
\lim_{\epsilon\to 0} \, \frac{Z_\epsilon}{\sqrt{\epsilon}}
\,=\, \lim_{\epsilon\to 0} \,  \frac{1}{\sqrt{\epsilon}} \,\int_{\bb R}
e^{-C\frac{x^2}{2\epsilon}} \,  \Phi(x) \, dx
\,=\, \lim_{\epsilon\to 0} \,  \frac{1}{\sqrt{C}} \,\int_{\bb R}
e^{-\frac{x^2}{2}} \,  \Phi(x \sqrt{\epsilon}/\sqrt{C}) \, dx
\,= \, \sqrt \frac{2\pi}{C}  \,.
\end{equation}

\begin{assertion}
The sequence of probability measures $(\mu_\epsilon)_{\epsilon>0}$
converges weakly to $\delta_0$.
\end{assertion}

\begin{proof}
Fix $\phi\in C^{2}(\bb T)$.  Since $\Phi$ is bounded by one, a
change of variables yields
\begin{equation*}
|\langle \mu_\epsilon, \phi\rangle - \phi(0)|
\leq \frac{\sqrt{\epsilon}}{Z_\epsilon}
\int_{-1/2\sqrt{\epsilon}}^{1/2\sqrt{\epsilon}} 
|\phi(\sqrt{\epsilon}x) - \phi(0)| \, e^{-C\frac{x^2}{2}} \, dx\,.  
\end{equation*}
As $Z_\epsilon$ is of order $\sqrt{\epsilon}$ and $\phi$ is
continuous, by the dominated convergence theorem, the right-hand
side of the previous expression vanishes as $\epsilon\to 0$, which
proves the assertion.
\end{proof}

To complete the proof of the lemma, we slightly perturb
$\mu_\epsilon$ in order to apply Lemma \ref{l03} to estimate
$\ms I_\epsilon(\mu_\epsilon)$.

\begin{assertion}
\label{estim_00_limsup1}
Given $\delta>0$, let $\mu^\delta_\epsilon\in \mss P(\bb T)$ be the
probability measure given by the density
\begin{equation*}
m_{\epsilon,\delta}(x):= \frac{m_\epsilon(x)+\delta}{1+\delta}.
\end{equation*}
Then,
\begin{equation*}
\lim_{\delta\to 0}
\left|\;\ms I_\epsilon (\mu_\epsilon^\delta) \;-\;
\frac{1}{4\epsilon}\int_{\bb T} \mss a\,
\left\{\epsilon \Big[\frac{m_{\epsilon, \delta}'}{
m_{\epsilon,\delta}}
\,+\, \frac{\mss a'}{\mss a}\Big]
\,-\, \frac{\mss b}{\mss a}\right\}^2 d\mu_\epsilon^\delta \;\right|
\;=\, 0\,.
\end{equation*}
\end{assertion}

\begin{proof}
Since $m_{\epsilon,\delta}(x)\geq \delta/(1+\delta)>0$, $x\in \bb T$,
by Lemma \ref{l03}, 
\begin{equation*}
\ms I_\epsilon(\mu^\delta_\epsilon)=
\frac{1}{4\epsilon}\int_{\bb T}\mss a\, \left\{\epsilon\,
\Big[\frac{m_{\epsilon, \delta}'}{ m_{\epsilon,\delta}}\,+\,
\frac{\mss a'}{\mss a}\Big] \,+\, \frac{B/(\mss a\,
m_{\epsilon,\delta})}{ \<1/ \mss a \, 
m_{\epsilon,\delta} \> } \,-\, \frac{\mss b}{\mss
a}\right\}^2d\mu_{\epsilon}^{\delta}\,. 
\end{equation*}
Expand the square. To get the desired result it is enough to show that
the integral of the terms involving the factor
$B/(\mss a\, m_{\epsilon,\delta})$ vanishes as $\delta\to 0$.

Since $m_\epsilon$ vanishes on a set of strictly positive Lebesgue
measure, the term $\<1/ \mss a \, m_{\epsilon,\delta} \> $ is bounded
below by $C_0 \delta^{-1}$, and above by $C_1 \delta^{-1}$ for finite
constants $0<C_0<C_1<\infty$.  Therefore, by the definitions of
$\mu^{\delta}_{\epsilon}$ and $B$, 
\begin{equation*}
\lim_{\delta\to 0} \frac{1}{4\epsilon}\, \int_{\bb T}\mss a\,
\left\{\frac{B/(\mss a\, m_{\epsilon,\delta})}
{\<1/ \mss a \, m_{\epsilon,\delta} \>
}\right\}^2\mu^{(\delta)}_\epsilon
\, =\, \lim_{\delta\to 0}\frac{1}{4\epsilon}
\frac{B^2}{\<1/ \mss a \, m_{\epsilon,\delta} \>} = 0\,,
\end{equation*}
and
\begin{equation*}
\lim_{\delta\to 0}\frac{1}{4\epsilon}\int_{\bb T}
- 2\mss b\,\frac{B/(\mss a\, m_{\epsilon,\delta})}
{\<1/ \mss a \, m_{\epsilon,\delta} \>}
d\mu^{(\delta)}_{\epsilon}(x)
\,=\,  \lim_{\delta\to 0}-\frac{B^2}{2\epsilon}
\frac{1}{\<1/ \mss a \, m_{\epsilon,\delta} \>} \,=\, 0\, .
\end{equation*}

On the other hand, since $m$ and $\mss a$ are periodic, 
\begin{align*}
\frac{1}{2\epsilon} \,
\frac{B}{\< 1/\mss a \, m_{\epsilon,\delta} \>}
\int_{\bb T} \Big[\frac{m_{\epsilon,\delta}'(x)}{m_{\epsilon,\delta}}(x)
\,+\, \frac{\mss a'(x)}{\mss a(x)}\Big]\,dx \,=\, 0\,.
\end{align*}
This completes the proof of the  assertion.  
\end{proof}

Since $\mu_\epsilon^\delta$ converges weakly to $\mu_\epsilon$,
by the lower semicontinuity of $\ms I_\epsilon$,
\begin{equation}
\label{13}
\ms I_\epsilon (\mu_\epsilon) \,\le\, \liminf_{\delta\to 0}
\ms I_\epsilon (\mu^\delta_\epsilon) \,.
\end{equation}
Since $m'_{\epsilon,\delta} = m'_{\epsilon}/(1+\delta)$, by the
monotone convergence theorem and the previous assertion,
\begin{equation}
\label{14}
\liminf_{\delta\to 0}
\ms I_\epsilon (\mu^\delta_\epsilon)
\,=\,
\frac{1}{4\epsilon}\int_{\bb T} \mss a\,
\left\{\epsilon \Big[\frac{m_{\epsilon}'}{
m_{\epsilon}}
\,+\, \frac{\mss a'}{\mss a}\Big]
\,-\, \frac{\mss b}{\mss a}\right\}^2 d\mu_\epsilon
\,=\,
\frac{1}{4\epsilon}\int_{\bb T} \mss a\,
\left\{\epsilon \, \frac{m_{\epsilon}'}{m_{\epsilon}}
\,+\, \frac{\epsilon\, \mss a' - \mss b}{\mss a} \, \right\}^2 d\mu_\epsilon
\,.
\end{equation}
It remains to estimate the right-hand side as $\epsilon\to 0$.

We estimate term by term.  We start with the cross term. We claim that
\begin{equation*}
\lim_{\epsilon\to 0} \frac{1}{2}\int_{\bb T}
\big[\, \epsilon \, \mss a'(x) - \mss b(x)\,\big]\,
m'_\epsilon(x) \; dx \,=\, \frac{\mss b'(0)}{2} \,\cdot
\end{equation*}
Indeed, as $\mu_\epsilon$ converges to the Dirac mass at the origin,
an integration by parts yields the result.

In the second diagonal term, the expressions involving
$\mss a'(\cdot)$ are easy to estimate. Since $\mss b(\cdot)$ vanishes
at the critical points, and $\mss a(\cdot)$ is strictly positive,
\begin{equation*}
\lim_{\epsilon \to 0} \frac{1}{4}\int_{\bb T}
\left\{\epsilon\, \frac{\mss a'(x)^2 }{\mss a(x)}
\,-\,2\, \mss b(x)\,  \mss a'(x)  \right\}
\, \mu_\epsilon(dx)  \,=\, 0\,.
\end{equation*}

The next two assertions take care of the remaining terms.

\begin{assertion}
\label{a_estim_2}
With the notation introduced above,
\begin{equation*}
\lim_{\epsilon\to 0} \frac{1}{4\epsilon}\int_{\bb T} \mss a\,
\Big\{\epsilon \, \frac{m_{\epsilon}'}{m_{\epsilon}} \,\Big \}^2
d\mu_\epsilon \, = \,  \frac{|\mss b'(0)|}{4}  \,\cdot
\end{equation*}
\end{assertion}

\begin{proof}
By definition of $m_{\epsilon}$, and since $\Phi$ has compact support
on $(-1/2, 1/2)$, the left-hand side is equal to
\begin{equation*}
\frac{\epsilon}{4\, Z_\epsilon} \, \int_{\bb R}
\mss a(x)\, \left\{-\,\frac{Cx}{\epsilon}
\;+\; \frac{\Phi'(x)}{\Phi (x)}\, \right\}^2\,
\Phi (x)\, e^{-C\frac{x^2}{2\epsilon}}\;dx \,.
\end{equation*}
The cross term is easy to handle. After a change of variables it
becomes
\begin{equation*}
-\, \frac{C\, \epsilon}{2\, Z_\epsilon} \, \int_{\bb R}
\mss a(\epsilon x)\, x\, 
\Phi'(\epsilon x) \, e^{-Cx^2/2}\;dx \,.
\end{equation*}
Since $Z_\epsilon$ is of order $\sqrt{\epsilon}$, this expression
vanishes as $\epsilon\to 0$.

On the other hand, since, by assumption, $\Phi'(x)^2/\Phi (x)$ is a
bounded function whose support is contained in
$(-1/2 +\delta_0, -1/2+\delta_1) \cup (1/2 -\delta_1, 1/2-\delta_0)$
and $1/2 -\delta_1\ge 1/4$, the second diagonal term is bounded by
\begin{equation*}
C_0\,  \frac{\epsilon}{Z_\epsilon} \, \int_{|x|\ge 1/4}
\mss a(x)\,\, e^{-Cx^2/2\epsilon}\;dx \,.
\end{equation*}
As $Z_\epsilon$ is of order $\sqrt{\epsilon}$, this expression
vanishes as $\epsilon\to 0$.

We turn to the first diagonal term which is equal to
\begin{equation*}
\frac{C^2}{4\, \epsilon\, Z_\epsilon} \, \int_{\bb R}
\mss a(x)\, x^2\, 
\Phi (x) \, e^{-Cx^2/2\epsilon}\;dx
\,=\,
\frac{C^2\sqrt{\epsilon}}{4\, Z_\epsilon} \, \int_{\bb R}
\mss a(\epsilon x)\, \Phi (\epsilon x) \,  x^2\, 
e^{-Cx^2/2}\;dx \,.
\end{equation*}
Since $\Phi(\cdot)$, $\mss a(\cdot)$ are continuous at the origin,
$\Phi(0)=1$, $C=|\mss b'(0)|/\mss a(0)$, by \eqref{15} the previous
expression converges, as $\epsilon\to 0$, to $(1/4) |\mss b'(0)|$. To
complete the proof of the assertion, it remains to recollect all
previous estimates.
\end{proof}

A similar argument yields that
\begin{equation*}
\lim_{\epsilon\to 0} \frac{1} {4\epsilon} \int_{\bb T}
\frac{\mss b^2(x)}{\mss a(x)}\, \mu_\epsilon (dx) 
\, =\, \frac{|\mss b'(0)|} {4}\,\cdot 
\end{equation*}

By \eqref{13}, \eqref{14}, and recollecting all previous estimates
yields that
\begin{equation*}
\limsup_{\epsilon\to 0}  \ms I_\epsilon(\mu_\epsilon) \leq
\frac{\mss b'(0)}{2} + \frac{|\mss b'(0)|}{4} +  \frac{|\mss b'(0)|}{4}
\,=\, \mss b'(0)\vee 0\,.
\end{equation*}
We have established Lemma \ref{l4_limsup}, and, thus, also proven
Proposition \ref{p4_limsup}.

\begin{proof}[Proof of Proposition \ref{s04}.]
The first assertion of the proposition is a direct consequence of
Propositions \ref{p4.1} and \ref{p4_limsup}.  By the definition of the
functional $\ms J^{(0)}$ given in \eqref{2nd_rate_f}, \eqref{2f_1},
the $0$-level set of the functional $\ms J^{(0)}$ consists of all
convex combinations of Dirac measures supported on the elements of
$\mc M$.
\end{proof}

\section{Higher order expansion}
\label{sec5}

In this section we introduce the main technical and theoretical tools
needed to prove Proposition \ref{s03} and Theorem \ref{mt3}. The
presentation is divided in two parts.

In the first subsection, we introduce a useful representation formula
for the level two large deviation rate functionals $\bb I^{(p)}$,
$1\leq p \leq \widehat{\mf q}$, associated to the reduced Markov
chains $\widehat{\bb X}_p(\cdot)$ defined in \eqref{p_generator}. This
formula, originally derived in \cite{l23}, will be used in Section
\ref{sec5.2} to verify the third condition in the definition of
$\Gamma-$expansion, and in Section \ref{sec6} to construct the
recovery sequence for the $\Gamma-\limsup$ at the intermediate
time-scales.

In the second subsection, we state Proposition
\ref{resolvent_limit_point_2}, which was essentially proven in
\cite{lm}. This result puts at our disposal the metastable properties
of the process $X_\epsilon(\cdot)$ that will be needed throughout the
remainder of the article. In particular, it plays a key role in
Section \ref{sec5.2}, where it is used to construct the test functions
required for the proof of the $\Gamma-\liminf$ of
$\theta_\epsilon^{(p)}\ms I_\epsilon$, for
$1\leq p\leq \widehat{\mf q}$, and in Section \ref{sec8}, where it is
used to define the densities of the recovery sequence needed to
establish the $\Gamma-\limsup$ of
$\theta_\epsilon^{(p)} \ms I_\epsilon$ at the last metastable scale,
under the assumption that the reduced chain
$\widehat{\bb X}_{\mf q}(\cdot)$ is irreducible and non-reversible.

\subsection*{The rate function $\bb I^{(p)}$ for the reduced Markov
chain $\widehat{\bb X}_p(\cdot)$:}
\label{sec5.1}

Fix $1\leq p \leq \widehat{\mf q}$ and recall the definition of the rate
functional $\bb I^{(p)}$ given in \eqref{25}. In this subsection we
introduce a formula for $\bb I^{(p)}$ that will be useful in the proof of
the $\Gamma-$convergence of $\theta_\epsilon^{(p)}\ms
J_\epsilon$. This requires some preliminary notation.

For every set $\mf C\subset \mc S_p$, denote by
$\widehat{\bb L}_{p,\mf C}$ the generator of the process
$\widehat{\bb X}_p(\cdot)$ reflected on $\mf C$,
\begin{equation}
\label{37}
{\color{blue}(\widehat{\bb L}_{p, \mf C}f)(\mc M)} = \sum_{\mc M'\in
\mf C}R_p(\mc M,\mc M')\{f(\mc M')-f(\mc M)\}\,, \quad f: \mf C\to
\mathbb R\,,\quad \mc M\in \mf C. 
\end{equation}   
Similarly, given a set $\mf C\in \mc S_p$, let
$\bb I^{(p)}_{\mf C}: \mss P(\mf C)\to [0,+\infty]$ be the large
deviations rate functional associated to $\widehat{\bb L}_{p,\mf C}$
and given by
\begin{equation}
\label{20}
{\color{blue}\bb I^{(p)}_{\mf C}(\omega')} := \sup_{h}-\sum_{\mc
M\in \mf C}\frac{[\widehat{\bb L}_{p, \mf C}h](\mc M)}{h(\mc
M)}\omega'(\mc M), \hspace{5mm}  \omega'\in \mss P(\mf C),  
\end{equation}
where the supremum is taken with respect to all positive functions
$h:\mf C\to \bb R$.

Fix a probability measure $\omega\in \mss P(\mc S_p)$. Let
${\color{blue}\mc S_p^{(\omega)}}\subset \mc S_p$ be the support of
$\omega$. We do not assume that $\mc S_p^{(\omega)}$ is a proper
subset of $\mc S_p$. Denote by
{\color{blue}$\widehat{\bb X}^{(\omega)}_p(\cdot)$} the Markov chain
induced by $\widehat{\bb L}_{p,\mc S_p^{(\omega)}}$. Let
${\color{blue}\mf E_{1}^{p},...,\mf E_{o_p}^{p}}\subset \mc
S_{p}^{(\omega)}$ be the equivalence classes of
$\widehat{\bb X}_p^{(\omega)}(\cdot)$, and
${\color{blue}\mf D_{1}^{p},..., \mf D_{l_p}^{p}} \subset \mc
S_{p}^{(\omega)}$ the equivalence classes of
$\widehat{\bb X}_p^{(\omega)}(\cdot)$ containing two or more
elements. Here, ${\color{blue}o_p}$ and ${\color{blue}l_p}$ are
nonnegative integers representing the number of equivalent classes and
equivalent classes with two or more elements, respectively.

Some elementary observations.  If the support of $\omega(\cdot)$ is
not a singleton and $\widehat{\bb X}_p^{(\omega)}$ is irreducible,
then $o_p=l_p=1$, and $\mf E_1^{p}=\mf D_1^{p} = \mc
S_p^{(\omega)}$. If $\omega$ is strictly positive, then
$\widehat{\bb X}_p^{(\omega)} = \widehat{\bb X}_p$. If, additionally,
$\widehat{\bb X}_p^{(\omega)}$ is irreducible, then, by construction,
$p = \widehat{\mf q}$.

For every set $\mf C\subset \mc S_p$, let $\omega(\cdot|\mf C)$ be the
measure $\omega$ conditioned on $\mf C$,
\begin{equation*}
{\color{blue}\omega(\mc M|\mf C)} := \frac{\omega(\mc M)}{\omega(\mf
C)}\,,\quad \mc M\in \mf C.   
\end{equation*}          
Lemmata A.7 and A.9 in \cite{l23} provide the following 
formula for the rate functional $\bb I^{(p)}$:
\begin{align}
\label{reduced_ldf0}
\bb I^{(p)}(\omega) &= \sum_{r=1}^{l_p}\omega(\mf D^{p}_{r})\;\bb
I^{(p)}_{\mf D^{p}_{r}}(\omega(\;\cdot\;|\;\mf D^{p}_{r})) \;+\;
\sum_{\mc M\in \mc S_p^{(\omega)}}\sum_{\mc M'\neq \mc
S_p^{(\omega)}}\omega(\mc M)R_p(\mc M,\mc M') \notag
\\ 
&+ \sum_{r=1}^{o_p}\sum_{s\neq r}\sum_{\mc M\in \mf
E^{p}_{r}}\sum_{\mc M'\in \mf E^{p}_s}\omega(\mc M)R_p(\mc M,\mc M').   
\end{align}

Notice that for every $1\leq r \leq l_p$, the reflected generator
$\widehat{\bb L}_{p,\mf D_r^p}$ is irreducible and the measure
$\omega(\cdot|\mf D_r^p)$, is strictly positive on $\mf D_r^p$.
Denote by ${\color{blue}\nu_{\mf D_r}^p}\in \mss P(\mf D_r)$ the
stationary state of the Markov chain induced by
$\widehat{\bb L}_{p,\mf D_r}$.  By Lemma A.3 in \cite{l23}, for every
$1\leq r \leq l_p$ there exists a strictly positive function
${\color{blue}\widehat{h}_r \colon \mf D^p_r\to \bb R}$, unique up to
a multiplicative constant, such that
\begin{equation}
\label{reflected_equiv}
\bb I^{(p)}_{\mf D^p_r}(\omega(\cdot|\mf D^p_r))
= - \sum_{\mc M \in \mf D^p_r}\frac{[\widehat{\bb L}_{p,\mf
D^p_r}\;\widehat{h}_r]
(\mc M)}{\widehat{h}_r(\mc M)}\; \omega(\mc M|\mf D^p_r)\,.  
\end{equation} 

Denote by $\widehat{\bb L}_{p, \widehat{h}_r}$ the reflected generator
$\widehat{\bb L}_{p, \mf D^p_r}$ tilted by the function
$\widehat{h}_r$,
\begin{equation}
\label{51}
{\color{blue}(\widehat{\bb L}_{p,\widehat{h}_r}f)(\mc M)} :=
\sum_{\mc M'\in \mf D_r^p}
\frac{R_p(\mc M,\mc M')\, \widehat{h}_r(\mc M')}
{\widehat{h}_r(\mc M)}\{f(\mc M')-f(\mc M)\}\,,
\end{equation}
for $f\colon \mf D_r^p \to \bb R$. Let
{\color{blue}$R_{p,\widehat{h}_r}(\mc M,\mc M')\,:=
\,\widehat{h}_r(\mc M)^{-1}\,R_p(\mc M,\mc M')\,\widehat{h}_r(\mc
M')$} be the associated tilted jump-rates. Clearly, the Markov chain
induced by the tilted generator $\widehat{\bb L}_{p,\widehat{h}_r}$ is
irreducible as well.  By \cite[Lemma A.2]{l23}, the probability
measure $\omega(\cdot|\mf D_r^p)$ is the unique stationary state for
the tilted generator $\widehat{\bb L}_{p,\widehat{h}_r}$:
\begin{equation}
\label{h_tilted_invariant_m}
\sum_{\mc M \in \mf D_r^p}(\widehat{\bb L}_{p,\widehat{h}_r})f(\mc M)\,
\omega(\mc M|\mf D_r^p) = 0\,,
\end{equation}
for every $f\colon \mf D_r^p \to \bb R$. 

If the measure $\omega(\cdot\, |\mf D_r^p) $ satisfies the detailed
balance conditions for the jump rates $R_{p,\widehat{h}_r}$, by Lemma
\ref{pro_red_stat} and Theorem 5 in \cite{dv75} the function
$\widehat{h}_r$ is given by
\begin{equation*}
\label{sol_var_ld}
\widehat{h}_r(\mc M) := \sqrt{\frac{\omega(\mc M|\mf
D^p_r)}{\nu^p_{\mf D^p_r}(\mc M)}}\,,
\end{equation*}
where, recall, $\nu_{\mf D_r}^p$ is the
stationary state of the Markov chain induced by
$\widehat{\bb L}_{p,\mf D_r}$.

\subsection*{Resolvent characterization of metastability}
\label{sec_Meta}

In this subsection, we introduce the main tool needed for the proof of
the $\Gamma-$convergence of $\theta_\epsilon^{(p)}\ms I_\epsilon$,
$1\leq p\leq \widehat{\mf q}$. 

Denote by {\color{blue}$N(x,r)$}, $x\in \bb R$, $r>0$, the connected
component of the set $\{y\in \bb R: S(y)<S(x)+r\}$ that contains $x$,
and by {\color{blue}$B(x,r)$} the open ball of radius $r$ centered at
$x$. Fix ${\color{blue}\hat{r}}>0$ such that
$B(m_k,\hat{r})\subsetneq (0,1)$ for every $0\leq k\leq N-1$. Recall
the definition of the height $\mf h_1$ introduced in \eqref{f02}. Given
$0< r <\mf h_1\wedge \hat{r}$, denote by
$\mc N_{r}(m_k)\subset \bb T$, $0\leq k \leq N-1$, the set defined by
\begin{equation*}
\label{0_well}
{\color{blue}\mc N_{r}(m_k)} := \Pi(N(x,r)\cap B(m_k,r)).
\end{equation*}
Fix $0<{\color{blue}r_0} < (1/2) \, (\mf h_1\wedge \hat{r})$,
and denote by $\mc E_1(m_k)$, $0\leq k \leq N-1$, the set given by
\begin{equation}
\label{1_well}
{\color{blue}\mc E_1(m_k)}:= \mc N_{r_0}(m_k).
\end{equation}
For each $m_k\in \mc M$ we say that $\mc E_1(m_k)$ is a first order
well. By the definition of $r_0$,
\begin{equation}
\label{ord_well_1}
\mc E_1(m_k)\cap \mc E_1(m_l) \,=\, \varnothing,
\end{equation} 
for $k,l\in S_1$, $k\neq l$.

Fix $1\leq p \leq \widehat{\mf q}$, and recall from condition (c) above equation
\eqref{17} that $\mf u_p$ represents the number of elements of
$\mc S_p$. Let {\color{blue}$S_p =\{0, \dots, \mf u_p-1\}$}. Define the
$p-$order well $\mc E_p(k)$, $k\in S_p$, by
\begin{equation*}
\label{pth_well}
{\color{blue}\mc E_p(k)} := \bigcup_{m\in \mc M_p(k)}\mc E_1(m).
\end{equation*}
Note that the set $\mc E_p(k)$ is disconnected if $\mc M_p(k)$ is not
a singleton.  By property $\mc P_2(p)$ in Appendix \ref{secA1},
\eqref{ord_well_1} and the definition of $r_0$, we also have
$\mc E_p(k)\cap \mc E_p(l) = \varnothing$, for $k,l\in S_p$,
$k\neq l$. Note that we can induce on $\{\mc E_p(k)\}_{k\in S_p}$ the
order given in $\mc S_p$. Fix $r\in (2r_{0},\mf h_1\wedge \hat{r})$,
and for every $k\in S_{p}$, choose a smooth function
$\zeta^{p,k} \colon \bb T \to \bb R$ such that
\begin{equation}
\label{resolv_smooth}
\mss 1\{\mc E_{p}(k) \} \leq
\cb{\zeta^{p,k} } \leq \mss 1 \{\mc N^{p}_{r}(k) \}\,,
\quad {\color{blue}\mc N^{p}_{r}(k)} := \bigcup_{\mf m\in \mc M_p(k)}\mc N_r(\mf m).
\end{equation}

Fix a function $g\colon \mc S_p \to \bb R$, and let
$G_g: \bb T \to \bb R$ be its smooth lifting to $\bb T$ given by
\begin{equation}
\label{smooth_lifted}
{\color{blue}G_g(x)} = \sum_{k\in S_p}
g(\mc M_p(k)) \, \zeta_{p,k}(x).
\end{equation}    
Note that $G_g$ is smooth, constant on each well $\mc E_p(k)$,
$k\in S_p$, and vanishes on
$\bb T\setminus \bigcup_{k\in S_p}\mc N_r(k)$. For $\lambda>0$, let
$\phi_{p,\epsilon}$ be the unique solution of the resolvent equation
\begin{equation}
\label{smooth_resolvent_eq}
(\lambda - \theta_{\epsilon}^{(p)}
\ms L_\epsilon) \, {\color{blue} \phi_{p,\epsilon}} = G_g.
\end{equation}
Since $G_g$ is smooth, the function $\phi_{p,\epsilon}$ belongs to the
domain of $\ms L_\epsilon$. In particular, it is of class
$C^2(\bb T)$.  We are now in a position to state the main tool in the
proof of the $\Gamma-$convergence of the rate functional
$\theta_\epsilon^{(p)}\ms I_\epsilon$.

\begin{proposition}
\label{resolvent_limit_point_2}
For every $1\leq p \leq \widehat{\mf q}$, $\lambda>0$, and
$g\colon \mc S_p\to \bb R$, the unique solution $\phi_{p,\epsilon}$ of
the resolvent equation \eqref{smooth_resolvent_eq} satisfies
\begin{equation}
\label{resolvent_limit_point_2:eq}
\lim_{\epsilon \to 0}\max_{k\in S_p}\sup_{x\in \ms
E_p(k)}|\phi_{p,\epsilon}(x) - f_p(\mc M_p(k))| = 0, 
\end{equation}   
where $f_{p} \colon \mc S_p \to \bb R$ is the unique solution of the
reduced resolvent equation
\begin{equation*}
\label{red_res_equation}
(\lambda - \widehat{\bb L}_p)f_{p} = g.
\end{equation*}
\end{proposition}

\begin{proof}
Fix $1\leq p\leq \widehat{\mf q}$, $\lambda>0$ and a function
$g\colon\mc S_p \to \bb R$. Let $G^{*}\colon\bb T\to \bb R$ be the function
defined by
\begin{equation*}
G^{*}(x) = \sum_{k\in S_p}g(\mc M_p(k))1_{\mc E_p(k)}(x).
\end{equation*}
Denote by $\phi_\epsilon^{*}$ the solution of the resolvent equation
\eqref{smooth_resolvent_eq} with $G^{*}$ on the right hand side. The
same argument given in the proof of \cite[Theorem 3.3]{lm}, yields
\begin{equation*}
\lim_{\epsilon \to 0}\max_{k\in S_p}\sup_{x\in \ms
E_p(k)}|\phi_{p,\epsilon}^{*}(x) - f_p(\mc M_p(k))| = 0. 
\end{equation*}
To finish the proof, note that \cite[Proposition 4.2]{ls22} and the
above equation, imply \eqref{resolvent_limit_point_2:eq}.
\end{proof}

\begin{remark}
\label{ref_resolv_meta}
For a discussion on the connection between the resolvent condition and
the metastable behaviour of Markov processes, we refer the reader to
\cite{lms23b} for the case of Markov chains on countable state
spaces, and \cite{ls22} for the case of metastable diffusions.
\end{remark}

We finish this subsection with three more properties of the solution of
the resolvent equation \eqref{smooth_resolvent_eq}. Recall that $\bb P_x^\epsilon$ denotes the law of the diffusion $X_\epsilon(\cdot)$, starting from $x\in \bb T$, and write {\color{blue}$\bb E_x^\epsilon$} for the corresponding expectation. Similarly, let
{\color{blue}$\bb P_x^{p,\epsilon}$} be the law of the re-scaled
diffusion
${\color{blue}X_\epsilon^{(p)}(t)} :=
X_\epsilon(\theta_\epsilon^{(p)}t)$, $t\geq 0$, starting from
$x\in \bb T$, and denote by {\color{blue}$\bb E_x^{p, \epsilon}$} the
associated expectation. Clearly,
$\theta_\epsilon^{(p)}\ms L_\epsilon$ is the infinitesimal generator
of $X_\epsilon^{(p)}(\cdot)$. Then, by \cite[Section 6.5]{Fri}, the
solution of \eqref{smooth_resolvent_eq} admits the following
stochastic representation:
\begin{equation*}
\label{stoc:rep}
\phi_{p,\epsilon}(x) = \bb
E^{p, \epsilon}_{x}\Big[\, \int_{0}^{\infty}e^{-\lambda s}\,
G_g ( X_\epsilon^{(p)} (s)) \; ds\,\Big]\;.
\end{equation*} 
The following result is a simple consequence of this formula.

\begin{lemma}
\label{l09}
Fix $\lambda>0$, $1\leq p\leq \widehat{\mf q}$ and a bounded function
$g:\mc S_p \to \bb R$. Let $\phi_\epsilon$ be the solution of
\eqref{smooth_resolvent_eq}. Then,
\begin{equation*}
\sup_{\epsilon>0}\sup_{x\in \bb T}|\phi_\epsilon(x)| \leq
\frac{1}{\lambda}\sup_{x\in \bb T}|G_g(x)| <+\infty. 
\end{equation*}
Moreover, if $g$ is nonnegative, then $\phi_\epsilon$ is also
nonnegative.
\end{lemma}

The next result provides a representation of the solution of the
resolvent equation \eqref{smooth_resolvent_eq} in terms of level
$p$-minima hitting
times

For every closed set $A\subset \bb T$, let $\tau(A)$ be the first time
the trajectory $\mtt x(\cdot)$ reaches the set $A$,
\begin{equation*}
\cb{\tau(A)} := \inf\{t\geq 0: \mtt x(t)\in A\}\,.
\end{equation*}
If $A$ is a singleton or a pair, say $A=\{a\}$, $A=\{a,b\}$ for some
$a,b\in \bb T$, we write $\cb{\tau(a)}$, $\cb {\tau(a,b)}$ for
$\tau(\{a\})$, $\tau(\{a,b\})$, respectively.

Fixed a level $1\leq p\leq \widehat{\mf q}$, let $\mc M_p\subset \mc M$ be the set of local minima of level $p$,
\begin{equation*}
    {\color{blue}\mc M_p}:= \bigcup_{k\in S_p}\mc M_p(k)\,.
\end{equation*}
For each point
$x \in \bb T\setminus \mc M_p$, which is not a local minima of level
$p$, denote by {\color{blue}$m_p^-(x), m_p^+(x)$} the closest level
$p$ minima to the left, right of $x$, respectively. That is,
$m_p^{-}(x)$, $m_p^+(x)\in \mc M_p$ and
\begin{equation*}
x\in (\,m_p^-(x),\,m_p^+(x)\,)\,, \quad
[\,m_p^-(x),\,m_p^+(x)\,]\,\cap\, \mc M_p \,=\,
\{\,m_p^-(x),\,m_p^+(x)\,\} \,.
\end{equation*}

\begin{lemma}
\label{s10}
Fix $\lambda>0$, $1\leq p\leq \widehat{\mf q}$ and a bounded function
$g\colon \mc S_p \to \bb R$. Let $\phi_{p,\epsilon}$ be the solution of
\eqref{smooth_resolvent_eq}. Then, for every
$x\in \bb T\setminus \mc M_p$,
\begin{align*}
\label{eq:resolv:replacement}
&\Big| \,\phi_{p,\epsilon}(x) \,-\, \phi_{p,\epsilon}(m_p^-(x))
\, \bb P_x^\epsilon \big[\, \tau(m_p^-(x))<\tau(m_p^+(x)) \,\big]
\,-\,\phi_{p,\epsilon}(m_p^+(x))\, 
\bb P_x^\epsilon \big[\, \tau(m_p^+(x))<\tau(m_p^-(x)) \,\big] \,\Big|
\notag\\
&\quad \leq 2\, \Vert G_g\Vert_\infty \,\bb E_x^\epsilon
\Big[\tau(m_p^-(x),m_p^+(x))\,/\,\theta^{(p)}_\epsilon\Big]\,. 
\end{align*}
In particular, 
\begin{equation*}
\liminf_{\epsilon\to 0}\inf_{x\in \bb T}\phi_{p,\epsilon}(x)
\,\ge \, \min_{\mc M\in \mc S_p}f_p(\mc M) \,.
\end{equation*}
\end{lemma}

\begin{proof}
The first assertion of the lemma follows from \cite[Proposition
4.1]{lm}. We turn to the second. Fix $x\in \bb T\setminus \mc M_p$. By
the first assertion of the lemma,
\begin{align*}
\phi_{p,\epsilon}(x)
\, \geq \,  \min_{m\in \mc M_{p}}\phi_{p,\epsilon}(m)
\,-\,2\, \Vert G_g\Vert_\infty\, 
\bb E_x^\epsilon\Big[\tau(m_p^-(x),m_p^+(x))\,/\,\theta^{(p)}_\epsilon\Big]\,.
\end{align*}
By Lemma \ref{hitting_1}, the previous expression is bounded below by
\begin{equation*}
\min_{m\in \mc M_{p}}\phi_{p,\epsilon}(m)
\,-\, \frac{C}{\epsilon}e^{[h_{p-1}-h_p]/\epsilon}\,,
\end{equation*}
for some finite constant $C$ which doesn't depend on $\epsilon$ nor on
$x$. Here, we adopted the convention that $h_0 = 0$.  Since $\mc M_p$
is a finite set and $h_p>h_{p-1}$, to complete the proof it remains to
let $\epsilon\to 0$ and to recall the statement of Proposition
\ref{resolvent_limit_point_2}.
\end{proof}

\section{$\Gamma-\liminf$  in the metastable time-scales} 
\label{sec5.2}

In this subsection we prove that $\ms J^{(p)}$ is a $\Gamma-\liminf$
of $\theta_\epsilon^{(p)}\ms I_\epsilon$ for
$1\leq p\leq \widehat{\mf q}$. This is the content of Proposition
\ref{p_liminf}. The core of the proof resides in the case where
$\ms J^{(p)}$ is finite, since the remaining situation will trivially
follow by the definition of a $\Gamma-\liminf$. The proof is divided
in two parts. First, we use \eqref{reduced_ldf0} and
\eqref{reflected_equiv} to show that, for $2\leq p\leq \widehat{\mf q}$, the
functional $\ms J^{(p)}(\mu)$ is finite if, and only if, $\mu$ belongs
to the zero-level set of $\ms J^{(p-1)}$. This verifies property (3)
in Definition \ref{gamma_expansion_def}. Next, we use the resolvent
characterization of metastability to construct a test function that
yield the desired $\Gamma-\liminf$ bound.

Fix $1\leq p\leq \widehat{\mf q}$ and recall the definition of the metastable
states $\mu^p_{\mc M_p(k)}(\cdot)$, $k\in S_p$, introduced in 
\eqref{21}. Denote by $\mss P_{\mc M_p}\subset \mss P(\bb T)$
the convex-hull of the $p-$order metastable states:
\begin{equation*}
{\color{blue}\mss P_{\mc M_p}}:=  \Big \{\sum_{l\in S_{p}}
\omega(\mc M_{p}(l)) \,
\mu^p_{\mc M_{p} (l)}(\cdot):\; \omega\in \mss P(\ms
S_{p})\Big\} \,. 
\end{equation*}
The next result characterizes the $0-$level set of $\ms J^{(p)}$ in
terms of $\mss P_{\mc M_{p+1}}$.

\begin{proposition}
\label{higher_ord:0level}
For every $0\leq p < \widehat{\mf q}$, the $0-$level set of $\ms J^{(p)}$ is the
convex hull of the $(p+1)-$order metastable states. That is,
\begin{equation*}
\left\{\mu\in \mss P(\bb T): \ms J^{(p)}(\mu) = 0\right\}\,=\, \mss P_{\mc M_{p+1}}.
\end{equation*}
\end{proposition}

\begin{proof}
For $p=0$, the result follows directly by the definitions of
$\ms J^{(0)}$, $\mu^1_{m_k}$, $k\in S_1$, given in \eqref{2nd_rate_f},
\eqref{21}, respectively.

We turn to the case where $1\leq p< \widehat{\mf q}$. Fix
$\mu\in \mss P_{\mc M_{p+1}}$, and choose
$\omega_{p+1}\in \mss P(\mc S_{p+1})$ such that
\begin{equation}
\label{eq:higher0level}
\mu(\cdot) = \sum_{\mc M\in \mc S_{p+1}}\omega_{p+1}(\ms
M)\, \mu_{\mc M}^{p+1}(\cdot). 
\end{equation}
Let $\omega_p\in \mss P(\mc S_p)$ be the measure given by
\begin{equation*}
\omega_p(\mc M)\::=\: \begin{cases}
\omega_{p+1}(\mc M_{p+1}(k))\,
\nu^p_{\mc M_{p+1}(k)}(\mc M), &\text{if }
\mc M\in \mc M_{p+1}(k), \text{ for }k\in S_{p+1}\\ 
0, &\text{otherwise}\ ,
\end{cases}
\end{equation*}
where $\nu^p_{\mc M_{p+1}(k)} (\cdot)$, $k\in S_{p+1}$, is the measure
introduced in \eqref{stat_recur}. By \eqref{24} and Proposition
\ref{lifted_red_stat}, 
\begin{equation*}
\mu(\cdot) = \sum_{\mc M\in \mc S_p}\omega_p(\mc M)\,
\mu_{\mc M}^p(\cdot).
\end{equation*}
By \eqref{eq:higher0level} and \eqref{24}, every closed irreducible
class of $\widehat{\bb X}^{(\omega_p)}_p(\cdot)$ that is contained in
the support of $\omega_p$ is a closed irreducible class of
$\widehat{\bb X}_p$. Moreover, $\omega_p(\mc M)=0$ for every
$\widehat{\bb X}_p(\cdot)-$transient state $\mc M$. Then, since
$\nu_{\mf D_r^p}(\cdot)$ is the stationary state of the Markov chain
induced by $\widehat{\bb L}_{p, \mf D_r^p}$, $1\leq r\leq l_p$, by
\eqref{23b} and \eqref{reduced_ldf0},
\begin{equation*}
\ms J^{(p)}(\mu) \,=\, \bb I^{(p)}(\omega_p) \,=\, \sum_{r =
1}^{l_p}\bb I_{\mf D_r^p}^{(p)}(\nu_{\mf D_r^p}^p) \,=\, 0. 
\end{equation*}

We turn to the converse. Fix $\mu\in \mss P(\bb T)$ and assume that
$\ms J^{(p)}(\mu)=0$. By the definition of $\ms J^{(p)}$,
$\mu(\cdot) = \sum_{\mc M\in \mc S_p}\omega(\mc M)\mu_{\ms
M}^{p}(\cdot)$, for some $\omega\in \mss P(\mc S_p)$. By
\eqref{reduced_ldf0},
\begin{align}
\label{eq333}
0 = \bb I^{(p)}(\omega) &= \sum_{r=1}^{l_p}\omega(\mf D^{p}_{r})\;\bb
I^{(p)}_{\mf D^{p}_{r}}(\omega(\;\cdot\;|\;\mf D^{p}_{r})) \;+\;
\sum_{\mc M\in \mc S_p^{(\omega)}}\sum_{\mc M'\neq \ms
S_p^{(\omega)}}\omega(\mc M)R_p(\mc M,\mc M') \notag
\\ 
&+ \sum_{r=1}^{o_p}\sum_{s\neq r}\sum_{\mc M\in \mf
E^{p}_{r}}\sum_{\mc M'\in \mf E^{p}_s}\omega(\mc M)R_p(\mc M,\mc M').   
\end{align}

Since each sum on the right-hand side of the above expression is
nonnegative, each sum must be equal to $0$. Let $\mf D\subset \mc S_p$
be an equivalent class of $\widehat{\bb X}_p(\cdot)$ such that
$\mf D\cap \mc S_p^{(\omega)}\neq \varnothing$. Since the second sum
on the right-hand side of \eqref{eq333} vanishes, $\mf D$ is
completely contained in the support of $\omega$,
$\mf D\subset \mc S_p^{(\omega)}$. As the third sum is equal to $0$,
$\mc S_p^{(\omega)}$ cannot contain transient states of
$\widehat{\bb X}_p(\cdot)$. Thus, $\mf D$ must be a closed irreducible
class of $\widehat{\bb X}_p(\cdot)$. Moreover, as the first sum also
vanishes, $\omega(\mc M)\,=\,c_{\mf D} \, \nu_{\mf D}(\mc M)$,
$\mc M\in \mf D$, where, recall, $\nu_{\mf D}$ is the stationary state
of $\widehat{\bb L}_{p, \mf D}$, and $c_{\mf D}$ is some positive
constant that depends only on $\mf D$. Clearly,
$\omega(\mf D) = c_{\mf D}$ so that
\begin{equation}
\label{eq888}
\omega(\mc M)\,=\ \omega(\mf D) \, \nu_{\mf D}(\mc M)\,, \quad 
\mc M\in \mf D \,.
\end{equation}

To finish the proof, let
$\omega_{p+1}\in \mss P(\mc S_{p+1})$ be the measure defined by
\begin{equation*}
\omega_{p+1}(\mc M_{p+1}(k)) = \omega(\mc M_{p+1}(k))\,,
\quad k\in S_{p+1}\,.
\end{equation*}    
By definition of the mesure $\mu(\cdot)$, since $\omega (\mc M) =0$ if
$\mc M$ does not belong to a closed irreducible class of
$\widehat{\bb X}_p(\cdot)$, and by \eqref{eq888},
\begin{align*}
\mu(\cdot) \,=\, \sum_{\mc M'\in \mc S_{p}}
\omega(\mc M') \, \mu_{\mc M'}^{p}(\cdot)
\,&=\,
\sum_{\mc M\in \mc M_{p+1}(k)}
\sum_{\mc M'\in \mc S_{p} : \mc M' \subset \mc M} 
\omega(\mc M') \, \mu_{\mc M'}^{p}(\cdot)
\\
\,&=\,
\sum_{\mc M\in \mc M_{p+1}(k)} \omega(\mc M) \, 
\sum_{\mc M'\in \mc S_{p} : \mc M' \subset \mc M} 
\nu_{\mc M}(\mc M') \, \mu_{\mc M}^{p}(\cdot)\,.
\end{align*}
By Proposition \ref{lifted_red_stat} and the definition of
$\omega_{p+1}(\cdot) $, this expression is equal to
\begin{align*}
\sum_{\mc M'\in \mc S_{p+1}}\omega_{p+1}(\mc M')\, \mu_{\ms
M'}^{p+1}(\cdot),  
\end{align*}
which completes the proof of the proposition. 
\end{proof}

We are now in position to prove the main result of this subsection.

\begin{proposition}
\label{p_liminf}
For every $0\leq p \leq \widehat{\mf q}$, the functional $\ms J^{(p)}$, defined
in \eqref{2nd_rate_f}, \eqref{23}, \eqref{23b}, is a $\Gamma-\liminf$ of
$\theta_{\epsilon}^{(p)} \ms I_\epsilon$.
\end{proposition}

\begin{proof}
We proceed by induction on $p$. The case $p=0$ is the content of
Proposition \ref{p4.1}. Fix $1\leq p\leq \widehat{\mf q}$, and assume that
$\ms J^{(q)}$ is a $\Gamma-\liminf$ of
$\theta_\epsilon^{(q)}\ms J_\epsilon$ for every $0\leq q< p$. Fix a
measure $\mu\in \mss P(\bb T)$ and a sequence
$(\mu_\epsilon:\epsilon>0)\subset \mss P(\bb T)$ such that
$\mu_\epsilon$ converges weakly to $\mu$. If
$\mu \notin \mss P_{\mc M_{p}}$, then by Proposition
\ref{higher_ord:0level}, $\ms J^{(p-1)}(\mu)>0$. Thus, by the
inductive hypothesis and \eqref{27},
\begin{equation*}
+\infty = \liminf_{\epsilon \to
0}\frac{\theta^{p+1}}{\theta_{\epsilon}^{(p)}}\ms J^{(p)}(\mu)
\leq \liminf_{\epsilon \to
0}\frac{\theta^{p+1}}{\theta_{\epsilon}^{(p)}}\liminf_{\epsilon
\to 0}\theta_{\epsilon}^{(p)}\ms I_\epsilon(\mu_\epsilon) \leq
\liminf_{\epsilon\to 0}\theta_{\epsilon}^{(p+1)}\ms
I_\epsilon(\mu_\epsilon). 
\end{equation*} 
Hence,
$\ms J^{(p)}(\mu) \leq \liminf_{\epsilon\to
0}\theta_{\epsilon}^{(p+1)}\ms I_\epsilon(\mu_\epsilon)$, if $\mu$
isn't a convex combination of $p-$order metastable states.

We turn to the case where $\mu\in \mss P_{\mc M_p}$. Choose
$\omega\in \mss P(\mc C)$ such that
\begin{equation*}
\mu(\cdot) = \sum_{\mc M\in \mc S_p}\omega(\mc M)\mu^p_{\mc M}(\cdot).
\end{equation*}
Fix a strictly positive function $h:\mc S_p\to \bb R$. Let
$h_0 := \min_{\mc M\in \mc S_p}h(\mc M)$, choose a finite constant
$\lambda$ such that
$\lambda> \max_{\mc M\in \mc S_p}|(\widehat{\bb L}_p h)(\mc M)|/h_0$, and
let $H\colon \mc S_p \to \bb R$ be the function  given by
\begin{equation*}
H(\mc M_p(k)):= [(\lambda - \widehat{\bb L}_p)h](\mc M_p(k)),
\hspace{5mm} \forall k\in S_p. 
\end{equation*}
Note that $H$ is strictly positive. Let
$\Phi_{\epsilon}^{h}: \bb T \to \bb R$ be the solution of the
resolvent equation \eqref{smooth_resolvent_eq} with $G:=G_H$ on the
right-hand side, where $G_H$ is given by \eqref{smooth_lifted}. Hence,
$\Phi_{\epsilon}^{h} \in C^{2}(\bb T)$ and, by Lemma
\ref{l09}, $\Phi^{h}_{\epsilon}$ is nonnegative. Let
$\delta>0$ be an arbitrary positive constant and let
$\Phi^{h,\delta}_{\epsilon}:\bb T\to \bb R$ be the function given by
\begin{equation*}
\Phi^{h,\delta}_{\epsilon}(x) := \Phi^{h}_{\epsilon}(x) + \delta.
\end{equation*}
Then, $\Phi^{h,\delta}_{\epsilon}$ is a positive function of class
$C^2(\bb T)$. Thus, by the definitions of $\ms I_\epsilon$ and $H$,
\begin{align}
\label{l51}
\theta_{\epsilon}^{(p)} \ms I_\epsilon (\mu_{\epsilon}) 
&\geq - \int_{\bb T}\frac{[\theta_{\epsilon}^{(p)}\ms L_{\epsilon}
\Phi_{\epsilon}^{h,\delta}](x)}{\Phi_{\epsilon}^{h, \delta}(x)}d\mu_{\epsilon}(x) 
= -\sum_{k\in S_{p}}[(\widehat{\bb L}_p)h](\mc M_{p}(k))\int_{\ms
N^{p}_{r}(k)}\frac{\zeta^{p,k}(x)}{\Phi^{h,\delta}_{\epsilon}(x)}d\mu_{\epsilon}(x)
\\ 
&+ \sum_{k\in S_{p}}\int_{\mc N^{p}_{r}(k)}\frac{\lambda \{h(\ms
M_p(k))\zeta^{p,k}(x) -
\Phi_{\epsilon}^{h,\delta}(x)\}}{\Phi^{h,\delta}_{\epsilon}(x)}d\mu_{\epsilon}(x)
-\int_{\bb T \setminus \{\bigcup_{k\in S_p} \ms
N^{p}_{r}(k)\}}\frac{\lambda
\Phi_{\epsilon}^{h,\delta}(x)}{\Phi_{\epsilon}^{h,\delta}(x)}d\mu_{\epsilon}(x).\notag 
\end{align}

Since $\mu_\epsilon$ converges weakly to $\mu$, and $\mu$ is
concentrated on the set $\mc M_p\subset \bigcup_{k\in S_p}\mc N^p_r(k)$,
the third term on the right hand side of \eqref{l51} vanishes as
$\epsilon \to 0$. We turn to the second term. By the definition of
$\Phi^{h,\delta}_\epsilon$, we can write this expression as
\begin{equation}
\label{l52}
\sum_{k\in S_{p}}\int_{\mc N^{p}_{r}(k)}\frac{\lambda \{h(\ms
M_p(k))\zeta^{p,k}(x) -
\Phi_{\epsilon}^{h}(x)\}}{\Phi^{h,\delta}_{\epsilon}(x)}d\mu_{\epsilon}(x)
- \int_{\mc N^{p}_{r}(k)}\frac{\lambda
\delta}{\Phi^{h}_{\epsilon}(x)+\delta}d\mu_{\epsilon}(x).    
\end{equation}
By Proposition \ref{resolvent_limit_point_2}, by the definition of
$\zeta^{p,k}$ given in \eqref{resolv_smooth}, and the assumption that
$\mu_\epsilon$ converges weakly to $\mu$, the first term in
\eqref{l52} vanishes as $\epsilon\to 0$. Similarly, the second one
converges to
\begin{equation*}
\sum_{k\in S_{p}}\frac{\lambda \delta}{h(\mc M_p(k))+\delta}
\, \mu(\mc E^p(k)) \,.
\end{equation*}
By \eqref{21}, $\mu^p_{\mc M}(\mc M) =1$ for all $\mc M \in \ms
S_p$. Thus, by definition of $\mu$,
$\mu(\mc E^p(k)) = \mu (\mc M_p(k)) = \omega(\mc M_p(k))$. Hence, the
previous sum is equal to
\begin{align*}
\sum_{k\in S_{p}}\frac{\lambda \delta}{h(\ms
M_p(k))+\delta} \, \omega(\mc M_p(k)).  
\end{align*}
Similarly, by Proposition \ref{resolvent_limit_point_2} and
\eqref{resolv_smooth}, the first term on the right hand side of
\eqref{l51} converges to
\begin{equation*}
-\sum_{k\in S_p}\frac{[\widehat{\bb L}_ph](\mc M_p(k))}{h(\ms
M_p(k))+\delta}\;\mu(\mc E^p(k)) = -\sum_{k\in S_p}\frac{[\widehat{\bb
L}_ph](\mc M_p(k))}{h(\mc M_p(k))+\delta}\;\omega(\mc M_p(k)),  
\end{equation*} 

Recollecting the previous estimates one obtains
\begin{equation*}
\liminf_{\epsilon \to 0}\theta_{\epsilon}^{(p)}\ms
I_\epsilon(\mu_{\epsilon})\geq -\sum_{k\in S_p}\frac{[\widehat{\bb
L}_ph](\mc M_p(k))}{h(\mc M_p(k))+\delta}\;\omega(\mc M_p(k))  -
\sum_{k\in S_{p}}\frac{\lambda \delta}{h(\mc M_p(k))+\delta}\omega(\ms
M_p(k)). 
\end{equation*}
By \eqref{25}, sending $\delta \to 0$ first and then taking the
supremum over all positive functions $h\colon\mc S_p \to \bb R$,
yields that
\begin{equation*}
\liminf_{\epsilon \to 0}\theta_{\epsilon}^{(p)}\ms
I_\epsilon(\mu_{\epsilon})\geq \bb I^{(p)}(\omega)\,.
\end{equation*}
This completes the proof of the proposition.
\end{proof}

\section{The $\Gamma-\limsup$ in the time-scales
$\theta^{(p)}_\epsilon$, $1\le p< \widehat{\mf q}$}
\label{sec6}

In this section, we prove the $\Gamma-\limsup$ for all metastable time
scales with the exception of the last
one.  
The main result reads as follows.

\begin{proposition}
\label{higher_speed_glimsup}
For every $1\leq p < \widehat{\mf q}$, the functional $\ms J^{(p)}$, defined
in \eqref{23}, \eqref{23b}, is a $\Gamma-\limsup$ of
$\theta_{\epsilon}^{(p)} \, \ms I_\epsilon$.
\end{proposition}

Fix $1\leq p < \widehat{\mf q}$. Since $p$ is fixed, most of the time it is
omitted from the notation. For every $\mu\in \mss P(\bb T)$ such that
$\ms J^{(p)}(\mu) = \infty$, and any sequence $\mu_{\epsilon}$ weakly
converging to $\mu$, the bound
\begin{equation*}
\limsup_{\epsilon\to 0}
\theta_{\epsilon}^{(p)} \,  \ms I_\epsilon(\mu_\epsilon)
\,\leq\,  \ms J^{(p)}(\mu)
\end{equation*}
holds trivially. Hence, we may assume that $\ms J^{(p)}(\mu)$ is
finite.  In view of \eqref{23b}, this means that it is enough to
consider measures $\mu \in \mss P(\bb T)$ of the form
\begin{equation}
\label{29}
\mu(\cdot) \,=\, \sum_{k\in S_p}\omega(\mc M_p(k))
\, \mu^p_{\mc M_p(k)}(\cdot) 
\end{equation} 
for some $\omega\in \mss P(\mc S_p)$, where
$\mu^p_{\mc M_p(k)}(\cdot),$ is the measure introduced in
\eqref{21}. By Lemma B5 in \cite{l23}, we can assume that
$\omega(\mc M)>0$ for every $\mc M\in \mc S_p$. In the next
subsections, we construct a recovery sequence for $\mu$.

The recovery sequence $(\mu_\epsilon)_{\epsilon>0}$ constructed below
is a convex combination of measures $\mu_{\mf D, \epsilon}$ supported
on neighborhoods of the equivalence classes $\mf D$ of
$\widehat{\bb X}_p$. For each $\mf D \subset \mc S_p$, the measure
$\mu_{\mf D, \epsilon}$ is absolutely continuous, with density
$m_{\mf D, \epsilon}$ given by the product of a localized version of
the Gibbs density
\begin{equation*}
\frac{1}{\mss a(x)}\exp\{-S(x)/\epsilon\}
\end{equation*}
and the square of a function $J_{\mf D, \epsilon}$. The function
$J_{\mf D, \epsilon}$ is defined as the harmonic extension of a
function whose values on neighborhoods of each $\mc M \in \mf D$ are
prescribed by the solution of the variational problem \eqref{20} for
$\bb I_{\mf D}^{(p)}$.

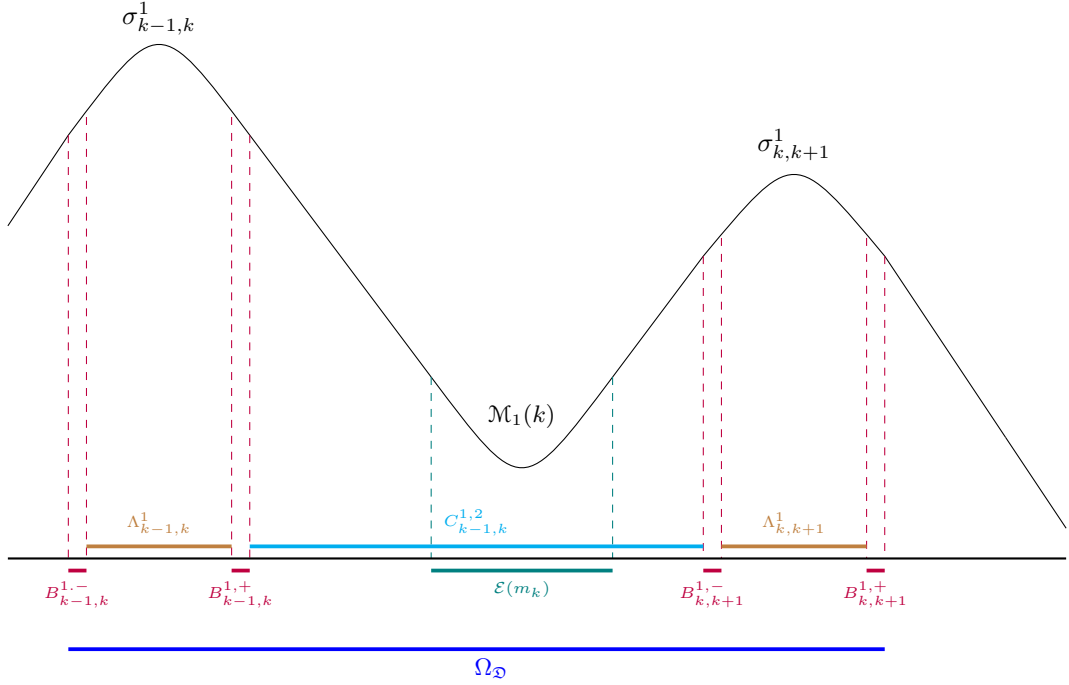
\begin{figure}
\centering
\begin{tikzpicture}[scale=0.8]
\draw[rounded corners] (0,1.5)--(1,3);
\draw[rounded corners] (1,3) .. controls (2.5,5) .. (4,3);  
\draw[rounded corners] (4,3) .. controls (5.5, 1) .. (7, -1);
\draw[rounded corners] (7,-1) .. controls (8.5, -3) .. (10,-1);
\draw[rounded corners] (10,-1) -- (11.5,1);
\draw[rounded corners] (11.5,1) .. controls (13,2.8) .. (14.5, 1);
\draw[rounded corners] (14.5,1)--(16.5,-2);
\draw[rounded corners] (16.5,-2)--(17.5,-3.5);

\fill(2.5,4.5)node[above]{$\sigma_{k-1,k}^1$};
\fill(13,2.4)node[above]{$\sigma_{k,k+1}^1$};
\fill(8.5,-2)node[above, font=\small]{$\mz M_1(k)$};

\draw[dashed, purple](1,3) -- (1,-4);
\draw[dashed, purple](1.3,3.3) -- (1.3,-4);
\draw[dashed, purple](3.7,3.3) -- (3.7,-4);
\draw[dashed, purple](4,3) -- (4,-4);

\draw[dashed, purple](11.5,1)--(11.5,-4);
\draw[dashed,purple](11.8,1.3)--(11.8,-4);
\draw[dashed, purple](14.2,1.3)--(14.2,-4);
\draw[dashed,purple](14.5,1)--(14.5,-4);

\draw[dashed, teal](7,-1)--(7,-4);
\draw[dashed, teal](10,-1)--(10,-4);

\draw[solid, thick, black](0,-4)--(17.5,-4);
\draw[line width=0.5mm, blue](1,-5.5)--(14.5,-5.5);
\fill(8,-5.5)node[below, blue, font=\small]{$\Omega_{\mf D}$};

\draw[line width=0.5mm, purple](1,-4.2)--(1.3,-4.2) node[midway, below, purple, font=\tiny]{$B_{k-1,k}^{1.-}$};
\draw[line width=0.5mm, brown](1.3,-3.8)--(3.7,-3.8) node[midway, above, brown, font=\tiny]{$\Lambda_{k-1,k}^1$};
\draw[line width = 0.5mm, purple](3.7,-4.2)--(4,-4.2) node[midway, below, purple, font=\tiny]{$B_{k-1,k}^{1,+}$};
\draw[line width = 0.5mm, cyan](4,-3.8)--(11.5,-3.8) node[midway, above, cyan, font=\tiny]{$C_{k-1,k}^{1,2}$};
\draw[line width=0.5mm, teal](7,-4.2)--(10,-4.2) node[midway, below, teal, font=\tiny]{$\mc E(m_k)$};
\draw[line width=0.5mm, purple](11.5, -4.2)--(11.8,-4.2) node[midway, below, purple, font=\tiny]{$B_{k,k+1}^{1,-}$};
\draw[line width=0.5mm, brown](11.8,-3.8)--(14.2,-3.8) node[midway, above, brown, font=\tiny]{$\Lambda_{k,k+1}^1$};
\draw[line width=0.5mm, purple](14.2,-4.2)--(14.5,-4.2) node[midway, below, purple, font=\tiny]{$B_{k,k+1}^{1,+}$};

\end{tikzpicture}
\caption{We illustrate an example of the intervals defined in
\eqref{28}, \eqref{max_neigh}, \eqref{core_neigh} and
\eqref{binding_sets} for the first metastable level $p=1$. Here,
$\mf D$ is a $\bb X_1(\cdot)-$equivalent class composed of only one
state, $\mz D = \mz M_1(k)$.}
\label{fig-f11}
\end{figure}

\begin{figure}
\centering
\begin{tikzpicture}[scale=0.6]

\draw[rounded corners](-2.6,3)..controls(-2.1,5)..(-1.6,3);
\draw[rounded corners](-1.6,3)--(-0.6,-1);
\draw[rounded corners](-0.6,-1)..controls(0,-3)..(0.5,-1);
\draw[rounded corners](0.5,-1)--(1,1);
\draw[rounded corners](1,1)--(1.5,3);

\draw[rounded corners] (1.5,3) .. controls (2.3,7) .. (3.1,3);  
\draw[rounded corners] (3.1,3) -- (3.3, 2.2);
\draw[rounded corners] (3.3,2.2) -- (3.6, 1);

\draw[rounded corners] (3.6, 1) -- (4.1, -1);
\draw[rounded corners] (4.1,-1) .. controls (4.6, -3) .. (4.9, -2);
\draw[rounded corners](4.9,-2) .. controls (5.2,-1)..(5.5,-2);
\draw[rounded corners](5.5,-2) .. controls (5.8,-3)..(6.3,-1);
\draw[rounded corners](6.3,-1)..controls (6.8,1)..(7.3,3);

\draw[rounded corners](7.3,3)..controls (7.8,5)..(8.3,3);
\draw[rounded corners](8.3,3)..controls (8.6,2)..(8.9,2.5);

\draw[rounded corners](8.9,2.5)..controls(9.2,3)..(9.5,2.5);
\draw[rounded corners](9.5,2.5)..controls(9.8,2)..(10.1,3);
\draw[rounded corners](10.1,3)..controls(10.6,5)..(11.1,3);
\draw[rounded corners](11.1,3)..controls (11.4,2).. (11.7,3);
\draw[rounded corners](11.7,3)..controls(12.2,5)..(12.7,3);

\draw[rounded corners](12.7,3)..controls (13.2,1)..(13.7,-1);
\draw[rounded corners](13.7,-1)..controls (14.3,-3)..(14.8,-1);

\draw[rounded corners](14.8,-1)--(15.3,1);
\draw[rounded corners](15.3,1)--(15.8,3);

\draw[rounded corners](15.8,3)..controls(16.5,6)..(17.2,3);

\draw[rounded corners](17.2,3)--(17.7,1);
\draw[rounded corners](17.7,1)--(18.2,-1);
\draw[rounded corners](18.2,-1)..controls(18.7,-3)..(19.2,-1);

\draw[rounded corners](19.2,-1)--(20.2,3);
\draw[rounded corners](20.2,3)..controls(20.7,5)..(21.2,3);

\draw[dashed, black](2.8,4.5)--(16.1,4.5);


\fill(2.3,6)node[above]{$\sigma_{\mf D}^{-}$};
\fill(16.5,5.3)node[above]{$\sigma_{\mf D}^{+}$};

\fill(7.8,4.5)node[above]{$\sigma_{k,k+1}^{p,1}$};
\fill(10.6,4.5)node[above]{$\sigma_{k,k+1}^{p,2}$};
\fill(12.2,4.5)node[above]{$\sigma_{k,k+1}^{p,3}$};

\fill(0,-2.5)node[below, font=\tiny]{$\mz M_p(k-1)$};
\fill(5.1,-2.5)node[below, font=\tiny]{$\mz M_p(k)$};
\fill(14.3,-2.5)node[below, font=\tiny]{$\mz M_p(k+1)$};
\fill(18.7,-2.5)node[below, font=\tiny]{$\mz M_p(k+2)$};


\draw[dashed](7.3,3.1)--(7.3,-5);
\draw[dashed](7.5,3.8)--(7.5,-5);

\draw[dashed](8.1,3.8)--(8.1,-5);
\draw[dashed](8.3,3.1)--(8.3,-5);

\draw[line width = 0.8mm, green](7.5,-4.9)--(8.1,-4.9);
\draw[line width = 0.8mm, red](7.3,-5.1)--(7.5,-5.1);
\draw[line width=0.8mm, red](8.1,-5.1)--(8.3,-5.1);

\draw[dashed](10.1,3.1)--(10.1,-5);
\draw[dashed](10.3,3.8)--(10.3,-5);

\draw[dashed](10.9,3.8)--(10.9,-5);
\draw[dashed](11.1,3.1)--(11.1,-5);

\draw[line width=0.8mm, red](10.1,-5.1)--(10.3,-5.1);
\draw[line width=0.8mm, green](10.3,-4.9)--(10.9,-4.9);
\draw[line width=0.8mm, red](10.9,-5.1)--(11.1,-5.1);

\draw[line width=0.8mm, blue](8.3,-4.9)--(10.1,-4.9);

\draw[dashed](11.7,3.1)--(11.7,-5);
\draw[dashed](11.9,3.8)--(11.9,-5);

\draw[dashed](12.5,3.8)--(12.5,-5);
\draw[dashed](12.7,3.1)--(12.7,-5);

\draw[line width=0.8mm, red](11.7,-5.1)--(11.9,-5.1);
\draw[line width=0.8mm, green](11.9,-4.9)--(12.5,-4.9);
\draw[line width=0.8mm, red](12.5,-5.1)--(12.7,-5.1);

\draw[line width=0.8mm, yellow](11.1,-4.9)--(11.7,-4.9);


\draw[dashed](1.6,3.2)--(1.6,-5);
\draw[dashed](1.8,4.2)--(1.8,-5);

\draw[dashed](2.8,4.2)--(2.8,-5);
\draw[dashed](3,3.2)--(3,-5);

\draw[line width=0.8mm, brown](3,-4.9)--(7.3,-4.9);

\draw[line width=0.8mm, red](1.6,-5.1)--(1.8,-5.1);
\draw[line width=0.8mm, green](1.8,-4.9)--(2.8,-4.9);
\draw[line width=0.8mm, red](2.8,-5.1)--(3,-5.1);

\draw[dashed](15.9,3.1)--(15.9,-5);
\draw[dashed](16.1,3.8)--(16.1,-5);
\draw[dashed](16.9,3.8)--(16.9,-5);
\draw[dashed](17.1,3.1)--(17.1,-5);

\draw[line width=0.8mm, red](15.9,-5.1)--(16.1,-5.1);
\draw[line width=0.8mm, green](16.1,-4.9)--(16.9,-4.9);
\draw[line width=0.8mm, red](16.9,-5.1)--(17.1,-5.1);
\draw[line width=0.8mm, pink](12.7,-4.9)--(15.9,-4.9);

\draw[line width=0.8mm](-2.6,-4.9)--(1.6,-4.9);
\draw[line width=0.8mm](17.1,-4.9)--(20.5,-4.9);

\draw[solid](-2.6,-5)--(20.5,-5);

\draw[line width=0.8mm, violet](1.6,-6)--(17.1,-6);

\fill(9.5,-6)node[below]{$\Omega_{\mf D}^p$};

\end{tikzpicture}
\caption{We illustrate the definition of the test function for a
higher-order metastable level $p\geq 2$. In this picture, $\mf D$ is a
closed irreducible class of $\bb X_p(\cdot)$, which is composed by the
states $\mz M_p(k)$ and $\mz M_p(k+1)$. The set of maxima
$\mz W_{\mf D}$, associated to $\mf D$, is equal to
$\{\sigma_{\mf D}^{-}, \sigma_{k,k+1}^{1}, \sigma_{k,k+1}^{2},
\sigma_{k,k+1}^{3}, \sigma_{\mf D}^{+}\}$. Here, the maxima
$\sigma_{\mf D}^{-}$ is the leftmost separating $\mz M_p(k-1)$ and
$\mz M_p(k)$, while $\sigma_{\mf D}^+$ is the rightmost maxima
separating $\mz M_p(k+1)$ and $\mz M_p(k+2)$. Since the set
$\mz W_{k-1,k}$ is a singleton,
$\sigma_{\mf D}^{-} = \sigma_{k-1,k}^{-} = \sigma_{k-1,k}^{+}$.
Similarly, since $\mz W_{k+1,k+2}$ is also a singleton,
$\sigma_{\mf D}^{+} = \sigma_{k+1,k+2}^{+} = \sigma_{k+1,k+2}^{-} =
\sigma_{k,k+1}^{\mf r_{k,k+1}+1}$.  The sets $\Lambda^r_{k-1,k}$
defined in \eqref{max_neigh}, are represented in green. The binding
sets, defined in \eqref{binding_sets}, are the ones in red.  The
intervals in brown, blue, yellow and pink represent the sets
$C^r_{k-1,k}$ defined in \eqref{core_neigh}. The set $\Omega_{\mf D}$,
represented in violet, is the neighborhood of $\mf D$ defined in
\eqref{28}. The test function $J_{\mf D,\epsilon}$ vanishes outside of
$\Omega_{\mf D}$. Inside the brown interval, $J_{\mf D}$ is constant
equal to $f_{\mf D}(\mc M_p(k))$. On the pink interval,
$J_{\mf D,\epsilon}$ is constant equal to $f_{\mf D}(\ms
M_p(k+1))$. On the intervals blue and yellow, the test function is a
convex combination of these two values, where the weights are given by
the function defined in \eqref{01}. In particular, restricted to the
blue interval, $J_{\mf D,\epsilon}$ is constant equal to
$\mss w_{k,k+1}^{1,2} \, f_{\mf D}(\mc M_p(k+1)) + [ 1 - \mss
w_{k,k+1}^{1,2} ]\, f_{\mf D}(\mc M_p(k))$; while on the yellow
interval, the test function is constant equal to
$\mss w_{k,k+1}^{2,3} \, f_{\mf D}(\mc M_p(k+1)) + [ 1 - \mss
w_{k,k+1}^{2,3} ]\, f_{\mf D}(\mc M_p(k))$. }
\label{fig7}
\end{figure}
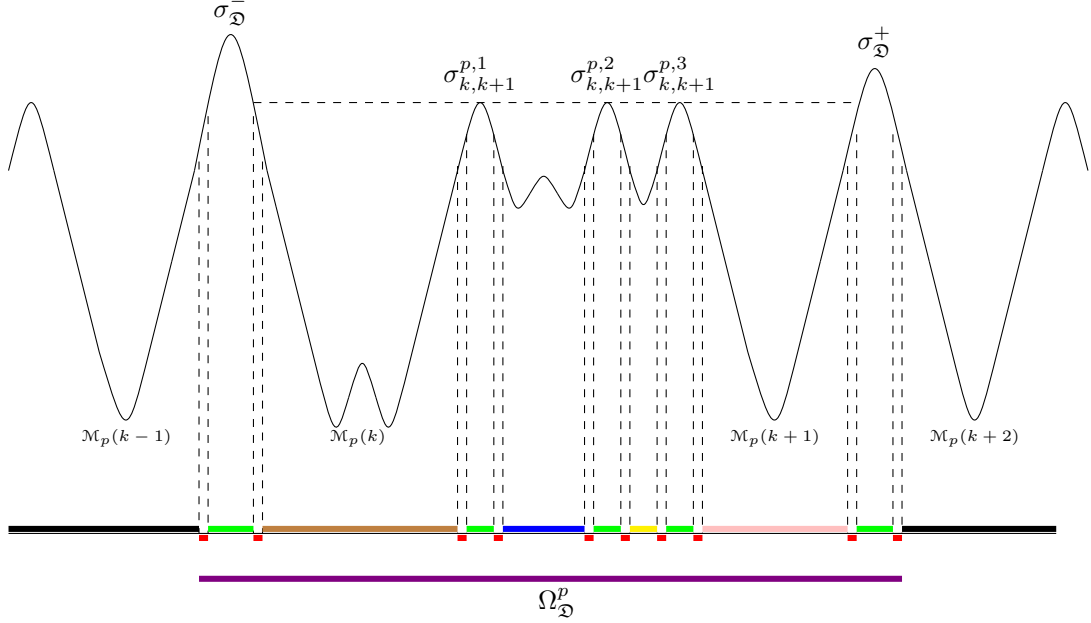

\subsection*{The test function}

For convenience, we first define the test function on $\bb R$, then we
project it onto the torus. We refer to figures \ref{fig-f11} and \ref{fig7} for the
notation introduced below. Fix an $\widehat{\bb X}_p-$equivalent class
$\mf D \subset \mc S_p$. Let $\mz D\subset \mz S_p$ be the
$\bb X_p-$equivalence class, such that
\begin{equation}
\label{lift_equiv_1}
\{\Pi(\mz M): \mz M\in \mz D\} = \mf D   \,, 
\quad\text{and}\quad 
\mz M_p(0) \leq \mz M \leq \mz M_p(\mf u_p -1)
\quad \forall \; \mz M\in \mz D.
\end{equation}
Let
$D\subset \llbracket 0,\, \mf u_p -1\rrbracket = \{0, \dots, \mf u_p
-1\}$ be the set of indices associated to the states contained in
$\mz D$:
${\color{blue}D}:= \left\{k\in \lb 0, \dots, \mf u_p -1\rb : \mz
M_p(k)\in \mz D\right\} \,=\, \left\{k\in \lb 0, \dots, \mf u_p -1\rb
: \mc M_p(k)\in \mf D\right\}$. Denote by $l_{\mf D}, r_{\mf D}$ the
leftmost, rightmost element of $D$, respectively,
${\color{blue}l_{\mf D}} := \min\{k\in D\}$,
${\color{blue} r_{\mf D}} := \max\{k\in D\}$. Hence, 
\begin{equation*}
D \,=\, \lb  l_{\mf D} ,\, r_{\mf D}\rb \,.
\end{equation*}
Recall the definition of the sets $\mz W_{k-1,k}^{(p)}$, $k\in S_p$,
introduced in \eqref{max_br}. Let $\mz W_{\mf D}$ be the set of
absolute maxima separating the metastable states $\mz M_p(k)$,
$k\in \lb  l_{\mf D} -1 ,\, r_{\mf D} +1 \rb$:  
\begin{equation}
\label{max_equiv}
\cb{\mz W_{\mf D} } \,:=\,
\bigcup_{k\,=\,l_{\mf D}}^{r_{\mf D}+1}\mz W_{k-1,k}^{(p)}\,.
\end{equation}

We construct the test function $J_{\mf D,\epsilon}$ in several
steps. We start defining it at neighborhoods of each local maxima
belonging to $\mz W_{\mf D}$. For every
$l_{\mf D} \le k \le r_{\mf D} +1$, denote the maxima of
$\mz W_{k-1,k}^{(p)}$ by
\begin{equation}
\label{43}
\sigma^{1}_{k-1,k} <
\dots < \sigma^{\mf r_{k-1,k}}_{k-1,k};
\end{equation}
where {\color{blue}$\mf r_{k-1,k}$} is the number of elements of
$\mz W_{k-1,k}^{(p)}$.

For $1\leq r\leq \mf r_{k-1,k}$, denote by
$\ms L_{k-1,k}^{r,\epsilon}$ the linearization of the generator
$\ms L_{\epsilon}$ around $\sigma_{k-1,k}^{r}$ (centered at the origin
and acting on the space $C^{2}(\bb R))$.
\begin{equation*}
\label{equil_pot_0}
(\ms L_{k-1,k}^{r,\epsilon} \psi)(x) \;=\; x\,\mss
b'(\sigma_{k-1,k}^{r})\,\,\psi'(x) \, +\,\epsilon\,\mss
a(\sigma_{k-1,k}^{r})\,\psi''(x)\,.
\end{equation*}
Let
${\color{blue}\varrho_{k-1,k}^{r,\epsilon} } \colon \bb R\to \bb [0,1]$
be the solution of the Poisson equation
\begin{equation*}
\ms L_{k-1,k}^{r,\epsilon} u = 0,\;\;
u(-\infty)=0, \;\; u(+\infty)=1.
\end{equation*}
Clearly, 
\begin{equation}
\label{kth_bridge:gaussian_ap}
\varrho_{k-1,k}^{r,\epsilon} (x) =
\sqrt{\frac{-\, S''(\sigma_{k-1,k}^{r})}{2\pi\epsilon}}
\, \int_{-\infty}^{x}
e^{ S''(\sigma_{k-1,k}^{r}) y^{2} / 2\epsilon} \; dy \,. 
\end{equation}

Denote by $\delta_{k-1,k}^{r}$ the scale given by
\begin{equation}
\label{21b}
\cb {\delta_{k-1,k}^{r}} \,: = \, \delta_{k-1,k}^{r}
(\epsilon) \,=\,
\sqrt{\frac{2 \,\epsilon \log
\epsilon^{-1}}{-\, S''(\sigma_{k-1,k}^{r})}}\;.
\end{equation}
The next  assertion states that the function
$\varrho_{k-1,k}^{r,\epsilon} (\cdot)$ at $\pm \, \delta_{k-1,k}^{r}$ is
very close to its asymptotic values. The proof of the assertion relies
on elementary Gaussian bounds.

\begin{assertion}
\label{equil_pot_bound}
There exists a finite constant $C_0$ such that
\begin{equation*}
\varrho_{k-1,k}^{r,\epsilon}  (-\, \delta_{k-1,k}^{r})
\,\le \, C_0\, \frac{\epsilon}{\sqrt{\log \epsilon^{-1}} } \,,
\quad
\varrho_{k-1,k}^{r,\epsilon}  (\delta_{k-1,k}^{r})
\,\ge \, 1 \,-\, C_0 \, \frac{\epsilon}{\sqrt{\log \epsilon^{-1}} } 
\end{equation*}
for all $l_{\mf D} \le k \le r_{\mf D} +1$,
$1\leq r\leq \mf r_{k-1,k}$, and $\epsilon>0$.
\end{assertion}

Fix $l_{\mf D} \le k \le r_{\mf D} +1$. Denote by
$\mss w_{k-1,k}^{r, r+1}$, $0\le r \le \mf r_{k-1,k}$, the
(asymptotic) probability that the diffusion $X_\epsilon(\cdot)$,
starting from a point
$x\in (\sigma_{k-1,k}^{r},\sigma_{k-1,k}^{r+1})$, hits
$\mz M_p(k)$ before $\mz M_p(k-1)$. By equations (2.27), (2.28) and
Remark 3.4 in \cite{lm},
\begin{equation}
\label{01}
{\color{blue}\mss w_{k-1,k}^{r,r+1} } \,:=\, 
\frac{\sum_{l=1}^{r}1/\sqrt{-S''(\sigma_{k-1,k}^{l})}}
{\sum_{l=1}^{\mf r_{k-1,k}}1/\sqrt{-S''(\sigma_{k-1,k}^{l})}}\,. 
\end{equation} 
Clearly, $\mss w_{k-1,k}^{0,1} = 0$ and
$\mss w_{k-1,k}^{\mf r, \mf r+1} =1$, where $\mf r = \mf r_{k-1,k}$.

Recall that $\omega(\,\cdot\, |\, \mf D)$ represents the measure
$\omega$ conditioned to $\mf D$, and $\nu^p_{\mf D}(\cdot)$ the
stationary state of the Markov chain $\widehat{\bb X}_p$ reflected at
$\mf D$.  Lemma \ref{pro_red_stat} provides an explicit formula for
$\nu^p_{\mf D}(\cdot)$.  Let $f_{\mf D} \colon \mc S_p\to \bb R$ be
function given by
\begin{equation}
\label{equivtest}
{\color{blue} f_{\mf D}(\mc M) } \,:=\, 
\sqrt{\frac{\omega(\mc M|\mf D)}
{\nu^p_{\mf D}(\mc M)}} \,
1\{ \mc M \in \mf D \} \,.
\end{equation}

Fix $l_{\mf D} \le k \le r_{\mf D} +1$, $1\leq r\leq \mf r_{k-1,k}$,
and recall the weights $\mss w_{k-1,k}^{r,r+1}$ introduced in
\eqref{01}.  Let $f_{k-1,k}^{r,r+1}$ be the asymptotic expectation of
$f_{\mf D}(X_\epsilon(\tau))$, where $\tau$ is the hitting time of the
set $\{ \mc M_p(k-1) , \mc M_p(k)\}$, when the process starts from a
point between $\sigma_{k-1,k}^{r}$, and $\sigma_{k-1,k}^{r+1}$:
\begin{equation}
\label{08}
\cb{ f_{k-1,k}^{r,r+1} }  \,:=\, 
\mss w_{k-1,k}^{r,r+1} \, f_{\mf D}(\mc M_p(k))
\,+\,
[\, 1 - \mss w_{k-1,k}^{r,r+1} \, \,] \,
f_{\mf D}(\mc M_p(k-1)) \,.
\end{equation}
Note that, by the definition of $f_{\mf D}(\cdot)$,
$f_{\mf D}(\mc M_p(l_{\mf D} -1)) = f_{\mf D}(\mc M_p(r_{\mf D} +1))
=0$.

Denote by $\rho_{k-1,k}^{r,\epsilon} (x)$ a function similar to
$\varrho_{k-1,k}^{r,\epsilon} (x)$, but which interpolates between
$f_{k-1,k}^{r-1,r}$ and $f_{k-1,k}^{r,r+1}$ instead of interpolating
between $0$ and $1$:
\begin{align}
\label{equil_pot}
{\color{blue} \rho_{k-1,k}^{r,\epsilon} (x) }
:= f_{k-1,k}^{r,r+1} \,  \varrho_{k-1,k}^{r,\epsilon} (x)\,+\,
f_{k-1,k}^{r-1,r}\, [\, 1- \varrho_{k-1,k}^{r,\epsilon} (x)\,]
\end{align}

We are now in a position to define the test function
$J_{\mf D,\epsilon} (\cdot)$. This is done writing the torus as the
union of disjoint sets and defining the test function in each of these
sets. We refer to figures \ref{fig-f11} and \ref{fig7} for an illustration of the
definitions.  Below $k$ is fixed and belongs to the interval $\lb\, l_{\mf
D} ,\,  r_{\mf D} +1\,\rb $.

For $1\leq r\leq \mf r_{k-1,k}$, let the scale $\eta_{k-1,k}^{r}$ be
given by
\begin{equation*}
{\color{blue}  \eta_{k-1,k}^{r} }
\,:=\, \frac{\epsilon}{\delta_{k-1,k}^{r}}
\;=\;
\sqrt{\frac{-\, \epsilon \, S''(\sigma_{k-1,k}^{r})}
{2  \,  \log \epsilon^{-1}}}\;.
\end{equation*}
Denote by $\sigma^{-}_{\mf D}$, $\sigma^{+}_{\mf D}$ are the leftmost,
rightmost maxima in $\mz W_{\mf D}$, respectively, 
\begin{equation}
\label{39}
{\color{blue}\sigma^{-}_{\mf D}} \,:= \,
\min\{\sigma \in \mz W_{\mf D}\}\,,\quad
{\color{blue}\sigma^{+}_{\mf D}}:= \max\{\sigma \in \mz W_{\mf D}\}\,. 
\end{equation}
Thus, $\sigma^{-}_{\mf D} = \sigma_{l_{\mf D} -1, l_{\mf D}}^{1}$,
$\sigma^{+}_{\mf D} = \sigma_{r_{\mf D} , r_{\mf D} +1}^{\mf r_{ \mf
D}}$, where
$\cb {\mf r_{\mf D} } := \mf r_{ r_{\mf D} , r_{\mf D} +1}$, are
respectively the leftmost, rightmost global maxima separating the sets
$\{\mz M_p(k) : k\in D\}$ and $\{\mz M_p(k) : k\not\in D\}$.  Denote
by $\Omega_{\mf D}$ the neighborhood of $\mf D$ defined by
\begin{equation}
\label{28}
{\color{blue}\Omega_{\mf D}}
\,:=\, (\sigma^{-}_{\mf D} 
- \delta_{l_{\mf D} -1, l_{\mf D}}^{1} 
- \eta_{l_{\mf D} -1, l_{\mf D}}^{1} \,,\,
\sigma^{+}_{\mf D}  
+\delta_{r_{\mf D} , r_{\mf D} +1}^{\mf r_{\mf D} } 
+\eta_{r_{\mf D} , r_{\mf D} +1}^{\mf r_{\mf D} })\,.
\end{equation}
Set the test function to vanish outside the interval $\Omega_{\mf D}$:
\begin{equation*}
\label{comp_neigh_test_f}
J_{\mf D, \epsilon} (x) = 0\,,\quad x\in \bb R\setminus \Omega_{\mf D}\,.     
\end{equation*}

Inside the set $\Omega_{\mf D}$, we first define the test function
close to the global maxima $\sigma_{k-1,k}^{r}$ and far from
them. Then, we interpolate smoothly between these intervals.  For
$1\leq r\leq \mf r_{k-1,k}$, let $\Lambda_{k-1,k}^{r}$ be the
$\delta_{k-1,k}^{r}$-neighborhood of $\sigma_{k-1,k}^{r}$
\begin{equation}
\label{max_neigh} 
{\color{blue} \Lambda_{k-1,k}^{r}}
\,:=\, \big[\,  \sigma_{k-1,k}^{r} - \delta_{k-1,k}^{r} \,,\,
\sigma_{k-1,k}^{r} + \delta_{k-1,k}^{r} \,\big] \;.
\end{equation}
The test function $J_{\mf D,\epsilon}$ is defined on
$\Lambda_{k-1,k}^{r}$ as
\begin{equation}
\label{kth_core_test_f}
{\color{blue}J_{\mf D,\epsilon} (x)} \,:=\,
\rho_{k-1,k}^{r,\epsilon} (x-\sigma_{k-1,k}^{r})\,,
\quad x\in \Lambda_{k-1,k}^{r} \,.
\end{equation}

Let $C_{k-1,k}^{r,r+1}$ be the interval between successive
maxima given by
\begin{equation}
\label{core_neigh}
 {\color{blue} C_{k-1,k}^{r,r+1} }
\,:=\, \big[\, \sigma_{k-1,k}^{r}  + \delta_{k-1,k}^{r}
+ \eta_{k-1,k}^{r}  \,,\,
\sigma_{k-1,k}^{r+1}  - \delta_{k-1,k}^{r+1}  -
\eta_{k-1,k}^{r+1} \,\big] \;.   
\end{equation}
In the interval $C_{k-1,k}^{r,r+1}$ the test function
$J_{\mf D, \epsilon} (\cdot)$ is constant and takes the value
$f_{k-1,k}^{r,r+1}$ introduced in \eqref{08}:
\begin{equation}
\label{kth_core_test_f_2}
{\color{blue} J_{\mf D,\epsilon}  (x)}  \,:=\, f_{k-1,k}^{r,r+1}\,. 
\end{equation}
In particular, $J_{\mf D, \epsilon}  (x) = f_{\mf D}(\mc M_p(k))$ for
$x\in C_{k-1,k}^{\mf r , \mf r+1}$, where $\mf r = \mf r_{k-1,k}$,
and $J_{\mf D,\epsilon}  (x) = f_{\mf D}(\mc M_p(k-1))$ for
$x\in C_{k-1,k}^{0 ,1}$.

To complete the definition of $J_{\mf D, \epsilon}$, consider the
binding sets $B_{k-1,k}^{r,-}$ $B_{k-1,k}^{r,+}$, $B_{k-1,k}^{r}$,
$B_{\mf D}$ given by
\begin{equation}
\label{binding_sets}
\begin{gathered}
{\color{blue} B_{k-1,k}^{r,-}  }  \,:=\,
(\sigma_{k-1,k}^{r}  - \delta_{k-1,k}^{r}  -
\eta_{k-1,k}^{r}  \,,\,
\sigma_{k-1,k}^{r}  - \delta_{k-1,k}^{r} )\;,
\\
{\color{blue} B_{k-1,k}^{r,+} }  \,:=\,
(\sigma_{k-1,k}^{r}   + \delta_{k-1,k}^{r}  \,,\,
\sigma_{k-1,k}^{r}   + \delta_{k-1,k}^{r}  +
\eta_{k-1,k}^{r}   )\;,
\\
{\color{blue}B_{k-1,k}^{r} } := B_{k-1,k}^{r,-}  \cup
B_{k-1,k}^{r,+}  \,,\quad
{\color{blue} B_{\mf D}}  \,:=\,
\bigcup_{k = l_{\mf D}  }^{r_{\mf D} +1}
\bigcup_{r=1}^{\mf r_{k-1,1}} 
B_{k-1,k}^{r} \,.
\end{gathered}
\end{equation}
On these biding sets, let $J_{\mf D,\epsilon}$ be defined as a smooth
function ${\color{blue}\Phi_{\mf D,\epsilon}}$ in $C^{2}(B_{\mf D})$
that satisfies the following properties
\begin{itemize}
\item $\Phi_{\mf D,\epsilon}$ turns $J_{\mf D,\epsilon}$ into an
element in $C^{2}(\bb R)$,

\item $\Phi_{\mf D,\epsilon}'(x)\neq 0$ for every $x\in B_{\mf D}$.

\item
$\sup_{x\in B_{\mf D} }|\Phi_{\mf D, \epsilon}'(x)| =
O(\sqrt{\epsilon})$.
\end{itemize}
\noindent Assertion \ref{equil_pot_bound} guarantees the existence of
such a function since the length of each interval $B_{k-1,k}^{r, \pm}$
is of order $\eta_{k-1,k}^{r}$.

Recall from \eqref{28} the definition of the interval
$\Omega_{\mf D}$.  Since there is more than one $\bb X_p$-equivalent
class, there exist $a>0$ and $\epsilon_0>0$ such that
$\Omega_{\mf D} \subset (\sigma^-_{\mf D} - a, \sigma^-_{\mf D}-a +
1)$ for all $0<\epsilon<\epsilon_0$. Thus, $\Omega_{\mf D}$ can be
considered as a subset of the torus, and $J_{\mf D, \epsilon}$ as a
function on the torus of class $C^{2}(\bb T)$.

\subsection*{The local Gibbs measure}

In this subsection, $\mf D\subset \mc S_p$ is a
$\widehat{\bb X}_p-$equivalence class and $\mz D$ its lifted version
to $\bb R$ satisfying \eqref{lift_equiv_1}. Recall the notation
introduced in \eqref{28}.  Fix a positive number $a>0$ sufficiently
small such that
\begin{equation*}
\big( \, \sigma^-_{\mf D} - a \,,\,
\sigma^-_{\mf D} \,\big) \,
\cap \mz M \,=\, \varnothing\,,
\quad
\big(\, \sigma^+_{\mf D} \,,\,
\sigma^+_{\mf D} + a\,\big)
\,\cap \, \mz M = \varnothing\,.
\end{equation*}
Clearly,
$\Omega_{\mf D} \subset ( \sigma^-_{\mf D} - a, \sigma^+_{\mf D} + a)$
for $\epsilon$ sufficiently small. Let
$\cb {S_{\mf D}}\colon \bb T\to \bb R$ be a function of class
$C^2(\bb T)$ such that
\begin{equation*}
\label{local_potential}
S_{\mf D}(\Pi(x)) \,=\,  S(x)\;, \quad
x\,\in\,  ( \sigma^-_{\mf D} - a \,,\,  \sigma^+_{\mf D} + a) \,.
\end{equation*}

From now on, we work on the torus. Denote by
$\mc M_p(\mf D)\subset \mc M_p$ the set of local minima of level $p$
belonging to $\mf D$,
\begin{equation*}
{\color{blue}\mc M_p(\mf D)} := \bigcup_{\mc M\in \mf D}\mc M\,,
\end{equation*}
and by $\mc E(\mc M_p(\mf D))\subset \bb T$ the union of the
associated metastable wells,
\begin{equation*}
{\color{blue} \mc E(\mc M_p(\mf D))}:=
\bigcup_{m\in \mc M_p(\mf D)}\mc E(m)\,.
\end{equation*}
The next result follows directly from Proposition
\ref{depth_equivalent}, and the definitions of the sets
$\Omega_{\mf D}$, $\mc E(\mc M_p(\mf D))$.

\begin{proposition}
\label{deepest_min}
For every $\widehat{\bb X}_p(\cdot)-$equivalence class
$\mf D\subset \mc S_p$,
\begin{gather*}
\min\big\{\, S_{\mf D}(m)\;:\;
m\in \big[\, \Omega_{\mf D} \cap \mc M\, \big ]
\setminus \mc M_p(\mf D) \, \big \}\;
>\; S_{\mf D}(\mc M_p(\mf D)) \,,
\\
\inf_{x\,\in\, \Omega_{\mf D}\setminus \ms
E(\mc M_p(\mf D))} S_{\mf D} (x)
\,>\, S_{\mf D}(\mc M_p(\mf D))\,.
\end{gather*}
\end{proposition}

Recall the definition of the weights $\pi_1(j)$, $\pi_p(k)$,
$j\in S_1$, $k\in S_p$, introduced in \eqref{22}, \eqref{p_weights},
respectively.

\begin{corollary}
\label{wells_conv}
Fix a $\widehat{\bb X}_p(\cdot)-$equivalence class
$\mf D\subset \mc S_p$. For every continuous function
$\phi\in C(\bb T)$,
\begin{equation*}
\int_{\Omega_{\mf D}}
\frac{\phi(x)}{\mss a(x)} \,
e^{-S_{\mf D}(x)/\epsilon}\, dx \,=\,
[1+O(\epsilon)]\,
\sqrt{\epsilon}\,
e^{-S_{\mf D}(\mc M_p(\mf D))/\epsilon}
\, \sum_{j:m_j\in\mc M_p(\mf D)}\phi(m_j)\pi_1(j)\,.
\end{equation*}
In particular,
\begin{equation*}
\int_{\Omega_{\mf D}}
\frac{1}{\mss a(x)}
\, e^{-S_{\mf D}(x)/\epsilon}dx \,=\,
[1+O(\epsilon)]\,
\sqrt{\epsilon}\,e^{-S_{\mf D}(\ms
M_p(\mf D))/\epsilon} \, \sum_{k\in D}\pi_p(k)\,.
\end{equation*}
\end{corollary}

\begin{proof}
Fix $\phi \in C(\bb T)$, and write
\begin{align*}
\int_{\Omega_{\mf D}}
\frac{\phi(x)}{\mss a(x)} \,
e^{-S_{\mf D}(x)/\epsilon}\,dx
\,=\,  \int_{\mc E(\mc M_p(\mf D))}\frac{\phi(x)}{\mss a(x)}
\, e^{-S_{\mf D}(x)/\epsilon}\,dx 
\,+\, \int_{\Omega_{\mf D}\setminus \mc E(\mc M_p(\mf D))}
\frac{\phi(x)}{\mss a(x)}
\, e^{-S_{\mf D}(x)/\epsilon}\,dx \,.
\end{align*}
We treat each integral on the right-hand side separately. Since $\phi$
is bounded and $\mss a$ is bounded below by a strictly positive
constant, the second integral is less than or equal to
\begin{align*}
C_0\,
e^{-S_{\mf D}(\mc M_p(\mf D))/\epsilon}
\sup_{x\in \Omega_{\mf D}\setminus \ms
E(\mc M_p(\mf D))}\,
e^{\,[S_{\mf D}(\mc M_p(\mf D))- S_{\mf D}(x)]\,/\,\epsilon}
\end{align*}
for some finite constant $C_0$. By Proposition \ref{deepest_min}, this
expression is bounded by
\begin{align*}
e^{-\,[\,S_{\mf D}(\ms
M_p(\mf D))\,+\,h'\,]\,/\,\epsilon}\,, 
\end{align*}
for some positive constant $h'>0$.

We turn to the first integral. As $\phi(\cdot)$, and $\mss a (\cdot)$
are continuous functions, by the definition of $\mc E(\mc M_p(\mf D))$
and Laplace's method, the first integral is equal to
\begin{align*}
\sum_{m\in \mc M_p(\mf D)}\int_{\mc E(m)}\frac{\phi(x)}{\mss a(x)}
e^{-S_{\mf D}(x)/\epsilon}\,dx
&=\, e^{-S_{\mf D}(\mc M_p(\mf D))/\epsilon}\sum_{m\in \ms
M_p(\mf D)}\frac{\phi(m)}{\mss
a(m)}\sqrt{\frac{2\pi\epsilon}{S''(m)}}\left(1+O(\epsilon)\right)\\
&=\, (1+O(\epsilon))\,\sqrt{\epsilon}\,
e^{-S_{\mf D}(\mc M_p(\mf D))/\epsilon}
\sum_{k\in D}\sum_{m\in \mc M_p(k)}\phi(m)\pi_1(k).
\end{align*}
To conclude the proof, in remains to recollect the previous estimates,
and recall the definition of $\pi_p(k)$ given in \eqref{p_weights}.
\end{proof}

\subsection*{A recovery sequence $\mu_\epsilon$ for $\mu$}

In this subsection we construct a recovery sequence $\mu_\epsilon$ for
the measure $\mu$. Fix a $\widehat{\bb X}_p(\cdot)-$equivalence class
$\mf D\subset \mc S_p$, and let
${\color{blue}\mu_{\mf D,\epsilon}}\in \mss P(\bb T)$ be the
probability measure given by the density
\begin{equation}
\label{recovery_seq_equiv}
{\color{blue}m_{\mf D,\epsilon}(x)}
\,:= \, \frac{\big[J_{\mf D, \epsilon} (x)\big]^2
\,e^{-S_{\mf D}(x)/\epsilon}}{\mss a(x) \, c_{\mf D}(\epsilon)}\,,
\end{equation}
where $J_{\mf D, \epsilon}$ and $S_{\mf D}$ are the
functions defined in the previous subsection, and
$c_{\mf D}(\epsilon)$ is a normalization constant that turns
$\mu_{\mf D, \epsilon}$ into a probability measure.

Recall the definition of the measure $\mu$ introduced in
\eqref{29} and that we assumed $\omega$ to be strictly
positive. Thus, with the notation introduced in Section
\ref{sec5}, $\mc S^{(\omega)}_p = \mc S_p$ and the sets $\mf D_r^p$,
$1\leq r\leq l_p$, are exactly the
$\widehat{\bb X}_p(\cdot)-$equivalent classes with two or more
elements. Let ${\color{blue}k_p}\geq 0$ be the number of
$\widehat{\bb X}_p-$equivalent classes having exactly one element, and
denote them by {\color{blue}$\mf C_r$}, $1\leq r\leq k_p$. Define
$\mu_\epsilon\in \mss P(\bb T)$ as the convex combination, with respect
to $\omega$, of the measures
$\{\mu_{\mf D_r}\}_{r=1}^{l_p}\cup \{\mu_{\mf C_r}\}_{r=1}^{k_p}$:
\begin{equation}
\label{recovery_seq_inter}
\mu_\epsilon := \sum_{r=1}^{l_p}\omega(\mf D_r)
\, \mu_{\mf D_r, \epsilon}\,+\,
\sum_{r=1}^{k_p}\omega(\mf C_r)\, \mu_{\mf C_r, \epsilon}\,.
\end{equation}

\begin{assertion}
\label{s05}
For every $\widehat{\bb X}_p(\cdot)-$equivalent class $\mf D$,
\begin{gather*}
\lim_{\epsilon\to 0}\, \frac{1}{\sqrt{\epsilon}}\,
e^{S_{\mf D}(\mc M_p(\mf D))/\epsilon}\,
c_{\mf D}(\epsilon)\, \, =\,
\sum_{k\in D}\pi_p(k)\,.     
\end{gather*}
\end{assertion}

\begin{proof}
Fix a $\widehat{\bb X}_p(\cdot)-$equivalence class $\mf D$. By the definitions of
$c_{\mf D}(\epsilon)$, $\Omega_{\mf D}$, and $J_{\mf D,\epsilon}$,
\begin{align*}
c_{\mf D}(\epsilon) = \sum_{k\in D}
f_{\mf D}(\mc M_p(k))^2\,
\int_{\mc E_p(k)}\frac{1}{\mss a(x)} \, e^{-S_{\mf D}(x)/\epsilon}\,dx
+ \int_{\Omega_{\mf D}\setminus\mc E(\mc M_p(\mf D))}
\frac{1}{\mss a(x)}\big[J_{\mf D,\epsilon}(x)\big]^2
\, e^{-S_{\mf D}(x)/\epsilon}\,dx\,.
\end{align*}
We treat each integral separately. Since $J_{\mf D,\epsilon}$ is
uniformy bounded, by the first part of the proof of Corollary
\ref{wells_conv}, the second integral is bounded above by
$e^{-[\,S_{\mf D}(\mc M_p(\mf D))\,+\,h'\,]\,/\,\epsilon}$, for some
positive constant $h'>0$.

We turn to the first integral. By the proof of Corollary
\ref{wells_conv}, this integral is equal to
\begin{gather*}
[\,1 + O(\epsilon)\,]\, e^{-S_{\mf D}(\mc M_p(\mf D))/\epsilon}\,
\sqrt{\epsilon}\,
\sum_{k\in D}f_{\mf D}(\mc M_p(k))^2 \, \pi_p(k)\,.
\end{gather*}
To complete the proof, it remains to recall the definitions of
$f_{\mf D}$ and $\nu_{\mf D}^p$, given in \eqref{equivtest} and Lemma
\ref{pro_red_stat}, respectively; and to recollect the previous
estimates.
\end{proof}

\begin{lemma}
\label{weak_conv_p_limsup}
The sequence $\{\mu_\epsilon\}_{\epsilon>0}$, given by
\eqref{recovery_seq_inter}, weakly converges to $\mu$ as $\epsilon \to 0$.
\end{lemma}

\begin{proof}
Fix $f\in C(\bb T)$. By the definition of $\mu_\epsilon$,
\begin{align*}
\int_{\bb T}f(x)\;d\mu_\epsilon &
= \sum_{r=1}^{l_p}\omega(\mf D_r)\,
\frac{1}{c_{\mf D_r}(\epsilon)} \,
\int_{\Omega_{\mf D_r}}\;f(x)\; J_{\mf D_r,\epsilon} (x) ^2\;
\frac{e^{-S_{\mf Dr}(x)/\epsilon}}{\mss a(x)}\;dx \notag
\\  
&+ \sum_{l = 1}^{k_p}\omega(\mf C_l)\,
\frac{1}{c_{\mf C_l}(\epsilon)}\int_{\Omega_{\mf C_l}}f(x)\;
J_{\mf C_l,\epsilon} (x)^2 \;
\frac{e^{-S_{\mf C_l}(x)/\epsilon}}{\mss a(x)}\;dx. 
\end{align*}
A calculation, similar to the one performed in the proof of Assertion
\ref{s05}, yields that this sum is equal to
\begin{align*}
&\sum_{r = 1}^{l_p}\sum_{k\in D_r}
\omega(\mc M_p(k)) \, \frac{1}{\pi_1(\mc M_p(k)) }\, 
\sum_{\mf m\in \mc M_p(k)}f(\mf m) \, \pi_1(\{\mf m\}) 
\\
& + \sum_{r = 1}^{k_p}
\sum_{k\in C_r}
\omega(\mc M_p(k)) \, \frac{1}{\pi_1(\mc M_p(k)) }\, 
\sum_{\mf m\in \mc M_p(k)}f(\mf m) \, \pi_1(\{\mf m\}) 
+ r(\epsilon)\,,
\end{align*}
where $r(\epsilon)$ is an error that vanishes as $\epsilon \to 0$. By
\eqref{29} and the expression \eqref{21} for the measure
$\mu^p_{\mc M_p(k)}$, this expression is equal to
$\int_{\bb T}f(x)\;d\mu + r(\epsilon)$, which completes the proof
of the lemma.
\end{proof}

\subsection*{Proof of Proposition \ref{higher_speed_glimsup}}

The following lemma, whose proof is postponed to the next subsection,
collects the required asymptotic bounds to establish the proposition.

\begin{lemma}
\label{main_lemma_limsup_p}
Let $\mu_\epsilon$ be the sequence defined in
\eqref{recovery_seq_inter}. Fix a
$\widehat{\bb X}_p(\cdot)-$equivalence class $\mf D\subset \mc S_p$.

\begin{enumerate}
\item If $\mf D$ has two or more elements, then
\begin{align*}
& \limsup_{\epsilon\to 0}\theta_{\epsilon}^{(p)} \,
\ms I_\epsilon(\mu_{\mf D,\epsilon})\,
\leq\, \bb I_{\mf D}^{(p)}(\omega(\,\cdot\,|\,\mf D))
\\
&\quad 
\,+\, \omega(\mc M_p(l_{\mf D} ) \, |\, \mf D)\,
R_p(\mc M_p(l_{\mf D}) \,, \mc M_p(l_{\mf D}-1))
\,+\, \omega(\mc M_p(r_{\mf D} ) \, |\, \mf D)\,
R_p(\mc M_p(r_{\mf D}) \,,  \mc M_p(r_{\mf D}+1)) \,.
\end{align*}
    
\item If $\mf D$ is a singleton and its element is a transient state, then
\begin{equation*}
\limsup_{\epsilon \to 0}\theta_{\epsilon}^{(p)} \,
\ms I_\epsilon(\mu_{\mf D,\epsilon}) \,\leq
\, R_p(\mc M_p(l_{\mf D}) \,, \mc M_p(l_{\mf D}-1))
\,+\, R_p(\mc M_p(r_{\mf D}) \,, \mc M_p(r_{\mf
D}+1))\,. 
\end{equation*} 

\item If $\mf D$ is a singleton and its element is an absorbing state,
then
\begin{equation*}
\limsup_{\epsilon \to 0}\theta_{\epsilon}^{(p)} \,
\ms I_\epsilon(\mu_{\mf D,\epsilon}) = 0\,.
\end{equation*} 
\end{enumerate}  
\end{lemma}

We are now in position to prove Proposition \ref{higher_speed_glimsup}.

\begin{proof}[Proof of Proposition \ref{higher_speed_glimsup}]
Fix $1\leq p<\widehat{\mf q}$. By Lemma \ref{weak_conv_p_limsup}, and in view of
the definition \eqref{23b} of $\ms J^{(p)} (\cdot)$, and formula
\eqref{reduced_ldf0} for $\bb I^{(p)} (\cdot)$, to complete the proof
of Proposition \ref{higher_speed_glimsup} it remains to show that
\begin{align}
\label{31}
\limsup_{\epsilon\to 0}\theta_{\epsilon}^{(p)}\,
\ms I_\epsilon(\mu_\epsilon)
\, \leq \, \sum_{r=1}^{l_p}\omega(\mf D_{r})\;
\bb I^{(p)}_{\mf D_{r}}(\omega(\;\cdot\;|\;\mf D_{r}))
\;+\; \sum_{r=1}^{o_p}\sum_{s\neq r}
\sum_{\mc M\in \mf E_{r}}\sum_{\mc M'\in \mf E_s}
\omega(\mc M)\, R_p(\mc M,\mc M') \,.
\end{align}
Split the second sum distinguishing transient states belonging to an
equivalent class with two or more elements from the singletons to
rewrite it as
\begin{equation*}
\sum_{r=1}^{l_p}\sum_{\mc M\in \mf D_r}
\sum_{\mc M'\in \mc S_p\setminus \mf D_r}
\omega(\mc M)\,  R_p(\mc M, \mc M') \,+\, \sum_{r = 1}^{k_p}
\sum_{\mc M\in \mf C_r} \sum_{\mc M'\in \mc S_p\setminus \mc M}
\omega(\mc M) \, R_p(\mc M,\mc M')\,.
\end{equation*}

We turn to the proof of \eqref{31}.  By convexity and the definition
of $\mu_{\epsilon}$,
\begin{equation*}
\theta_{\epsilon}^{(p)} \, \ms I_\epsilon(\mu_\epsilon)
\,\leq\, \sum_{r=1}^{l_p}\omega(\mf D_r)\,
\theta_\epsilon^{(p)} \, \ms I_\epsilon(\mu_{\mf D_r, \epsilon})
\,+\,\sum_{r=1}^{k_p}\omega(\mf C_r)\,
\theta_\epsilon^{(p)} \, \ms I_\epsilon(\mu_{\mf C_r, \epsilon})\,.
\end{equation*}
By item (1) in Lemma \ref{main_lemma_limsup_p} and the fact that
$\widehat{\bb X}_p(\cdot)$ jumps to nearest neighbors only,
\begin{equation*}
\limsup_{\epsilon \to 0}\sum_{r =1}^{l_p}
\omega(\mf D_r) \, \theta_{\epsilon}^{(p)} \,
\ms I_\epsilon(\mu_{\mf D_r,\epsilon})
\,\leq\,
\sum_{r=1}^{l_p} \omega(\mf D_r)\,
\bb I^{(p)}_{\mf D_{r}}(\omega(\,\cdot\,|\,\mf D_r))
\,+\, \sum_{r=1}^{l_p} \sum_{\mc M\in \mf D_r}
\sum_{\mc M'\in \mc S_p\setminus \mf D_r^p}\,
\omega(\mc M) \,  R_p(\mc M, \mc M')\,.
\end{equation*} 
By items (2) and (3) of Lemma \ref{main_lemma_limsup_p}, 
\begin{equation*}
\limsup_{\epsilon \to 0} \sum_{r = 1}^{k_p}
\omega(\mf C_r) \, \theta_{\epsilon}^{(p)} \,
\ms I_\epsilon(\mu_{\mf C_r ,\epsilon})
\,\leq\, \sum_{r = 1}^{k_p} \sum_{\mc M\in \mf C_r}
\sum_{\mc M'\in \mc S_p\setminus \mc M}
\omega(\mc M) \, R_p(\mc M,\mc M') \,.
\end{equation*}
To complete the proof the proposition, it remains to recollect the
previous estimates.
\end{proof}

\subsection*{Proof of Lemma \ref{main_lemma_limsup_p}}

The proof requires some preliminary estimates. We only focus on item
(1) since a similar argument applies to items (2) and (3).  Fix an
equivalent class $\mf D\subset \mc S_p$ having two or more
elements. As in Section \ref{sec4.2}, we slightly perturb
$\mu_{\mf D,\epsilon}$ in order to apply Lemma \ref{l03} to estimate
$\ms I_\epsilon (\mu_{\mf D,\epsilon})$.

Given $\delta>0$, let $\mu^\delta_{\mf D,\epsilon}\in \mss P(\bb T)$ be
the probability measure given by the density
\begin{equation*}
m_{\mf D,\epsilon}^\delta(x) \,:=\,
\frac{m_{\mf D,\epsilon}(x)+\delta}{1+\delta},
\end{equation*}
where $m_{\mf D,\epsilon}$ has been introduced in
\eqref{recovery_seq_equiv}. The proof of Assertion
\ref{estim_00_limsup1} yields the next result.

\begin{assertion}
\label{estim_limsup_high}
For each $\epsilon>0$, 
\begin{equation*}
\lim_{\delta\to 0}
\Big|\; \ms I_\epsilon (\mu_{\mf D, \epsilon}^\delta) \;-\;
\frac{1}{4\epsilon}\int_{\bb T} \mss a\,
\Big\{\epsilon \Big[\frac{(m_{\mf D,\epsilon}^\delta)'}{
m_{\mf D,\epsilon}^\delta}
\,+\, \frac{\mss a'}{\mss a}\Big]
\,-\, \frac{\mss b}{\mss a} \Big\}^2
d\mu_{\mf D,\epsilon}^\delta \;\Big|
\;=\, 0\,.
\end{equation*}    
\end{assertion}

Next results follows from  this estimate and the definition of the
test function $J_{\mf D,\epsilon}$.

\begin{assertion}
\label{a1:sup_high}
For each $\epsilon>0$, 
\begin{align*}
\ms I_\epsilon(\mu_{\mf D,\epsilon})
\,\leq \,
\sum_{k = l_{\mf D} }^{\, r_{\mf D} +1} 
\sum_{r=1}^{\mf r_{k-1,k}}
\int_{B_{\sigma_{k-1,k}^{p,r}}^p\cup
\,\Lambda_{\sigma_{k-1,k}^{p,r}}^p}
\big[J'_{\mf D,\epsilon} (x) \big]^2\;
\frac{e^{-S_{\mf D}(x)/\epsilon}}
{c_{\mf D}(\epsilon)}\;dx \,.
\end{align*}
\end{assertion}

\begin{proof}
Since $\mu_{\mf D,\epsilon}^\delta$ converges weakly to
$\mu_{\mf D,\epsilon}$, by the lower-semicontinuity of
$\ms I_\epsilon (\cdot)$,
\begin{equation*}
\ms I_{\epsilon}(\mu_{\mf D,\epsilon})
\, \leq\,  \liminf_{\delta \to 0}
\ms I_\epsilon(\mu_{\mf D,\epsilon}^{\delta}). 
\end{equation*}
At this point, recall Assertion \ref{estim_limsup_high}.
Differentiate $m_{\mf D,\epsilon}^\delta$, expand the square and
apply the dominated convergence theorem, to get that
\begin{equation*}
\ms I_\epsilon(\mu_{\mf D,\epsilon} ) \,\leq\,
\epsilon \int_{\bb T} \big[ J'_{\mf D,\epsilon} (x) \big]^2 \;
\frac{e^{-S_{\mf D}(x)/\epsilon}}{c_{\mf D}(\epsilon)}\;dx\,.
\end{equation*}
To complete the proof, recall that $J_{\mf D,\epsilon}$ vanishes on
$\bb T\setminus \Omega_{\mf D}$ and that, by
\eqref{kth_core_test_f_2}, $J_{\mf D,\epsilon}'$ vanishes on the
intervals $C^{r,r+1}_{k-1,k}$.
\end{proof}

The next results states that in Assertion \ref{a1:sup_high} the
integrals over the binding sets are negligible.

\begin{assertion}
\label{a2:sup_high}
For each $l_{\mf D} \le k \le r_{\mf D} +1$, $1\leq r\leq \mf
r_{k-1,k}$, 
\begin{equation*}
\limsup_{\epsilon\to 0}\frac{\theta_{\epsilon}^{(p)}\,
\epsilon}{c_{\mf D}(\epsilon)} \,
\int_{B^r_{k-1,k} } \big[ J'_{\mf D,\epsilon} (x)\big]^2
\;e^{-S_{\mf D}(x)/\epsilon}\;dx \,=\, 0\,.
\end{equation*}
\end{assertion}

\begin{proof}
By \eqref{18}, \eqref{19} and \eqref{depth_equiv_class} 
\begin{equation}
\label{scale_bound}
\theta_\epsilon^{(p)} \,\leq\,
e^{\,[S_{\mf D}  (\sigma^r_{k-1,k} )- S_{\mf D} (\mc M_p(\mf
D))]\,/\,\epsilon}\,. 
\end{equation}
Mind that this inequality may be strict if the local maximum
$\sigma^r_{k-1,k}$ is the left or rightmost, that is, if it is equal
to $\sigma^-_{\mf D}$ or $\sigma^+_{\mf D}$.

By Assertion \ref{s05},  for each fixed $\epsilon>0$, the
expression appearing in the statement of the lemma is bounded above by 
\begin{align*}
C_0 \,\sqrt{\epsilon} \, 
\int_{B^r_{k-1,k} }  
\big[ J'_{\mf D,\epsilon} (x)\big]^2 
\;e^{\,[S_{\mf D} (\sigma^r_{k-1,k})-S_{\mf D}(x)]\,/\,\epsilon}\;dx \,.
\end{align*}
Here and below, $C_0$ represents a finite constant independent of
$\epsilon$, whose value may change from line to line.  By the
definition of $J_{\mf D,\epsilon} (\cdot)$ over the binding sets
$B^r_{k-1,k}$, given below \eqref{binding_sets}, this expression is
equal to
\begin{align*}
C_0 \,\sqrt{\epsilon} \, 
\int_{B^r_{k-1,k} }  
\big[ \Phi'_{\mf D,\epsilon} (x)\big]^2 
\;e^{\,[S_{\mf D} (\sigma^r_{k-1,k})-S_{\mf D}(x)]\,/\,\epsilon}\;dx \,.
\end{align*}

By a Taylor expansion of $S_{\mf D}(\cdot)$ around $\sigma^r_{k-1,k}$,
and by definition of $\delta^r_{k-1,k}$, $\eta^r_{k-1,k}$, on
$B^r_{k-1,k}$,
\begin{align*}
\exp \big\{\, [S_{\mf D} (\sigma^r_{k-1,k})-S_{\mf D}(x)]\,/\,\epsilon
\, \big \}
\, & = \,  [1+o(\epsilon)]\,
\exp \big\{\, -\, S''_{\mf D} (\sigma^r_{k-1,k}) \,
(\delta^r_{k-1,k} + \eta^r_{k-1,k})^2 / 2\epsilon \big\}
\\
& \le\,  C_0 \,
\exp \big\{\, -\, S''_{\mf D} (\sigma^r_{k-1,k}) \,
(\delta^r_{k-1,k}) ^2 / 2\epsilon \big\}\,.
\end{align*}
By definition of $\delta^r_{k-1,k}$, the last line is equal to
$C_0\, \epsilon^{-1}$. Thus, by the bound on the derivative of
$\Phi_{\mf D,\epsilon}$, and since the length of the interval
$B^{r, \pm}_{k-1,k}$ is $\eta^r_{k-1,k}$, the penultimate displayed
equation is bounded above by $C_0 \, \sqrt{\epsilon}\; \eta^r_{k-1,k}$,
which vanishes as $\epsilon\to 0$.
\end{proof}

The next Assertion completes the proof of Proposition
\ref{higher_speed_glimsup}. It requires a simple observation. As
$\mf D$ is a $\widehat{\bb X}_p(\cdot)$-equivalence class, the jump
rates from $\mc M_p(l_{\mf D})$, $\mc M_p(r_{\mf D})$ to
$\mc M_p(l_{\mf D}-1)$, $\mc M_p(r_{\mf D}+1)$, respectively, may be
positive or vanish. By \eqref{50b},
$R_p(\mc M_p(l_{\mf D}) , \mc M_p(l_{\mf D}-1))>0$ if 
\begin{equation}
\label{34}
S(\sigma_{l_{\mf D}-1,l_{\mf D}}^{r}) \,-\, S(\mc M_p(l_{\mf D})) \,=\, \mf h_p
\end{equation}
for some (an thus for all) $1\le r\le \mf r_{l_{\mf D}-1,l_{\mf
D}}$. It vanishes if
$S(\sigma_{l_{\mf D}-1,l_{\mf D}}^{r}) \,-\, S(\mc M_p(l_{\mf D}))
\,>\, \mf h_p$.  A similar observation holds for the jump rate
$R_p(\mc M_p(r_{\mf D}) , \mc M_p(r_{\mf D}+1))$.

\begin{assertion}
For each $\widehat{\bb X}_p(\cdot)-$equivalence class $\mf D\subset \ms
S_p$,
\begin{align*}
& \limsup_{\epsilon \to 0}\theta_{\epsilon}^{(p)}\, \epsilon\,
\sum_{k=  l_{\mf D}}^{r_{\mf D} +1}
\sum_{r=1}^{\mf r_{k-1,k}}
\int_{\Lambda_{k-1,k}^{r}}
\big[ \, J_{\mf D,\epsilon}' (x) \, \big]^2(x)\;
\frac{e^{-S_{\mf D}(x)/\epsilon}}
{c_{\mf D}(\epsilon)}\;dx
\, \leq \,
\bb I^{(p)}_{\mf D}(\omega(\,\cdot\,|\, \mf D))
\\
&\quad 
\,+\, \omega(\mc M_p(l_{\mf D}) \,|\, \mf D) \, R_p(\mc M_p(l_{\mf
D}),\mc M_p(l_{\mf D}-1))
\, +\,\omega(\mc M_p(r_{\mf D}) \,|\, \mf D)\,
R_p(\mc M_p(r_{\mf D} ), \mc M_p(\mf r_{\mf D}+1)) \,.
\end{align*}    
\end{assertion}

\begin{proof}
Fix $l_{\mf D} \le k\le r_{\mf D} +1$ and
$1\leq r\leq \mf r_{k-1,k-1}$. By \eqref{kth_core_test_f},
\eqref{equil_pot}, and \eqref{kth_bridge:gaussian_ap}, the integral
appearing in the statement of the assertion is equal to
\begin{align}
\label{35}
\cb{\bb W_{k-1,k}^{r} (\epsilon)} \,:=\,
K_{k-1,k}^{r}\, \sqrt{\frac{-S''(\sigma_{k-1,k}^{r})}{2\pi}} \,
\frac{\theta_{\epsilon}^{(p)}}
{c_{\mf D}(\epsilon)}\,
\int_{\Lambda_{k-1,k}^{r} }
e^{S''(\sigma_{k-1,k}^{r})(\sigma_{k-1,k}^{r}-x)^2/\epsilon}\,
e^{-S_{\mf D} (x)/\epsilon}\,dx\,, 
\end{align}
where
\begin{align*}
\label{constant_interm_lemma}
{\color{blue} K_{k-1,k}^{r} } \,  := \,
\sqrt{\frac{-S''(\sigma_{k-1,k}^{r})}{2\pi}}\,
\big[\, \mss w_{k-1,k}^{r,r+1} - \mss w_{k-1,k}^{r-1,r}\,\big]^2 
\, \big[\,  f_{\mf D}(\mc M_p(k)) \, -\, f_{\mf D}(\mc M_p(k-1))
\,\big]^2\,. 
\end{align*}
We consider two cases.

\smallskip\noindent{\it Case 1:} Assume that $k= l_{\mf D}$ and that
$R_p(\mc M_p(l_{\mf D}) , \mc M_p(l_{\mf D}-1))=0$. Under this
assumption, by \eqref{34},
$S(\sigma_{l_{\mf D}-1,l_{\mf D}}^{r}) \,-\, S(\mc M_p(\mf D )) > \mf
h_p$. Hence, by Assertion \ref{s05} and by the definition of
$\theta_{\epsilon}^{(p)}$, \eqref{35} is bounded above by
\begin{align*}
C_0\, e^{ [\mf h_p - S_{\mf D} (\sigma_{k-1,k}^{r})  + S_{\mf D} (\mc M_p(\mf D )) ] /\epsilon }
\, \sqrt{\frac{1}{\epsilon}}
\int_{\Lambda_{k-1,k}^{r}}\,
e^{S''(\sigma_{k-1,k}^{r})(\sigma_{k-1,k}^{r}-x)^2/\epsilon}\,
e^{[S_{\mf D} (\sigma_{k-1,k}^{r})  \,-\, S_{\mf D} (x) ] /\epsilon }\, dx\,.
\end{align*}
for some finite constant $C_0$ independent of $\epsilon$ and whose
value may change from line to line. We prove below that this
expression without the first exponential converges. Therefore, as
$S(\sigma_{l_{\mf D}-1,l_{\mf D}}^{r}) \,-\, S(\mc M_p(\mf D )) > \mf
h_p$ this expression with the first exponential vanishes in the
limit. Thus, we proved that
\begin{equation*}
\lim_{\epsilon\to 0}
\sum_{r=1}^{\mf r_{l_{\mf D}-1,l_{\mf D}}}
\bb W_{l_{\mf D}-1,l_{\mf D}}^{r} (\epsilon) \,=\, 0\,.
\end{equation*}
Since $R_p(\mc M_p(l_{\mf D}) , \mc M_p(l_{\mf D}-1))=0$, we may
rewrite this identity in the following form
\begin{align*}
& \lim_{\epsilon\to 0}
\sum_{r=1}^{\mf r_{l_{\mf D}-1,l_{\mf D}}}
\bb W_{l_{\mf D}-1,l_{\mf D}}^{r} (\epsilon)
\\
& \quad =\; - \,\nu_{\mf D}(\mc M_p(l_{\mf D} ))\, f_{\mf D}(\mc M_p(l_{\mf D}))\,
R_p(\mc M_p(l_{\mf D}), \mc M_p(l_{\mf D} -1)) \,
\big[f_{\mf D}(\mc M_p(l_{\mf D} -1 ))-f_{\mf D}(\mc M_p(l_{\mf D}))\big]
\end{align*}
The same argument applies to the case  $k= r_{\mf D}+1$ with
$R_p(\mc M_p(r_{\mf D}) , \mc M_p(r_{\mf D}+1))=0$.

\smallskip\noindent{\it Case 2:} Assume that $k= l_{\mf D}$ and 
$R_p(\mc M_p(l_{\mf D}) , \mc M_p(l_{\mf D}-1))>0$, or $l_{\mf D} <
k \le r_{\mf D}$, or $k= r_{\mf D} +1$ and 
$R_p(\mc M_p(r_{\mf D}) , \mc M_p(r_{\mf D}+1))>0$.

By Assertion \ref{s05} and \eqref{scale_bound}, the expression in
\eqref{35} is equal to
\begin{align*}
[\, 1 + o(\epsilon)\,]\,
\frac{K_{k-1,k}^{r}}{\sum_{k\in D}\pi_p(k)}
\, \sqrt{\frac{-S''(\sigma_{k-1,k}^{r})}{2\pi\epsilon}}
\int_{\Lambda_{k-1,k}^{r}}\,
e^{S''(\sigma_{k-1,k}^{r})(\sigma_{k-1,k}^{r}-x)^2/\epsilon}\,
e^{[S_{\mf D} (\sigma_{k-1,k}^{r})  \,-\, S_{\mf D} (x) ] /\epsilon }\, dx\,.
\end{align*}
A Taylor expansion of $S_{\mf D}$ around $\sigma^{r}_{k-1,k}$ yields that
\begin{equation*}
\frac{1}{\epsilon} \, \big[ \, S_{\mf D} (\sigma_{k-1,k}^{r})  \,-\, S_{\mf D} (x) \,\big] 
\,=\, -\, \frac{1}{2\epsilon} \, S''(\sigma^{r}_{k-1,k})\,
(\sigma^{r}_{k-1,k} - x)^2 \,+\, o_\epsilon(1)\;,
\end{equation*}
where the remainder is uniform on $x\in
\Lambda_{k-1,k}^{r}$. Moreover, by the definition of
$\delta_{k-1,k}^{r}$, introduced in \eqref{21b}, and a change of
variables,
\begin{equation*}
\label{eq3:lem_limsup_bound}
\lim_{\epsilon\to 0} \sqrt{\frac{-S''(\sigma_{k-1,k}^{r})}{2\pi\epsilon}}
\int_{\Lambda_{k-1,k}^{r}}\,
e^{S''(\sigma_{k-1,k}^{r}) (\sigma_{k-1,k}^{r}-x)^2/2\epsilon}\,\,dx \,=\, 1\,.
\end{equation*}
This proves that \eqref{35} converges to
\begin{align*}
& \frac{1}{\sum_{k\in D}\pi_p(k)} \, 
\sqrt{\frac{-S''(\sigma_{k-1,k}^{r})}{2\pi}}\,
\big[\, \mss w_{k-1,k}^{r,r+1} - \mss w_{k-1,k}^{r-1,r}\,\big]^2 
\, \big[\,  f_{\mf D}(\mc M_p(k)) \, -\, f_{\mf D}(\mc M_p(k-1))
\,\big]^2
\\
& \quad =\, \frac{1}{\sum_{k\in D}\pi_p(k)} \, 
\frac{1}{\sqrt{2\pi}}\,
\frac{1/\sqrt{-S''(\sigma_{k-1,k}^{r})}}
{\Big(\sum_{l=1}^{\mf
r_{k-1,k}}1/\sqrt{-S''(\sigma_{k-1,k}^{l})}\Big)^2}
\, \big[\,  f_{\mf D}(\mc M_p(k)) \, -\, f_{\mf D}(\mc M_p(k-1))
\,\big]^2 \,.
\end{align*}
where we used the definition \eqref{01} of $\mss w_{k-1,k}^{r,r+1}$ in
the last step.

Recall from \eqref{p_weights} the definition of $\sigma_p(j,j+1)$. It follows
from the just proved convergence of the expression in \eqref{35} and
from the identity
\begin{equation*}
\sum_{r=1}^{\mf r_{k-1,k}} \frac{1}{\sqrt{2\pi}}\,
\frac{1/\sqrt{-S''(\sigma_{k-1,k}^{r})}}
{\Big(\sum_{l=1}^{\mf
r_{k-1,k}}1/\sqrt{-S''(\sigma_{k-1,k}^{l})}\Big)^2} \,=\,
\frac{1}{\sigma_p(k-1,k)} 
\end{equation*}
that in Case 2,
\begin{align*}
\lim_{\epsilon\to 0} \sum_{r=1}^{\mf r_{k-1,k}}
\bb W_{k -1, k}^{r} (\epsilon) \,=\, 
\frac{1}{\sum_{l\in D}\pi_p(l)}\frac{1}{\sigma_p(k-1,k)}\,\big[f_{\mf
D}(\mc M_p(k))-f_{\mf D}(\mc M_p(k-1))\big]^{2}
\end{align*}
Expanding the square and since $\sigma_p(k,k-1) =
\sigma_p(k-1,k)$. the previous expression is equal to
\begin{align*}
& -\, \frac{1}{\sum_{l\in D}\pi_p(l)} \, \frac{1}{\sigma_p(k-1,k)}
\,f_{\mf D}(\mc M_p(k-1)) \, \big[ \, f_{\mf D}(\mc M_p(k))
- f_{\mf D}(\mc M_p(k-1)) \, \big]
\\ 
&  \quad - \, \frac{1}{\sum_{l\in D}\pi_p(l)}
\, \frac{1}{\sigma_p(k,k-1)}\,
f_{\mf D}(\mc M_p(k)) \,
\big[\, f_{\mf D}(\mc M_p(k-1))- f_{\mf D}(\mc M_p(k)) \, \big]\,.
\end{align*}
Multiply and divide the first line by $\pi_p(k-1)$. By the definition
of the function $f_{\mf D}$, the jump-rates $R_p$ and the stationary
state $\nu_{\mf D}^p$ introduced in \eqref{equivtest}, \eqref{50b},
and Lemma \ref{pro_red_stat}, respectively, the first term of the
right-hand side of the previous equation is equal to
\begin{equation*}
- \,\nu_{\mf D}(\mc M_p(k-1))\, f_{\mf D}(\mc M_p(k-1))\,
R_p(\mc M_p(k-1), \mc M_p(k)) \,
\big[f_{\mf D}(\mc M_p(k))-f_{\mf D}(\mc M_p(k-1))\big]\,.
\end{equation*}
Similarly,  the second line  is equal to
\begin{equation*}
- \,\nu_{\mf D}(\mc M_p(k))\,
f_{\mf D}(\mc M_p(k))\,R_p(\mc M_p(k), \mc M_p(k))
\, \big[f_{\mf D}(\mc M_p(k-1))-f_{\mf D}(\mc M_p(k))\big]\,.
\end{equation*}
Hence, in Case 2,
\begin{align*}
& \lim_{\epsilon\to 0} \sum_{r=1}^{\mf r_{k-1,k}}
\bb W_{k -1, k}^{r} (\epsilon) \,
\\ &\quad  =\,
- \,\nu_{\mf D}(\mc M_p(k-1))\, f_{\mf D}(\mc M_p(k-1))\,
R_p(\mc M_p(k-1), \mc M_p(k)) \,
\big[f_{\mf D}(\mc M_p(k))-f_{\mf D}(\mc M_p(k-1))\big]
\\
& \quad - \,\nu_{\mf D}(\mc M_p(k))\,
f_{\mf D}(\mc M_p(k))\,R_p(\mc M_p(k), \mc M_p(k-1))
\, \big[f_{\mf D}(\mc M_p(k-1))-f_{\mf D}(\mc M_p(k))\big]\,.
\end{align*}
Note that in the case $k=l_{\mf D}$, the first line vanishes because
$f_{\mf D}(\mc M_p(l_{\mf D}-1)) =0$. Hence we obtained the same
formula for the limit in both cases. We may thus proceed using only
the formula obtained in Case 2.

\smallskip\noindent{\it Conclusion:} Recall from \eqref{37} the
definition of the reflected generator $\widehat{\bb L}_{p, \mf D}$.
Since
$f_{\mf D}(\mc M_p(l_{\mf D}-1) = f_{\mf D}(\mc M_p(r_{\mf D}+1) =0$,
it follows from the computation performed in Case 2 that the
expression which appears on the left-hand side of the assertion is
equal to
\begin{align*}
& -\, \sum_{k \in D} \, \nu_{\mf D}(\mc M_p(k))\,
f_{\mf D}(\mc M_p(k))\, (\widehat{\bb L} _{p, \mf D} f_{\mf D} )(\ms
M_p(k))
\\
& \quad + \,\nu_{\mf D}(\mc M_p(l_{\mf D}))\, f_{\mf D}(\mc M_p(l_{\mf D}))^2\,
R_p(\mc M_p(l_{\mf D}), \mc M_p(l_{\mf D} -1)) 
\\
& \quad + \,\nu_{\mf D}(\mc M_p(r_{\mf D}))\, f_{\mf D}(\mc M_p(r_{\mf D}))^2\,
R_p(\mc M_p(r_{\mf D}), \mc M_p(r_{\mf D} +1))  \,. 
\end{align*}
By definition of $f_{\mf D} (\cdot)$, the last two lines are equal to
\begin{align*}
\omega (\mc M_p(l_{\mf D}) \,|\, \mf D )\,
R_p(\mc M_p(l_{\mf D}), \mc M_p(l_{\mf D} -1)) 
\, + \,
\omega (\mc M_p(r_{\mf D}) \,|\, \mf D )\,
R_p(\mc M_p(r_{\mf D}), \mc M_p(r_{\mf D} +1)) 
\,,
\end{align*}
and the first line can be rewritten as
\begin{align*}
& -\, \sum_{k \in D} \, \omega (\mc M_p(k) \,|\, \mf D )\,
\frac{1}{f_{\mf D}(\mc M_p(k))}
\, (\widehat{\bb L} _{p, \mf D} f_{\mf D} )(\ms
M_p(k)) \,. 
\end{align*}
To complete the proof of the assertion, it remains to recall the
formula \eqref{reflected_equiv} for
$\bb I^{(p)}_{\mf D}(\omega(\cdot|\mf D))$
\end{proof}

\section{The $\Gamma-\limsup$ in the last time-scale: The reversible
case}
\label{sec2}

In this section, we establish the $\Gamma-\limsup$ for the last
time-scale $\theta^{(\widehat{\mf q})}_\epsilon$ under the following hypothesis

\smallskip\noindent{\bf Hypothesis (Rev).}
Assume that one of the three conditions is fulfilled
\begin{itemize}

\item[(a)] The Markov chain $\widehat{\bb X}_{\widehat{\mf q}}(\cdot)$
is reducible.

\item[(b)] The Markov chain $\bb X_{\widehat{\mf q}}(\cdot)$ has one closed
irreducible class.

\item[(c)] The Markov chain $\widehat{\bb X}_{\widehat{\mf q}}(\cdot)$
is irreducible, and the Markov chain $\bb X_{\widehat{\mf q}}(\cdot)$
has infinitely many closed irreducible classes which are all
equivalent.
\end{itemize}

Before we proceed, we state some remarks on the cases (a)-(c).  Recall
that, above the statement of Postulate $\mc P_{10}$ in Appendix
\ref{secA1}, we denote by $\mf n_p$, $1\leq p\leq \mf q$, the number
of closed irreducible classes of the lifted chain $\bb X_p$, up to
\eqref{equiv_class}-equivalence. By definition of $\widehat{\mf q}$,
in all situations $\mf n_{\widehat{\mf q}} \le 1$.

\smallskip\noindent{\bf Condition (a):} Under condition (a), $S(1)$
may be strictly smaller than $S(0)$. Furthermore,
$\mf n_{\widehat{\mf q}} =1$, and the Markov chain
$\widehat{\bb X}_{\widehat{\mf q}}(\cdot)$ restricted to the unique
closed irreducible class is reversible.

\smallskip\noindent{\bf Condition (b):} Under condition (b), by
postulate $\mc P_4$ all local sets $\mz M_{\widehat{\mf q}}(k)$,
$k\in \bb Z$, are at the same depth. Thus, by periodicity,
$S(\mz M_{\widehat{\mf q}}(0)) = S(\mz M_{\widehat{\mf q}}(\mf
u_{\widehat{\mf q}})) = S(1 + \mz M_{\widehat{\mf q}}(0))$, so that
$S(1)=S(0)$. On the other hand, by definition,
$\mf n_{\widehat{\mf q}} =0$. Finally, the Markov chain
$\widehat{\bb X}_{\widehat{\mf q}}(\cdot)$ is irreducible and
reversible.

\smallskip\noindent{\bf Condition (c):} Under condition (c), $S(1)$
may be strictly smaller than $S(0)$, $\mf n_{\widehat{\mf q}} =1$, and
the Markov chain $\widehat{\bb X}_{\widehat{\mf q}}(\cdot)$ is
reversible.

\begin{proposition}
\label{p:g_limsup:rev}
Assume that condition {\rm (Rev)} is fulfilled.  Then, the functional
$\ms J^{(\widehat{\mf q})}$ is a $\Gamma-\limsup$ of
$\theta_\epsilon^{(\widehat{\mf q})}\, \ms I_\epsilon$.
\end{proposition}

Fix a measure $\mu \in \mss P(\bb T)$. If
$\ms J^{(\widehat{\mf q})}(\mu)=+\infty$, then the bound
\begin{equation*}
\limsup_{\epsilon\to 0}\theta_\epsilon^{(\widehat{\mf q})}\, 
\ms I(\mu_\epsilon) \leq \ms J^{(\widehat{\mf q})} (\mu)
\end{equation*}
holds trivially for any sequence that converges weakly to $\mu$. Thus,
we may assume that $\ms J^{(\widehat{\mf q})} (\mu)<+\infty$.  By the
definition of $\ms J^{(\widehat{\mf q})}$, this entails that the
measure $\mu$ is given by
\begin{equation*}
\mu(\cdot)
= \sum_{k\in S_{\widehat{\mf q}}}\omega(\mc M_{\widehat{\mf q}}(k))
\, \mu_{\mc M_{\widehat{\mf q}}(k)}^{\widehat{\mf q}}(\cdot)\,,
\end{equation*}
for some fixed probability measure
$\omega\in \mss P(\mc S_{\widehat{\mf q}})$. By Lemma B5 in \cite{l23},
we can assume that $\omega(\mc M)>0$ for every
$\mc M\in \mc S_{\widehat{\mf q}}$. 

\begin{proof}[Proof of Proposition \ref{p:g_limsup:rev} under the
conditions (a) and (b)]

If the Markov chain $\widehat{\bb X}_{\widehat{\mf q}}(\cdot)$ is reducible,
that is, condition (a) is in force, then each
$\widehat{\bb X}_{\widehat{\mf q}}$-equivalent class is surrounded on the left
and on the right by another $\widehat{\bb X}_{\widehat{\mf q}}$-equivalent
class. We may thus repeat the proof of Proposition
\ref{higher_speed_glimsup} without modifications, constructing a
recovery sequence associated to each
$\widehat{\bb X}_{\widehat{\mf q}}$-equivalent class, to prove the
$\Gamma-\limsup$.

Suppose condition (b) holds.  As the Markov chain
$\bb X_{\widehat{\mf q}}(\cdot)$ is irreducible, $S(1)=S(0)$ so that $B=0$. In
particular, $S(\cdot)$ is periodic and well defined as a function on
the torus $\bb T$.  In this case also the proof of Proposition
\ref{higher_speed_glimsup} applies, this time with minor
modifications.

Fix a probability measure $\omega\in \mss P(\mc S_{\widehat{\mf q}})$
such that $\omega(\mc M)>0$ for every
$\mc M\in \mc S_{\widehat{\mf q}}$.  The construction of the test
function follows the one presented in the previous section starting at
equation \eqref{lift_equiv_1}. Here $\mf D = \mc S_{\widehat{\mf q}}$,
so that $l_{\mc S_{\widehat{\mf q}}} = 0$,
$r_{\mc S_{\widehat{\mf q}}} = \mf u_{\widehat{\mf q}} -1$.

Denote by
$\cb{f_{\mc S_{\widehat{\mf q}}} } \colon \mc S_{\widehat{\mf q}} \to
\bb R_+$ the function defined in \eqref{equivtest} (with $\mf D$
replaced by $\mc S_{\widehat{\mf q}}$). Let
$F_{\mz S_{\widehat{\mf q}}} \colon \mz S_{\widehat{\mf q}}\to \bb
R_+$ be the lifting of $f_{\mc S_{\widehat{\mf q}}}$:
\begin{equation*}
\cb{F_{\mz S_{\widehat{\mf q}}}(\mz M_{\widehat{\mf q}}(j))}
\,:=\,  f_{\mc S_{\widehat{\mf q}}}(\mc M_{\widehat{\mf q}}(k))\,,
\quad j=k \mod{\mf u_{\widehat{\mf q}}-1}\,.
\end{equation*}
Clearly, $F_{\mz S_{\widehat{\mf q}}}$ is
$\mf u_{\widehat{\mf q}}-$periodic. The construction presented from
equation \eqref{equivtest} to \eqref{binding_sets} provides a strictly
positive and $1-$periodic test function
{\color{blue}$J_{\widehat{\mf q}, \epsilon} = J_{\mc S_{\widehat{\mf
q}},\epsilon}$} of class $C^2(\bb R)$, provided that the smooth
functions $\Phi_{\mc S_{\widehat{\mf q}},\epsilon}$ defined below
\eqref{binding_sets} are chosen to enforce this $1-$periodicity
property. Thus, $J_{\widehat{\mf q}, \epsilon}$ can be considered as
defined on the torus.

Let ${\color{blue}\mu_{\widehat{\mf q},\epsilon}}\in \mss P(\bb T)$ be
the probability measure given by the density:
\begin{equation*}
{\color{blue}m_{\widehat{\mf q},\epsilon}(x)}
\,:= \, \frac{\big[J_{\widehat{\mf q}, \epsilon} (x)\big]^2
\,e^{-S(x)/\epsilon}}{\mss a(x) \, c_{\widehat{\mf q}}(\epsilon)}\,,    
\end{equation*}
where {\color{blue} $c_{\widehat{\mf q}}(\epsilon)$} is a
normalization constant. Note that we exchanged the local potential
$S_{\mf D}$ that appear in \eqref{recovery_seq_equiv} by $S$.

By the proof of Lemma \ref{weak_conv_p_limsup}, the sequence
$\mu_{\widehat{\mf q},\epsilon}$ weakly converges to $\mu$. By the one
of Lemma \ref{main_lemma_limsup_p} item (1),
\begin{equation}
\label{36}
\limsup_{\epsilon\to 0}\theta_{\epsilon}^{(\widehat{\mf q})} \,
\ms I_\epsilon(\mu_{\widehat{\mf q},\epsilon})\,
\leq\, \bb I^{(\widehat{\mf q})}(\omega)
\,=\, \ms J^{(\widehat{\mf q})} (\mu)\,.
\end{equation}
\end{proof}

We turn the condition (c). Up to the end of the section, we assume
that the Markov chain $\bb X_{\widehat{\mf q}}(\cdot)$ has infinitely
many closed irreducible classes which are all equivalent, and that the
Markov chain $\widehat{\bb X}_{\widehat{\mf q}}(\cdot)$ is
irreducible.  As observed above, in this case $S(1)$ may be strictly
smaller than $S(0)$, see Figure \ref{fig8}.

It follows from the condition (c) that there exists $k\in \bb Z$ for
which
$\{\mz M_{\widehat{\mf q}}(k) , \dots, \mz M_{\widehat{\mf q}}(k+\mf
u_{\widehat{\mf q}}-1) \}$ is a $\bb X_{\widehat{\mf q}}$-closed
irreducible class. Without loss of generality we may assume that
$k=0$. As
$\{\mz M_{\widehat{\mf q}}(0) , \dots, \mz M_{\widehat{\mf q}}(\mf
u_{\widehat{\mf q}}-1) \}$ is a $\bb X_{\widehat{\mf q}}$-closed
irreducible class, by translation-invariance of the jump rates,
\begin{equation}
\label{46}
R_{\widehat{\mf q}} (
\mz M_{\widehat{\mf q}}(\mf u_{\widehat{\mf q}}-1) ,
\mz M_{\widehat{\mf q}}(\mf u_{\widehat{\mf q}}))
\,=\,
R_{\widehat{\mf q}} (\mz M_{\widehat{\mf q}}(\mf u_{\widehat{\mf q}})
, \mz M_{\widehat{\mf q}}(\mf u_{\widehat{\mf q}}-1)) \,=\, 0 \,.    
\end{equation}
A similar identity holds for jump rates between
$\mz M_{\widehat{\mf q}}(0)$ and $\mz M_{\widehat{\mf q}}(-1)$

We start with a description of the energy landscape. We refer to
Figure \ref{fig8} for an illustration.  Recall the definition of the
maxima $\sigma_{k,k+1}^{\widehat{\mf q},\pm}$, $k\in \bb Z$, given in
\eqref{rl_max:R}. Without loss of generality, we may assume that
\begin{equation*}
\label{q:rev:asump0}
\{0\} < \{\sigma_{-1,0}^{\widehat{\mf q},+}\}
< \mz M_{\widehat{\mf q}}(0)
< \mz M_{\widehat{\mf q}}(\mf u_{\widehat{\mf q}}-1)
< \{1\}\,,\quad (0,\sigma_{-1,0}^{\widehat{\mf q},+})
\cap \mc C = \varnothing \,. 
\end{equation*}
That is, $\sigma_{-1,0}^{\widehat{\mf q},+}$ is the leftmost critical
point, to the right of $0$.

As
$\{\mz M_{\widehat{\mf q}}(0) , \dots, \mz M_{\widehat{\mf q}}(\mf
u_{\widehat{\mf q}}-1) \}$ is a $\bb X_{\widehat{\mf q}}$-closed
irreducible class, by postulate $\mc P_8$, the sets
$\mz M_{\widehat{\mf q}}(k)$, $0\leq k\leq \mf u_{\widehat{\mf q}}-1$
have all the same depth,
\begin{equation*}
S(\mz M_{\widehat{\mf q}}(k)) = S(\mz M_{\widehat{\mf q}}(j))\,,
\quad 0\leq j\leq k\leq \mf u_{\widehat{\mf q}}-1\,.
\end{equation*}
This equality and definitions \eqref{18}, \eqref{19}, \eqref{50b}
imply that the barriers separating $\mz M_{\widehat{\mf q}}(0)$,
$\mz M_{\widehat{\mf q}}(1),\dots, \mz M_{\widehat{\mf q}}(\mf
u_{\widehat{\mf q}}-1)$ have all the same height,
\begin{equation*}
S(\sigma) = S(\sigma')\,,\quad
\sigma,\sigma'\in \bigcup_{j=0}^{\mf u_{\widehat{\mf q}}-2}
\mz W_{j,j+1}^{(\widehat{\mf q})}\,.
\end{equation*}
By \eqref{50b}, \eqref{46}, the barriers at the boundary of the closed
irreducible 
class are higher:
\begin{equation}
\label{49}
S(\sigma_{-1,0}^{\widehat{\mf q},+})
\wedge
S(\sigma_{\mf u_{\widehat{\mf q}}-1,\mf u_{\widehat{\mf q}}}^{\widehat{\mf q},-})
> S(\sigma) \quad
\text{for all}\;\;
\sigma \in \bigcup_{j=0}^{\mf u_{\widehat{\mf q}}-2}
\mz W_{j,j+1}^{(\widehat{\mf q})}  \,.
\end{equation}

\begin{assertion}
\label{q:rev:scale:as1:per_p}
Assume that case condition Rev(c) is in force. Then, there exist
$\kappa >0$, and points
$x_{\widehat{\mf q}}^- < x_{\widehat{\mf q}}^+$ such that
$\mss b(x_{\widehat{\mf q}}^-), \mss b(x_{\widehat{\mf q}}^+) \neq 0$;
\begin{equation*}
m_{\widehat{\mf q},\mf u_{\widehat{\mf q}}-1}^+ -1
\,<\,x_{\widehat{\mf q}}^-
\,<\, m_{\widehat{\mf q},0}^-
\,<\,m_{\widehat{\mf q},\mf u_{\widehat{\mf q}}-1}^+
\,<\, x_{\widehat{\mf q}}^+
\,<\, m_{\widehat{\mf q},0}^- +1
\,;
\end{equation*} 
\begin{equation*}
S(x_{\widehat{\mf q}}^-) = S(x_{\widehat{\mf q}}^+)
= S(\mz M_{\widehat{\mf q}}(0))
+ \mf h_{\widehat{\mf q}} + \kappa > S(x)
\quad \text{for all}\;\;
x\in (x_{\widehat{\mf q}}^-,x_{\widehat{\mf q}}^+)\,.
\end{equation*}
\end{assertion}

\begin{proof}
By condition Rev(c), the lifted chain $\bb X_{\widehat{\mf q}}(\cdot)$
has more than one closed irreducible class. Consequently, we may apply
the recursive procedure one more time to build a new layer at level
$\widehat{\mf q}+1$. The states of this new chain are the sets
$\{\mz M_{\widehat{\mf q}}(0), \dots, \mz M_{\widehat{\mf q}}(\mf
u_{\widehat{\mf q} -1})\}$ and their translate. The assertion follows
by Lemma 4.6 in \cite{lm}.
\end{proof}

By translation invariance,
$\sigma_{\mf u_{\widehat{\mf q}} -1,\mf u_{\widehat{\mf
q}}}^{\widehat{\mf q},+} = 1+\sigma_{-1,0}^{\widehat{\mf q},+}$, and
by definition,
$\sigma_{\mf u_{\widehat{\mf q}} -1,\mf u_{\widehat{\mf
q}}}^{\widehat{\mf q},-} \le \sigma_{\mf u_{\widehat{\mf q}} -1,\mf
u_{\widehat{\mf q}}}^{\widehat{\mf q},+}$. Thus, by moving
$\sigma_{-1,0}^{\widehat{\mf q},+}$ to $0+$ we move
$\sigma_{\mf u_{\widehat{\mf q}} -1,\mf u_{\widehat{\mf
q}}}^{\widehat{\mf q},+}$ (and therefore
$\sigma_{\mf u_{\widehat{\mf q}} -1,\mf u_{\widehat{\mf
q}}}^{\widehat{\mf q},-} $) to $1+$. Thus, by \eqref{49}, and by
changing the value of $\kappa$ if needed, we may assume
that
\begin{equation}
\label{47}
0\,<\,\sigma_{-1,0}^{\widehat{\mf q},+}
\,<\,x_{\widehat{\mf q}}^-
\,<\,m_{\widehat{\mf q},0}^-
\,<\, m_{\widehat{\mf q},\mf u_{\widehat{\mf q}}-1}^{+}
\,<\,x_{\widehat{\mf q}}^+
\,<\,1\,.
\end{equation}

\begin{figure}
\centering
\begin{tikzpicture}[scale=0.8]


\draw[rounded corners] (1,3) .. controls (2,6.3) .. (3,3);

\draw[rounded corners] (3,3) .. controls (3.5, 1) .. (4, -1);
\draw[rounded corners] (4,-1) .. controls (4.5, -3) .. (5,-1);
\draw[rounded corners] (5,-1) -- (5.5,1);
\draw[rounded corners] (5.5,1) .. controls (5.8,2) .. (6.1, 1);

\draw[rounded corners](6.1,1)--(6.5,-1);
\draw[rounded corners](6.5,-1)..controls(7,-3)..(7.5,-1);
\draw[rounded corners](7.5,-1)--(7.9,1);

\draw[rounded corners] (7.9, 1) .. controls (8.2, 2) .. (8.5, 1);

\draw[rounded corners] (8.5, 1) -- (9, -1);
\draw[rounded corners] (9,-1) .. controls (9.5, -3) .. (10, -1);
\draw[rounded corners] (10,-1) .. controls (10.5, 1) .. (11, -1);
\draw[rounded corners] (11,-1) .. controls (11.5,-3) .. (12.2, -1);
\draw[rounded corners] (12.2,-1) -- (12.9, 1);
\draw[rounded corners] (12.9,1) .. controls (14, 4.3) .. (15, 1);


\fill(2,5.5)node[above]{$\sigma_{-1,0}^{\widehat{\mf q},\pm}$};

\fill(14.2,3.5)node[above]{$\sigma_{1,0}^{\widehat{\mf q},\pm}+1$};

\fill(4.5,-2.6)node[below,red, font=\small]{$\mc M_{\widehat{\mf q}}(0)$};
\fill(7.1,-2.6)node[below,red,font=\small]{$\mc M_{\widehat{\mf q}}(1)$};
\fill(10.5,-2.6)node[below,red, font=\small]{$\mc M_{\widehat{\mf q}}(2)$};


\draw[to-to,line width=0.5mm, red](5.8,1.7)--(5.8,-2.6);
\fill(6.2,-1)node[below,red,font=\small]{$\mf h_{\widehat{\mf q}}$};

\draw[dashed](4.5,-2.6)--(11.6,-2.6);



\draw[dashed,black](3.2,2) -- (13.2,2);

\draw[dashed](3.2,2)--(3.2,-3.5);
\draw[dashed](13.2,2)--(13.2,-3.5);

\draw[line width=0.5mm, cyan](3.2,-3.5)--(13.2,-3.5);

\fill(3.2,-3.5)node[below, cyan]{$x_{\widehat{\mf q}}^-$};
\fill(13.2,-3.5)node[below,cyan]{$x_{\widehat{\mf q}}^+$};

\draw[to-to, blue](9,1.72)--(9,2);

\draw[dashed](5.75,1.72)--(9,1.72);

\fill(9.3,2.05)node[below, blue]{$\kappa$};



\draw[brown, rounded corners](1.3,1.9)..controls(1.9,4)..(2.1,2.1);
\draw[brown](2.1,2.1)--(2.4,-2.5);
\draw[,brown,rounded corners](2.4,-2.5)..controls(2.5,-3.5)..(2.6,-2.5);
\draw[brown](2.6,-2.5)--(2.85,0.8);
\draw[brown, rounded corners](2.85,0.8) .. controls (3,3) .. (3.1,2.6);

\fill(1.9,3.5)node[above, brown]{$\hat{\sigma}_1$};
\fill(3.1,3)node[above,brown]{$\hat{\sigma}_2$};
\fill(2.5,-3.2)node[below,brown]{$\hat{m}$};



\draw[brown, rounded corners](13.2,1.9)..controls(13.8,4)..(14,2.1);
\draw[brown](14,2.1)--(14.3,-2.5);
\draw[,brown,rounded corners](14.3,-2.5)..controls(14.4,-3.5)..(14.5,-2.5);
\draw[brown](14.5,-2.5)--(14.75,0.8);
\draw[brown, rounded corners](14.75,0.8) .. controls (14.9,3) .. (15,2.6);



\draw[line width=0.5mm](1,-4.5)--(15,-4.5);
\fill(1.7,-4.5)node[below]{$0$};
\fill(13.7,-4.5)node[below]{$1$};

\end{tikzpicture}
\caption{This picture illustrates an example of a
non-reversible energy landscape as well as the periodic correction
$\widehat{S}$ for the case where $B>0$ and the reduced chain
$\widehat{\bb X}_{\widehat{\mf q}}(\cdot)$ is reversible. In this
example, the state space of the reduced chain
$\widehat{\bb X}_{\widehat{\mf q}}$ is composed by the three states
$\mc M_{\widehat{\mf q}}(0), \mc M_{\widehat{\mf q}}(1)$ and
$\mc M_{\widehat{\mf q}}(2)$. The Markov chain cannot jump to the
right of $\mc M_{\widehat{\mf q}}(2)$, nor to the left of
$\mc M_{\widehat{\mf q}}(0)$. The points
$x_{\widehat{\mf q}}^-, x_{\widehat{\mf q}}^+$ and the height $\kappa$
are chosen according to Assertion \ref{q:rev:scale:as1:per_p}. Thus,
$S(x_{\widehat{\mf q}}^-) - S(\mc M_{\widehat{\mf q}}(0)) =
S(x_{\widehat{\mf q}}^+)- S(\mc M_{\widehat{\mf q}}(0)) = \mf
h_{\widehat{\mf q}}+\kappa$. Finally, the behaviour of the reversible
potential $\widehat{S}$ outside of the set
$[x_{\widehat{\mf q}}^-, x_{\widehat{\mf q}}^+]$ is represented in
brown. }
\label{fig8}
\end{figure}
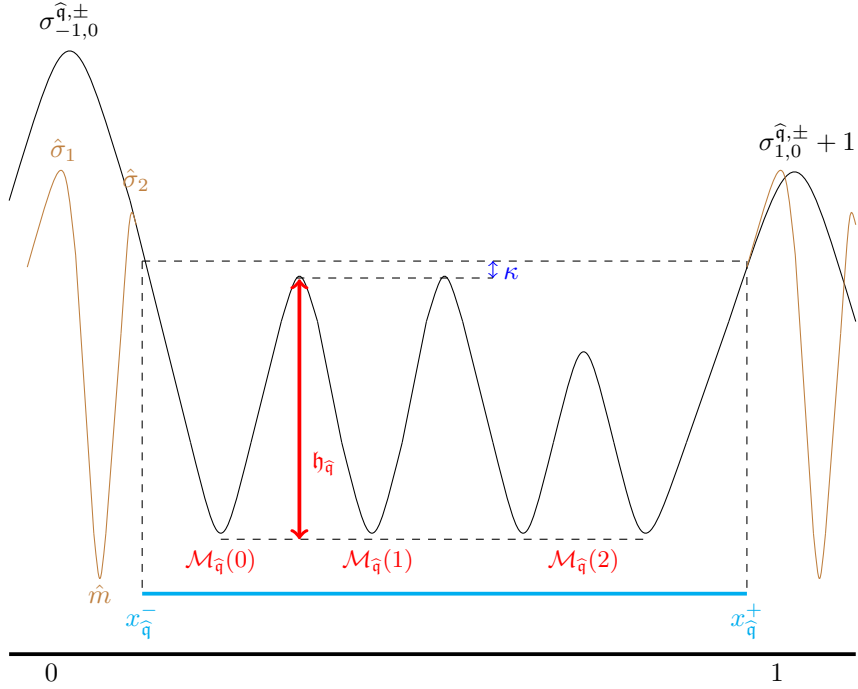

The idea of the proof is simple and goes as follows. We define a
$1-$periodic potential $\widehat{S}$ that is equal to $S$ on the set
$(x_{\widehat{\mf q}}^-, x_{\widehat{\mf q}}^+)$, but has an extra
well that is deeper than all of the wells of $S$ on
$(x_{\widehat{\mf q}}^-, x_{\widehat{\mf q}}^+)$. See Figure
\ref{fig8}. Intuitively, this new potential induces a reversible
diffusion $X_{\epsilon}^{*}$ on the torus having $\widehat{\mf q}+1$
metastable levels and whose metastable behaviour on the interval
$(x_{\widehat{\mf q}}^-,x_{\widehat{\mf q}}^+)$ matches the one of
$X_{\epsilon}$, up to level $\widehat{\mf q}$. The state space
$\mc S_{\widehat{\mf q}}$ of
$\widehat{\bb X}_{\widehat{\mf q}}(\cdot)$ can then be regarded as a
properly contained closed irreducible class of the reduced chain at
level $\widehat{\mf q}$ of $X_{\epsilon}^{*}$. The desired recovery
sequence will be the recovery sequence for $X_{\epsilon}^{*}$
restricted to $\Omega_{\mc S_{\widehat{\mf q}}}$, where
$\Omega_{\mc S_{\widehat{\mf q}}}$ is the set defined in \eqref{28}.

Let {\color{blue}$\widehat{S}:\bb R\to \bb R$} be a $1-$periodic
function of class $C^2(\bb R)$ satisfying the following conditions:
\begin{enumerate}
    \item $\widehat{S}(x) = S(x)$, for all $x\in (x_{\widehat{\mf q}}^-,x_{\widehat{\mf q}}^+)$;
    \item on the interval $(x_{\widehat{\mf q}}^+-1, x_{\widehat{\mf q}}^-)$, the function $\widehat{S}$ has exactly three critical points, denoted by $\hat{\sigma}_1, \hat{m}$ and $ \hat{\sigma}_2$, each of them being nondegenerate, and satisfying
    \begin{equation}\label{test_q_torus}
        0\,<\,\hat{\sigma}_1 \,<\,\hat{m}\,<\,\hat{\sigma}_2\,<\, x_{\widehat{\mf q}}^-\,<\,1\,;
    \end{equation}
    \item $\widehat{S}(\hat{m})\leq S(\mz M_{\widehat{\mf q}}(0))$;
    \item $\widehat{S}(\hat{\sigma}_1) = \widehat{S}(\hat{\sigma}_2)$ and $\widehat{S}(\hat{\sigma}_1) >  S(x_{\widehat{\mf q}}^-)$.
\end{enumerate}
Let ${\color{blue}\sigma_{\mc S_{\widehat{\mf q}}}^-}:= \hat{\sigma}_1$ and ${\color{blue}\sigma_{\mc S_{\widehat{\mf q}}}^+}:= \hat{\sigma}_2$; and denote by $\widehat{\mz W}_{\mc S_{\widehat{\mf q}}}\subset (0,1)$ be the set given by
\begin{equation*}
    {\color{blue}\widehat{\mz W}_{\mc S_{\widehat{\mf q}}}}:= \bigcup_{k=1}^{\mf u_{\widehat{\mf q}}-1}\mz W_{k-1,k}^{(\widehat{\mf q})}\cup\{\sigma_{\mc S_{\widehat{\mf q}}}^-,\sigma_{\mc S_{\widehat{\mf q}}}^+\}\,.
\end{equation*}
Let $\delta_{\mc S_{\widehat{\mf q}}}^{\pm}$, $\eta_{\mc S_{\widehat{\mf q}}}^{\pm}$, be the scales defined as
\begin{equation*}
    {\color{blue}\delta_{\mc S_{\widehat{\mf q}}}^{\pm}}:= \sqrt{\frac{2\epsilon\log \epsilon^{-1}}{-\widehat{S}''(\sigma_{\mc S_{\widehat{\mf q}}}^{\pm})}}\,,\quad {\color{blue}\eta_{\mc S_{\widehat{\mf q}}}^{\pm}}:= \frac{\epsilon}{\delta_{\mc S_{\widehat{\mf q}}}^{\pm}}\,,
\end{equation*}
and let $\Omega_{\mc S_{\widehat{\mf q}}}$ be the $\mc S_{\widehat{\mf q}}-$neighborhood given by
\begin{equation*}
    {\color{blue}\Omega_{\mc S_{\widehat{\mf q}}}}:= (\sigma_{\mc S_{\widehat{\mf q}}}^- - \delta_{\mc S_{\widehat{\mf q}}}^- - \eta_{\mc S_{\widehat{\mf q}}}^-\,, \sigma_{\mc S_{\widehat{\mf q}}}^+ + \delta_{\mc S_{\widehat{\mf q}}}^+ + \eta_{\mc S_{\widehat{\mf q}}}^+)\,.
\end{equation*}
Let {\color{blue}$J_{\mc S_{\widehat{\mf q}},\epsilon}$} be the test function constructed via the exact same procedure introduced in the previous subsection, but changing $S$ by $\widehat{S}$. By construction, $J_{\mc S_{\widehat{\mf q}},\epsilon}$ is a function of class $C^{2}(\bb R)$ whose support in contained in the set $\Omega_{\mc S_{\widehat{\mf q}}}$. By \eqref{test_q_torus}, this set is properly contained in the interval $(0,1)$. Thus, $J_{\mc S_{\widehat{\mf q}},\epsilon}$ can be understood as defined on the torus. 

Let $\mu_{\epsilon}\in \mss P(\bb T)$ be the probability measure given by the density
\begin{equation}\label{q:rev:recov_seq_2}
    {\color{blue}m_{\epsilon}(x)}:= \frac{[J_{\mc S_{\widehat{\mf q}},\epsilon}(x)]^2e^{\widehat{S}(x)/\epsilon}}{\mss a(x)c(\epsilon)}\,,
\end{equation}
where ${\color{blue}c(\epsilon)}$ is a normalization constant that turns $m_{\epsilon}$ into the density of a probability measure. The same proof of Lemma \ref{weak_conv_p_limsup} yields that $\mu_{\epsilon}$ weakly converges to $\mu$, provided that we set $l_{\widehat{\mf q}} = 1$ and $k_{\widehat{\mf q}}=0$. To conclude, it remains to show that 
\begin{equation*}
    \limsup_{\epsilon\to 0}\theta_{\epsilon}^{(\widehat{\mf q})}\ms I_\epsilon(\mu_\epsilon) \leq \bb I^{(\widehat{\mf q})}(\omega)\,.
\end{equation*}
The proof is similar to the one of item (1) in Lemma \ref{main_lemma_limsup_p}, provided that we replace the statement in Assertion \ref{a1:sup_high} by the following
\begin{proposition}\label{q:scale:nestim:ls}
Let $\mu_\epsilon$ be the sequence of measures defined in \eqref{q:rev:recov_seq_2}. Then,
\begin{equation}\label{q:scale:nestim:eq:ls}
\theta_\epsilon^{(\widehat{\mf q})}\ms I_\epsilon(\mu_\epsilon)
\,\leq \,
\sum_{k = 1}^{\mf u_{\widehat{\mf q}}-1} 
\sum_{r=1}^{\mf r_{k-1,k}}
\theta_\epsilon^{(\widehat{\mf q})}\int_{B_{\sigma_{k-1,k}^{\widehat{\mf q},r}}^{\widehat{\mf q}}\,\cup
\,\Lambda_{\sigma_{k-1,k}^{\widehat{\mf q},r}}^{\widehat{\mf q}}}
\big[J'_{\mc S_{\widehat{\mf q}},\epsilon} (x) \big]^2\;
\frac{e^{-S(x)/\epsilon}}
{c(\epsilon)}\;dx \,+\, R(\epsilon)\,,
\end{equation}
where $R(\epsilon)$ is an error term that vanishes as $\epsilon\to 0$.
\end{proposition}
Note that in the right-hand side of \eqref{q:scale:nestim:eq:ls} we removed the integrals over the neighborhoods associated to the local maxima $\sigma_{\mc S_{\widehat{\mf q}}}^{\pm}$ of $\widehat{S}$. Consequently, 
\begin{equation*}
    \bigcup_{k=1}^{\mf u_{\widehat{\mf q}}-1}\bigcup_{r=1}^{\mf r_{k-1,k}}B_{\sigma_{k-1,k}^{\widehat{\mf q},r}}^{\widehat{\mf q}}\cup
\,\Lambda_{\sigma_{k-1,k}^{\widehat{\mf q},r}}^{\widehat{\mf q}} \subset (x_{\widehat{\mf q}}^-,x_{\widehat{\mf q}}^+)\,.
\end{equation*}
Since the functions $S$ and $\widehat{S}$ coincide on $(x_{\widehat{\mf q}}^-,x_{\widehat{\mf q}}^+)$, this explains why we replaced $\widehat{S}$ by $S$ on the right-hand side of \eqref{q:scale:nestim:ls}. The remaining of the section focuses on proving this proposition.

Recall that the test function vanishes outside the set $\Omega_{\mc S_{\widehat{\mf q}}}$. The following result states that the discrepancy among the potentials $S$ and $\widehat{S}$ on the set $\Omega_{\mc S_{\widehat{\mf q}}}\setminus [x_{\widehat{\mf q}}^-, x_{\widehat{\mf q}}^+]$ is asymptotically negligible.
\begin{assertion}\label{q:scale:nestim:002}
$\limsup_{\epsilon \to 0}\frac{\theta_\epsilon^{(\widehat{\mf q})}}{\epsilon^2}\int_{\Omega_{\mc S_{\widehat{\mf q}}}\setminus [x_{\widehat{\mf q}}^-,x_{\widehat{\mf q}}^+]}\,\frac{e^{-\widehat{S}(x)/\epsilon}}{c(\epsilon)}\,dx\, =\, 0$.
\end{assertion}
\begin{proof}
    By item (4) in the definition of $\widehat{S}$ and \eqref{47},
    \begin{equation}\label{q:scale:as717:eq1}
        \inf_{x\in \Omega_{\mc S_{\widehat{\mf q}}}}\widehat{S}(x) = S(\mz M_{\widehat{\mf q}}(0)) + \mf h_{\widehat{\mf q}} + \kappa\,.
    \end{equation}
    This estimate, the fact that the test function vanishes outside of the set $\Omega_{\mc S_{\widehat{\mf q}}}$ and the proof of Assertion \ref{s05} imply that the normalization constant $c(\epsilon)$ given in \eqref{q:rev:recov_seq_2} satisfies the following asymptotic estimate
    \begin{equation}\label{q:scale:as717:eq2}
        \lim_{\epsilon \to 0}\frac{1}{\sqrt{\epsilon}}e^{S(\mz M_{\widehat{\mf q}}(0))/\epsilon}\,c(\epsilon) = \sum_{k\in S_{\widehat{\mf q}}}\pi_{\widehat{\mf q}}(k)\,.
    \end{equation}
    To conclude the proof, observe that \eqref{q:scale:as717:eq1} and \eqref{q:scale:as717:eq2} yield
    \begin{equation*}
        \frac{\theta_\epsilon^{(\widehat{\mf q})}}{\epsilon^2}\int_{\Omega_{\mc S_{\widehat{\mf q}}}\setminus [x_{\widehat{\mf q}}^-,x_{\widehat{\mf q}}^+]}\,\frac{e^{-\widehat{S}(x)/\epsilon}}{c(\epsilon)}\,dx\,\leq\, \frac{C}{\epsilon^{5/2}}e^{-\kappa/\epsilon}\,\to\,0\,,
    \end{equation*}
    as $\epsilon\to 0$, where $C>0$ is a finite and positive constant which doesn't depend on $\epsilon$.
\end{proof}

Given $\delta>0$, let $\mu_\epsilon^\delta\in \mss P(\bb T)$ be the probability measure given by the following density
\begin{equation*}
m_{\epsilon}^\delta(x) \,:=\,
\frac{m_{\epsilon}(x)+\delta}{1+\delta}\,.
\end{equation*}
By Assertion \ref{estim_limsup_high} and the first part of the proof of Assertion \ref{a1:sup_high}, it follows that
\begin{equation*}
    \theta_{\epsilon}^{(\widehat{\mf q})}\ms I_{\epsilon}(\mu_{\epsilon}) \leq \liminf_{\delta\to 0} \frac{\theta_\epsilon^{(\widehat{\mf q})}}{4\epsilon}\int_{\bb T} \mss a\,
\Big\{\epsilon \Big[\frac{(m_\epsilon^\delta)'}{
m_\epsilon^\delta}
\,+\, \frac{\mss a'}{\mss a}\Big]
\,-\, \frac{\mss b}{\mss a} \Big\}^2
d\mu_\epsilon^\delta\,.
\end{equation*}
By \eqref{kth_bridge:gaussian_ap} in the definition of the test function, $\sup_{x\in \bb T}|J_{\mc S_{\widehat{\mf q}},\epsilon}'(x)| = O(\epsilon^{-1/2})$. Then, differentiating $m_\epsilon^\delta$, expanding the square, applying the dominated convergence theorem and using the definition of $\widehat{S}$ yields 
\begin{equation*}
    \theta_{\epsilon}^{(\widehat{\mf q})}\ms I_{\epsilon}(\mu_{\epsilon}) \,\leq\,
\epsilon\, \theta_{\epsilon}^{(\widehat{\mf q})}\,\int_{x_{\widehat{\mf q}}^-}^{x_{\widehat{\mf q}}^+} \big[ J'_{\mc S_{\widehat{\mf q}},\epsilon} (x) \big]^2 \;
\frac{e^{-S(x)/\epsilon}}{c(\epsilon)}\;dx\,+\, C\,\frac{\theta_{\epsilon}^{(\widehat{\mf q})}}{\epsilon^2}\int_{\Omega_{\mc S_{\widehat{\mf q}}}\setminus [x_{\widehat{\mf q}}^-,x_{\widehat{\mf q}}^+]}\,\frac{e^{-S(x)/\epsilon}}{c(\epsilon)}\;dx,
\end{equation*}
for some finite and positive constant $C>0$ that doesn't depend on $\epsilon$. To conclude the proof of Proposition \ref{q:scale:nestim:ls}, it remains to apply Assertion \ref{q:scale:nestim:002} and the definition of $J_{\mc S_{\widehat{\mf q}},\epsilon}$ to the right-hand side of the previous displayed inequality.

\section{The $\Gamma-\limsup$ in the last time-scale: the
non-reversible case}
\label{sec8}

In this section, we prove the $\Gamma-\limsup$ in the remaining case.
Once this has been done, at the end of the section,  we prove Proposition \ref{s03} and Theorem
\ref{mt3}. 

In this section, we assume that in the last time-scale the reduced
Markov chain is non-reversible. Since we supposed that $B\ge 0$, in
view of Claims 2.A-2.C in \cite{lm}, this corresponds to the case
where
\begin{equation}
\tag{NonRev}
\begin{gathered}
B>0\,, \quad 
R_{\widehat{\mf q}}(\mz M_{\widehat{\mf q}}(k),
\mz M_{\widehat{\mf q}}(k+1))>0
\quad \text{for all $k\in \bb Z$}\,,
\\
\text{
$R_{\widehat{\mf q}}(\mz M_{\widehat{\mf q}}(j),
\mz M_{\widehat{\mf q}}(j-1))=0$ for some
$j\in \bb Z$}\, .
\end{gathered}
\end{equation}

The next lemma asserts that in the layer $\widehat{\mf q}$ condition
(Rev) or condition (Nonrev) hold. In particular, Section \ref{sec2}
and \ref{sec8} cover all cases of the last time scale.

\begin{lemma}
\label{s11}
The Markov chains $\bb X_{\widehat{\mf q}}(\cdot)$, $\widehat{\bb
X}_{\widehat{\mf q}}(\cdot)$ satisfy one of the four conditions (Rev),
(Nonrev). Moreover, $\mf q = \widehat{\mf q}$ if conditions (Rev)(b)
or (Nonrev) are in force, and $\mf q = \widehat{\mf q} +1$, otherwise.
\end{lemma}

\begin{proof}
By postulate $\mc P_7$, the Markov chain
$\bb X_{\widehat{\mf q}}(\cdot)$ jumps only to nearest-neighbor sites,
and its jump rates are periodic in the sense of postulate $\mc
P_9$. By definition of the last layer $\widehat{\mf q}$, the projected
chain $\widehat{\bb X}_{\widehat{\mf q}}(\cdot)$ has only one
irreducible class.

Denote by $\{j, \dots, k\}$, the unique
$\widehat{\bb X}_{\widehat{\mf q}}(\cdot)$-irreducible class.  We may
assume that $\{j, \dots, k\}$ is contained in
$\{0, \dots, \mf u_{\widehat{\mf q}} -1 \}$.

First, suppose that these two sets coincide:
$\{j, \dots, k\} = \{0, \dots, \mf u_{\widehat{\mf q}} -1 \}$. Being a
$\widehat{\bb X}_{\widehat{\mf q}}(\cdot)$ closed irreducible set, it
is also a $\bb X_{\widehat{\mf q}}(\cdot)$ closed irreducible set.
Thus, by postulate $\mc P_4$, the sets
$\mz M_{\widehat{\mf q}} (0), \dots, \mz M_{\widehat{\mf q}} (\mf
u_{\widehat{\mf q}} -1)$ are at the same depth. By definition of $B$,
$S(\mz M_{\widehat{\mf q}} (\mf u_{\widehat{\mf q}} -1)) = S(\mz
M_{\widehat{\mf q}} (0)) = S(\mz M_{\widehat{\mf q}} (\mf
u_{\widehat{\mf q}} )) +B$.

Suppose that $B>0$. In this case, by postulate $\mc P_8$, the Markov
chain $\bb X_{\widehat{\mf q}}(\cdot)$ does not jump from
$\mz M_{\widehat{\mf q}} (\mf u_{\widehat{\mf q}} )$ to
$\mz M_{\widehat{\mf q}} (\mf u_{\widehat{\mf q}} -1)$. By
periodicity, it also does not jump from $\mz M_{\widehat{\mf q}} (0)$
to $\mz M_{\widehat{\mf q}} (-1)$.  If it jumps from
$\mz M_{\widehat{\mf q}} (\mf u_{\widehat{\mf q}} -1)$ to
$\mz M_{\widehat{\mf q}} (\mf u_{\widehat{\mf q}})$, condition
(Nonrev) is fulfilled. Moreover, the Markov chain
$\bb X_{\widehat{\mf q}}(\cdot)$ has no closed irreducible class so
that $\mf n_{\widehat{\mf q}} =0$ and $\widehat{\mf q} = \mf q$.

It the Markov chain $\bb X_{\widehat{\mf q}}(\cdot)$ does not jump,
condition (Rev)(c) is fulfilled. The Markov chain
$\bb X_{\widehat{\mf q}}(\cdot)$ has infinitely many closed
irreducible classes which are all equivalent. In this case,
$\mf n_{\widehat{\mf q}} =1$ and $\mf q = \widehat{\mf q} +1$.

Suppose that $B=0$, so that
$S(\mz M_{\widehat{\mf q}} (\mf u_{\widehat{\mf q}} -1)) = S(\mz
M_{\widehat{\mf q}} (\mf u_{\widehat{\mf q}} ))$.  By postulate
$\mc P_7$, the Markov chain $\bb X_{\widehat{\mf q}}(\cdot)$ jumps
from $\mz M_{\widehat{\mf q}} (\mf u_{\widehat{\mf q}} )$ to
$\mz M_{\widehat{\mf q}} (\mf u_{\widehat{\mf q}} -1)$ if, and only if,
it jumps from $\mz M_{\widehat{\mf q}} (\mf u_{\widehat{\mf q}} -1)$
to $\mz M_{\widehat{\mf q}} (\mf u_{\widehat{\mf q}})$.  Suppose it
jumps. In this case, the Markov chain $\bb X_{\widehat{\mf q}}(\cdot)$
is irreducible, condition (Rev)(b) is fulfilled, and
$\mf q = \widehat{\mf q}$.

If it does not jump, the Markov chain
$\bb X_{\widehat{\mf q}}(\cdot)$ has infinitely many closed
irreducible classes which are all equivalent. Thus, (Rev)(c) holds, 
$\mf n_{\widehat{\mf q}} =1$ and $\mf q = \widehat{\mf q} +1$.

Assume now that the unique
$\widehat{\bb X}_{\widehat{\mf q}}(\cdot)$-closed irreducible class
$\{j, \dots, k\}$ is a proper subset of
$\{0, \dots, \mf u_{\widehat{\mf q}} -1 \}$. In this case, condition
(Rev)(a) is fulfilled, $\bb X_{\widehat{\mf q}}(\cdot)$ has infinitely many closed
irreducible classes which are all equivalent, 
$\mf n_{\widehat{\mf q}} =1$ and $\mf q = \widehat{\mf q} +1$.
\end{proof}

The main result of this section reads as follows.

\begin{proposition}
\label{l10}
Under the hypothesis {\rm (NonRev)}, the functional $\ms J^{(\widehat{\mf q})}$
is a $\Gamma-\limsup$ of $\theta_\epsilon^{(\widehat{\mf q})}\ms I_\epsilon$.
\end{proposition}

The proof of the $\Gamma-\limsup$ for this non-reversible dynamics
follows the strategy proposed in \cite{l23} for finite-state Markov
chains and can be applied to other models.  Fix a measure
$\mu \in \mss P(\bb T)$. If $\ms J^{(\widehat{\mf q})}(\mu)=+\infty$, then the
bound
\begin{equation*}
\label{q_limsup:trivial_bound}
\limsup_{\epsilon\to 0}\theta_\epsilon^{(\widehat{\mf q})}\,
\ms I_\epsilon (\mu_\epsilon) \leq \ms J^{(\widehat{\mf q})} (\mu)
\end{equation*}
holds trivially for any sequence $\mu_\epsilon$ that converges weakly
to $\mu$. Thus, by the definition of $\ms J^{(\widehat{\mf q})} (\cdot)$, we may
assume that the measure $\mu$ is given by
\begin{equation}
\label{q_limsup:mu_convex}
{\color{blue}\mu(\cdot)} = \sum_{k\in S_{\widehat{\mf q}}}
\omega(\mc M_{\widehat{\mf q}}(k)) \,
\mu_{\mc M_{\widehat{\mf q}}(k)}^{\widehat{\mf q}}(\cdot)\,,
\end{equation}
for some probability measure $\omega\in \mss P(\mc S_{\widehat{\mf q}})$.  The
measures $\mu_{\mc M_{\widehat{\mf q}}(k)}^{\widehat{\mf q}}(\cdot)$ appearing in this
formula were introduced in \eqref{21}.  By Lemma B5 in \cite{l23}, we
can assume that $\omega(\mc M _{\widehat{\mf q}}(k))>0$ for every
$k \in S_{\widehat{\mf q}}$. \smallskip

Fix a strictly positive measure $\omega \in \mss P(\mc S_{\widehat{\mf q}})$.
Let {\color{blue}$\widehat{h} \colon \mc S_{\widehat{\mf q}}\to \bb R_{+}$} be a
strictly positive function that satisfies \eqref{reflected_equiv} with
$l_{\widehat{\mf q}}=1$ and $\mf D_1^{\widehat{\mf q}} = \mc S_{\widehat{\mf q}}$. This function is
unique up to a multiplicative constant. Denote by {\color{blue}
$\bb L_{\widehat{\mf q},\widehat{h}}$} the tilted generator given by
\eqref{51}. By \eqref{h_tilted_invariant_m}, $\omega$
is the stationary state of the Markov chain with generator
$\bb L_{\widehat{\mf q},\widehat{h}}$.

Denote by $\cb{\zeta^{\widehat{\mf q},k}\colon \bb T \to \bb R}$,
$k\in S_{\widehat{\mf q}}$, the smooth functions satisfying
\eqref{resolv_smooth}.  Fix $\lambda>0$, and denote by
${\color{blue}\phi_{\widehat{\mf q},\epsilon}\colon \bb T \to \bb R}$ the unique
solution of the resolvent equation
\begin{equation}
\label{q:resolv_limsup}
(\lambda - \theta_{\epsilon}^{(\widehat{\mf q})} \ms L_{\epsilon}) \, \phi_{\widehat{\mf q},\epsilon}
= \sum_{k\in S_{\widehat{\mf q}}}\, [\, (\lambda - \widehat{\bb L}_{\widehat{\mf q}})\, 
\widehat{h}\, ](\mc M_{\widehat{\mf q}}(k)) \, \zeta^{\widehat{\mf q},k}\,.
\end{equation}
The function $\phi_{\widehat{\mf q},\epsilon}$ is clearly smooth.  As
$\widehat{h}$ is strictly positive, by Lemma
\ref{s10}, there exists $c_0>0$ and $\epsilon_0>0$
such that
\begin{equation*}
\phi_{\widehat{\mf q}, \epsilon} (x) \,\ge\, c_0 \quad \text{for all}\;\; x\in
\bb T\,, \;\; 0<\epsilon < \epsilon_0\,.
\end{equation*}

Denote by $\ms L_{\widehat{\mf q},\epsilon}$ the generator tilted by $\phi_{\mf
q, \epsilon}$, which already appeared just before the statement of
Lemma \ref{l02}, 
\begin{equation*}
\label{q:tilted:gen}
\begin{aligned}
{\color{blue} (\ms L_{\widehat{\mf q}, \epsilon}F)(x)} & :=
\frac{1}{\phi_{\widehat{\mf q}, \epsilon} (x)}\, \Big\{
[ \ms L_{\epsilon} ( F \phi_{\widehat{\mf q}, \epsilon}) ] (x) - F(x)\,
(\ms L_{\epsilon} \phi_{\widehat{\mf q}, \epsilon})  (x)\, \big\}
\\
\, &=\, 
\Big\{\mtt b (x)\,
+\,2\,\epsilon \,\mtt a (x) \,
\frac{\phi_{\widehat{\mf q},\epsilon}' (x)}{\phi_{\widehat{\mf q}, \epsilon} (x)}\Big\}
F'(x)\,+\,\epsilon\,\mtt a(x)\,F''(x)\,,
\end{aligned}
\end{equation*}
and by ${\color{blue}\mu_{\widehat{\mf q}, \epsilon}}\in \mss P(\bb T)$ the
unique stationary state of $\ms L_{\widehat{\mf q},\epsilon}$. By equation
\eqref{q:recovery:seq} below, the density of $\mu_{\widehat{\mf q},\epsilon}$ is
a strictly positive function of class $C^2(\bb T)$. Consequently, by
Lemma \ref{l03},
\begin{equation*}
\label{q:rate:estimate}
\theta_{\epsilon}^{(\widehat{\mf q})}\ms I_{\epsilon}(\mu_{\widehat{\mf q},\epsilon})
= -\int_{\bb T}\frac{\theta_\epsilon^{(\widehat{\mf q})}
\ms L_{\epsilon}\phi_{\widehat{\mf q},\epsilon}}{\phi_{\widehat{\mf q}, \epsilon}}\,
d\mu_{\widehat{\mf q}, \epsilon}\,.
\end{equation*}

In the remaining of the section, we show that $\mu_{\widehat{\mf q},\epsilon}$
is a recovery sequence for $\mu$: 

\begin{proposition}
\label{prop:q:wlimit}
The probability measure $\mu_{\widehat{\mf q},\epsilon}$ converges weakly to the
measure $\mu$.
\end{proposition}

\begin{proposition}
\label{prop:q:glimsup}
$\limsup_{\epsilon\to 0} \theta_{\epsilon}^{(\widehat{\mf q})}\, \ms
I_{\epsilon}(\mu_{\widehat{\mf q},\epsilon}) \leq \ms J^{(\widehat{\mf q})}(\mu)$.
\end{proposition}

The argument is organized as follows. First, we present an explicit
expression for the measure $\mu_{\widehat{\mf q},\epsilon}$. Then, we briefly
describe the main features of the energy landscape needed to estimate
the asymptotics of $\mu_{\widehat{\mf q},\epsilon}$. We complete the section
proving Propositions \ref{prop:q:wlimit} and \ref{prop:q:glimsup}, as
well as Theorem \ref{mt3}.

\subsection*{The invariant measure of $\ms L_{\widehat{\mf q},\epsilon}$}

In \cite{fg12}, the stationary measure of the process
$X_\epsilon(\cdot)$, for $\mtt a \equiv 1$, has been explicitly
computed. A simple adaptation gives the invariant measure for the
general situation. Let $\pi_\epsilon$,
$m_\epsilon : \bb R \to \bb R_+$ be given by
\begin{equation}
\label{79}
{\color{blue}m_\epsilon(x)} \;=\; \frac{1}{\mtt a(x)}\, \int_x^{x+1} 
e^{[S(y)-S(x)]/\epsilon}\, dy
\;, \quad
{\color{blue}\pi_\epsilon(x)} \;=\; \frac 1{c(\epsilon)} \, m_\epsilon(x)\;,
\end{equation}
where ${\color{blue}c(\epsilon)} := \int_{[0,1]} m_\epsilon(x)\, dx$.
As $m_\epsilon$ is $1$-periodic, it can be considered as defined on
$\bb T$, and $\pi_\epsilon$ as the density of a probability measure on
$\bb T$.  A simple computation shows that this measure
is the stationary state of the diffusion \eqref{03}. Similarly, let
$m_{\widehat{\mf q},\epsilon} \colon \bb T\to \bb R_+$ be given by
\begin{equation}
\label{q:recovery:seq}
{\color{blue}m_{\widehat{\mf q},\epsilon}(x)\,}:=\,
\frac{1}{Z_\epsilon} \, \frac{1}{\mtt a(x)} \, \phi_{\widehat{\mf q},\epsilon}(x)^2
\int_x^{x+1}\frac{e^{[S(y) - S(x)] /\epsilon}}{\phi_{\widehat{\mf q},\epsilon}(y)^2}\,dy \,.
\end{equation}
As $\phi_{\widehat{\mf q},\epsilon}$, $m_{\widehat{\mf q},\epsilon}$ is a strictly
positive function of class $C^2(\bb T)$. Integrating by parts gives
that
\begin{equation*}
\int_{\bb T}(\ms L_{\widehat{\mf q}, \epsilon}f)(x)\,m_{\widehat{\mf q},\epsilon}(x)\,dx\, =\,0\,,
\end{equation*}
for every $f\in C^2(\bb T)$. This proves the next claim. Recall that
$\mu_{\widehat{\mf q},\epsilon}$ represents the stationary state of the
diffusion induced by the generator $\ms L_{\widehat{\mf q}, \epsilon}$.

\begin{assertion}
\label{q:scale:as0}
It holds that
$\mu_{\widehat{\mf q},\epsilon} (dx) = m_{\widehat{\mf q},\epsilon} (x)\, dx$.
\end{assertion}

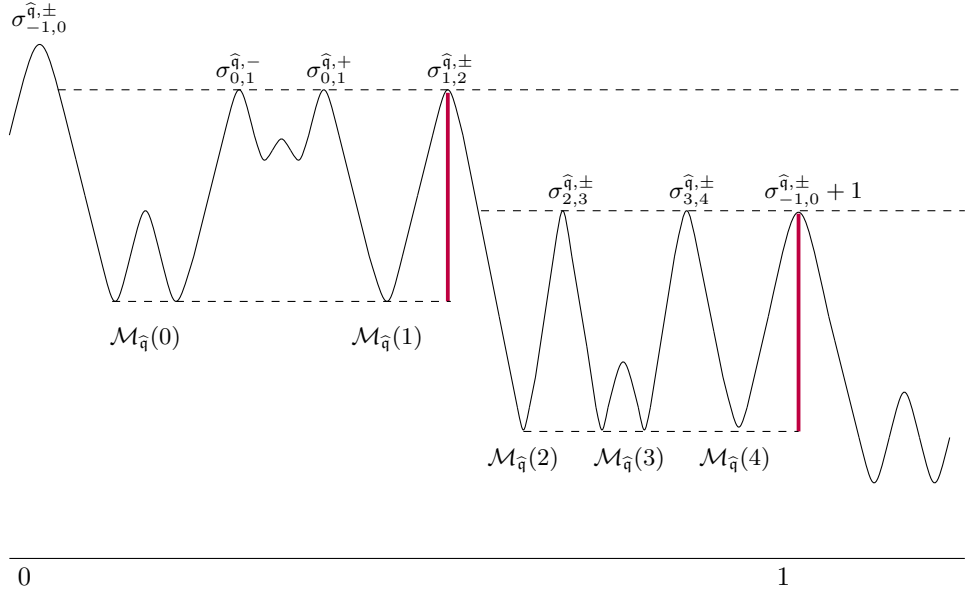
\begin{figure}
\centering
\begin{tikzpicture}[scale=0.4]

\draw[rounded corners] (-1,3) .. controls (0,7) .. (1,3);  
\draw[rounded corners] (1,3) -- (1.2, 2.2);
\draw[rounded corners] (1.2,2.2) -- (1.5, 1);

\draw[rounded corners] (1.5, 1) -- (2, -1);
\draw[rounded corners] (2,-1) .. controls (2.5, -3) .. (3, -1);

\draw[rounded corners](3,-1)..controls (3.5,1) .. (4,-1);
\draw[rounded corners](4,-1).. controls (4.5,-3)..(5.1,-1);
\draw[rounded corners](5.1,-1)..controls (5.6,1)..(6.1,3);

\draw[rounded corners](6.1,3)..controls (6.6,5)..(7.1,3);
\draw[rounded corners](7.1,3)..controls (7.4,2)..(7.7,2.5);
\draw[rounded corners](7.7,2.5)..controls(8,3)..(8.3,2.5);
\draw[rounded corners](8.3,2.5)..controls(8.6,2)..(8.9,3);
\draw[rounded corners](8.9,3)..controls(9.4,5)..(9.9,3);

\draw[rounded corners](9.9,3)..controls (10.4,1)..(10.9,-1);
\draw[rounded corners](10.9,-1)..controls (11.5,-3)..(12,-1);

\draw[rounded corners](12,-1)--(12.5,1);
\draw[rounded corners](12.5,1)--(13,3);

\draw[rounded corners](13,3)..controls(13.5,5)..(14,3);

\draw[rounded corners]
(14,3)--(14.6,0);
\draw[rounded corners]
(14.6,0)--(15.2,-3);
\draw[rounded corners]
(15.2,-3)--(15.8,-6);
\draw[rounded corners]
(15.8,-6)..controls(16,-7)..(16.2,-6);
\draw[rounded corners]
(16.2,-6)--(16.4,-5);
\draw[rounded corners]
(16.4,-5)--(17,-1);
\draw[rounded corners]
(17,-1)..controls(17.3,1)..(17.6,-1);
\draw[rounded corners]
(17.6,-1)--(18.1,-4);
\draw[rounded corners]
(18.1,-4)--(18.4,-6);
\draw[rounded corners]
(18.4,-6)..controls(18.6,-7)..(18.8,-6);
\draw[rounded corners]
(18.8,-6)..controls(19.3,-4)..(19.8,-6);
\draw[rounded corners]
(19.8,-6)..controls(20,-7)..(20.2,-6);
\draw[rounded corners]
(20.2,-6)--(21,-1);
\draw[rounded corners]
(21,-1)..controls(21.4,1)..(21.8,-1);
\draw[rounded corners]
(21.8,-1)--(22.3,-3.5);

\draw[rounded corners]
(22.3,-3.5)--(22.6,-5);
\draw[rounded corners]
(22.6,-5)..controls(23.1,-7)..(23.4,-6);
\draw[rounded corners]
(23.4,-6)--(24.1,-3);
\draw[rounded corners]
(24.1,-3)..controls(25.1,1.6)..(26.1,-3);
\draw[rounded corners]
(26.1,-3)--(26.3,-3.8);
\draw[rounded corners]
(26.3,-3.8)--(26.6,-5);
\draw[rounded corners]
(26.6,-5)--(27.1,-7);
\draw[rounded corners]
(27.1,-7)..controls(27.6,-9)..(28.1,-7);
\draw[rounded corners]
(28.1,-7)..controls(28.6,-5)..(29.1,-7);
\draw[rounded corners]
(29.1,-7)..controls(29.6,-9)..(30.1,-7);

\draw[solid](-1,-11)--(30.6,-11);

\fill(-0.5,-11)node[below]{$0$};
\fill(24.6,-11)node[below]{$1$};

\fill(3.5,-3)node[below,font=\small]{$\mc M_{\widehat{\mf q}}(0)$};
\fill(11.5,-3)node[below,font=\small]{$\mc M_{\widehat{\mf q}}(1)$};

\fill(16,-7)node[below,font=\small]{$\mc M_{\widehat{\mf q}}(2)$};
\fill(19.5,-7)node[below,font=\small]{$\mc M_{\widehat{\mf q}}(3)$};
\fill(23,-7)node[below,font=\small]{$\mc M_{\widehat{\mf q}}(4)$};

\draw[dashed](0.6,4.5)--(30.6,4.5);
\draw[dashed](14.6,0.5)--(30.6,0.5);

\fill(0,6)node[above,font=\small]{$\sigma_{-1,0}^{\widehat{\mf q},\pm}$};
\fill(6.6,4.3)node[above,font=\small]{$\sigma_{0,1}^{\widehat{\mf q},-}$};
\fill(9.6,4.3)node[above,font=\small]{$\sigma_{0,1}^{\widehat{\mf q},+}$};

\fill(13.6,4.3)node[above,font=\small]{$\sigma_{1,2}^{\widehat{\mf q},\pm}$};

\fill(17.6,0.3)node[above,font=\small]{$\sigma_{2,3}^{\widehat{\mf q},\pm}$};
\fill(21.6,0.3)node[above,font=\small]{$\sigma_{3,4}^{\widehat{\mf q},\pm}$};
\fill(25.6,0.3)node[above,font=\small]{$\sigma_{-1,0}^{\widehat{\mf q},\pm}+1$};

\draw[line width=0.5mm, purple](13.5,4.4)--(13.5,-2.5);
\draw[dashed](2.4,-2.5)--(13.6,-2.5);

\draw[line width=0.5mm, purple](25.1, 0.4)--(25.1,-6.8);
\draw[dashed](16,-6.8)--(25.1,-6.8);

\end{tikzpicture}
\caption{Here we illustrate an example of an energy landscape of a
non-reversible diffusion whose reduced Markov chain
$\widehat{\bb X}_{\widehat{\mf q}}(\cdot)$, at the last level of the
recursive construction, is irreducible and non-reversible. The depth
of the landscape, denoted by $\mf h_{\widehat{\mf q}}$, is represented
in purple. The state space of $\bb X_{\widehat{\mf q}}(\cdot)$ is
composed by five states:
$\mc S_{\widehat{\mf q}} = \{\mc M_{\widehat{\mf q}}(0), \dots, \mc
M_{\widehat{\mf q}}(4)\}$. Among them, $\mc M_{\widehat{\mf q}}(0)$
and $\mc M_{\widehat{\mf q}}(2)$ are the only ones from which
$\widehat{\bb X}_{\widehat{\mf q}}(\cdot)$ cannot jump to the left. In
this example, $\mf I_{0}=\{0,1\}$, $\mf I_{1}=\{1\}$,
$\mf I_{2}=\{2,3,4\}$, $\mf I_{3}=\{3,4\}$ and $\mf I_{4}=\{4\}$.  }
\label{fig10}
\end{figure}

\subsection*{The energy landscape at level $\widehat{\mf q}$}

In this subsection we present some properties of the energy landscape
for the diffusion $X_\epsilon(\cdot)$ in the last time-scale
$\theta_\epsilon^{(\widehat{\mf q})}$. We refer to Figure \ref{fig10} for an
illustration of the notation introduced below.

Recall that $\mz M_{\widehat{\mf q}}(i)$, $i\in \bb Z$, represent the lifting of
the sets $\mc M_{\widehat{\mf q}}(j)$, $j\in S_{\widehat{\mf q}}$, to $\bb R$:
\begin{equation*}
\mz M_{\widehat{\mf q}}(j) < \mz M_{\widehat{\mf q}}(j+1)\,, \quad
\Pi (\mz M_{\widehat{\mf q}}(j)) \,=\, \Pi (\mz M_{\widehat{\mf q}}(i)) 
\;\;\text{if} \;\;  j\equiv i \;\;\text{mod}\; \mf u_{\widehat{\mf q}} \,.
\end{equation*}
We may assume, without loss of generality, that
\begin{equation*}
\{0\} < \{\sigma_{-1,0}^{\widehat{\mf q},+}\} < \mz M_{\widehat{\mf q}}(0)
< \mz M_{\widehat{\mf q}}(\mf u_{\widehat{\mf q}}-1) < \{1\}\,,
\end{equation*}
where $\sigma_{0,-1}^{\widehat{\mf q},+}$ is given by \eqref{rl_max:R}.

We start with three simple observations. Recall from \eqref{max_br} the
definition of the set of absolute maxima $\mz W_{j,j+1}^{(\widehat{\mf q})}$,
$j\in \bb Z$.  As the process may always jump to the right, by
postulate $\mc P_8(\widehat{\mf q})$, and the definition \eqref{50b} of the jump
rates, 
\begin{equation}
\label{60}
S (\mz M_{\widehat{\mf q}}(j+1)) \,\le\, S(\mz M_{\widehat{\mf q}}(j))  \quad
\text{and}\quad
S(\mz W_{j,j+1}^{(\widehat{\mf q})}) - S(\mz M_{\widehat{\mf q}}(j)) = h^{\widehat{\mf q}, +}_j
= \mf h_{\widehat{\mf q}}
\quad 
\text{for all}\;\; j\in \bb Z\,.
\end{equation}
As $B>0$, $S(x+1) < S(x)$ so that
\begin{equation}
\label{59}
\sup_{y \ge x } S(y) \,=\, \sup_{y \in [x, x+1) } S(y) \quad 
\text{for all}\;\; x\in \bb R\,.
\end{equation}

We claim that
\begin{equation}
\label{62}
S(\sigma_{-1,0}^{\widehat{\mf q},+}) =
\sup_{x \ge \sigma_{-1,0}^{\widehat{\mf q},+} } S(x) \,.
\end{equation}
Assume by contradiction that this is not the case. Then, we may choose
$\hat{\sigma}\in (\sigma_{-1,0}^{\widehat{\mf q},+},\sigma_{-1,0}^{\widehat{\mf q},+}+1)$
such that $S(\hat{\sigma})>S(\sigma_{-1,0}^{\widehat{\mf q},+})$. Recall the
definition of $m_{\widehat{\mf q}}^{\pm}(\hat{\sigma})$ given below Lemma
\ref{l09}. Let $0\leq k\leq \mf u_{\widehat{\mf q}}-1$ be such that
$m_{\widehat{\mf q}}^{-}(\hat{\sigma})\in \mz M_{\widehat{\mf q}}(k)$.  By \eqref{60},
$S(\mz M_{\widehat{\mf q}}(0))\geq S(\mz M_{\widehat{\mf q}}(k)))$. Then
\begin{equation*}
\mf h_{\widehat{\mf q}} \,\leq\, h_0^{\widehat{\mf q},-}
\, =\, S(\sigma_{-1,0}^{\widehat{\mf q},+})
- S(\mz M_{\widehat{\mf q}}(0))
\,<\, S(\hat{\sigma})
- S(m_{\widehat{\mf q}}^-(\mz M_{\widehat{\mf q}}(k))
\,\leq\, h_k^{\widehat{\mf q},+}\,.
\end{equation*}
By \eqref{50b}, this implies that
$R_{\widehat{\mf q}}(\mz M_{\widehat{\mf q}}(k),\mz M_{\widehat{\widehat{\mf q}}}(k+1))= 0$,
contradicting {\rm (NonRev)}. This proves the claim. \smallskip

Recall from \eqref{17} the definition of the points $m_{\widehat{\mf q},j}^\pm$,
$j\in \bb Z$.
We claim that
\begin{equation}
\label{58}
\sup_{x \ge m_{\widehat{\mf q},j}^+} S(x)
\,=\, S(\sigma)   
\quad
\text{for all} \;\;
\sigma\in \mz W_{j,j+1}^{(\widehat{\mf q})}\,,\;\;
j\in \bb Z\,.
\end{equation}

Fix $j\in \bb Z$ and $\sigma\in \mz W_{j,j+1}^{(\widehat{\mf q})}$. As,
$\sigma \ge m_{\widehat{\mf q},j}^+$,
$S(\sigma) \le \sup_{x \ge m_{\widehat{\mf q},j}^+} S(x)$.

Conversely, fix $x_0> m^+_{\widehat{\mf q},j}$. Recall from Lemma \ref{s10} the
definition of the points $\cb{m^\pm_{\widehat{\mf q}} (x_0)}$.  By definition,
$m^+_{\widehat{\mf q}} (x_0) \ge x_0$, and
$m^+_{\widehat{\mf q}} (x_0)\in \mz M_{\widehat{\mf q}}(k)$ for some $k\in \bb Z$. As
$x_0> m^+_{\widehat{\mf q},j}$, $k>j$. Clearly,
$x_0 \in [m^+_{\widehat{\mf q}, k-1}, m^+_{\widehat{\mf q}, k}]$.  Thus,
\begin{equation*}
S(x_0) \le \, \sup \big\{ S(y) : m^+_{\widehat{\mf q}, k-1} \le y \le m^-_{\mf
q, k}\,\big\}
\,\vee\,
\sup \big\{ S(y) : m^-_{\widehat{\mf q}, k} \le y \le m^+_{\widehat{\mf q}, k}
\,\big\}\,.
\end{equation*}
The first term on the right-hand side is $S(\mz W^{(\widehat{\mf q})}_{k-1,k})$
by definition of this set. By Lemma \ref{l15}, the second term is also
bounded by $S(\mz W^{(\widehat{\mf q})}_{k-1,k})$, so that $S(x_0) \le S(\mz
W^{(\widehat{\mf q})}_{k-1,k})$.

Therefore, as $k>j$, by both properties stated in \eqref{60},
\begin{equation*}
S(x_0) \le S(\mz W^{(\widehat{\mf q})}_{k-1,k})
= \mf h_{\widehat{\mf q}} + S(\mz M_{\widehat{\mf q}}(k-1))
\le \mf h_{\widehat{\mf q}} + S(\mz M_{\widehat{\mf q}}(j))
= S(\mz W^{(\widehat{\mf q})}_{j,j+1}) = S(\sigma)\,.
\end{equation*}
This completes the proof of \eqref{58}. \smallskip

Fix $k\in S_{\widehat{\mf q}}$. Let $\mf I_{k} \subset \bb Z$ be the largest
component to the right of $k$, which contains $k$ and in which it is
always possible to jump to the left:
\begin{equation}
\label{48}
{\color{blue} \mf I_k}:= \lb k \,, \mf b_k\rb\, \subset \bb Z
\end{equation}
if $R_{\widehat{\mf q}}(\mz M_{\widehat{\mf q}}(\mf b_k+1),\mz M_{\widehat{\mf q}}(\mf b_k))=0$, and
$R_{\widehat{\mf q}}(\mz M_{\widehat{\mf q}}(j),\mz M_{\widehat{\mf q}}(j-1))>0$ for all
$k+1\le j\le \mf b_k$. By definition, if
$R_{\widehat{\mf q}}(\mz M_{\widehat{\mf q}}(k+1),\mz M_{\widehat{\mf q}}(k))=0$, $\mf b_k=k$
and $\mf I_{k} =\{k\}$.  

Let $\mc M_{\widehat{\mf q}}(\mf I_k)\subset \mc M_{\widehat{\mf q}}$ be
the set of minima belonging to the states with indices in
$\mf I_k$, and denote by $\mz M_{\widehat{\mf q}}(\mf I_k)$ its
lifting:
\begin{equation*}
{\color{blue} \mc M_{\widehat{\mf q}}(\mf I_k)}
:= \bigcup_{j\in \mf I_k}\mc M_{\widehat{\mf q}}(j)\,,
\quad
{\color{blue} \mz M_{\widehat{\mf q}}(\mf I_k)}
:= \bigcup_{j\in \mf I_k}\mz M_{\widehat{\mf q}}(j)\,.
\end{equation*}
By construction (which guarantees that the jump rates between
nearest-neighbor elements of $\mf I_k$ are positive), and postulates
$\mc P_8(\widehat{\mf q})$ in Appendix \ref{secA1}, all the local minima in
$\mz M_{\widehat{\mf q}}(\mf I_k)$ have the same depth:
\begin{equation}
\label{56}
S(m) = S(m') \quad \text{for all}\;\;
m\,, m'\in \mz M_{\widehat{\mf q}}(\mf I_k)\,.
\end{equation}
As the jump rate to the right is always positive, by \eqref{50b},
$h_k^{\widehat{\mf q},+} = \mf h_{\widehat{\mf q}}$ for all $k\in S_{\widehat{\mf q}}$. Hence, by
\eqref{56},
\begin{equation}
\label{57}
S(\sigma) = S(\sigma') \quad \text{for all}\;\;
\sigma\,, \sigma'\in \bigcup_{j\in \mf I_k} \mz W^{\widehat{\mf q}}_{j,j+1}\,.
\end{equation}

Let $\mss M_{\widehat{\mf q}}(\mf I_k)\subset \mz M$ be the set of minima that
are in the interval
$(\sigma_{k-1,k}^{\widehat{\mf q},+},\sigma_{\mf b_k ,\mf b_k+1}^{\widehat{\mf q},-})$ but do not
belong to the set $\mz M_{\widehat{\mf q}}(\mf I_k)$,
\begin{equation*}
\label{q:scale:higher_min_equiv}
{\color{blue}\mss M_{\widehat{\mf q}}(\mf I_k)} :=
\big[ \, \mz M \cap
(\sigma_{k-1,k}^{\widehat{\mf q},+},\sigma_{\mf b_k,\mf b_k+1}^{\widehat{\mf q},-}) \,\big]
\setminus \mz M_{\widehat{\mf q}}(\mf I_k) \;.
\end{equation*}

Denote by $V_{\widehat{\mf q}}:\bb R\to \bb R_+$ the quasi-potential given by
\begin{equation*}
\label{quasi_potential}
{\color{blue}V_{\widehat{\mf q}}(x)} := S(x) \,-\, \sup_{y\in [x,+\infty)}S(y)\,.
\end{equation*}
By \eqref{59}, the quasi-potential is well defined, and it is a
$1-$periodic function. By \eqref{62},
$V_{\widehat{\mf q}} (\sigma_{-1,0}^{\widehat{\mf q},+})=0$.

The next result follows directly by the definition of $\mf I_k$,
$k\in S_{\widehat{\mf q}}$, and \eqref{58}. It relates the depth of the energy
landscape with the dynamics of the reduced Markov chain
$\widehat{\bb X}_{\widehat{\mf q}}(\cdot)$.

\begin{assertion}
\label{s09}
For every $j\in S_{\widehat{\mf q}}$, $m\in \mz M_{q}(j)$,
$V_{\widehat{\mf q}} (m) = V_{\widehat{\mf q}} (m_{\widehat{\mf q},j}^+)$. Denote this common value
by $\cb{V_{\widehat{\mf q}}(\mz M_{q}(j))}$.  For every $j\in S_{\widehat{\mf q}}$,
\begin{equation*}
V_{\widehat{\mf q}}(\mz M_{q}(j)) 
= S(\mz M_{\widehat{\mf q}}(j)) - S(\mz W_{j,j+1}^{(\widehat{\mf q})})
\,=\, -\, \mf h_{\widehat{\mf q}} \,.
\end{equation*}
\end{assertion}

\begin{proof}
Fix $j\in \mz S_{\widehat{\mf q}}$. By \eqref{58} and
the definition of $V_{\widehat{\mf q}}$,
\begin{align*}
\label{q:scale:as87:eq1}
V_{\widehat{\mf q}}(m_{\widehat{\mf q}, j}^+)
\,=\, S(m_{\widehat{\mf q},j}^+) \,-\, S(\mz W_{j,j+1}^{(\widehat{\mf q})})
\,=\, S(\mz M_{\widehat{\mf q}}(j)) \,-\, S(\mz W_{j,j+1}^{(\widehat{\mf q})})\,.     
\end{align*}

Assume that $\mz M_{\widehat{\mf q}}(j)$ has more than one element. Choose
$m\in \mz M_{\widehat{\mf q}}(j)\setminus \{m_{\widehat{\mf q}, j}^+\}$. By postulate
$\mc P_5$, \eqref{27} and \eqref{58},
\begin{equation}
\label{54}
S(\mz W_{j,j+1}^{(\widehat{\mf q})}) = \sup_{x\ge m} S(x)\,.
\end{equation}
Consequently, since all of the minima in $\mz M_{\widehat{\mf q}}(j)$ have the
same depth, by the definition of $V_{\widehat{\mf q}}$ and \eqref{54},
$V_{\widehat{\mf q}}(m) = V_{\widehat{\mf q}} (m_{\widehat{\mf q},j}^+)$.  To complete the proof, it
remains to recall the second identity in \eqref{60}.
\end{proof}

\subsection*{Proof of Proposition \ref{prop:q:wlimit}}

In this subsection, we establish Proposition \ref{prop:q:wlimit}. The
proof relies on two lemmata which characterize the asymptotic behavior
of the measure $\mu_{\widehat{\mf q},\epsilon}$. The proofs are based on the
properties of the solution of the resolvent equation presented in
Subsection \ref{sec_Meta}.

We first show that the measure $\mu_{\widehat{\mf q},\epsilon}$ is concentrated
on the neighborhoods of the local minima of order $\widehat{\mf q}$, and that
the second integral in its definition can be carried over small
neighborhoods of the global maxima separating the minima.

Recall the definition of $r_0$, the height of the metastable wells,
given slightly above \eqref{1_well}. By Lemma \ref{l15}, there exists
$0<h_0<+\infty$ such that
\begin{equation}
{\color{blue} h_0} \,<\, \inf_{k\in  \lb 0,\mf u_{\widehat{\mf q}}-1 \rb}
S(\sigma_{k,k+1}^{\widehat{\mf q},-})\,-\,S_{\widehat{\mf q},k}^+\,,
\quad \text{where}\;\; {\color{blue} S^+_{\widehat{\mf q} ,k}} \,:=\,
\max \{ S(x) : x\in  [m^-_{\widehat{\mf q} ,k} , \sigma^{\widehat{\mf q},-}_{k,k+1} ) \cap \mz W\}\,.
\end{equation}
Fix $0< {\color{blue}s_0} <r_0\wedge h_0$, and let
{\color{blue}$\eta_0>0$} be a positive constant such that
\begin{equation*}
0\leq \sup_{x\in (\sigma-\eta_0,\sigma+\eta_0)}S(\sigma)-S(x) < s_0
\quad \text{for every $\sigma\in \mz W$} \,.
\end{equation*}
As $s_0 + r_0 < 2r_0 < \mf h_1$, and $s_0<h_0$,
\begin{equation*}
\label{q:scale:prop_eta2}
\begin{gathered}
(\sigma-\eta_0 \,, \sigma+\eta_0)\cap \mc E = \varnothing
\quad \forall\, \sigma\in \mz W  \,, \;\;
\text{so that}\;\; 
(\sigma-\eta_0 \,,\sigma+\eta_0)\cap \mz C
= \{\sigma\}  \,, 
\\
\text{and}\quad 
\inf_{x\in (\sigma - \eta_0,\sigma+\eta_0)}S(x)
> S_{\widehat{\mf q},k}^{+}  \;\; \text{for all}\;\;
\sigma\in \mz W_{k,k+1}^{(\widehat{\mf q})}\,,\;
k \in \bb Z\,.
\end{gathered}
\end{equation*}

For each $k\in \bb Z$ and $x\in \mc E_{\widehat{\mf q}}(k)$, let
$\mc W_{\widehat{\mf q},k}^+$ and $\widehat{\mc W}_{\widehat{\mf q},k}^+(x)$ be the sets
given by
\begin{equation*}
{\color{blue}\mc W_{\widehat{\mf q},k}^{+}}:= \bigcup_{j\in \mf I_k}
\bigcup_{\sigma\in \mz W_{j,j+1}^{(\widehat{\mf q})}}
(\sigma-\eta_0 \,,\sigma+\eta_0)\,,
\quad {\color{blue}\widehat{\mc W}_{\widehat{\mf q},k}^+(x)}
= [x,x+1)\setminus \mc W_{\widehat{\mf q},k}^{+}\,.
\end{equation*}

\begin{assertion}
\label{s08}
There exists a positive constant $r_1> 0$, independent of $\epsilon$
such that
\begin{gather*}
\inf_{x\in [0,1)\setminus
\bigcup_{k\in S_{\widehat{\mf q}}}\mc E_{\widehat{\mf q}}(k)} V_{\widehat{\mf q}}(x)
> -  \, \mf h_{\widehat{\mf q}} + r_1
\\
\max_{k\in S_{\widehat{\mf q}}}
\sup_{x\in \mc E_{\widehat{\mf q}}(k) }
\sup_{y\in \widehat{\mc W}_{\widehat{\mf q},k}^{+}(x)}
\big[\, S(y)-S(x) \,\big] < \mf h_{\widehat{\mf q}} - r_1\, .
\end{gather*}
\end{assertion}

\begin{proof}
Fix $x\in [0,1)\setminus \bigcup_{k\in S_{\widehat{\mf q}}}\mc E_{\mf
q}(k)$. Let $0\leq j \leq \mf u_{\widehat{\mf q}}-1$ such that
$m_{\widehat{\mf q}}^-(x)\in \mz M_{\widehat{\mf q}}(j)$. Either
$m_{\widehat{\mf q}}^-(x)\leq x\leq \sigma_{j,j+1}^{\widehat{\mf q},+}$ or
$\sigma_{j,j+1}^{\widehat{\mf q},+} < x$. In the second case,
$m_{\widehat{\mf q}}^+(x)\in \mz M_{\widehat{\mf q}}(j+1)$ and
$\sigma_{j,j+1}^{\widehat{\mf q},+}\leq x\leq m_{\widehat{\mf q}}^+(x)$.

Assume that the first case is in force. The second one can be treated
similarly. Recall the definition of the set $\mtt M_{\widehat{\mf q}}(j)$ given
below Lemma \ref{l15}. Let
\begin{equation*}
{\color{blue}M_j^1}:= \min_{m\in \mtt M_{\widehat{\mf q}}(j)}S(m)\,,
\quad {\color{blue}M_{j,j+1}^2}:=
\inf\{S(z): z\in [\sigma_{j,j+1}^{\widehat{\mf q},-},\,
\sigma_{j,j+1}^{\widehat{\mf q},+}]\}\,,
\quad {\color{blue}M_j^3}:= S(\mz M_{\widehat{\mf q}}(j)) + r_0\,,
\end{equation*}
where $r_0$ has been introduced above \eqref{1_well}.  Denote by $M_j$
the minimum over these quantities:
\begin{equation*}
{\color{blue}M_j}:= \min\{M_j^1, M_{j,j+1}^2, M_j^3\}\,.
\end{equation*}
If $x\in [\sigma_{j,j+1}^{\widehat{\mf q},-},\, \sigma_{j,j+1}^{\widehat{\mf q},+}]$, then
$S(x) \ge M_{j,j+1}^2 \ge M_j$, while if
$x\in [m_{\widehat{\mf q}}^-(x),\, \sigma_{j,j+1}^{\widehat{\mf q},-}]$, then
\begin{equation*}
S(x) \,\ge\,  \inf \Big\{ S(y) : y \in [m_{\widehat{\mf q}}^-(x),\, \sigma_{j,j+1}^{\mf
q,-}] \setminus \bigcup_{k\in S_{\widehat{\mf q}}} \mc E_{\widehat{\mf q}}(k) \, \Big\}\,.
\end{equation*}
As $m_{\widehat{\mf q}}^-(x)$ belongs to $\mz M_{\widehat{\mf q}}(j)$, the infimum on the
right-hand side is attained either at an element of
$\mtt M_{\widehat{\mf q}}(j)$, or at the boundary of a set $\mc E_{1}(m)$,
$m\in \mz M_{\widehat{\mf q}}(j)$. The right-hand side is thus bounded below by
$M_j^1 \wedge M_{j,j+1}^2 \ge M_j$. Thus, $S(x) \geq M_j$. On the
other hand, by Proposition \ref{l13}, Lemma \ref{l27}, and the
definition of the wells,
${\color{blue}s_j}:= M_j - S(\mz M_{\widehat{\mf q}}(j))>0$. In conclusion,
\begin{equation}
\label{-q:scale:eq_1}
S(x) \geq M_j = S(\mz M_{\widehat{\mf q}}(j)) + s_j\,.
\end{equation}

Fix $y\ge x$. As $x\ge m^-_{\widehat{\mf q}} (x)\in \mz M_{\widehat{\mf q}}(j)$, by
\eqref{54},
\begin{equation*}
S(y) - S(x) \le \sup_{z\ge x} S(z) - S(x) \le 
\sup_{z\ge  m^-_{\widehat{\mf q}} (x)} S(z) - S(x) 
\leq S(\mz W_{j,j+1}) - S(x)\,.
\end{equation*}
Recall from the second identity of Assertion \ref{s09} that
$S(\mz W_{j,j+1}) - S(\mz M_{\widehat{\mf q}}(j)) = \mf h_{\widehat{\mf q}}$.  Adding and
subtracting $S(\mz M_{\widehat{\mf q}}(j))$, we can write the right-hand side of
the previous expression as
\begin{equation*}
\mf h_{\widehat{\mf q}} + S(\mz M_{\widehat{\mf q}}(j)) - S(x)
\leq \mf h_{\widehat{\mf q}} - s_j\,,
\end{equation*}
where the inequality follow by \eqref{-q:scale:eq_1}. Therefore,
\begin{equation*}
-\, V_{\widehat{\mf q}}(x) = \sup_{y\ge x} S(y) - S(x) \leq \mf h_{\widehat{\mf q}} - s_j\,.
\end{equation*}
This proves the first assertion with $r_1= \min_j s_j$.

We turn to the second one. Recall the definition of $\eta_0$ given at
the beginning of this subsection. Fix $k\in S_{\widehat{\mf q}}$ and
$x\in \mc E_{\widehat{\mf q}}(k)$. By the proof of \eqref{54},
\begin{equation*}
\sup_{y\in [x,x+1]}S(y) = \sup_{y\in [m_{\widehat{\mf q},k}^+, m_{\widehat{\mf q},k}^++1]}S(y)\,.
\end{equation*}
Let $W_k$ be the positive number given by
\begin{equation*}
{\color{blue}W_k}:= \max\Big\{\, S(\sigma-\eta_0) ,\,
S(\sigma +\eta_0) \,: \sigma\in
\bigcup_{j\in \mf I_k}\mz W_{j,j+1}^{(\widehat{\mf q})}\Big\}\,.
\end{equation*}
By \eqref{57}, there exists a positive constant $s_k^{*}>0$ such that
$S(\mz W_{k,k+1}^{(\widehat{\mf q})}) - W_k > s_k^{*}$. By the definition
of the set $\widehat{\mz W}_{\widehat{\mf q}, k}^+(x)$,
\begin{equation*}
\sup_{y\in \widehat{\mz W}_{\widehat{\mf q}, k}^+(x)} S(y)\,
\leq\, W_k \,<\, S(\mz W_{k,k+1}^{(\widehat{\mf q})}) -s_k^{*}\,.
\end{equation*}
Subtracting $S(x)$ in both sides of the previous expression, and
since, by the definition of $\mc E_{\widehat{\mf q}}(k)$,
$S(x)\geq S(\mz M_{\widehat{\mf q}}(k))$,
\begin{equation*}
\sup_{y\in \widehat{\mz W}_{\widehat{\mf q}, k}^+(x)} S(y) - S(x)
< S(\mz W_{k,k+1}^{(\widehat{\mf q})}) - S(x) - s_k^{*}
\le S(\mz W_{k,k+1}^{(\widehat{\mf q})}) - S(\mz M_{\widehat{\mf q}}(k)) - s_k^{*}
= \mf h_{\widehat{\mf q}} - s_{k}^{*}\,,
\end{equation*}
where the last identity follows from the second identity of Assertion
\ref{s09}. This proves the second claim of the assertion with $r_1=
\min\{s_k^{*}: k\in S_{\widehat{\mf q}}\}$. 
\end{proof}

\begin{lemma}
\label{q:scale:mu_1est}
Fix a continuous function $f\in C(\bb T)$. Then, there exists a positive constant
$\mf h^*>0$, independent of $\epsilon$ such that
\begin{align*}
e^{-\mf h_{\widehat{\mf q}}/\epsilon}\,
\Big|\, & \int_{\bb T}f(x) \, \frac{1}{\mtt a(x)} \,
\phi_{\widehat{\mf q},\epsilon}(x)^2
\int_x^{x+1}\frac{e^{[S(y) - S(x)] /\epsilon}}{\phi_{\widehat{\mf q},\epsilon}(y)^2}
\,dy\,dx \\
&\quad - \sum_{k\in S_{\widehat{\mf q}}} \int_{\mc E_{\widehat{\mf q}}(k)} f(x)
\, \frac{1}{\mtt a(x)} \, \phi_{\widehat{\mf q},\epsilon}(x)^2
\sum_{j\in \mf I_k}\sum_{\sigma\in \mc W_{j,j+1}^{(\widehat{\mf q})}}
\int_{\sigma-\eta_0}^{\sigma+\eta_0}
\frac{e^{[S(y)-S(x)] /\epsilon}}{\phi_{\widehat{\mf q},\epsilon}(y)^2}\,dy\,dx\,\Big|
\,\leq\, e^{-\mf h^{*}/\epsilon}\,.
\end{align*}
\end{lemma}

\begin{proof}
We need to show that
\begin{equation*}
e^{- \mf  h_{\widehat{\mf q}}/\epsilon}\,
\Big|\,
\int_{\bb T\,\setminus\, \bigcup_{k\in S_{\widehat{\mf q}}}\mc E_{\widehat{\mf q}}(k)}
f(x) \, \frac{1}{\mtt a(x)} \, e^{-S(x)/\epsilon}
\phi_{\widehat{\mf q},\epsilon}(x)^2 \, \int_x^{x+1}
\frac{e^{S(y)/\epsilon}}{\phi_{\widehat{\mf q},\epsilon}(y)^2} \,dy\,dx
\, \Big| \,\leq\, e^{-\mf h^{*}/\epsilon}
\end{equation*}
and
\begin{equation*}
e^{- \mf  h_{\widehat{\mf q}}/\epsilon}\, \Big|\, \sum_{k\in S_{\widehat{\mf q}}}\int_{\mc E_{\widehat{\mf q}}(k)}
f(x) \frac{1}{\mtt a(x)}e^{-S(x)/\epsilon}\,
\phi_{\widehat{\mf q},\epsilon}(x)^2\, \int_{\widehat{\mc W}_{\widehat{\mf q},k}^+(x)}
\frac{e^{S(y)/\epsilon}}{\phi_{\widehat{\mf q},\epsilon}(y)^2}\,dy\,dx
\, \Big| \,\leq\, e^{-\mf h^{*}/\epsilon}
\end{equation*}
for some $\mf h^{*}>0$.

We consider only the first term, as the second one can be estimated in
the same way. The argument relies on Assertion
\ref{s08}.  By Lemmata
\ref{l09} and \ref{s10}, and the
fact that $f$ and $\mtt a$ are bounded functions, the left-hand side
of the penultimate displayed equation is bounded above by
\begin{align*}
C_0 \, e^{- \mf  h_{\widehat{\mf q}}/\epsilon}\,
\int_{\bb T\,\setminus\, \bigcup_{k\in S_{\widehat{\mf q}}}\mc E_{\widehat{\mf q}}(k)}\,
e^{-S(x)/\epsilon}\int_x^{x+1}e^{S(y)/\epsilon}\,dy\,\,dx 
\leq C_0\, e^{- \mf h_{\widehat{\mf q}}/\epsilon}
\sup_{x\in [0,1)\setminus \bigcup_{k\in S_{\widehat{\mf q}}}\mc E_{\widehat{\mf q}}(k)}
e^{-V_{\widehat{\mf q}}(x)}\,. 
\end{align*}
for some finite and positive constant $C_0$, independent from
$\epsilon$. By Assertion \ref{s08}, the
right-hand side of the previous expression is bounded by
$C_0 \, \exp\{ - r_1/\epsilon\}$, which completes the proof.
\end{proof}

\begin{lemma}
\label{l07}
For each $k\in S_{\widehat{\mf q}}$, $j\in \mf I_k$,
\begin{align*}
\lim_{\epsilon\to 0} \frac{1}{\sqrt{\epsilon}} \,
e^{-[ \mf h_{\widehat{\mf q}}+S(\mz M_{\widehat{\mf q}}(k))]/\epsilon}\,
\sum_{\sigma\in \mc W_{j,j+1}^{(\widehat{\mf q})}}
\int_{\sigma-\eta_0}^{\sigma+\eta_0}
\frac{e^{S(y)/\epsilon}}{\phi_{\widehat{\mf q},\epsilon}(y)^2}\,dy
\,=\, 
\frac{\sigma_{\widehat{\mf q}}(j,j+1)}
{\widehat{h}(\mc M_{\widehat{\mf q}}(j)) \, \widehat{h}(\mc M_{\widehat{\mf q}}(j+1))} \,,
\end{align*}
where $\sigma_{\widehat{\mf q}}(j,j+1)$ has been introduced in
\eqref{p_weights}. 
\end{lemma}

\begin{proof}
Fix $k\in S_{\widehat{\mf q}}$, $j\in \mf I_k$.  We start with a remark. By
the identity \eqref{56},
$S(\mz M_{\widehat{\mf q}}(j)) = S(\mz M_{\widehat{\mf q}}(k))$, and by Assertion
\ref{s09},
$S(\sigma) = \mf h_{\widehat{\mf q}}+S(\mz M_{\widehat{\mf q}}(j))$ for all
$\sigma\in \mz W_{j,j+1}^{(\widehat{\mf q})}$, so that
\begin{equation}
\label{41}
S(\sigma) \,=\, \mf h_{\widehat{\mf q}} \,+\, S(\mz M_{\widehat{\mf q}}(j))
\,=\, \mf h_{\widehat{\mf q}} \,+\, S(\mz M_{\widehat{\mf q}}(k))
\quad \text{for all}\;\; 
\sigma\in \mz W_{j,j+1}^{(\widehat{\mf q})} \,.
\end{equation}

The proof is performed in four steps.  Let
\begin{equation*}
{\color{blue}H_{\epsilon}} := 
\sum_{\sigma\in \mc W_{j,j+1}^{(\widehat{\mf q})}}
\int_{\sigma-\eta_0}^{\sigma+\eta_0}
\frac{e^{S(y)/\epsilon}}{\phi_{\widehat{\mf q},\epsilon}(y)^2}\,dy\,.
\end{equation*}

\smallskip
\noindent{\bf Step 1: Replacing $H_{\epsilon}$ by $\widehat{H}_{\epsilon}$}
\smallskip

Fix $\sigma\in \mz W_{j,j+1}^{(\widehat{\mf q})}$, and recall the definition of
the minima $m_{\widehat{\mf q},j}^{\pm}$ given in \eqref{17} and let
$\psi_{j,\sigma,\epsilon} \colon [\sigma-\eta_0,\sigma+\eta_0]\to \bb
R$ be the function defined by
\begin{equation*}
{\color{blue}\psi_{j,\sigma,\epsilon}(x)} :=
\phi_{\widehat{\mf q},\epsilon}(m_{\widehat{\mf q},j}^+)\,
\bb P_x^\epsilon\big[\tau(m_{\widehat{\mf q},j}^+)<\tau(m_{\widehat{\mf q}, j+1}^-)\big]
+ \phi_{\widehat{\mf q}, \epsilon}(m_{\widehat{\mf q},j+1}^-)
\, \bb P_x^\epsilon\big[\tau(m_{\widehat{\mf q}, j+1}^-)<\tau(m_{\widehat{\mf q}, j}^+)\big]\,.
\end{equation*}
We claim that we can replace $\phi_{\widehat{\mf q},\epsilon} (\cdot)$ by
$\psi_{j,\sigma,\epsilon}(\cdot)$ in the definition of
$H_{\epsilon}$: 
\begin{equation}
\label{40}
\lim_{\epsilon\to 0} \frac{1}{\sqrt{\epsilon}} \,
e^{-[ \mf h_{\widehat{\mf q}}+S(\mz M_{\widehat{\mf q}}(k))]/\epsilon}\,
\big|\, H_{\epsilon} - \widehat{H}_{\epsilon}\,\big|
\,=\,  0\,.
\end{equation}
where
\begin{align*}
{\color{blue} \widehat{H}_{\epsilon}} := 
\sum_{\sigma\in \mc W_{j,j+1}^{(\widehat{\mf q})}}
\int_{\sigma-\eta_0}^{\sigma+\eta_0}
\frac{e^{S(y)/\epsilon}}{\psi_{j,\sigma,\epsilon}(y)^2}\,dy\,.
\end{align*}

To prove \eqref{40}, fix $\sigma\in \mc W_{j,j+1}^{(\widehat{\mf q})}$.  By the
definition of $\eta_0$, $\sigma$ is the unique maxima in the interval
$(\sigma-\eta_0,\sigma+\eta_0)$.  Thus,
\begin{align*}
& \int_{\sigma-\eta_0}^{\sigma+\eta_0}
\Big|\, \frac{1}{\phi_{\widehat{\mf q},\epsilon}(y)^2}
-\frac{1} {\psi_{j,\sigma,\epsilon}(y)^2}  \,\Big| \, 
e^{S(y)/\epsilon}\,dy
\,\le\, e^{S(\sigma)/\epsilon}
\int_{\sigma-\eta_0}^{\sigma+\eta_0}
\Big|\, \frac{1}{\phi_{\widehat{\mf q},\epsilon}(y)^2}
-\frac{1} {\psi_{j,\sigma,\epsilon}(y)^2}  \,\Big| \, 
\,dy\,.
\end{align*}

Therefore, by \eqref{41},
\begin{align*}
&\frac{e^{-[ \mf h_{\widehat{\mf q}}+S(\mz M_{\widehat{\mf q}}(k))]/\epsilon}}
{\sqrt{\epsilon}} \int_{\sigma-\eta_0}^{\sigma+\eta_0}
\Big|\, \frac{1}{\phi_{\widehat{\mf q},\epsilon}(y)^2}
-\frac{1} {\psi_{j,\sigma,\epsilon}(y)^2}  \,\Big| \, 
e^{S(y)/\epsilon}\,dy \notag
\\
&\quad
\le\, \frac{1}{\sqrt{\epsilon}}\,
\int_{\sigma-\eta_0}^{\sigma+\eta_0}
\Big| \, \frac{(\phi_{\widehat{\mf q},\epsilon}(y)\,
-\,\psi_{j,\sigma,\epsilon}(y)) \,
(\phi_{\widehat{\mf q},\epsilon}(y)+\psi_{j,\sigma,\epsilon}(y))}
{\psi_{j,\sigma,\epsilon}(y)^2 \, \phi_{\widehat{\mf q},\epsilon}(y)^2}\, \Big|\, dy\,.
\end{align*}
By Lemmata \ref{l09} and \ref{s10},
and since
$\bb E_x^{\epsilon}[\tau(m_{\widehat{\mf q},j}^+,m_{\widehat{\mf q},j+1}^-) ] \leq \bb
E_x^\epsilon[\tau(\mc M_{\widehat{\mf q}})]$, the previous expression is less
than or equal to
\begin{align*}
\frac{C_0}{\sqrt{\epsilon}} \int_{\sigma-\eta_0}^{\sigma+\eta_0}
\big| \, \phi_{\widehat{\mf q},\epsilon}(y)\,-\,\psi_{j,\sigma,\epsilon}(y)) \,
\big|\,dy
\, \leq \, \frac{C_0\, \eta_0 }{\sqrt{\epsilon}}
\,\sup_{x\in \bb T}\bb E_{x}^{\epsilon}[\tau(\mc M_{\mf
q})/\theta_{\epsilon}^{(\widehat{\mf q})}]
\end{align*}
for some finite constant $C_0$ whose value changed from from the left
to the right-hand side. By Lemma \ref{hitting_1}, the last expression
is bounded above by
$C_0\, \eta_0 \, \epsilon^{-1/2} \, \exp\{ -[\mf h_{\widehat{\mf q}}- \mf h_{\mf
q -1}]/\epsilon \}$. This completes the proof of \eqref{40}.  \smallskip

By \eqref{41} and \eqref{40}, to prove the lemma it is enough to
show that 
\begin{equation}
\label{q:scale:wconv_est}
\lim_{\epsilon \to 0}\,\Big|\,
\sum_{\sigma\in \mc W_{j,j+1}^{(\widehat{\mf q})}}
\frac{1}{\sqrt{\epsilon}} \,
\int_{\sigma-\eta_0}^{\sigma+\eta_0}
\frac{e^{[S(y) - S(\sigma)] /\epsilon}}
{\psi_{j,\sigma,\epsilon}(y)^2}\,dy
\,-\, \frac{\sigma_{\widehat{\mf q}}(j,j+1)}
{\widehat{h}(\mc M_{\widehat{\mf q}}(j)) \, \widehat{h}(\mc M_{\widehat{\mf q}}(j+1))}
\, \Big| = 0\,.    
\end{equation}

Since the diffusion $X_{\epsilon} (\cdot)$ is irreducible,
\begin{equation}
\label{q:scale:psi_2}
\psi_{j,\sigma,\epsilon}(x) =
\phi_{\widehat{\mf q},\epsilon}(m_{\widehat{\mf q}, j}^+) +
\bb P_{x}^\epsilon\big[\, \tau(m_{\widehat{\mf q},j+1}^-)
<\tau(m_{\widehat{\mf q},j}^+)\,\big]
\,
\big[ \, \phi_{\widehat{\mf q},\epsilon}(m_{\widehat{\mf q},j+1}^-)
\,-\,\phi_{\widehat{\mf q},\epsilon}(m_{\widehat{\mf q},j}^+)\, \big]\,.
\end{equation}
The argument is splited in two cases.

\smallskip
\noindent{\bf Step 2: Assume that
$\widehat{h}(\mc M_{\widehat{\mf q}}(j)) = \widehat{h}(\mc M_{\widehat{\mf q}}(j+1))$.}
\smallskip

Recall from \eqref{p_weights} the definition of
$\sigma_{\widehat{\mf q}}(j,j+1)$.  Fix $\sigma\in \mz W_{j,j+1}^{(\widehat{\mf q})}$.
Add and subtract $1/\phi_{\widehat{\mf q},\epsilon}(m_{\widehat{\mf q}, j}^+)^2$ from
$1/\psi_{j,\sigma,\epsilon}(y)^2$ to estimate the expression appearing
on the left-hand side of \eqref{q:scale:wconv_est} by
\begin{align*}
& \frac{1}{\sqrt{\epsilon}}\,
\, \int_{\sigma-\eta_0}^{\sigma+\eta_0}
e^{[S(y) - S(\sigma)] /\epsilon} \, 
\Big|\, \frac{1}{\psi_{j,\sigma,\epsilon}(y)^2} -
\frac{1}{ \phi_{\widehat{\mf q},\epsilon}(m_{\widehat{\mf q},j}^+)^2}\, \Big| \,dy
\\
&\quad + \,
\Big|\, 
\frac{1}{\phi_{\widehat{\mf q},\epsilon}(m_{\widehat{\mf q},j}^+)^2}
\frac{1}{\sqrt{\epsilon}}\, 
\int_{\sigma -\eta_0}^{\sigma+\eta_0} \,
e^{[S(y) - S(\sigma)] /\epsilon}  \,dy
\,-\, \frac{1} {\widehat{h}(\mc M_{\widehat{\mf q}}(j))^2}\, 
\sqrt{\frac{2\pi}{-S''(\sigma)}} \, \Big| \,.
\end{align*}
By the definition \eqref{q:resolv_limsup} of
$\phi_{\widehat{\mf q},\epsilon}(\cdot)$, Proposition
\ref{resolvent_limit_point_2}, and by Laplace's method, the second
term vanishes as $\epsilon\to 0$.  By Lemmata
\ref{l09} and \ref{s10}, the first
one is bounded by
\begin{equation*}
\frac{C_0}{\sqrt{\epsilon}}\,
\, \int_{\sigma-\eta_0}^{\sigma+\eta_0}
e^{[S(y) - S(\sigma)] /\epsilon} \,
\big|\, \psi_{j,\sigma,\epsilon}(y)  -
\phi_{\widehat{\mf q},\epsilon}(m_{\widehat{\mf q},j}^+) \, \big| \,dy
\end{equation*}
for some finite constant $C_0$. By \eqref{q:scale:psi_2}, Proposition
\ref{resolvent_limit_point_2}, and since, by hypothesis,
$\widehat{h}(\mc M_{\widehat{\mf q}}(j)) = \widehat{h}(\mc M_{\widehat{\mf q}}(j+1))$,
this expression vanishes as $\epsilon\to 0$. This proves
\eqref{q:scale:wconv_est} in the case where $\widehat{h}(\mc M_{\mf
q}(j)) = \widehat{h}(\mc M_{\widehat{\mf q}}(j+1))$.

We turn to the case
$\widehat{h}(\mc M_{\widehat{\mf q}}(j)) \neq \widehat{h}(\mc M_{\mf
q}(j+1))$. By \eqref{q:scale:psi_2} and Lemma \ref{s06}, 
\begin{equation*}
\label{q:scale:der:psi_2}
\psi_{j,\sigma,\epsilon}'(x) =
\frac{e^{S(x)/\epsilon}}
{\int_{m_{\widehat{\mf q},j}^+}^{m_{\widehat{\mf q},j+1}^-} \, e^{S(y)/\epsilon}\,dy}
\, \big[ \, \phi_{\widehat{\mf q},\epsilon}(m_{\widehat{\mf q},j+1}^-)
\,-\,\phi_{\widehat{\mf q},\epsilon}(m_{\widehat{\mf q},j}^+)\, \big]\,. 
\end{equation*} 
As
$\widehat{h}(\mc M_{\widehat{\mf q}}(j)) \neq \widehat{h}(\mc M_{\widehat{\mf q}}(j+1))$,
by Proposition \ref{resolvent_limit_point_2},
$\psi_{j,\sigma,\epsilon}'$ is bounded away from $0$. We may thus
write $\psi_{j,\sigma,\epsilon}^{-2}$ as
$-[\psi_{j,\sigma,\epsilon}']^{-1}[\psi_{j,\sigma,\epsilon}^{-1}]'$.
This identity together with \eqref{41} yields that
\begin{align}
\label{42}
& \frac{1} {\sqrt{\epsilon}}\,
\sum_{\sigma\in \mz W_{j,j+1}^{(\widehat{\mf q})}}
\int_{\sigma-\eta_0}^{\sigma+\eta_0}
\frac{e^{[S(y) - S(\sigma)] /\epsilon}}
{\psi_{j,\sigma,\epsilon}(y)^2}\,dy\\
& \quad = - \, \frac{1} {\sqrt{\epsilon}}\,
\frac{\int_{m_{\widehat{\mf q},j}^+}^{m_{\widehat{\mf q},j+1}^-}
e^{[S(y) - S(\sigma)] /\epsilon} \,dy}
{\phi_{\widehat{\mf q},\epsilon}(m_{\widehat{\mf q}, j+1}^-)
-\phi_{\widehat{\mf q},\epsilon}(m_{\widehat{\mf q},j}^+)}
\sum_{\sigma\in \mz W_{j,j+1}^{(\widehat{\mf q})}}
\int_{\sigma-\eta_0}^{\sigma+\eta_0}
\Big(\frac{1}{\psi_{j,\sigma,\epsilon}(y)}\Big)'\,dy
\nonumber
\\
& \quad = - \, \frac{1} {\sqrt{\epsilon}}\,
\frac{\int_{m_{\widehat{\mf q},j}^+}^{m_{\widehat{\mf q},j+1}^-}
e^{[S(y) - S(\sigma)] /\epsilon}  \,dy}
{\phi_{\widehat{\mf q},\epsilon}(m_{\widehat{\mf q}, j+1}^-)
-\phi_{\widehat{\mf q},\epsilon}(m_{\widehat{\mf q},j}^+)}
\sum_{\sigma\in \mz W_{j,j+1}^{(\widehat{\mf q})}}
\Big(\frac{1}{\psi_{j,\sigma,\epsilon}(\sigma+\eta_0)}
\,-\,\frac{1}{\psi_{j,\sigma,\epsilon}(\sigma - \eta_0)}\Big)\,.
\nonumber
\end{align}

\smallskip
\noindent{\bf Step 3: Assume that
$\widehat{h}(\mc M_{\widehat{\mf q}}(j)) \neq  \widehat{h}(\mc M_{\widehat{\mf q}}(j+1))$
and $\mz W_{j,j+1}^{(\widehat{\mf q})}$ is a singleton.}
\smallskip

If $\mz W_{j,j+1}^{(\widehat{\mf q})}$ has exactly one maxima $\sigma$, by
\eqref{41} and the definition \eqref{q:scale:psi_2} of
$\psi_{j,\sigma,\epsilon}(\cdot)$, the previous expression is equal to
\begin{align*}
-\, \frac{1}{\sqrt{\epsilon}} \, 
\frac{\int_{m_{\widehat{\mf q},j}^+}^{m_{\widehat{\mf q},j+1}^-}
e^{[S(y) - S(\sigma)] /\epsilon}  \,dy}
{\psi_{j,\sigma,\epsilon}(\sigma+\eta_0)\, \psi_{j,\sigma,\epsilon}(\sigma-\eta_0)}
\, \Big[\, \bb P_{\sigma - \eta_0}^\epsilon \big[\,
\tau(m_{\widehat{\mf q},j+1}^-)<\tau(m_{\widehat{\mf q}, j}^+)\, \big]
\,-\,\bb P_{\sigma+\eta_0}^\epsilon \big[\, \tau(m_{\widehat{\mf q},j+1}^-)<\tau(m_{\mf
q, j}^+\, \big]\, \Big]\,. 
\end{align*}
By Laplace method, Lemma \ref{s06}, Proposition
\ref{resolvent_limit_point_2}; and the definition \eqref{p_weights} of
$\sigma_{\widehat{\mf q}}(j,j+1)$, the previous expression is equal to
\begin{align}
\label{q:scale:wconv:eq8}
\sqrt{\frac{2\pi}{-S''(\sigma)}}\,
\frac{1}{\widehat{h}(\mc M_{\widehat{\mf q}}(j+1))
\, \widehat{h}(\mc M_{\widehat{\mf q}}(j))}\,
+ R(\epsilon) 
= \frac{\sigma_{\widehat{\mf q}}(j,j+1)}
{\widehat{h}(\mc M_{\widehat{\mf q}}(j+1))\, \widehat{h}(\mc M_{\widehat{\mf q}}(j))}\,
\,+\, R(\epsilon)\,, 
\end{align}
where $R(\epsilon)$ denotes an error that vanishes as $\epsilon\to 0$.
This proves \eqref{q:scale:wconv_est} when
$\widehat{h}(\mc M_{\widehat{\mf q}}(j)) \neq \widehat{h}(\mc M_{\widehat{\mf q}}(j+1))$,
and $\mz W_{j,j+1}^{(\widehat{\mf q})}$ is a singleton.

\smallskip
\noindent{\bf Step 4: Assume that
$\widehat{h}(\mc M_{\widehat{\mf q}}(j)) \neq  \widehat{h}(\mc M_{\widehat{\mf q}}(j+1))$
and $\mz W_{j,j+1}^{(\widehat{\mf q})}$ is not a singleton.}
\smallskip

Recall from \eqref{43} that we denote by
$\sigma^{1}_{j,j+1} < \dots < \sigma^{\mf r_{j,j+1}}_{j, j+1}$ the
elements of $\mz W_{j,j+1}^{(\widehat{\mf q})}$.  Re-ordering the terms in
the sum appearing in the last line of \eqref{42} yields that the first
line of \eqref{42} is equal to the sum $I_1(\epsilon) +
I_2(\epsilon)$, where 
\begin{equation*}
I_1(\epsilon)  \,=\, - \, \frac{1}{\sqrt{\epsilon}}
\frac{\int_{m_{\widehat{\mf q},j}^+}^{m_{\widehat{\mf q},j+1}^-}
e^{[S(y - S(\sigma)] )/\epsilon}\,dy}
{\phi_{\widehat{\mf q},\epsilon}(m_{\widehat{\mf q},j+1}^-)
-\phi_{\widehat{\mf q},\epsilon}(m_{\widehat{\mf q},j}^+)}
\,\Big( \,
\frac{1}{\psi_{j,\sigma,\epsilon}(\sigma_{j,j+1}^{\widehat{\mf q},+}+\eta_0)}
\,-\,
\frac{1}{\psi_{j,\sigma,\epsilon}(\sigma_{j,j+1}^{\widehat{\mf q},-} - \eta_0)}
\, \Big ) 
\end{equation*}
and
\begin{align*}
I_2(\epsilon) \, =\, - \, \frac{1}{\sqrt{\epsilon}}
\frac{\int_{m_{\widehat{\mf q},j}^+}^{m_{\widehat{\mf q},j+1}^-}
e^{[S(y - S(\sigma)] )/\epsilon}\,dy}
{\phi_{\widehat{\mf q},\epsilon}(m_{\widehat{\mf q},j+1}^-)
-\phi_{\widehat{\mf q},\epsilon}(m_{\widehat{\mf q},j}^+)}
\,\sum_{r=1}^{\mf r_{j,j+1}-1}
\Big\{ \frac{1}{\psi_{j,\sigma,\epsilon}(\sigma_{j,j+1}^{\widehat{\mf q},r}+\eta_0)}
- \frac{1}{\psi_{j,\sigma,\epsilon}(\sigma_{j,j+1}^{\widehat{\mf q},r+1} -
\eta_0)}\, \Big\}\,.
\end{align*}
By \eqref{q:scale:wconv:eq8}, 
\begin{align*}
\lim_{\epsilon\to 0} I_1(\epsilon) \, =\,
\frac{\sigma_{\widehat{\mf q}}(j,j+1)}
{\widehat{h}(\mc M_{\widehat{\mf q}}(j+1))\, \widehat{h}(\mc M_{\widehat{\mf q}}(j))} \,.
\end{align*}
By Laplace's method and Proposition \ref{resolvent_limit_point_2}, the
factor multiplying the sum in $I_2(\epsilon)$ is of order $1$ since we
assumed that $\widehat{h}(\mc M_{\widehat{\mf q}}(j))\neq \widehat{h}(\mc M_{\mf
q}(j+1))$.

We turn to the remaining term. Fix $1\leq r\leq \mf r_{j,j+1}-1$. By
the definition \eqref{q:scale:psi_2} of
$\psi_{j,\sigma,\epsilon} (\cdot)$,
\begin{align*}
\big|\, \psi_{j,\sigma,\epsilon}(\sigma_{j,j+1}^{\widehat{\mf q},r}+\eta_0)^{-1}
\,-\,\psi_{j,\sigma,\epsilon}(\sigma_{j,j+1}^{\widehat{\mf q},r+1}-\eta_0)^{-1}
\, \big|
= \Big| \, \frac{\phi_{\widehat{\mf q},\epsilon}(m_{\widehat{\mf q},j+1}^-)
- \phi_{\widehat{\mf q},\epsilon}(m_{\widehat{\mf q},j}^+)}
{\psi_{j,\sigma,\epsilon}(\sigma_{j,j+1}^{\widehat{\mf q},r}+\eta_0) \,
\psi_{j,\sigma,\epsilon}(\sigma_{j,j+1}^{\widehat{\mf q},r+1}-\eta_0)} \,
\Big|\, A_\epsilon
\end{align*}
where
\begin{align*}
A_\epsilon \,=\,
\Big|\, \bb P_{\sigma_{j,j+1}^{\widehat{\mf q},r}+\eta_0}^\epsilon
\Big(\tau(m_{\widehat{\mf q},j+1}^-)<\tau(m_{\widehat{\mf q},j}^+)\Big)
\,-\,\bb P_{\sigma_{j,j+1}^{\mf
q,r+1}-\eta_0}^\epsilon\Big(\tau(m_{\widehat{\mf q},j+1}^-)
<\tau(m_{\widehat{\mf q},j}^+)\Big)\, \Big| \,.
\end{align*}
By Lemmata \ref{l09} and \ref{s10},
the first term on the right-hand side of the penultimate displayed
equation is bounded, uniformly in $\epsilon$. We turn to $A_\epsilon$.
By Lemma \ref{s06}, 
\begin{align*}
A_\epsilon \,=\, 
\int_{\sigma_{j,j+1}^{\widehat{\mf q},r}+\eta_0}
^{\sigma_{j,j+1}^{\widehat{\mf q},r+1}-\eta_0}
e^{S(y)/\epsilon}\,dy\,\Bigg/
\,\int_{m_{\widehat{\mf q},j}^{+}}^{m_{\widehat{\mf q},j+1}^{-}}
e^{S(y)/\epsilon}\,dy\,.
\end{align*}
Multiply the numerator and the denominator by
$\epsilon^{-1/2}\, \exp \{- S(\sigma)/\epsilon\}$ for some
$\sigma\in \mz W_{j,j+1}^{(\widehat{\mf q})}$.  Laplace's method yields that the
denominator converges to a positive constant. Since the torus has
length $1$, and $\sigma_{j,j+1}^{\widehat{\mf q},r}$,
$\sigma_{j,j+1}^{\widehat{\mf q},r+1}$ are consecutive global maxima in the
interval $[m_{\widehat{\mf q},j}^+ ,\, m_{\widehat{\mf q},j+1}^-]$, the numerator is
bounded by
\begin{equation*}
\frac{1}{\sqrt{\epsilon}}\exp\Big(\,
\big[S(\sigma_{j,j+1}^{\widehat{\mf q},r}+\eta_0)\vee S(\sigma_{j,j+1}^{\widehat{\mf q},r+1}-\eta_0)
\,-\,S(\sigma)\big]\,/\,\epsilon\,\Big) 
\end{equation*}
As $\epsilon\to 0$, this expression vanishes, which concludes the
proof of the lemma.
\end{proof}

For each $k\in S_{\widehat{\mf q}}$, let 
\begin{equation*}
{\color{blue}H_{k,\epsilon}} := \sum_{j\in \mf I_k}
\sum_{\sigma\in \mc W_{j,j+1}^{(\widehat{\mf q})}}
\int_{\sigma-\eta_0}^{\sigma+\eta_0}
\frac{e^{S(y)/\epsilon}}{\phi_{\widehat{\mf q},\epsilon}(y)^2}\,dy\,.
\end{equation*}
By the previous lemma,
\begin{align}
\label{44}
\lim_{\epsilon\to 0} \frac{1}{\sqrt{\epsilon}} \,
e^{-[ \mf h_{\widehat{\mf q}}+S(\mz M_{\widehat{\mf q}}(k))]/\epsilon}\, H_{k,\epsilon}
\,=\, \sum_{j\in \mf I_k}
\frac{\sigma_{\widehat{\mf q}}(j,j+1)}
{\widehat{h}(\mc M_{\widehat{\mf q}}(j)) \, \widehat{h}(\mc M_{\widehat{\mf q}}(j+1))} \,\cdot
\end{align}

\begin{corollary}
\label{l08}
For each  continuous function $f\in C(\bb T)$,
\begin{align*}
& \lim_{\epsilon \to 0}
\frac{1}{\epsilon}e^{-\mf h_{\widehat{\mf q}}/\epsilon}
\sum_{k\in S_{\widehat{\mf q}}}\int_{\mc E_{\widehat{\mf q}}(k)}f(x)
\frac{1}{\mtt a(x)} \, \phi_{\widehat{\mf q},\epsilon}(x)^2
\sum_{j\in \mf I_k }\sum_{\sigma\in \mc W_{j,j+1}^{(\widehat{\mf q})}}
\int_{\sigma-\eta_0}^{\sigma+\eta_0}\frac{e^{ [S(y) - S(x)] /\epsilon}}{
\phi_{\widehat{\mf q},\epsilon}(y)^2}\,dy\,dx \notag
\\
& \quad = \sum_{k\in S_{\widehat{\mf q}}}\widehat{h}(\mc M_{\widehat{\mf q}}(k))^2
\sum_{j\in \mf I_k}\frac{\sigma_{\widehat{\mf q}}(j,j+1)}
{\widehat{h}(\mc M_{\widehat{\mf q}}(j)) \, \widehat{h}(\mc M_{\widehat{\mf q}}(j+1))}
\sum_{m\in\mc M_{\widehat{\mf q}}(k)}\pi_1 (\{m\}) \, f(m)
\end{align*}
\end{corollary}

\begin{proof}
Multiplying and dividing by $\exp \{-S(\mz M_{\widehat{\mf q}}(k))/\epsilon\}$
we can rewrite the left-hand side of the statement of the lemma as
\begin{equation}
\label{q:scale:lastlem:eq1}
\sum_{k\in S_{\widehat{\mf q}}}\frac{1}{\sqrt{\epsilon}} \,
e^{-[ \mf h_{\widehat{\mf q}}+S(\mz M_{\widehat{\mf q}}(k))]/\epsilon}
\,H_{k,\epsilon}\,\int_{\mc E_{\widehat{\mf q}}(k)} f(x) \,
\frac{1}{\mtt a(x)} \, 
\frac{1}{\sqrt{\epsilon}}\, 
e^{[S(\mz M_{\widehat{\mf q}}(k))-S(x)]/\epsilon}
\, \phi_{\widehat{\mf q},\epsilon}(x)^2\,dx \,.
\end{equation}

For each $k\in S_{\widehat{\mf q}}$, by Laplace's method, the definition
\eqref{q:resolv_limsup} of $\phi_{\widehat{\mf q},\epsilon}(\cdot)$, and
Proposition \ref{resolvent_limit_point_2},
\begin{align*}
& \lim_{`e\to 0} \int_{\mc E_{\widehat{\mf q}}(k)} f(x) \,
\frac{1}{\mtt a(x)} \, 
\frac{1}{\sqrt{\epsilon}}\, 
e^{[S(\mz M_{\widehat{\mf q}}(k))-S(x)]/\epsilon}
\, \phi_{\widehat{\mf q},\epsilon}(x)^2\,dx
\\
&\quad \,=\, \widehat{h}(\mc M_{\widehat{\mf q}}(k))^2 \, 
\sum_{m\in\mc M_{\widehat{\mf q}}(k)} f(m)\, \frac{1}{\mtt a(m)}\,
\sqrt{\frac{2\pi}{S''(m)}} \,.
\end{align*}
Thus, by \eqref{44}, as $\epsilon\to 0$,
\eqref{q:scale:lastlem:eq1} converges to
\begin{align*}
\sum_{k\in S_{\widehat{\mf q}}} \Big( \sum_{j\in \mf I_k}
\frac{\sigma_{\widehat{\mf q}}(j,j+1)}
{\widehat{h}(\mc M_{\widehat{\mf q}}(j)) \,
\widehat{h}(\mc M_{\widehat{\mf q}}(j+1))}  \Big)
\, \widehat{h}(\mc M_{\widehat{\mf q}}(k))^2
\sum_{m\in\mc M_{\widehat{\mf q}}(k)}\frac{1}{\mtt a(m)}
\, \sqrt{\frac{2\pi}{S''(m)}} \, f(m) \,.
\end{align*}
To complete the proof of the  corollary it remains to recall from
\eqref{22} the definition of the weights $\pi_1(\{m\})$.
\end{proof}

\subsection*{Proof of the main propositions}

The proof of the next result is presented in Appendix \ref{secA4}.

\begin{lemma}
\label{s07}
There exists a constant $Z_{\widehat{\mf q}}>0$, such that, for every
$k\in S_{\widehat{\mf q}}$
\begin{equation*}
\frac{1}{\omega(\mc M_{\widehat{\mf q}}(k))} \,
\pi_{\widehat{\mf q}}(k)\, \widehat{h}(\mc M_{\widehat{\mf q}}(k))^2
\sum_{j\in \mf I_k} \frac{\sigma_{\widehat{\mf q}}(j,j+1)}
{\widehat{h}(\mc M_{\widehat{\mf q}}(j))\widehat{h}(\mc M_{\widehat{\mf q}}(j+1))}
\,=\, Z_{\widehat{\mf q}}\,.
\end{equation*}
\end{lemma}

\begin{proof}[Proof of Proposition \ref{prop:q:wlimit}]
By the definition of $Z_\epsilon$, Lemma \ref{q:scale:mu_1est} and
Corollary \ref{l08}, with $f=1$,
\begin{align*}
\lim_{\epsilon\to 0} \frac{1}{\epsilon}e^{- \mf h_{\widehat{\mf q}}/\epsilon} \, Z_\epsilon
= \sum_{k\in S_{\widehat{\mf q}}}\pi_{\widehat{\mf q}}(k)\, 
\widehat{h}(\mc M_{\widehat{\mf q}}(k))^2 \sum_{j\in \mf I_k}
\frac{\sigma_{\widehat{\mf q}}(j,j+1)}
{\widehat{h}(\mc M_{\widehat{\mf q}}(j)) \, \widehat{h}(\mc M_{\widehat{\mf q}}(j+1))}\,. 
\end{align*}
By Lemma \ref{s07}, and since $\omega$ is a
probability measure on $\mc S_{\widehat{\mf q}}$, this expression is equal to
$Z_{\widehat{\mf q}}$. 
    
Fix a continuous function $f\in C(\bb T)$. By \eqref{q:recovery:seq},
\begin{align*}
\int_{\bb T}f(x)\,d\mu_{\widehat{\mf q},\epsilon}(x)
= \frac{1}{Z_\epsilon} \, \int_{\bb T}
f(x)\, \frac{1}{\mtt a(x)} \, e^{-S(x)/\epsilon}
\phi_{\widehat{\mf q},\epsilon}(x)^2 \int_x^{x+1}
\frac{e^{S(y)/\epsilon}}{\phi_{\widehat{\mf q},\epsilon}(y)^2}\,dy\,dx\,.
\end{align*}
By the first part of the proof,
$Z_\epsilon = [\, 1 + o(\epsilon)\,]\, Z_{\widehat{\mf q}}\, \epsilon \, e^{\mf
h_{\widehat{\mf q}}/\epsilon} $. On the other hand, by Lemma
\ref{q:scale:mu_1est} and Corollary \ref{l08},
\begin{align*}
& \lim_{\epsilon\to 0}
\frac{1}{\epsilon}\, e^{- \mf h_{\widehat{\mf q}}/\epsilon} \int_{\bb T}
f(x)\, \frac{1}{\mtt a(x)} \, e^{-S(x)/\epsilon}
\phi_{\widehat{\mf q},\epsilon}(x)^2 \int_x^{x+1}
\frac{e^{S(y)/\epsilon}}{\phi_{\widehat{\mf q},\epsilon}(y)^2}\,dy\,dx
\\
&\quad =\,
\sum_{k\in S_{\widehat{\mf q}}}\omega(\mc M_{\widehat{\mf q}}(k))\, 
\frac{\pi_{\widehat{\mf q}}(k) \, \widehat{h}(\mc M_{\widehat{\mf q}}(k)^2)}
{\omega(\mc M_{\widehat{\mf q}}(k))}
\sum_{j\in \mf I_k}
\frac{\sigma_{\widehat{\mf q}}(j,j+1)}
{\widehat{h}(\mc M_{\widehat{\mf q}}(j)) \, \widehat{h}(\mc M_{\widehat{\mf q}}(j+1))}
\sum_{m\in\mc M_{\widehat{\mf q}}(k)}
\frac{\pi_1(\{m\})}{\pi_{\widehat{\mf q}}(k)}f(m)
\end{align*}
By Lemma \ref{s07}, the definition
\eqref{q_limsup:mu_convex} of $\mu$, and \eqref{21}, this expression
is equal to
\begin{align*}
Z_{\widehat{\mf q}}\, \sum_{k\in S_{\widehat{\mf q}}}\omega(\mc M_{\widehat{\mf q}}(k))\,
\sum_{m\in\mc M_{\widehat{\mf q}}(k)}\frac{\pi(\{m\})}{\pi_{\widehat{\mf q}}(k)}f(m)
= Z_{\widehat{\mf q}}\,  \int_{\bb T}f(x)\,d\mu(x)\,.
\end{align*} 
The assertion of the proposition follows from the two previous
estimates.
\end{proof}

The first part of the previous proof characterizes the asymptotic
behavior of $Z_\epsilon$. This result together with Lemma
\ref{q:scale:mu_1est} and the formula \eqref{q:recovery:seq} for the
measure $\mu_{\widehat{\mf q},\epsilon}$ yields that for any bounded function
$g\colon \bb T \to \bb R$,
\begin{equation}
\label{45}
\lim_{\epsilon\to 0} \int_{\mc E_{\widehat{\mf q}}^c}g(x)\,d\mu_{\widehat{\mf q},\epsilon}(x) \,=\, 0\,.
\end{equation}
where $\cb{ \mc E_{\widehat{\mf q}}} := \cup_{k\in S_{\widehat{\mf q}}} \mc E_{\widehat{\mf q}}(k)$.

\begin{proof}[Proof of Proposition \ref{prop:q:glimsup}]
By equation \eqref{q:resolv_limsup}, adding and subtracting $\lambda$
gives that
\begin{align*}
\label{q:scale:limsup:lastequation}
-\int_{\bb T}
\frac{\theta_\epsilon^{(\widehat{\mf q})} (\ms L_\epsilon
\phi_{\widehat{\mf q},\epsilon}) (x)}
{\phi_{\widehat{\mf q},\epsilon}(x)}\,d\mu_{\widehat{\mf q},\epsilon}(x)
\, & =\,
-\sum_{k\in S_{\widehat{\mf q}}}\int_{\bb T}
\frac{ ( \widehat{\bb L}_{\widehat{\mf q}}\widehat{h}) (\mc M_{\widehat{\mf q}}(k))}
{\phi_{\widehat{\mf q},\epsilon}(x)} \, \zeta^{\widehat{\mf q}, k}(x)\,
d\mu_{\widehat{\mf q},\epsilon}(x)
\\
& +\, \lambda\,\sum_{k\in S_{\widehat{\mf q}}}
\int_{\bb T}\frac{\widehat{h}(\mc M_{\widehat{\mf q}}(k))
}{\phi_{\widehat{\mf q},\epsilon}(x)}\zeta^{\widehat{\mf q}, k}(x)\,
d\mu_{\widehat{\mf q},\epsilon}(x)\,-\,\lambda \,,
\nonumber
\end{align*}
where $\zeta^{\widehat{\mf q},k}$, $k\in S_{\widehat{\mf q}}$, was defined in
\eqref{resolv_smooth}. By Lemma \ref{l09} and \eqref{45}, we may
restrict these integrals to $\mc E_{\widehat{\mf q}}$. By Proposition
\ref{resolvent_limit_point_2}, on the set $\mc E_{\widehat{\mf q}} (k)$, the
function $\phi_{\widehat{\mf q},\epsilon}(\cdot)$ is uniformly close to
$\widehat{h}(\mc M_{\widehat{\mf q}}(k))$. We may thus replace in the previous
integrals $\phi_{\widehat{\mf q},\epsilon}(x)$ by $\widehat{h}(\mc M_{\mf
q}(k))$ at a cost which vanishes as $\epsilon\to 0$. After the
replacement, we may apply \eqref{45} to extend the integral to the
torus. In conclusion, the right-hand side is equal to
\begin{align*}
-\sum_{k\in S_{\widehat{\mf q}}}
\frac{ ( \widehat{\bb L}_{\widehat{\mf q}}\widehat{h}) (\mc M_{\widehat{\mf q}}(k))}
{\widehat{h} (\mc M_{\widehat{\mf q}}(k))}
\int_{\bb T} \zeta^{\widehat{\mf q}, k}(x)\,
d\mu_{\widehat{\mf q},\epsilon}(x)
\, +\, \lambda\,\sum_{k\in S_{\widehat{\mf q}}}
\int_{\bb T} \zeta^{\widehat{\mf q}, k}(x)\,
d\mu_{\widehat{\mf q},\epsilon}(x)\,-\,\lambda  \,+\, o(\epsilon)\,.
\end{align*}
By Proposition \ref{prop:q:wlimit}, as $\epsilon\to 0$, this
expression converges to
\begin{align*}
-\sum_{k\in S_{\widehat{\mf q}}}
\frac{ ( \widehat{\bb L}_{\widehat{\mf q}}\widehat{h}) (\mc M_{\widehat{\mf q}}(k))}
{\widehat{h} (\mc M_{\widehat{\mf q}}(k))}
\int_{\bb T} \zeta^{\widehat{\mf q}, k}(x)\,
d\mu (x)
\, +\, \lambda\,\sum_{k\in S_{\widehat{\mf q}}}
\int_{\bb T} \zeta^{\widehat{\mf q}, k}(x)\,
d\mu (x)\,-\,\lambda  \,.
\end{align*}
By definitions of the measure $\mu$ and the function $\zeta^{\widehat{\mf q},
k}$, $\int_{\bb T} \zeta^{\widehat{\mf q}, k}\, d\mu = \omega(\mc M_{\mf
q}(k))$. The previous expression is thus equal to
\begin{align*}
-\sum_{k\in S_{\widehat{\mf q}}}
\frac{ ( \widehat{\bb L}_{\widehat{\mf q}}\widehat{h}) (\mc M_{\widehat{\mf q}}(k))}
{\widehat{h} (\mc M_{\widehat{\mf q}}(k))} \,
\omega(\mc M_{\widehat{\mf q}}(k))\,. 
\end{align*}
By definition of $\widehat{h}$, it satisfies \eqref{reflected_equiv},
and the previous expression is equal to $\bb I^{(\widehat{\mf q})}(\omega) $,
which completes the proof of the proposition.
\end{proof}

\subsection*{Proof of Proposition \ref{s03} and Theorem \ref{mt3}}

Recall the Definition \ref{gamma_expansion_def}. By \eqref{27}, condition
1 in the definition of $\Gamma-$expansion is in force. By Propositions
\ref{l01} and \ref{s04}, condition 2 is in force for $p=-1$ and
$p=0$. By Propositions \ref{p_liminf}, \ref{higher_speed_glimsup}, \ref{p:g_limsup:rev},
\ref{l10}, and Lemma \ref{s11}, condition 2 also holds for $1\leq p\leq \widehat{\mf q}$. In particular, we have established Proposition \ref{s03}. Condition 3
follows by Propositions \ref{l01}, \ref{s04}, \ref{higher_ord:0level}
and \eqref{23}, \eqref{23b}. Condition 4 is part of the content of
Proposition \ref{l01}.

To complete the proof, it remains to show that condition 5 is in
force. This follows by \eqref{25} and the fact that
$\widehat{\bb X}_{\widehat{\mf q}}(\cdot)$ has only one closed irreducible
class. In fact, denote by $\mf R_{\widehat{\mf q}}$ the unique closed
irreducible class of $\widehat{\bb X}_{\widehat{\mf q}}(\cdot)$. Let
$\omega^{*}\in \mss P(\mc S_{\widehat{\mf q}})$ be the probability measure
defined by
\begin{equation} 
\omega^{*}(\mc M) \;:=\; \begin{cases}
\nu^{\widehat{\mf q}}_{\mf R_{\widehat{\mf q}}}(\mc M),
&\text{if } \mc M\in \mf R_{\widehat{\mf q}},\\ 
0, &\text{ otherwise}\;,
\end{cases}
\end{equation} 
where $\nu^{\widehat{\mf q}}_{\mf R_{\widehat{\mf q}}}(\cdot)$ is the
unique invariant measure of the process
$\widehat{\bb X}_{\widehat{\mf q}}(\cdot)$ restricted to
$\mf R_{\widehat{\mf q}}$. Note that
$\nu^{\widehat{\mf q}}_{\mf R_{\widehat{\mf q}}} = \omega^{*}$ if
$\widehat{\bb X}_{\widehat{\mf q}}(\cdot)$ is irreducible. Since,
$\bb I^{(\widehat{\mf q})}(\omega) = 0$ if, and only if,
$\omega = \omega^{*}$; then,
\begin{equation*}
\mu(\cdot) = \sum_{\mc M\in \mc S_{\widehat{\mf q}}}
\omega^{*}(\mc M)\mu_{\mc M}^{\widehat{\mf q}}(\cdot)
\end{equation*}
is the unique element in the $0-$level set of
$\ms J^{(\widehat{\mf q})}$.  \qed

\appendix

\section{Metastable structure}
\label{secA1}

In this section we present some of the properties of the hierarchical
metastable structure used in the article. The proofs of these properties
can be found in \cite[Appendix A]{lm}.

At each step $q\geq 1$, the recursive procedure generates a layer 
\begin{equation*}
\Gamma_q = (\mtt P_q, \mf h_q, \bb L_q),
\end{equation*}
where $\mtt P_q:= \mz S_q\cup \{\mz T_q\}$ is a partition of the set
of $S-$local minima $\mz M$, $\mz S_q:= \{\mz M_q(j): j\in \bb Z\}$,
$\mf h_q>0$ is an energy barrier, and $\bb L_q$ is the generator of a
$\mz S_q-$valued Markov chain, denoted by $\bb X_q(\cdot)$. If the
Markov chain $\bb X_q$ has two or more closed irreducible classes, we
may iterate the procedure. The recursion ends at the first step
$\mf q$ at which the method yields a Markov chain $\bb X_{\mf q}$ with
only one or no closed irreducible classes.

In Section \ref{secNR} we explicitly defined the energy barrier
$\mf h_q$ and the reduced generator $\bb L_q$. Before stating the main
properties of the metastable structure, we give an explicit definition
of the partition $\mz S_q$ for the case $q\geq 2$. We refer the reader
to Figure \ref{fig-1f} for an illustration. Denote by
$\color{blue} \mf R_{q-1}(k)$, $k\in K$, where $\color{blue} K$ is a
subset of $\bb N$, the closed irreducible classes of the Markov chain
$\bb X_{q-1}(\cdot)$. Let
\begin{equation}
\label{24}
{\color{blue}\mz M^*_{q}(k)} \,:=\,
\bigcup_{i\colon \mz M_{q-1}(i)  \in \widehat{\mf
R}_{q-1}(k)} \mz M_{q-1}(i) \;, \quad
{\color{blue}\mz M_{q}} \,:=\, \bigcup_{k\in K} \mz M^*_{q}(k)
\;.
\end{equation}
Hence, $\mz M_{q}$ contains all the $S$-local minima $m$ belonging to
some $\bb X_{q-1}$-recurrent state $\mz M_{q}(k)$. In \cite[Section
2]{lm} we showed that $K$ is either empty, a singleton set or
countably infinite. In the later case, the family
$\{\mz M_{q}^{*}(k)\}_{k\in K}$ can be relabeled in such a way that
its members are ordered and the set $\mz M_{q}^{*}(0)$ contains the
leftmost nonnegative minima in $\mz M_{q}$. Denote this new relabeled
family by $\mz M_{q}(k)$, $k\in \bb Z$. The state space of $\bb X_q$
is this relabeled family,
\begin{equation*}
\mz S_{q} = \{\mz M_{q}(k)\}_{k\in \bb Z}.
\end{equation*}

We may now introduce the properties of the metastable structure. We
start with the ones related to each partition $\mz S_q$ generated by
the recursion.

\begin{enumerate}
\item[[$\mc P_1(q)$\!\!\!]] For every $k\in \bb Z$,
$\mz M_{q}(k)\neq \varnothing$ and $\mz M_{q}(k) \subset \mz M$.
\end{enumerate}

Two subsets $C$, $D$ of $\bb R$ are said to be {\it well  ordered}, if
$C\cap D = \varnothing$ and $x<y$ or $x>y$ for every $x\in C$ and
$y\in D$. In the first case this relation is represented by $\color{blue}
C<D$, while in the second one it is denoted by $\color{blue} C>D$.
\begin{definition}
\label{order}
Let $I$ be a subset of $\bb Z$ and consider a family $\mz F$ of
subsets of $\bb R$ indexed by $I$
\begin{equation*}
\mz F = \{D_{k}\}_{k\in I}, \hspace{5mm} D_{k}\subset \bb R,
\hspace{1.5mm} \forall k\in I.
\end{equation*}
We say that $\mz F$ is well ordered if $ D_{k}<D_{l}$ for every
$k,l\in \bb I$ with $k<l$.
\end{definition}

\begin{enumerate}

\item[[$\mc P_2(q)$\!\!\!]] $\mz S_q$ is well ordered according to
Definition \ref{order} and its labeling is such that
$m_{\j_q} \in \mz M_q(0)$, where
\begin{equation}
\label{24b}
{\color{blue} j_q } \,:=\, \min\{k\ge 0 : m_k \in \mz M_q\}\;.
\end{equation}
\end{enumerate}

Two subsets $A, B$ of $\bb R$ are said to be equivalent, a relation
represented by $A\sim B$, if
\begin{equation}
\label{equiv_class}
B=A+k,    
\end{equation}
for some $k\in \bb Z$. Here $A+k = \{x+k: x\in A\}$. Recall that $\mf u_q\in \bb N$ represents the number of elements in the state space $\mc S_p$ of the projected chain $\widehat{\bb X}_q$. In other words, $\mf u_q\in \bb N$ is the number of \eqref{equiv_class}-equivalent classes of the set $\mz S_q$.

\begin{enumerate}

\item[[$\mc P_3(q)$\!\!\!]]  $\mz M_q(j + \mf u_q) = \mz M_q(j) + 1$
for all $j\in \bb Z$. In particular,
$\mz M_q(j + \mf u_q) \sim \mz M_q(j)$.

\end{enumerate}

Let $C$ be a nonempty subset of $\bb R$. We say that the elements in
$C$ have the same \emph{depth} if $S(x) = S(y)$ for all $x$,
$y \in C$. For any set $A\subset \bb R$ where all elements share the
same depth, let $S(A)$ denote the common value of $S$ at $A$:
\begin{equation}
\label{16}
{\color{blue} S(A)} \,:=\,  S(x),\;\; \forall\;x\in A.
\end{equation}

\begin{enumerate}

\item[[$\mc P_4(q)$\!\!\!]] For all $k\in \bb Z$, the local minima in
$\mz M_{q}(k)$ have the same \emph{depth}.

\end{enumerate}

We turn to the properties related to the energy barrier. Given two
subsets $A$, $B$ of $\bb R$ such that $A<B$, define the \emph{barrier}
between $A$ and $B$ as
\begin{equation*}
{\color{blue}\Lambda_0 (A,B)} \,:= \,
\sup_{x\in [x^{+}_{A}, x^{-}_{B}]}S(x),
\end{equation*}
where $x^{+}_{A}$ is the rightmost element of the closure of $A$, and
$x^{-}_{B}$ is the leftmost element on the closure of $B$:
${\color{blue}x^{+}_{A}}:= \sup\{x\in A\}$,
${\color{blue}x^{-}_{B}}:= \inf\{x\in B\}$.  For subsets $C$, $D$ of
$\bb R$ such that $C<D$ or $D<C$, let
\begin{equation}
\label{09}
{\color{blue} \Lambda (C,D)}  \,:= \, \Lambda_0  (C,D)\;\;\text{if}\;\;
C<D\quad\text{and}\quad 
\Lambda (C,D)  \,:= \, \Lambda_0  (D,C)\;\;\text{if}\;\; D<C\;.
\end{equation}

\begin{enumerate}

\item[[$\mc P_5(q)$\!\!\!]]  For every $k\in \bb Z$ for which
$\mz M_{q}(k)$ has two or more elements, and for every
$m',m''\in \mz M_{q}(k)$
\begin{equation*}
\Lambda (\{m'\},\{m''\}) - S(\mz M_{q}(k))\, \le\, \mf h_{q-1} \;,
\end{equation*}
where $\Lambda(\cdot,\cdot)$ is as in \eqref{09}. 
\item[[$\mc P_6(q)$\!\!\!]] For every $k\in \bb Z$,
$\displaystyle \Lambda (\mz M_{q}(k),\mz M_{q}(k\pm
1))-S(\mz M_{q}(k)) \ge \mf h_{q}.$

\end{enumerate}

Now, we turn to the properties associated to the Markov chains
constructed through the recursive procedure.

\begin{enumerate}

\item[[$\mc P_7(q)$\!\!\!]] For every $k$, $l \in \bb Z$,
$R_q(\mz M_{q}(k),\mz M_{q}(l))>0$ if, and only if, $l = k \pm 1$
and
$\displaystyle \Lambda (\mz M_{q}(k), \mz M_{q}(k\pm 1))-S(\mz
M_{q}(k)) = \mf h_{q}$.

\end{enumerate}

\begin{enumerate}

\item[[$\mc P_8(q)$\!\!\!]] Let $a=\pm 1$. If
$R_q(\mz M_{q}(k),\mz M_{q}(k+a))>0$ for some $k\in \bb Z$, then
\begin{equation*}
S(\mz M_{q}(k))\geq S(\mz M_{q}(k+a)).
\end{equation*}
Moreover, if $\widehat{R}_q(\mz M_{q}(k+a),\mz M_{q}(k)) = 0$, then the
previous inequality is strict.

\item[[$\mc P_9(q)$\!\!\!]]
The jump rates are $\mf u_q$-periodic:
$R_q(\mz M_{q}(k+ \mf u_q),\mz M_{q}(k+ \mf u_q \pm 1)) = R_q(\mz
M_{q}(k),\mz M_{q}(k\pm 1))$ for all $k\in \bb Z$.

\end{enumerate}

Let $\mz K_q$ be the set of $\bb X_q$-closed irreducible classes, and
denote by {\color{blue}$\mf n_q$} the number of elements in $\mz K_q$,
modulo \eqref{equiv_class}. Let $\mf n_0 = N$. Then,
\begin{enumerate}

\item[[$\mc P_{10}(q)$\!\!\!]] $\mf n_{q} < \mf n_{q-1}$. 

\end{enumerate}

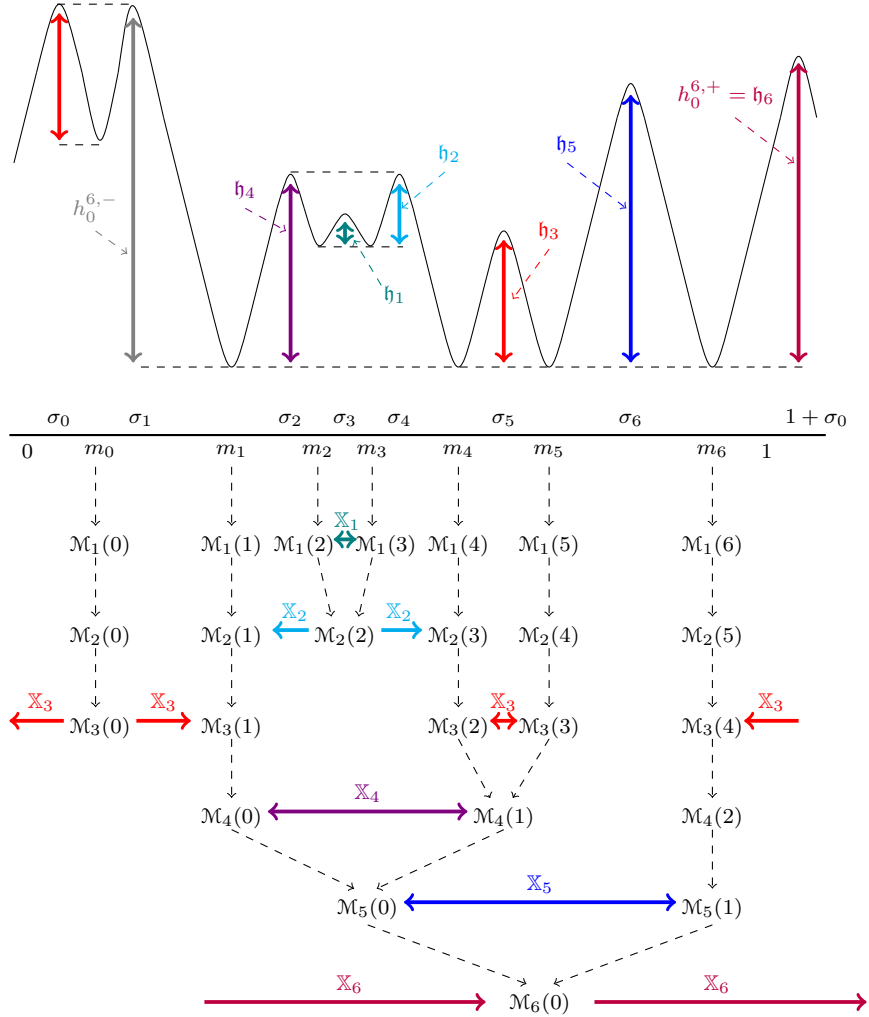
\begin{figure}
\centering
\begin{tikzpicture}[scale=0.6]
\draw[rounded corners](-0.3,2)--(0.2,4);
\draw [rounded corners] (0.2,4) .. controls (0.7,6) .. (1.2,4);
\draw[rounded corners] (1.2,4) .. controls (1.6,2) .. (2,4);
\draw[rounded corners] (2,4) .. controls (2.3,6.1) .. (3,3);  
\draw[rounded corners] (3,3) .. controls (3.5, 1) .. (4, -1);
\draw[rounded corners] (4,-1) .. controls (4.5, -3) .. (5,-1);
\draw[rounded corners] (5,-1) -- (5.5,1);
\draw[rounded corners] (5.5,1) .. controls (5.8,2) .. (6.1, 1);
\draw[rounded corners] (6.1, 1) .. controls (6.4, 0) .. (6.7, 0.5);
\draw[rounded corners] (6.7, 0.5) .. controls (7, 1) .. (7.3, 0.5);
\draw[rounded corners] (7.3, 0.5) .. controls (7.6, 0) .. (7.9, 1);
\draw[rounded corners] (7.9, 1) .. controls (8.2, 2) .. (8.5, 1);
\draw[rounded corners] (8.5, 1) -- (9, -1);
\draw[rounded corners] (9,-1) .. controls (9.5, -3) .. (10, -1);
\draw[rounded corners] (10,-1) .. controls (10.5, 1) .. (11, -1);
\draw[rounded corners] (11,-1) .. controls (11.5,-3) .. (12, -1);
\draw[rounded corners] (12,-1) .. controls (12.5, 1) .. (13, 3);
\draw[rounded corners] (13,3) .. controls (13.3, 4) .. (13.6, 3);
\draw[rounded corners] (13.6,3) .. controls (14.1,1) .. (14.6,-1);
\draw[rounded corners] (14.6,-1) .. controls (15.1,-3).. (15.6,-1);
\draw[rounded corners] (15.6,-1) .. controls (16.1,1) .. (16.6,3);
\draw[rounded corners](16.6,3) .. controls (17,4.8) .. (17.4,3);

\draw[to-to, line width=0.5mm, blue](13.3,3.5)--(13.3,-2.4);
\draw[to-to, line width = 0.5mm, teal](7,0.7)--(7,0.2);
\draw[to-to, line width = 0.5mm, cyan](8.2,1.54)--(8.2,0.2);
\draw[to-to, line width=0.5mm, violet](5.8,1.54)--(5.8,-2.4);
\draw[to-to, line width=0.5mm, red](10.5,0.3)--(10.5,-2.4);
\draw[to-to, line width = 0.5mm, gray](2.33,5.2)--(2.33,-2.4);
\draw[to-to, line width=0.5mm, purple](17,4.2) -- (17,-2.4);

\draw[to-to,line width=0.5mm, red](0.7,5.3)--(0.7,2.5);

\draw[-to, dashed, gray](1.5,0.5)--(2.2,0); \fill(1.5,0.5)node[above,
gray, font=\footnotesize]{$h_0^{6,-}$}; 

\draw[-to, dashed, violet](4.8,1)--(5.7,0.5);
\fill(4.8,1)node[above,violet,font=\footnotesize]{$\mf h_4$};

\draw[-to, dashed, teal](7.8,-0.7)--(7.2,0.3);
\fill(8.1,-0.5)node[below, teal, font=\footnotesize]{$\mf h_1$};

\draw[-to, dashed, cyan](9.2,1.8)--(8.25,1);
\fill(9.3,1.8)node[above, cyan, font=\footnotesize]{$\mf h_2$};

\draw[-to, dashed, red](11.5,0.1)--(10.7,-1);
\fill(11.5,0.1)node[above,red,font=\footnotesize]{$\mf h_3$};

\draw[-to, dashed, blue](11.9,2)--(13.2,1);
\fill(11.9,2)node[above, blue, font=\footnotesize]{$\mf h_5$};

\draw[-to, dashed, purple](15.6,3)--(16.9,2);
\fill(15.4,3)node[above,purple,font=\footnotesize]{$h_0^{6,+} = \mf h_6$};

\draw[dashed,black](0.65,5.5)--(2.4,5.5);
\draw[dashed, black](5.8,1.8)--(8.25,1.8);
\draw[dashed, black](2.5,-2.5)--(17.1,-2.5);
\draw[dashed, black](0.7,2.4)--(1.7,2.4);
\draw[dashed, black](6.4,0.15)--(8.28,0.15);

\draw[solid, thick, black](-0.4,-4)--(17.6,-4);
\fill(0,-4)node[below,font=\footnotesize]{$0$};
\fill(16.3,-4)node[below,font=\footnotesize]{$1$};

\fill(6.4,-4)node[below, font=\footnotesize]{$m_2$};
\fill(7.6,-4)node[below, font=\footnotesize]{$m_3$};
\fill(9.5,-4)node[below, font=\footnotesize]{$m_4$};
\fill(11.5,-4)node[below, font=\footnotesize]{$m_5$};
\fill(15.1,-4)node[below, font=\footnotesize]{$m_6$};
\fill(4.5,-4)node[below, font=\footnotesize]{$m_{1}$};
\fill(1.6,-4)node[below, font=\footnotesize]{$m_{0}$};

\fill(0.7,-4)node[above, font=\footnotesize]{$\sigma_{0}$};
\fill(2.5,-4)node[above, font=\footnotesize]{$\sigma_{1}$};
\fill(5.8,-4)node[above, font=\footnotesize]{$\sigma_2$};
\fill(7,-4)node[above, font=\footnotesize]{$\sigma_3$};
\fill(8.2,-4)node[above, font=\footnotesize]{$\sigma_4$};
\fill(10.5,-4)node[above, font=\footnotesize]{$\sigma_5$};
\fill(13.3,-4)node[above, font=\footnotesize]{$\sigma_6$};
\fill(17.4,-4)node[above, font=\footnotesize]{$1+\sigma_0$};

\draw[-to, dashed](6.4,-4.7)--(6.4,-6);
\fill(6.1,-6)node[below, font=\footnotesize]{$\mz M_1(2)$};
\draw[-to, dashed](7.6,-4.7)--(7.6,-6);
\fill(7.9,-6)node[below, font=\footnotesize]{$\mz M_1(3)$};

\draw[to-to, teal, line width=0.5mm] (6.75,-6.3) -- (7.25,-6.3); 
\fill(7.05,-6.3)node[above, teal, font=\footnotesize]{$\bb X_1$};

\draw[-to, dashed](1.5,-4.7)--(1.5,-6);
\fill(1.6,-6)node[below, font=\footnotesize]{$\mz M_1(0)$};
\draw[-to, dashed](4.5,-4.7)--(4.5,-6);
\fill(4.5,-6)node[below, font=\footnotesize]{$\mz M_1(1)$};
\draw[-to, dashed](9.5,-4.7)--(9.5,-6);
\fill(9.5,-6)node[below, font=\footnotesize]{$\mz M_1(4)$};
\draw[-to, dashed](11.5,-4.7)--(11.5,-6);
\fill(11.5,-6)node[below, font=\footnotesize]{$\mz M_1(5)$};
\draw[-to, dashed](15.1,-4.7)--(15.1,-6);
\fill(15.1,-6)node[below, font=\footnotesize]{$\mz M_1(6)$};

\draw[-to, dashed](6.4,-6.7)--(6.7,-8);
\draw[-to, dashed](7.6,-6.7)--(7.3,-8);
\fill(7,-8)node[below, font=\footnotesize]{$\mz M_2(2)$};

\draw[-to, cyan, line width=0.5mm](7.8,-8.3)--(8.7,-8.3);
\fill(8.2,-8.3)node[above, cyan, font=\footnotesize]{$\bb X_2$};
\draw[-to, cyan, line width=0.5mm](6.2,-8.3)--(5.4,-8.3);
\fill(5.9,-8.3)node[above, cyan, font=\footnotesize]{$\bb X_2$};

\draw[-to, dashed](1.5,-6.7)--(1.5,-8);
\fill(1.6,-8)node[below, font=\footnotesize]{$\mz M_2(0)$};
\draw[-to, dashed](4.5,-6.7)--(4.5,-8);
\fill(4.5,-8)node[below, font=\footnotesize]{$\mz M_2(1)$};
\draw[-to, dashed](9.5,-6.7)--(9.5,-8);
\fill(9.5,-8)node[below, font=\footnotesize]{$\mz M_2(3)$};
\draw[-to, dashed](11.5,-6.7)--(11.5,-8);
\fill(11.5,-8)node[below, font=\footnotesize]{$\mz M_2(4)$};
\draw[-to, dashed](15.1,-6.7)--(15.1,-8);
\fill(15.1,-8)node[below, font=\footnotesize]{$\mz M_2(5)$};

\draw[-to, dashed](1.5,-8.7)--(1.5,-10);
\fill(1.6,-10)node[below, font=\footnotesize]{$\mz M_3(0)$};

\draw[-to, red, line width=0.5mm](2.4,-10.3)--(3.6,-10.3);
\fill(3,-10.3)node[above, red, font=\footnotesize]{$\bb X_3$};

\draw[-to, red, line width=0.5mm](0.8,-10.3)--(-0.4,-10.3);
\fill(0.3,-10.3)node[above, red, font=\footnotesize]{$\bb X_3$};

\draw[-to, dashed](4.5,-8.7)--(4.5,-10);
\fill(4.5,-10)node[below, font=\footnotesize]{$\mz M_3(1)$};
\draw[-to, dashed](9.5,-8.7)--(9.5,-10);
\fill(9.5,-10)node[below,font=\footnotesize]{$\mz M_3(2)$};

\draw[to-to, red, line width=0.5mm](10.2,-10.3)--(10.8,-10.3);
\fill(10.5,-10.3)node[above, red, font=\footnotesize]{$\bb X_3$};

\draw[-to, dashed](11.5, -8.7)--(11.5,-10);
\fill(11.5,-10)node[below, font=\footnotesize]{$\mz M_3(3)$};
\draw[-to, dashed](15.1,-8.7)--(15.1,-10);
\fill(15.1,-10)node[below, font=\footnotesize]{$\mz M_3(4)$};

\draw[-to, red, line width=0.5mm](17,-10.3)--(15.8,-10.3);
\fill(16.4,-10.3)node[above, red, font=\footnotesize]{$\bb X_3$};

\draw[-to, dashed](4.5,-10.7)--(4.5,-12);
\fill(4.5,-12)node[below, font=\footnotesize]{$\mz M_4(0)$};
\draw[-to, dashed](9.5,-10.7)--(10.2,-12);
\draw[-to, dashed](11.5,-10.7)--(10.7,-12);
\fill(10.5,-12)node[below, font=\footnotesize]{$\mz M_4(1)$};

\draw[to-to, violet, line width=0.5mm](5.3,-12.3)--(9.7,-12.3);
\fill(7.5,-12.3)node[above, violet, font=\footnotesize]{$\bb X_4$};

\draw[-to, dashed](15.1,-10.7)--(15.1,-12);
\fill(15.1,-12)node[below, font=\footnotesize]{$\mz M_4(2)$};

\draw[-to, dashed](4.5,-12.7)--(7.2,-14);
\draw[-to,dashed](10.5,-12.7)--(7.7,-14);
\fill(7.5,-14)node[below,font=\footnotesize]{$\mz M_5(0)$};
\draw[-to, dashed](15.1,-12.7)--(15.1,-14);
\fill(15.1,-14)node[below,font=\footnotesize]{$\mz M_5(1)$};
\draw[to-to, blue, line width=0.5mm](8.3,-14.3)--(14.3,-14.3);
\fill(11.3,-14.3)node[above, blue, font=\footnotesize]{$\bb X_5$};

\draw[-to,dashed](7.5,-14.8)--(11,-16.1);
\draw[-to,dashed](15.1,-14.8)--(11.6,-16.1);
\fill(11.3,-16.1)node[below, font=\footnotesize]{$\mz M_6(0)$};

\draw[-to, purple, line width=0.5mm](12.5,-16.5)--(18.5,-16.5);
\draw[-to, purple, line width=0.5mm](3.9,-16.5)--(10.1,-16.5);
\fill(7.1,-16.5)node[above, purple, font=\footnotesize]{$\bb X_6$};
\fill(15.2,-16.5)node[above,purple,font=\footnotesize]{$\bb X_6$};

\end{tikzpicture}
\caption{Here we represent an example of a metastable structure with 6
levels on $\bb R$, but 5 projected levels on $\bb T$. At level 1, the states $\mz M_1(k)$, $0\leq k\leq 6$,
are given by $\mz M_1(k) = \{m_k\}$. These comprise all of the states contained in the interval $[0,1)$. The sets $\mz M_1(2)$ and
$\mz M_1(3)$ belong to the same $\bb X_1(\cdot)-$recurrent
class, while the remaining states are absorbing. For
level 2, $\mz M_2(0) = \{m_{0}\}$, $\mz M_2(1)=\{m_{1}\}$,
$\mz M_2(2) = \{m_2, m_3\}$, and $\mz M_2(k) = \{m_{k+1}\}$, for
$k\in \{3,4,5\}$. At this level, $\mz M_2(2)$ is a
$\bb X_2(\cdot)-$transient state, while the remaining states
are absorbing. The
states of level 3 contained in the interval $[0,1)$ are
$\mz M_3(0)=\{m_{0}\}$, $\mz M_3(1)=\{m_{1}\}$,
$\mz M_3(2) = \{m_4\}$, $\mz M_3(3)=\{m_5\}$ and $\mz
M_3(4)=\{m_6\}$. At this level, $\mz M_3(0)$ is
$\bb X_3(\cdot)-$transient, the states $\mz M_3(1)$ and
$\mz M_3(4)$ are absorbing, and $\{\mz M_3(2), \mz M_3(3)\}$ is a
closed irreducible class. Observe that, the projected chain $\widehat{\bb X}_3$ can jump from $\mc M_3(0) = \Pi(\mz M_3(0))$ to $\mc M_3(4) = \Pi(\mz M_3(4))$. Turning to level 4, we see the first relabeling (recall
property $\mc P_2$ in Appendix \ref{secA1}). Here $\mss j_4 = 1$ and we find the states
$\mz M_4(0) = \{m_{1}\}$, $\mz M_4(1) = \{m_4, m_5\}$ and
$\mz M_4(2) = \{m_6\}$. The state $\mz M_4(2)$ is absorbing, while the set
$\{\mz M_4(0), \mz M_4(1)\}$ is a closed
irreducible class. At level 5,
$\mz M_5(0) = \{m_{1},\; m_4,\; m_5\}$ and $\mz M_5(1) = \{m_6\}$
compose the unique close irreducible class inside the interval $[0,1)$. By periodicity, $\bb X_5$ has infinite closed irreducible classes and all of them are equivalent. Therefore, $\mf n_5 = 1$ and this is the last level for the projected metastable structure on $\bb T$, $\widehat{\mf q}=5$, while there is one more level for the structure in $\bb R$, $\mf q = 6$. The state space of the Markov chain $\bb X_6$ is $\ms S_6 = \{\ms M_6(0)+k: k\in \bb Z\}$, the chain is transient and can only jump to the right.}
\label{fig-1f}
\end{figure}

\subsection*{Further properties of the metastable structure}

Recall the definition of the set $\mz W^{(p)}_{k,k+1}$,
$\sigma_{k, k+1}^{p,-}$, $\sigma_{k, k+1}^{p,+}$, $k\in \bb Z$, given
in \eqref{max_br} and \eqref{rl_max:R}, respectively.

The following result, \cite[Lemma A5]{lm}, states that the energetic
barriers at $\sigma^{p,+}_{k-1,k}$, $\sigma^{p,-}_{k,k+1}$ are higher
than the ones in the interior of the interval
$(\sigma^{p,+}_{k-1,k}\,,\, \sigma^{p,-}_{k,k+1})$.  This result
implies that the diffusion $X_\epsilon (t)$, in time-scales of smaller
order than $\theta^{(p)}_\epsilon$, cannot escape the interval
$(\sigma^{p,+}_{k-1,k}\,,\, \sigma^{p,-}_{k,k+1})$. 

\begin{lemma}
\label{l15}
For each $2\le p \le \mf q$, $k\in \bb Z$,
\begin{gather*}
S(\sigma^{p,+}_{k-1,k}) \,>\, {\color{blue} S^-_{p,k}}
\,:=\, \max \{ \, S(x) :
x\in (\sigma^{p,+}_{k-1,k} ,  m^+_{p,k}] \cap \mz W\, \big\}\;, 
\\
S(\sigma^{p,-}_{k,k+1})  
\,>\, {\color{blue} S^+_{p,k}} \,:=\,
\max \{ S(x) : x\in  [m^-_{p,k} , \sigma^{p,-}_{k,k+1} ) \cap \mz W\} \;.
\end{gather*}
\end{lemma}

Let $\mss M_p(k)$ be the set of local
minima of $S(\cdot)$ that belong to the interval
$[\sigma^{p,+}_{k-1,k} , \sigma^{p,-}_{k,k+1}]$, but are not contained in
$\mz M_{p}(k)$ (see Figure \ref{fig-f3}):
\begin{equation*}
{\color{blue} \mss M_p(k) } \,:=\,
\big(\, \mz M \cap [\sigma^{p,+}_{k-1,k} ,
\sigma^{p,-}_{k,k+1}] \,\big)  \setminus \mz M_{p}(k) \;.
\end{equation*}

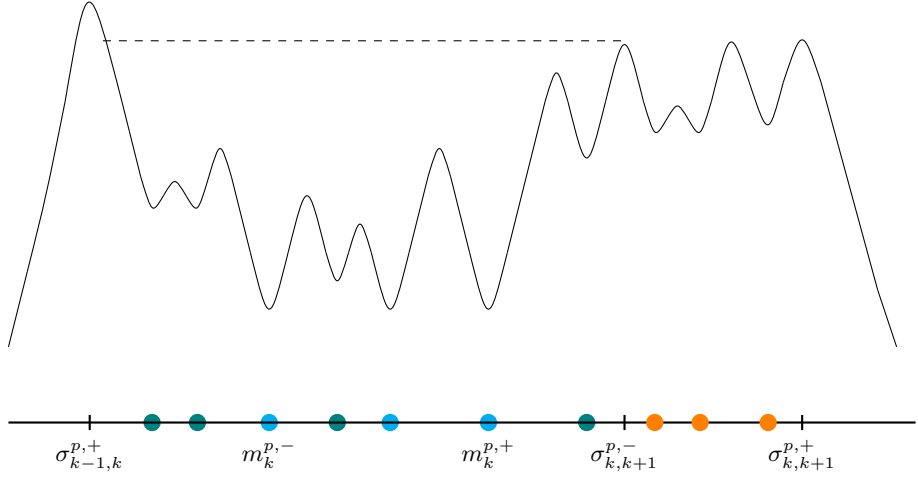
\begin{figure}
\centering
\begin{tikzpicture}[scale=0.5]
\draw [rounded corners] (0,-3.5) -- (0.5,-1.5);
\draw[rounded corners] (0.5,-1.5) .. controls (1,0.5) .. (1.5,3);
\draw[rounded corners] (1.5,3) .. controls (2.1,6.5) .. (3,3);  
\draw[rounded corners] (3,3) -- (3.2, 2.2);
\draw[rounded corners] (3.2,2.2) -- (3.5, 1);
\draw[rounded corners] (3.5, 1) .. controls (3.8, 0) .. (4.1, 0.5);
\draw[rounded corners] (4.1, 0.5) .. controls (4.4, 1) .. (4.7, 0.5);
\draw[rounded corners] (4.7, 0.5) .. controls (5, 0) .. (5.3, 1);
\draw[rounded corners] (5.3, 1) .. controls (5.6, 2) .. (5.9, 1);
\draw[rounded corners] (5.9, 1) -- (6.4, -1);
\draw[rounded corners] (6.4,-1) .. controls (6.9, -3) .. (7.4, -1);
\draw[rounded corners] (7.4,-1) .. controls (7.9, 1) .. (8.4, -1);
\draw[rounded corners](8.4,-1) .. controls (8.7,-2)..(9,-1);
\draw[rounded corners](9,-1).. controls (9.3,0)..(9.6,-1);
\draw[rounded corners] (9.6,-1) .. controls (10.1,-3) .. (10.6, -1);
\draw[rounded corners](10.6,-1)--(11.1,1);
\draw[rounded corners](11.1,1)..controls (11.4,2)..(11.7,1);
\draw[rounded corners](11.7,1)--(12.2,-1);
\draw[rounded corners](12.2,-1).. controls (12.7,-3)..(13.2,-1);
\draw[rounded corners](13.2,-1)..controls (13.7,1)..(14.2,3);
\draw[rounded corners](14.2,3)..controls (14.5,4)..(14.8,3);
\draw[rounded corners](14.8,3)..controls (15.3,1)..(15.8,3);
\draw[rounded corners](15.8,3)..controls (16.3,5)..(16.8,3);
\draw[rounded corners](16.8,3)..controls (17.1,2)..(17.4,2.5);
\draw[rounded corners](17.4,2.5)..controls(17.7,3)..(18,2.5);
\draw[rounded corners](18,2.5)..controls(18.3,2)..(18.6,3);
\draw[rounded corners](18.6,3)..controls(19.1,5)..(19.6,3.5);
\draw[rounded corners](19.6,3.5)..controls (20.1,2)..(20.5,3.5);
\draw[rounded corners](20.5,3.5)..controls (21,5)..(21.5,3.5);
\draw[rounded corners](21.5,3.5)--(23,-2);
\draw[rounded corners](23,-2)--(23.5,-3.5);

\draw[dashed, black](2.5,4.6)-- (16.3,4.6);

\fill(16.3,-5.7)node[below, font=\small]{$\sigma^{p,-}_{k,k+1}$};
\fill(21,-5.7)node[below, font=\small]{$\sigma^{p,+}_{k,k+1}$};
\draw[thick] (16.3,-5.7) -- (16.3,-5.3);
\draw[thick] (21,-5.7) -- (21,-5.3);

\draw[thick,fill,cyan](6.9,-5.5)circle(.2);
\draw[thick,fill,cyan](12.7,-5.5) circle(.2);
\draw[thick,fill,cyan](10.1,-5.5)circle(.2);

\fill(6.9,-5.7)node[below, font=\small]{$m^{p,-}_k$};
\fill(12.7,-5.7)node[below, font=\small]{$m^{p,+}_k$};

\draw[thick, fill, teal](3.8,-5.5) circle(.2);
\draw[thick, fill, teal](5,-5.5) circle(.2);
\draw[thick, fill, teal](8.7,-5.5)circle(.2);
\draw[thick, fill, teal](15.3,-5.5) circle(.2);

\draw[solid, thick, black](0,-5.5)--(24,-5.5);

\draw (2.15,-5.7) node[below, font=\small]{$\sigma^{p,+}_{k-1,k}$};
\draw[thick] (2.15,-5.7) -- (2.15,-5.3);

\draw[thick, fill, orange](17.1,-5.5)circle(.2);
\draw[thick, fill, orange](18.3,-5.5)circle(.2);
\draw[thick, fill, orange](20.1,-5.5)circle(.2);

\end{tikzpicture}
\caption{Fix $p>1$ and $k\in \bb Z$. The state $\mz M_p(k)$ is
represented by the cyan circles, while the set $\mss M_p(k)$ is
composed by the green circles.  As stated in Proposition \ref{l13},
the local minima in $M_p(k)$ have higher energy than the ones
contained in $\mz M_p(k)$, which are all of the same height. With
respect to the escape barriers of the state $\mz M_p(k)$, we have only
one global maxima separating $\mz M_p(k)$ and $\mz M_p(k-1)$, while
there are three between $\mz M_p(k)$ and $\mz M_p(k+1)$. Here,
$\sigma_{k-1,k}^{p,+}=\sigma_{k-1,k}^{p,-}$, while
$\sigma_{k,k+1}^{p,-}\,<\,\sigma_{k,k+1}^{p,+}$. The orange circles
represent the local minima in between the left-most and right-most
maxima separating the states $\mz M_p(k)$ and $\mz M_p(k+1)$. As
stated in Lemma \ref{l27}, these minima are higher than the ones
contained in $\mz M_p(k)$.  }

\label{fig-f3}
\end{figure}

The following proposition, \cite[Proposition A2]{lm}, states that
every local minima in $\mss M_p(k)$ is strictly \textit{higher} than the
ones contained in $\mz M_p(k)$. See Figure \ref{fig-f3}.

\begin{proposition}
\label{l13}
For all $k\in \bb Z$,
\begin{equation*}
\min \{ S(m) : m\in \mss M_p(k)\,\} \,>\, S(\mz M_{p}(k))\;.
\end{equation*}
\end{proposition}

The next result, \cite[Lemma A11]{lm}, states that the local minima in
between the leftmost and rightmost absolute maxima separating any two
states, say $\mz M_p(k)$ and $\mz M_p(k+1)$, are higher than the ones
belonging $\mz M_p(k) \cup \mz M_p(k+1)$. See Figure \ref{fig-f3}. The
case $p=1$ is excluded because
$\sigma_{k,k+1}^{1,-}=\sigma_{k,k+1}^{1,+}$, for all $k\in \bb Z$.

\begin{lemma}
\label{l27}
Fix $2\le p\le \mf q$, $k\in \bb Z$, and suppose that
$\sigma^{p,-}_{k,k+1} <\sigma^{p,+}_{k,k+1}$. Then,
\begin{equation*}
\min \big\{ S(x) : x \in
[\sigma^{p,-}_{k,k+1} \,,\, \sigma^{p,+}_{k,k+1} ] \,\big\} \, >\,
\max\{S(\mc M_{p}(k)),S(\mc M_p(k+1))\} \;.
\end{equation*}
\end{lemma}

We conclude this subsection with a result regarding the depth of the
local minima belonging to the same $\bb X_p-$equivalent class, for any
intermediate layer $1\leq p<\mf q$. Fix a $\bb X_p-$equivalent class
$\mf D\subset \mz S_p$. Recall the definition of the set
$\mz W_{\mf D}^{(p)}$ given in \eqref{max_equiv}, and the definitions
of the local maxima $\sigma_{\mf D}^-, \sigma_{\mf D}^{+}$ given in
\eqref{39}. Let $\mz M_p(\mf D)\subset \mz M_p$ be the set of local
minima of level $p$ belonging to the equivalence class $\mf D$
\begin{equation*}
{\color{blue}\mz M_p(\mf D)}:= \bigcup_{\mz M\in \mf D}\mz M. 
\end{equation*}
Note that, by properties $\mc P_4$ and $\mc P_8$ in the previous
subsection, all the local minima in $\mz M_p(\mf D)$ have the same
depth,
\begin{equation}
\label{depth_equiv_class}
S(m) = S(m')\,,\quad m,m'\in \mz M_p(\mf D)\,.
\end{equation}
Let $\mss M_p(\mf D)$ be set of local minima of $S(\cdot)$ that belong
to the interval $[\sigma_{\mf D}^-,\sigma_{\mf D}^+]$, but are not
contained in $\mz M_p(\mf D)$,
\begin{equation*}
{\color{blue}\mss M_p(\mf D)} := \big(\, \mz M
\cap [\sigma_{\mf D}^-,\sigma_{\mf D}^+] \,\big)
\setminus \mz M_{p}(\mf D) \;.
\end{equation*}
The next result states that every local minima in $\mss M_p(\mf D)$ is
strictly higher than the ones in $\mz M_p(\mf D)$.

\begin{proposition}
\label{depth_equivalent}
Fix $1\leq p<\mf q$. For every $\bb X_p(\cdot)-$equivalent
class $\mf D\subset \mz S_p$,
 \begin{equation*}
\min \{ S(m) : m\in \mss M_p(\mf D)\,\} \,>\, S(\mz M_{p}(\mf D))\;.
\end{equation*}
\end{proposition}

\begin{proof}
Fix $m\in \mss M_p(\mf D)$. Then, either
$m\in \mz M\cap [\sigma_{k,k+1}^{p,-}, \sigma_{k,k+1}^{p,+}]$ or
$m\in \mss M_p(j)$, for some $k$, $j\in \bb Z$. In the first case, the
result follows by Lemma \ref{l27} and \eqref{depth_equiv_class}. In
the second case, it follows by Proposition \ref{l13} and
\eqref{depth_equiv_class}.
\end{proof}

\section{Metastable states}
\label{secA2}

In this section we introduce a useful formula for the metastable
states defined in \eqref{21}. 
Recall that
$S_p$ denotes the index set of the state space $\mc S_p$,
$1\leq p\leq \mf q$. 
This formula seeks to relate the metastable states
$\mu_{\mc M_{p+1}(k)}^{p+1}$, $1\leq p<\mf q$, $k\in S_{p+1}$, with
the stationary states of $\widehat{\bb X}_{p}(\cdot)$. This requires
some preparation. Fix $1\leq p\leq \mf q$. Let $\mf D\subset \mc S_p$
be a $\widehat{\bb X}_p(\cdot)-$equivalent class with two or more
elements. Recall from Section \ref{sec5} the definition of the
generator $\widehat{\bb L}_{p, \mf D}$, introduced in \eqref{37}, and
that we denote by $\nu_{\mf D}^p \in \mss P(\mf D)$ the stationary
state of the Markov chain induced by $\widehat{\bb L}_{p,\mf D}$.

\begin{lemma}
\label{pro_red_stat}
Assume that the Markov chain $\widehat{\bb X}_p(\cdot)$ is
reducible. Then, the generator $\widehat{\bb L}_{p,\mf D}$ induces a
reversible process and the stationary state $\nu_{\mf D}^p$ is given by
\begin{equation*}
\nu^{p}_{\mf D}(\mc M_{p}(j)) =
\frac{1}{Z_p(\mf D)}\, \pi_{p}(j)
\,,\quad \mc M_p(j)\in \mf D\,, \quad
\text{where}\;\;
Z_p(\mf D) \,:=\, \sum_{i\in S_{p}: \mc M_{p}(i)\in \mf D}\pi_{p}(i) \,,
\end{equation*}
and $\pi_p(j)$ is given by \eqref{p_weights}.
\end{lemma}

\begin{proof}
Let $\nu^*$ be the probability measure defined by
\begin{equation*}
\nu^*(\mc M_p(k)) \,:=\,
\frac{1}{Z_p(\mf D)}\, \pi_p(k) \,,
\quad \mc M_p(k)\in \mf D.
\end{equation*}
We claim that $\nu^*$ satisfies the detailed balance conditions. Fix
$k\in S_p$ such that $\mc M_p(k), \mc M_p(k+1)\in \mf D$. Since
$\mf D$ is an $\widehat{\bb X}_p$-equivalent class, and
$\widehat{\bb X}_p(\cdot)$ is a reducible Markov chain jumping only to
nearest neighbor sets, $\widehat{R}_p(\mc M_p(k),\mc M_p(k+1))>0$ and
$\widehat{R}_p(\mc M_p(k+1),\mc M_p(k))>0$. Thus, by \eqref{50b},
$h^{p,+}_k=h^{p,-}_{k+1}= \mf h_p$, and 
\begin{align*}
& \nu^{*}(\mc M_p(k))\, \widehat{R}_{p}(\mc M_p(k),\mc M_p(k+1))
= \frac{1}{Z_p(\mf D)}\, \pi_p(k)  \,
\frac{1}{\pi_{p}(k)\, \sigma_p(k,k+1)}
\\
& \quad = \frac{1}{Z_p(\mf D)}\, \pi_p(k+1)  \,
\frac{1}{\pi_{p}(k+1)\, \sigma_p(k,k+1)}
= \nu^{*}(\mc M_p(k+1)) \, \widehat{R}_{p}(\mc M_p(k+1),\mc M_p(k))\,, 
\end{align*}
as claimed.
\end{proof}

Fix $1\leq p < \mf q$ and $k\in S_{p+1}$. Denote by
${\color{blue}\widehat{\mf R}_k^p}$ the
$\widehat{\bb X}_p(\cdot)-$closed irreducible class generating the
state $\mc M_{p+1}(k)$ in the sense of \eqref{24}. Let
$\nu_{\mc M_{p+1}(k)}^p\in \mss P(\widehat{\mf R}_k^p)$,
$k\in S_{p+1}$, be the probability measure defined by
\begin{equation}
\label{stat_recur}
{\color{blue}\nu_{\mc M_{p+1}(k)}^p(\mc M)}
\,=\, \nu_{\widehat{\mf R}_k^p}^p(\mc M)\,,
\quad \mc M\in \widehat{\mf R}_k^p.
\end{equation}
The next result provides a recursive formula for the measures
$\mu^{p}_{\mc M_p(r)}(\cdot)$, $0\le r <\mf u_p$, introduced in
\eqref{21}.

\begin{proposition}
\label{lifted_red_stat}
Fix $1\leq p<\mf q$ and $k\in S_{p+1}$. Then,
\begin{equation*}
\mu_{\mc M_{p+1}(k)}^{p+1}(\cdot)
\,=\, \sum_{r\in S_p: \mc M_p(r)\subset \mc M_{p+1}(k)}
\nu^{p}_{\mc M_{p+1}(k)}(\mc M_p(r)) \, \mu^{p}_{\mc M_p(r)}(\cdot)\,.
\end{equation*}
\end{proposition}

\begin{proof}
The case $p=1$ follows from the definition \eqref{21} of the measures
$\mu_{\mc M_{2}(k)}^{2}(\cdot)$, $\mu_{\mc M_{1}(r)}^{1}(\cdot)$, and
Lemma \ref{pro_red_stat}.

Fix $2\le p<\mf q$, $k\in S_{p+1}$. By \eqref{21} and \eqref{24}, 
\begin{align*}
\mu_{\mc M_{p+1}(k)}^{p+1} 
= \sum_{l\in S_p: \mc M_p(l)\subset \mc M_{p+1}(k)}
\frac{\pi_{p}(l)}{\pi_{p+1}(k)}
\sum_{m \in \mc M_p(l)}
\frac{\pi_1(\{m\})}{\pi_p(l)}\,\delta_{m}
= \sum_{l\in S_p: \mc M_p(l)\subset \mc M_{p+1}(k)}
\frac{\pi_{p}(l)}{\pi_{p+1}(k)} \, \mu_{\mc M_p(l)}^p\,.
\end{align*}
It remains to apply Lemma \ref{pro_red_stat} to complete the proof. 
\end{proof}

\section{Hitting times}
\label{secA3}

In this section, we present some estimates of the
$X_\epsilon(\cdot)$-hitting time of wells. The following result,
\cite[Lemma 4.3]{lm}, states that the set $\mc M_p$ is reached in the
time-scale $\epsilon^{-1}e^{\mf h_{p-1}/\epsilon}$.

\begin{lemma}
\label{hitting_1}
There exists a finite constant $C_0 = C_0(\mf q, \mtt a(\cdot))$ such
that
\begin{equation*}
\sup_{x\in \bb T}
\bb E_{x}^{\epsilon} [\, \tau (\mc M_p) \, ]
\,\le\, C_0 \, \epsilon^{-1}\, e^{\mf h_{p-1}/\epsilon}
\end{equation*}
for all $1\le p\le \mf q$, provided $\mf h_0=0$.
\end{lemma}

Recall from the statement of Lemma \ref{s10} that
for each $x\in \bb T$, we denoted by $m_p^-(x)$, and $m_p^+(x)$, the
closest minima of level $p$ to the left, and right of $x$,
respectively.  Let {\color{blue}$m_p^{\pm,*}(x)$} be their lifted
versions in $\bb R$ satisfying $\Pi(m_p^{\pm,*}(x)) = m_p^{\pm}(x)$,
$m_p^{-,*}(x)\in (0,1)$, and
$m_p^{+,*}(x)\in (m_p^{-,*}(x), m_p^{-,*}(x)+1)$. By definition of
$m_p^\pm(x)$, if we denote by $\cb{x^{*}}$ the real number such that
$x^{*}\in (0,1]$ and $\Pi(x^{*}) = x$,
$m_p^{-,*}(x)<x^*< m_p^{+,*}(x)$.

The next result
provides an explicit expression for the probability that the diffusion
$X_\epsilon$, conditioned on starting at some arbitrary point
$x\in \bb T$, reaches $\mc M_p$ by hitting $m_p^-(x)$ first.

\begin{lemma}
\label{s06}
Fix $1\leq p\leq \mf q$. Then, for every $x\in \bb T$
\begin{equation*}
\bb P_{x}^\epsilon\big[\, \tau(m_p^-(x))< \tau(m_p^+(x)\,\big]
= \frac{\int_{m_p^{-,*}(x)}^{x^{*}}e^{S(y)/\epsilon}\,dy}
{\int_{m_p^{-,*}(x)}^{m_p^{+,*}(x)}e^{S(y)/\epsilon}\,dy}\,\cdot
\end{equation*}
In particular, for every $k\in S_p$,
$x^-\in (m_{p,k}^+, \sigma_{k,k+1}^{p,-})$, and
$x^+\in (\sigma_{k,k+1}^{p,+}, m_{p,k+1}^-)$, there exist positive
constant $C_{x^{\pm}}$, $h_{x^{\pm}}$ which depend on $x^{\pm}$, but
not on $\epsilon$, such that
\begin{equation*}
\bb P_{x^-}^\epsilon\big[\, \tau(m_{p,k+1}^-)<\tau(m_{p,k}^+)\,\big]
\leq \frac{C_{x^-}}{\sqrt{\epsilon}}e^{-h_{x^-}/\epsilon}\,,
\quad  \bb P_{x^+}^\epsilon \big[\, \tau(m_{p,k}^-)<\tau(m_{p,k+1}^+)\,\big]
\leq \frac{C_{x^{+}}}{\sqrt{\epsilon}}e^{- h_{x^+}/\epsilon}
\end{equation*}
for all $\epsilon>0$.
\end{lemma}

\begin{proof}
Fix $x\in \bb T$. For each $\epsilon>0$, let
$\xi_\epsilon \colon [m_p^-(x),m_p^+(x)]$ be the function given by
\begin{equation*}
\xi_\epsilon(z) = \frac{\int_{m_p^{-,*}(x)}^{z}
e^{S(y)/\epsilon}\,dy}{\int_{m_p^{-,*}(x)}^{m_p^{+,*}(x)}
e^{S(y)/\epsilon}\,dy}\,,\quad z\in [m_p^-(x),m_p^+(x)]\,.
\end{equation*}
By a direct computation, it can be verified that $\xi$ solves the
following equation
\begin{equation*}
\begin{cases}
(\ms L_{\epsilon}u)(z) = 0\,, \quad
\text{for}\ z\in (m_p^-(x),m_p^+(x)) \\
u(m_p^-(x)) = 0\,, \quad  u(m_p^+(x)) = 1\,.
\end{cases}
\end{equation*}
Therefore, the process $\xi_\epsilon(X_\epsilon(t\wedge \tau))$, 
$\tau:=\tau(m_p^-(x), m_p^+(x))$, is a bounded martingale with
respect to $\bb P_x^\epsilon$. By the optional stopping
theorem,
\begin{equation*}
\bb P_{x}^\epsilon\big[\, \tau(m_p^-(x))< \tau(m_p^+(x)) \,\big]
= \xi_\epsilon(x) =
\frac{\int_{m_p^{-,*}(x)}^{x^{*}}e^{S(y)/\epsilon}\,dy}
{\int_{m_p^{-,*}(x)}^{m_p^{+,*}(x)}e^{S(y)/\epsilon}\,dy}\,.    
\end{equation*}

We turn to the second statement. Fix $k\in S_p$ and
$x^-\in (m_{p,k}^+,\sigma_{k,k+1}^{p,-})$. Let $h_{0}$ be the maximum
of $S$ in the interval $[m_{p,k}^+, x]$. As $x<
\sigma_{k,k+1}^{p,-}$, $h_0 < S(\sigma_{k,k+1}^{p,-}) = h^{p,+}_k$.
By the first assertion of the lemma, estimating the numerator by the
maximum value of the function multiplied by the length of the interval
(which is bounded by $1$), and Laplace's method, 
\begin{equation*}
\bb P_{x^-}^\epsilon\big[\, \tau(m_{p,k+1}^-)<\tau(m_{p,k}^+)\,\big]
\leq \frac{C}{\sqrt{\epsilon}}e^{[h_0- h^{p,+}_k]/\epsilon}\,, 
\end{equation*}
for some positive constant $C>0$ which doesn't depend on
$\epsilon$. This establishes the first estimate of the second
assertion of the lemma. The other one is proved similarly.
\end{proof}

\section{Proof of Lemma \ref{s07}}
\label{secA4}

The proof is carried out in three main steps. We first express the sum
on the left-hand side of Lemma \ref{s07} in terms of the jump rates of
the tilted chain.

Recall from Section \ref{sec8} that $\omega \in \mss P(\mc S_{\mf q})$
is a strictly positive measure, and that
$\widehat{h} \colon \mc S_{\widehat{\mf q}}\to \bb R_{+}$ is a
strictly positive function, unique up to a multiplicative constant,
that satisfies \eqref{reflected_equiv} with $l_{\widehat{\mf q}}=1$
and $\mf D_1^{\widehat{\mf q}} = \mc S_{\widehat{\mf q}}$. Recall,
furthermore, that
$\widehat{\bb L}_{\widehat{\mf q},\widehat{h}} =: \cb{\widehat{\bb
L}_{\widehat{\mf q}}}$ represents the tilted generator given by
\eqref{51}. Denote by $\cb{\widehat{R}_{\widehat{\mf q}}}$ the
associated nearest-neighbor jump rates, and write
$\cb{\widehat{R}_{\widehat{\mf q}}(k,k\pm 1)}:=
\widehat{R}_{\widehat{\mf q}}(\mc M_{\widehat{\mf q}}(k), \mc
M_{\widehat{\mf q}}(k\pm 1))$,
$\cb{R_{\widehat{\mf q}}(k,k\pm 1)}:= R_{\widehat{\mf q}}(\mc
M_{\widehat{\mf q}}(k), \mc M_{\widehat{\mf q}}(k\pm 1))$, for
$k\in S_{\widehat{\mf q}}$. By \eqref{h_tilted_invariant_m}, $\omega$
is the stationary state of the Markov chain with generator
$\widehat{\bb L}_{\widehat{\mf q}}$.

We claim that for all $k\in S_{\widehat{\mf q}}$,
\begin{equation}
\label{52}
\begin{aligned}
& \frac{1}{\omega(\mc M_{\widehat{\mf q}}(k))} \,
\pi_{\widehat{\mf q}}(k)\, \widehat{h}(\mc M_{\widehat{\mf q}}(k))^2
\sum_{j\in \mf I_k} \frac{\sigma_{\widehat{\mf q}}(j,j+1)}
{\widehat{h}(\mc M_{\widehat{\mf q}}(j))\widehat{h}(\mc M_{\widehat{\mf q}}(j+1))}
\\
& \quad
\,=\,
\frac{1}{\omega(\mc M_{\widehat{\mf q}}(k))} \,
\sum_{j\in \mf I_k}
\frac{\widehat{R}_{\widehat{\mf q}}(k+1,k) \cdots \widehat{R}_{\widehat{\mf q}}(j,j-1)}
{\widehat{R}_{\widehat{\mf q}}(k,k+1) \cdots \widehat{R}_{\widehat{\mf q}}(j-1,j)}\,
\frac{1}{\widehat{R}_{\widehat{\mf q}}(j,j+1)}\,\cdot
\end{aligned}
\end{equation}
Mind that for $j=k$ the sum becomes $\widehat{R}_{\widehat{\mf q}}(k,k+1)^{-1}$.

The proof of this identity is elementary. Fix $k\in S_{\widehat{\mf q}}$. First
observe by \eqref{50b} that
\begin{equation*}
\frac{\pi_{\widehat{\mf q}}(k)}{\pi_{\widehat{\mf q}}(j)}\,=\,
\frac{R_{\widehat{\mf q}}(k+1,k) \cdots R_{\widehat{\mf q}}(j,j-1)}
{R_{\widehat{\mf q}}(k,k+1) \cdots R_{\widehat{\mf q}}(j-1,j)}
\quad \text{for}\;\; j\in \mf I_k\,, \;\; j\neq k \,.
\end{equation*}
To conclude, use \eqref{50b} again to write that
$\pi_{\widehat{\mf q}}(j) \, \sigma_{\widehat{\mf q}}(j,j+1) = R_{\widehat{\mf q}}(j,j+1)^{-1}$, and
express the jump rates $R_{\widehat{\mf q}}(i,i\pm 1)$ in terms of the tilted
ones $\widehat{R}_{\widehat{\mf q}}(i,i\pm 1)$.

As $\omega(\cdot)$ is a stationary state, and the Markov chain jumps
only to nearest-neighbor sites, for every $k\in S_{\widehat{\mf q}}$,
\begin{align*}
&\omega(\mc M_{\widehat{\mf q}}(k-1))\,
\widehat{R}_{\widehat{\mf q}} (k-1,k)\,+\,
\omega(\mc M_{\widehat{\mf q}}(k+1))\,
\widehat{R}_{\widehat{\mf q}} (k+1,k)
\\
&\quad \,=\,
\omega(\mc M_{\widehat{\mf q}}(k))\,
\widehat{R}_{\widehat{\mf q}} (k,k-1)\,+\,
\omega(\mc M_{\widehat{\mf q}}(k))\,
\widehat{R}_{\widehat{\mf q}} (k,k+1)\,.
\end{align*}
Therefore,
\begin{equation}
\label{53}
Z\, :=\, \omega(\mc M_{\widehat{\mf q}}(k))\, \widehat{R}_{\widehat{\mf q}} (k,k+1)
\,-\,
\omega(\mc M_{\widehat{\mf q}}(k+1))\,
\widehat{R}_{\widehat{\mf q}} (k+1,k)
\quad\text{is constant}\,.
\end{equation}
If $Z$ were equal to $0$ the process would be reversible, which, by
the hypothesis (NonRev), is not the case. As the Markov chain may not
jump to the left at certain sites, $Z>0$.

We claim that the sum in \eqref{52} is equal to $1/Z$. Denote it by
$\cb{\mtt F(k)}$.  The proof is by induction. Fix $j\in S_{\widehat{\mf q}}$,
and recall from \eqref{48} the definition of the set $\mf I_j$, and
that $\mf b_j$ represents the right-most element of $\mf I_j$.

By inspection, $\mtt F(\mf b_j) = 1/Z$. Assume that $\mtt F(i)= 1/Z$
for $k+1\le i\le \mf b_j$ and that $k\in \mf I_j$. By the definition
of $\mtt F(\cdot)$,
\begin{equation*}
\mtt F(k) \,=\, \frac{1}{\omega(\mc M_{\widehat{\mf q}}(k))}\,
\Big\{\, \frac{1}{\widehat{R}_{\widehat{\mf q}} (k,k+1)} \,+\,
\frac{\widehat{R}_{\widehat{\mf q}} (k+1,k)}{\widehat{R}_{\widehat{\mf q}} (k,k+1)}
\, \omega(\mc M_{\widehat{\mf q}}(k+1))\, \mtt F(k+1)\, \Big\}\,.
\end{equation*}
Thus $\mtt F(k) = 1/Z$ if, and only if,
\begin{equation*}
\omega(\mc M_{\widehat{\mf q}}(k))\, \widehat{R}_{\widehat{\mf q}} (k,k+1)\, (1/Z)
\,=\, 
\Big\{\, 1  \,+\,
\widehat{R}_{\widehat{\mf q}} (k+1,k) 
\, \omega(\mc M_{\widehat{\mf q}}(k+1))\, \mtt F(k+1)\, \Big\}\,.
\end{equation*}
By the induction hypothesis, $\mtt F(k+1) = 1/Z$. We may thus replace
$\mtt F(k+1)$ by
$[\omega(\mc M_{\widehat{\mf q}}(k))\, \widehat{R}_{\widehat{\mf q}} (k,k+1) \,-\,
\omega(\mc M_{\widehat{\mf q}}(k+1))\, \widehat{R}_{\widehat{\mf q}} (k+1,k)]^{-1}$. 
After this replacement a straightforward calculation shows that the
previous identity holds. This completes the proof of the Lemma
\ref{s07}.

\subsection*{Acknowledgments}

C. L. has been partially supported by FAPERJ CNE E-26/201.117/2021, by
CNPq, Brazil Bolsa de Produtividade em Pesquisa PQ 305779/2022-2, and
by the French National Research Agency (ANR) via Project
ANR-25-CE40-6875-01 (ANR DySLoS).

\end{document}